\documentclass[11pt,letter]{amsart}
\usepackage[
    top=2.5cm, bottom=2.4cm, inner=2.5cm, outer=2.5cm,
    marginparwidth=2cm %necessary to avoid breaking fixme notes.
    ]{geometry}

\usepackage{amssymb,latexsym, amsmath, amsxtra, bbm, bm}
\usepackage[dvips]{graphics}
\usepackage{xypic}
\usepackage{verbatim}
\usepackage[abs]{overpic}
\usepackage{hyperref}
\usepackage{mathtools}

\allowdisplaybreaks[3]

\theoremstyle{plain}
        \newtheorem{theorem}{Theorem}[section]
        \newtheorem*{theorem*}{Theorem}
        \newtheorem*{conj*}{Conjecture}
        \newtheorem{lemma}[theorem]{Lemma}
        \newtheorem{prop}[theorem]{Proposition}
        \newtheorem{cor}[theorem]{Corollary}

        \newtheorem{thmx}{Theorem}
\theoremstyle{definition}
        \newtheorem{definition}[theorem]{Definition}
        \newtheorem{rem}[theorem]{Remark}
         \newtheorem{rems}[theorem]{Remarks}
         
         \newtheorem*{assumptions}{The Assumptions}

\theoremstyle{remark}
        \newtheorem*{remark}{Remark}

\numberwithin{equation}{section}
\numberwithin{theorem}{section}
\numberwithin{table}{section}
\numberwithin{figure}{section}

\providecommand{\defn}[1]{\emph{#1}}

\renewcommand{\leq}{\leqslant}

\renewcommand{\geq}{\geqslant}

\newcommand{\diam}  {\operatorname{diam}}

\newcommand{\inte}  {\operatorname{inte}}
\newcommand{\inter}  {\operatorname{int}}

\newcommand{\id} {\operatorname{id}}

\newcommand{\card} {\operatorname{card}}
\newcommand{\supp}{\operatorname{supp}}
\newcommand{\R}{\mathbb{R}}
\newcommand{\B}{\mathbb{B}}    
\newcommand{\C}{\mathbb{C}}      

\newcommand{\N}{\mathbb{N}}      
\newcommand{\Z}{\mathbb{Z}}      

\newcommand{\D}{\mathbb{D}}     
 
\renewcommand{\H}{\mathbb{H}}    
\providecommand{\abs}[1]{\lvert#1\rvert}
\newcommand{\Abs}[1]{\left\lvert#1\right\rvert}
\providecommand{\Absbig}[1]{\bigl\lvert#1\bigr\rvert}
\providecommand{\Absbigg}[1]{\biggl\lvert#1\biggr\rvert}
\providecommand{\AbsBig}[1]{\Bigl\lvert#1\Bigr\rvert}
\providecommand{\AbsBigg}[1]{\Biggl\lvert#1\Biggr\rvert}
\providecommand{\norm}[1]{\|#1\|}
\providecommand{\Norm}[1]{\left\|#1\right\|}

\renewcommand{\:}{\colon}

\renewcommand{\b}{\mathfrak{b}}
\newcommand{\w}{\mathfrak{w}}
\renewcommand{\c}{\mathfrak{c}}

\newcommand{\I}{\mathbf{i}}
\newcommand{\crit}{\operatorname{crit}}
\newcommand{\post}{\operatorname{post}}

\newcommand{\CC}{\mathcal{C}}

 \newcommand{\DD}{\mathbf{D}}

\newcommand{\X} {\mathbf{X}}

\newcommand{\E} {\mathbf{E}}

\newcommand{\V} {\mathbf{V}}
 
\newcommand{\W} {\mathbf{W}}

\newcommand{\CCC}{C}

\newcommand{\PPP}{\mathcal{P}}

\newcommand{\MMM}{\mathcal{M}}

\newcommand{\Holder}[1] {\CCC^{0,#1}}

\newcommand{\Hseminorm}[3] {\Abs{#3}_{#1}}

\newcommand{\Hnorm}[3] {\Norm{#2}_{\Holder{#1}{#3}}}

\newcommand{\vertiii}[1]{{\left\vert\kern-0.25ex\left\vert\kern-0.25ex\left\vert #1 
    \right\vert\kern-0.25ex\right\vert\kern-0.25ex\right\vert}}

\newcommand{\vertiiibig}[1]{{\bigl\vert\kern-0.25ex\bigl\vert\kern-0.25ex\bigl\vert #1 
    \bigr\vert\kern-0.25ex\bigr\vert\kern-0.25ex\bigr\vert}}
    
\newcommand{\vertiiiBig}[1]{{\Bigl\vert\kern-0.25ex\Bigl\vert\kern-0.25ex\Bigl\vert #1 
    \Bigr\vert\kern-0.25ex\Bigr\vert\kern-0.25ex\Bigr\vert}}

\newcommand{\RR}{\mathcal{L}}

\newcommand{\lcm}{\operatorname{lcm}}

\newcommand{\XX}{\mathbb{X}}

\newcommand{\wt}[1]{\widetilde{#1}}

\newcommand{\minus}{\scalebox{0.6}[0.6]{$-\!\!\;$}}

\newcommand{\circsmall}{\scalebox{0.6}[0.6]{$\circ$}}

\newcommand{\Orb}{\mathfrak{P}}

\newcommand{\ti}{\vartriangle}

\newcommand{\ee}{\boldsymbol{\shortparallel}}

\newcommand{\e}{\boldsymbol{\shortmid}}

\newcommand{\po}{\bullet}

\newcommand{\DS}{\mathfrak{D}}

\newcommand{\Inv}{\operatorname{Inv}}

\newcommand{\bigO}{\mathcal{O}}

\begin{document}
\title[Prime orbit theorems for expanding Thurston maps]{Prime orbit theorems for expanding Thurston maps: \\Dirichlet series and orbifolds}
\author{Zhiqiang~Li \and Tianyi~Zheng}
\thanks{Z.~Li was partially supported by NSFC Nos.~12101017, 12090010, 12090015, and BJNSF No.~1214021.}

\address{Zhiqiang~Li, School of Mathematical Sciences \& Beijing International Center for Mathematical Research, Peking University, Beijing 100871, China}
\email{zli@math.pku.edu.cn}
\address{Tianyi~Zheng, Department of Mathematics, University of California, San Diego, San Diego, CA 92093--0112}
%\thanks{The author was partially supported by NSF grant .}
\email{tzheng2@math.ucsd.edu}
%\date{\today}



\begin{abstract}
We obtain an analog of the prime number theorem for a class of branched covering maps on the $2$-sphere $S^2$ called expanding Thurston maps, which are topological models of some non-uniformly expanding rational maps without any smoothness or holomorphicity assumptions. More precisely, we show that the number of primitive periodic orbits, ordered by a weight on each point induced by an (eventually) positive real-valued H\"{o}lder continuous function on $S^2$ that is not cohomologous to a constant, is asymptotically the same as the well-known logarithmic integral. In particular, our results apply to postcritically-finite rational maps for which the Julia set is the whole Riemann sphere.
\end{abstract}

\maketitle

\tableofcontents

\section{Introduction}

\subsection{History and motivations}

Counting is probably one of the very first mathematical activities that predates any written history of humankind. It remains at the core of virtually all fields of mathematics to count important objects in the field and study their statistical properties. 

One useful idea in such studies is to code the important objects in a function in the form of a polynomial or a series. Perhaps the most famous of such functions is the \emph{Riemann zeta function}
\begin{equation*}
	\zeta_{\operatorname{Riemann}} (s) \coloneqq \sum\limits_{n=1}^{+\infty} \frac{1}{n^s}  
	=  \prod\limits_{p \text{ prime}}  ( 1 - p^{-s}  )^{-1},    \qquad\qquad  \Re(s) > 1,
\end{equation*}
whose analytic properties were studied by B.~Riemann in the 19th century, even though the product formula was already known to L.~Euler in the 18th century. Analytic properties of the Riemann zeta function are closely related to the distribution of prime numbers. It is known that the assertion that $\zeta_{\operatorname{Riemann}}$ has a non-vanishing holomorphic extension on the line $\Re(s) = 1$ except for a simple pole at $s=1$ is equivalent to the famous Prime Number Theorem of Ch.~J.~de~la~Vall\'ee-Poussin and J.~Hadamard stating that the number $\pi(T)$ of primes no larger than $T>0$ satisfies 
\begin{equation*}
	\pi(T) \sim \operatorname{Li} (T)  \sim \frac{T}{\log T},  \qquad\qquad  \text{as } T \to +\infty,
\end{equation*}
where $\operatorname{Li}(y)$ is the \defn{Eulerian logarithmic integral function}
\begin{equation}      \label{eqDefLogIntegral}
	\operatorname{Li} (y) \coloneqq \int_2^y\! \frac{1}{\log u} \,\mathrm{d} u,  \qquad\qquad y>0.
\end{equation}
A more careful study of $\zeta_{\operatorname{Riemann}}$ reveals that a condition of H.~von~Koch from 1901 \cite{vK01} on the error term in the Prime Number Theorem, namely,
\begin{equation*}
	\pi(T) = \operatorname{Li} (T)  + \bigO \bigl( \sqrt{T} \log T \bigr),  \qquad\qquad  \text{as } T \to +\infty,
\end{equation*}
is equivalent to the Riemann hypothesis (see also \cite[Section~5.1]{BCRW08}).

The idea of studying zeta functions was first introduced by A.~Selberg in 1956 from number theory into geometry, where (primitive) closed geodesics serve the role of prime numbers. He defined a zeta function
\begin{equation}    \label{eqDefZetaFnSelberg}
	\zeta_{\operatorname{Selberg}} (s) \coloneqq   \prod\limits_{\gamma \in \Orb }   \prod\limits_{ n=0 }^{+\infty}   \bigl( 1 - e^{ -(s + n) l(\gamma) }  \bigr),    \qquad\qquad  \Re(s) > 1,
\end{equation}
where $\Orb$ denotes the set of primitive closed geodesics and $l(\gamma)$ is the length of the geodesic $\gamma$ \cite{Se56}.

H.~Huber established the first \emph{Prime Geodesic Theorem}, as an analog of the Prime Number Theorem, for surfaces of constant negative curvature in 1961, where A.~Selberg's work \cite{Se56} was implicitly used.

\begin{theorem*}[H.~Huber \cite{Hu61}]
	Let $M$ be a compact surface of constant curvature $-1$, and by $\pi(T)$ we denote the number of primitive closed geodesics $\gamma$ of length $l(\gamma) \leq T$. 
	
	Then there exists $\alpha \in (0,1)$ such that 
	\begin{equation*}
		\pi(T) = \operatorname{Li}  \bigl(  e^T  \bigr)  + \bigO \bigl( e^{\alpha T} \bigr),  \qquad\qquad  \text{as } T \to +\infty.
	\end{equation*} 
\end{theorem*}

The zeta functions were then introduced into dynamics by M.~Artin and B.~Mazur \cite{AM65} in 1965 for diffeomorphisms and by S.~Smale \cite{Sm67} in 1967 for Anosov flows, where (primitive) periodic orbits serve the role of prime numbers. S.~Smale used A.~Selberg's formulation in the context of geodesic flows on surfaces of constant negative curvature due to the direct correspondence between closed geodesics on the surface and periodic orbits of the geodesic flow. A related formulation of zeta functions for flows was later proposed and studied by D.~Ruelle \cite{Rue76a, Rue76b, Rue76c} in 1976, which behaves better under changes of time scale in the more general context of Axiom A systems. More precisely, for Anosov flows, the Ruelle zeta function is defined as
\begin{equation}   \label{eqDefZetaFnRuelle}
	\zeta_{\operatorname{Ruelle}} (s) \coloneqq   \prod\limits_{\gamma \in \Orb }   \bigl( 1 - e^{ - s l(\gamma) }  \bigr)^{-1} ,    \qquad\qquad  \Re(s) > 1,
\end{equation}
where $\Orb$ denotes the set of primitive periodic orbits of the flow and $l(\gamma)$ is the length of the orbit $\gamma$. With this interpretation of $\Orb$ and $l(\gamma)$, we have
\begin{equation}   \label{eqRuelleSelberg}
	\zeta_{\operatorname{Ruelle}} (s)  
	= \frac{ \zeta_{\operatorname{Selberg}}(s+1)}{ \zeta_{\operatorname{Selberg}}(s) }
\end{equation}
when both sides are defined.

Extensive researches have been carried out in geometry and dynamics in establishing Prime Geodesic Theorems on various spaces and Prime Orbit Theorems for various flows and other dynamical systems. We recall but a few such results here and by no means claim to give a complete review of the literature.

We denote by $\pi(T)$ the number of primitive periodic orbits $\gamma$ of ``length'' (appropriately interpreted for the corresponding dynamical system) $l(\gamma) \leq T$. By a Prime Orbit Theorem without an error term, we mean the assertion that there exists a constant $h>0$ such that 
$\pi(T) \sim  \operatorname{Li}  \bigl(  e^{hT}  \bigr)$ as $T \to +\infty$. By a Prime Orbit Theorem with an exponential error term, we mean the assertion that there exist constants $h>0$ and $\delta\in(0,h)$ such that 
$\pi(T)=  \operatorname{Li}  \bigl(  e^{hT}  \bigr)  + \bigO \bigl(  e^{(h-\delta)T}  \bigr)$ as $T \to +\infty$. 

Generalizing the first order asymptotics of H.~Huber for geodesic flows over compact surfaces of constant negative curvature, G.~A.~Margulis established in his thesis in 1970 \cite{Mar04} (see also \cite{Mar69}) a Prime Orbit Theorem without an error term for the geodesic flows over compact Riemannian manifolds with variable negative curvature, and more generally, for weak-mixing Anosov flows preserving a smooth volume. Similar results were obtained by P.~Sarnak in his thesis in 1980 for non-compact surfaces of finite volume \cite{Sa80}.

For geodesic flows over convex-cocompact surfaces of constant negative curvature, a Prime Orbit Theorem without an error term was obtained conditionally by L.~Guillop\'e \cite{Gu86} and later unconditionally by S.~P.~Lalley \cite{La89}.

The exponential error terms in the Prime Orbit Theorems in the contexts above (except in H.~Huber's result) were out of reach until D.~Dolgopyat's seminal work on the exponential mixing of Anosov flows in his thesis \cite{Dol98}, where he developed an ingenuous approach to get new upper bounds on the norms of the complex Ruelle (transfer) operators on some appropriate function spaces. M.~Pollicott and R.~Sharp \cite{PoSh98} combined these bounds with techniques from number theory to get a Prime Orbit Theorem with an exponential error term for the geodesic flows over compact surfaces of variable negative curvature. For related works on closed geodesics satisfying some homological constraints, see R.~Phillips and ~P.~Sarnak \cite{PhSa87}, S.~P.~Lalley \cite{La89}, A.~Katsuda and T.~Sunada \cite{KS90}, M.~Pollicott \cite{Po91}, R.~Sharp \cite{Sh93}, M.~Babillot and F.~Ledrappier \cite{BabLe98}, M.~Pollicott and R.~Sharp \cite{PoSh98}, N.~Anantharaman \cite{An00a, An00b}, etc.

The elegant idea of M.~Pollicott and R.~Sharp in \cite{PoSh98} used in establishing the error term for their Prime Orbit Theorem is summarized in a nutshell below:
\begin{enumerate}
	\smallskip
	\item[(1)] Obtain a quantitative bound for each term in the additive form of the Ruelle zeta function $\zeta_{\operatorname{Ruelle}}$ (compare with (\ref{eqDefZetaFn}) and (\ref{eqZetaFnOrbForm})) in terms of the operator norm of the Ruelle operator via an argument of D.~Ruelle \cite{Rue90} that matches the preimage points and periodic points of the symbolic dynamics induced by the Bowen--Ratner symbolic coding for the geodesic flows.
	
	\smallskip
	\item[(2)] By combining the bound above with D.~Dolgopyat's bound \cite{Dol98} on the norms of the Ruelle operator on some appropriate function spaces, derive a non-vanishing holomorphic extension to $\zeta_{\operatorname{Ruelle}}$ on a vertical strip $h-\epsilon \leq \Re(s) \leq h$, for some $\epsilon>0$, except for a simple pole at $s=h$, where $h\in\R$ is the smallest number such that the additive form of $\zeta_{\operatorname{Ruelle}} (s)$ converges on $\{ s\in\C \,|\, \Re(s)> h\}$, and additionally, obtain a quantitative bound of $\abs{ \zeta_{\operatorname{Ruelle}} }$ on this strip.
	
	\smallskip
	\item[(3)] Establish the Prime Orbit Theorem with an exponential error term from the bound of $\abs{ \zeta_{\operatorname{Ruelle}} }$ above via standard arguments from analytic number theory.
\end{enumerate}

Variations and simplifications of this general strategy of M.~Pollicott and R.~Sharp, relying on the machinery of D.~Dolgopyat, have been adapted by many authors in various contexts, see for example, F.~Naud \cite{Na05}, L.~N.~Stoyanov \cite{St11}, P.~Giulietti, C.~Liverani, and M.~Pollicott \cite{GLP13}, H.~Oh and D.~Winter \cite{OW16, OW17}, D.~Winter \cite{Wi16}, etc. The importance of the analytic properties of various dynamical zeta functions in understanding the distribution of periodic orbits now becomes apparent. Not surprisingly, in view of the connection between \cite{Dol98} and \cite{PoSh98}, dynamical zeta functions are also closely related to the decay of correlations and resonances. As M.~Pollicott has put it, these are basically ``two sides of the same coin''. For related researches on the side of decay of correlations, see for example,  D.~Dolgopyat \cite{Dol98}, C.~Liverani \cite{Liv04}, A.~Avila, S.~Gou\"ezel, and J.~C.~Yoccoz \cite{AGY06}, L.~N.~Stoyanov \cite{St01, St11}, V.~Baladi and C.~Liverani \cite{BalLiv12}, V.~Baladi, M.~Demers, and C.~Liverani \cite{BDL18}, etc.

In the context of convex-cocompact surfaces $M$ of constant negative curvature, i.e., $M = \Gamma \backslash \H^2$ being the quotient of a \emph{classical Fuchsian Schottky group} $\Gamma$ (see \cite[Section~4.1]{Na05}) acting on the hyperbolic plane $\H^2$, F.~Naud \cite{Na05} established in 2005 a Prime Orbit Theorem with an exponential error term by producing some vertical strip in $\C$ on which the Selberg zeta function $\zeta_{\operatorname{Selberg}}$ (resp.\ the Ruelle zeta function $\zeta_{\operatorname{Ruelle}}$) has a non-vanishing holomorphic extension except a simple zero (resp.\ pole, see (\ref{eqRuelleSelberg})). For stronger results on the zero free strip and distribution of zeros in these contexts, see the recent works of J.~Bourgain, A.~Gamburd, and P.~Sarnak \cite{BGS11}, F.~Naud \cite{Na14}, H.~Oh and D.~Winter \cite{OW16}, S.~Dyatlov and J.~Zahl \cite{DZ16}, J.~Bourgain and S.~Dyatlov \cite{BD17}.

In the context of subgroups of the group of orientation preserving isometries of higher dimensional real hyperbolic space $\H^n$ and more general settings, T.~Roblin \cite{Ro03} proved a Prime Orbit Theorem without an error term for geometrically finite subgroups, G.~A.~Margulis, A.~Mohammadi, and H.~Oh \cite{MMO14} established an exponential error term for geometrically finite subgroups under additional conditions, and D.~Winter \cite{Wi16} showed a Prime Orbit Theorem with an exponential error term for convex-cocompact subgroups. A form of Prime Orbit Theorem without an error term for abelian covers of some hyperbolic manifolds was established by H.~Oh and W.~Pan \cite{OP18}.

In the same work \cite{Na05}, F.~Naud also established the first Prime Orbit Theorem with an exponential error term in complex dynamics, for a class of hyperbolic polynomials $z^2 + c$, $c\in(-\infty, -2)$. One key feature of this class of polynomials is that their Julia sets are Cantor sets. For an earlier work on dynamical zeta functions for a class of sub-hyperbolic quadratic polynomials, see V.~Baladi, Y.~Jiang, and H.~H.~Rugh \cite{BJR02}. For hyperbolic rational maps, S.~Waddington studied a variation of the Ruelle zeta function defined by strictly preperiodic points instead of periodic points (compare with (\ref{eqDefZetaFn}) and (\ref{eqZetaFnOrbForm})), and established a corresponding form of Prime Orbit Theorem without an error term in \cite{Wad97}.

The study of iterations of polynomials and rational maps, known as complex dynamics, dates back to the work of G.~K{\oe}nigs, E.~Schr\"oder, and others in the 19th century. This subject was developed into an active area of research, thanks to the remarkable works of S.~Latt\`{e}s, C.~Carath\'eodory, P.~Fatou, G.~Julia, P.~Koebe, L.~Ahlfors, L.~Bers, M.~Herman, A.~Douady, D.~P.~Sullivan, J.~H.~Hubbard, W.~P.~Thurston, J.-C.~Yoccoz, C.~McMullen, J.~Milnor, M.~Lyubich, M.~Shishikura, and many others.

In the early 1980s, D.~P.~Sullivan \cite{Su85, Su83} introduced a ``dictionary'', known as \emph{Sullivan's dictionary} nowadays, linking the theory of complex dynamics with another classical area of conformal dynamical systems, namely, geometric group theory, mainly concerning the study of Kleinian groups acting on the Riemann sphere. Many dynamical objects in both areas can be similarly defined and results similarly proven, yet essential and important differences remain.

The Prime Orbit Theorems with exponential error terms of F.~Naud in \cite{Na05} can be considered as another new correspondence in Sullivan's dictionary. Despite active researches on dynamical zeta functions and Prime Orbit Theorems in many areas of dynamical systems, especially the works of L.~N.~Stoyanov \cite{St11}, G.~A.~Margulis, A.~Mohammadi, and H.~Oh \cite{MMO14}, and D.~Winter \cite{Wi16} on the group side of Sullivan's dictionary, the authors are not aware of similar entries in complex dynamics since F.~Naud \cite{Na05}, until the recent work of H.~Oh and D.~Winter \cite{OW17}. At a suggestion of D.~P.~Sullivan regarding holonomies, H.~Oh and D.~Winter established a Prime Orbit Theorem (as well as the equidistribution of holonomies) with an exponential error term for hyperbolic rational maps in \cite{OW17}. A rational map is \emph{hyperbolic} if the closure of the union of forward orbits of critical points is disjoint from its Julia set. The Julia set of a hyperbolic rational map has zero area. A rational map is forward-expansive on some neighborhood of its Julia set if and only if it is hyperbolic. The novelty and emphasis of this paper differ from that of \cite{OW17}; see Subsection~\ref{subsctPlan} for more details.

In Sullivan's dictionary, Kleinian groups, i.e., discrete subgroups of M\"obius transformations on the Riemann sphere, correspond to rational maps, and convex-cocompact Kleinian groups correspond to rational maps that exhibit strong expansion properties such as hyperbolic rational maps, semi-hyperbolic rational maps, and postcritically-finite sub-hyperbolic rational maps. See insightful discussions on this part of the dictionary in \cite[Chapter~1]{BM17}, \cite[Chapter~1]{HP09}, and \cite[Section~1]{LM97}.

One important question in conformal dynamical systems is: ``\emph{What is special about conformal dynamical systems in a wider class of dynamical systems characterized by suitable metric-topological conditions?}''

W.~P.~Thurston gave an answer to this question in his celebrated combinatorial characterization theorem of \emph{postcritically-finite}  rational maps (i.e., the union of forward orbits of critical points is a finite set) on the Riemann sphere among a class of more general topological maps, known as Thurston maps nowadays \cite{DH93}. A \emph{Thurston map} is a (non-homeomorphic) branched covering map on the topological $2$-sphere $S^2$ whose finitely many critical points are all preperiodic (see Subsection~\ref{subsctThurstonMap} for a precise definition). Thurston's theorem asserts that a Thurston map is essentially a rational map if and only if there exists no so-called \emph{Thurston obstruction}, i.e., a collection of simple closed curves on $S^2$ subject to certain conditions \cite{DH93}. 

Under Sullivan's dictionary, the counterpart of Thurston's theorem in geometric group theory is Cannon's Conjecture \cite{Ca94}. This conjecture predicts that an infinite, finitely presented Gromov hyperbolic group $G$ whose boundary at infinity $\partial_\infty G$ is a topological $2$-sphere is a Kleinian group. Gromov hyperbolic groups can be considered as metric-topological systems generalizing the conformal systems in the context of geometric group theory, namely, convex-cocompact Kleinian groups. Inspired by Sullivan's dictionary and their interest in Cannon's Conjecture, M.~Bonk and D.~Meyer, along with others, studied a subclass of Thurston maps by imposing some additional condition of expansion. A new characterization theorem of rational maps from a metric space point of view is established in this context by M.~Bonk and D.~Meyer \cite{BM10, BM17}, and by P.~Ha\"issinsky and K.~M.~Pilgrim \cite{HP09}. Roughly speaking, we say that a Thurston map is \emph{expanding} if for any two points $x, \, y\in S^2$, their preimages under iterations of the map get closer and closer. For each expanding Thurston map, we can equip the $2$-sphere $S^2$ with a natural class of metrics, called \emph{visual metrics}. As the name suggests, these metrics are constructed in a similar fashion as the visual metrics on the boundary $\partial_\infty G$ of a Gromov hyperbolic group $G$. See Subsection~\ref{subsctThurstonMap} for a more detailed discussion on these notions.

\begin{theorem*}[M.~Bonk \& D.~Meyer \cite{BM10, BM17}, P.~Ha\"issinsky \& K.~M.~Pilgrim \cite{HP09}] 
	An expanding Thurston map is conjugate to a rational map if and only if the sphere $(S^2,d)$ equipped with a visual metric $d$ is quasisymmetrically equivalent to the Riemann sphere $\widehat\C$ equipped with the chordal metric.
\end{theorem*}   

See \cite[Theorem~18.1~(ii)]{BM17} for a proof. For an equivalent formulation of Cannon's conjecture from a similar point of view, see \cite[Conjecture~5.2]{Bon06}. The definition of the chordal metric is recalled in Remark~\ref{rmChordalVisualQSEquiv} and the notion of quasisymmetric equivalence in Definition~\ref{defQuasiSymmetry}.

We remark on the subtlety of the expansion property of expanding Thurston maps by pointing out that such maps are never forward-expansive due to the critical points. In fact, each expanding Thurston map without periodic critical points is \emph{asymptotically $h$-expansive}, but not \emph{$h$-expansive}; on the other hand, expanding Thurston maps with at least one periodic critical point are not even asymptotically $h$-expansive \cite{Li15}. Asymptotic $h$-expansiveness and $h$-expansiveness are two notions of weak expansion introduced by M.~Misiurewicz \cite{Mi73} and R.~Bowen \cite{Bow72}, respectively. Note that forward-expansiveness implies $h$-expansiveness, which in turn implies asymptotic $h$-expansiveness \cite{Mi76}.

Thanks to the fundamental works of W.~P.~Thurston, M.~Bonk, D.~Meyer, P.~Ha\"issinsky, and K.~M.~Pilgrim, the dynamics and geometry of expanding Thurston maps and similar topological branched covering maps has attracted a considerable amount of interests, with motivations from both complex dynamics as well as Sullivan's dictionary. Under the dictionary, an expanding Thurston map corresponds to a Gromov hyperbolic group whose boundary at infinity is the topological $2$-sphere, and the special case of a rational expanding Thurston map (i.e., a postcritically-finite rational map whose Julia set is the whole Riemann sphere) corresponds to a convex-cocompact Kleinian group whose limit set is homeomorphic to a $2$-sphere (i.e., a cocompact lattice of $\mathrm{PSL}(2,\C)$) (see \cite[Chapter~1]{BM17}, \cite[Section~1]{Yi15}, and compare with \cite[Chapter~1]{HP09}).

Lastly, we want to remark that we have not been able to make connections to another successful approach to dynamical zeta functions dating back to the work of J.~Milnor and W.~P.~Thurston in 1988 on the \emph{kneading determinant} for real $1$-dimensional dynamics with critical points \cite{MT88}. The Milnor--Thurston kneading theory has been developed and used by many authors since then, for example, V.~Baladi and D.~Ruelle \cite{BR96}, V.~Baladi, A.~Kitaev, D.~Ruelle, and S.~Semmes \cite{BKRS97}, M.~Baillif \cite{Bai04}, M.~Baillif and V.~Baladi \cite{BB05}, H.~H.~Rugh \cite{Rug16}, and V.~Baladi \cite[Chapter~3]{Bal18}.

\subsection{Main results}

Complex dynamics is a vibrant field of dynamical systems, focusing on the study of iterations of polynomials and rational maps on the Riemann sphere $\widehat{\C}$. It is closely connected, via \emph{Sullivan's dictionary} \cite{Su85, Su83}, to geometric group theory, mainly concerning the study of Kleinian groups.

In complex dynamics, the lack of uniform expansion of a rational map arises from critical points in the Julia set. One natural class of non-uniformly expanding rational maps are called \emph{topological Collet--Eckmann maps}, whose basic dynamical properties have been studied by S.~Smirnov, F.~Przytycki, J.~Rivera-Letelier, Weixiao Shen, etc.\ (see \cite{PRLS03, PRL07, PRL11, RLS14}). In this paper, we focus on a subclass of topological Collet--Eckmann maps for which each critical point is preperiodic and the Julia set is the whole Riemann sphere. Actually, the most general version of our results is established for topological models of these maps, called \emph{expanding Thurston maps}. Thurston maps were studied by W.~P.~Thurston in his celebrated characterization theorem of postcritically-finite rational maps among such topological models \cite{DH93}. Thurston maps and Thurston's theorem, sometimes known as the fundamental theorem of complex dynamics, are indispensable tools in the modern theory of complex dynamics. Expanding Thurston maps were studied extensively by M.~Bonk, D.~Meyer \cite{BM10, BM17} and P.~Ha\"issinsky, K.~M.~Pilgrim \cite{HP09}.

The investigations of the growth rate of the number of periodic orbits (e.g.\ closed geodesics) have been a recurring theme in dynamics and geometry. 

Inspired by the seminal works of F.~Naud \cite{Na05} and H.~Oh, D.~Winter \cite{OW17} on the growth rate of periodic orbits, known as Prime Orbit Theorems, for hyperbolic (uniformly expanding) polynomials and rational maps, we establish in this paper the first Prime Orbit Theorems (to the best of our knowledge) in a non-uniformly expanding setting in complex dynamics. On the other side of Sullivan's dictionary, see related works \cite{MMO14, OW16, OP18}. For an earlier work on dynamical zeta functions for a class of sub-hyperbolic quadratic polynomials, see V.~Baladi, Y.~Jiang, and H.~H.~Rugh \cite{BJR02}. See also related work of S.~Waddington \cite{Wad97} on strictly preperiodic points of hyperbolic rational maps.

Given a map  $f\: X \rightarrow X$ on a metric space $(X,d)$ and a function $\phi\: S^2 \rightarrow \R$,  we define the weighted length $l_{f,\phi} (\tau)$ of a primitive periodic orbit 
\begin{equation*}
\tau \coloneqq \bigl\{x, \, f(x), \,  \cdots, \,  f^{n-1}(x) \bigr\} \in \Orb(f)
\end{equation*}
as
\begin{equation}  \label{eqDefComplexLength}
l_{f,\phi} (\tau) \coloneqq \phi(x) + \phi(f(x)) + \cdots + \phi \bigl(f^{n-1}(x) \bigr).
\end{equation}
We denote by
\begin{equation}   \label{eqDefPiT}
\pi_{f,\phi}(T) \coloneqq \card \{ \tau \in \Orb(f)  : l_{f,\phi}( \tau )  \leq T \}, \qquad   T>0,
\end{equation}
the number of primitive periodic orbits with weighted lengths up to $T$. Here $\Orb(f)$ denotes the set of all primitive periodic orbits of $f$ (see Section~\ref{sctNotation}). 

Note that the Prime Orbit Theorems in \cite{Na05, OW17} are established for the \emph{geometric potential} $\phi= \log \abs{f'}$. For hyperbolic rational maps, the Lipschitz continuity of the geometric potential plays a crucial role in \cite{Na05, OW17}. In our non-uniform expanding setting, critical points destroy the continuity of $\log\abs{f'}$. So we are left with two options to develop our theory, namely, considering 
\begin{enumerate}
\smallskip
\item[(a)] H\"older continuous $\phi$ or

\smallskip
\item[(b)] the geometric potential $\log\abs{f'}$.
\end{enumerate}
Despite the lack of H\"older continuity of $\log\abs{f'}$ in our setting, its value is closely related to the size of pull-backs of sets under backward iterations of the map $f$. This fact enables an investigation of the Prime Orbit Theorem in case (b), which will be investigated in an upcoming series of separate works starting with \cite{LRL}. 

The current paper is the first of a series of three papers (together with \cite{LZhe23b, LZhe23c}) focusing on case (a), in which the incompatibility of H\"older continuity of $\phi$ and non-uniform expansion of $f$ calls for a close investigation of metric geometries associated to $f$.

Latt\`{e}s maps are rational Thurston maps with parabolic orbifolds (see \cite[Chapter~3]{BM17}). They form a well-known class of rational maps. We first formulate our theorem for Latt\`{e}s maps. 

\begin{thmx}[Prime Orbit Theorem for Latt\`{e}s maps]   \label{thmLattesPOT}
Let $f\: \widehat{\C} \rightarrow \widehat{\C}$ be a Latt\`{e}s map on the Riemann sphere $\widehat{\C}$. Let $\phi \: \widehat{\C} \rightarrow \R$ be eventually positive and continuously differentiable. Then there exists a unique positive number $s_0>0$ with $P(f,-s_0 \phi) = 0$ and there exists $N_f\in\N$ depending only on $f$ such that the following statements are equivalent:
\begin{enumerate}
\smallskip
\item[(i)] $\phi$ is not cohomologous to a constant in the space $\CCC\bigl(\widehat{\C} \bigr)$ of real-valued continuous functions on $\widehat{\C}$.

\smallskip
\item[(ii)] For each $n\in \N$ with $n\geq N_f$, we have 
\begin{equation*}
\pi_{F,\Phi}(T) \sim \operatorname{Li}\bigl( e^{s_0 T} \bigr) \text{ as } T \to + \infty,
\end{equation*}
where  $F\coloneqq f^n$ and $\Phi\coloneqq \sum_{i=0}^{n-1} \phi \circ f^i$.

\smallskip
\item[(iii)] For each $n\in \N$ with $n\geq N_f$, there exists a constant $\delta \in (0, s_0)$ such that
\begin{equation*}
\pi_{F,\Phi}(T) = \operatorname{Li}\bigl( e^{s_0 T} \bigr)  + \bigO \bigl( e^{(s_0 - \delta)T} \bigr) \text{ as } T \to + \infty ,
\end{equation*}
where  $F\coloneqq f^n$ and $\Phi\coloneqq \sum_{i=0}^{n-1} \phi \circ f^i$.
\end{enumerate} 
Here $P(f,\cdot)$ denotes the topological pressure, and $\operatorname{Li} (y) \coloneqq \int_2^y\! \frac{1}{\log u} \,\mathrm{d} u$, $y>0$, is the \defn{Eulerian logarithmic integral function}.
\end{thmx}

See Definitions~\ref{defCohomologous} and~\ref{defEventuallyPositive} for the definitions of co-homology and eventually positive functions, respectively.

The implication (i)$\implies$(iii) relies crucially on some local properties of the metric geometry of the visual sphere induced by the Latt\`{e}s maps, and is not expected (by the authors) to hold in general. We postpone the proof of the implication (i)$\implies$(iii) to the next paper \cite{LZhe23b}, in which the exponential error term similar to that in (iii) for a class of more general rational Thurston maps will be established under a condition called \emph{$\alpha$-strong non-integrability condition}. In the third paper \cite{LZhe23c} in this series, we show that this condition is generic.

In fact, the equivalence of the first two conditions is proved for more general postcritically-finite rational maps without periodic critical points. The following theorem is an immediate consequence of a more general result in Theorem~\ref{thmPrimeOrbitTheorem}.

\begin{thmx}[Prime Orbit Theorems for rational expanding Thurston maps]  \label{thmPrimeOrbitTheorem_rational}
Let $f\: \widehat{\C} \rightarrow \widehat{\C}$ be a postcritically-finite rational map without periodic critical points. Let $\sigma$ be the chordal metric or the spherical metric on the Riemann sphere $\widehat{\C}$, and $\phi \in \Holder{\alpha} \bigl( \widehat{\C}, \sigma \bigr)$ be an eventually positive real-valued H\"{o}lder continuous function with an exponent $\alpha\in (0,1]$. Then there exists a unique positive number $s_0>0$ with topological pressure $P(f,-s_0 \phi) = 0$ and there exists $N_f\in\N$ depending only on $f$ such that for each $n\in \N$ with $n\geq N_f$, the following statements hold for $F\coloneqq f^n$ and $\Phi\coloneqq \sum_{i=0}^{n-1} \phi \circ f^i$:
\begin{enumerate}
	\smallskip
	\item[(i)] $\pi_{F,\Phi}(T) \sim \frac{n s_0 c \exp(n s_0 c)}{\exp(n s_0 c) - 1} \operatorname{Li}\bigl( e^{s_0 T } \bigr) \sim \frac{n (\deg f)^n }{ ( \deg f)^n -1} \cdot \frac{(\deg f)^{T/c}}{T/c}$ as $T \to + \infty$ and $s_0 = \frac1c \log (\deg f)$ if $\phi$ is cohomologous to a constant $c>0$ in the space $\CCC\bigl(\widehat{\C} \bigr)$ of real-valued continuous functions on $\widehat{\C}$.
	
	\smallskip
	\item[(ii)] $\pi_{F,\Phi}(T) \sim \operatorname{Li}\bigl( e^{s_0 T} \bigr)$ as $T \to + \infty$ if $\phi$ is not cohomologous to a constant in the space $\CCC\bigl(\widehat{\C} \bigr)$ of real-valued continuous functions on $\widehat{\C}$.
	\end{enumerate}
\end{thmx}

Our strategy to overcome the obstacles presented by the incompatibility of the non-uniform expansion of our rational maps and the H\"older continuity of the weight $\phi$ (e.g.\ (a)~the set of $\alpha$-H\"older continuous functions is not invariant under the Ruelle operator $\RR_\phi$, for each $\alpha\in(0,1]$; (b) the weakening of the regularity of the temporal distance compared to that of the potential)  is to investigate the metric geometry of various natural metrics associated to the dynamics such as visual metrics, the canonical orbifold metric, and the chordal metric. Such considerations lead us beyond conformal, or even smooth, dynamical settings and into the realm of topological dynamical systems. More precisely, we will work in the abstract setting of branched covering maps on a topological $2$-sphere $S^2$ (see Subsections~\ref{subsctBranchedCoveringMaps} and~\ref{subsctThurstonMap}) without any smoothness assumptions. A \emph{Thurston map} is a postcritically-finite branched covering map on $S^2$. Thurston maps can be considered as topological models of the corresponding rational maps. 

Via Sullivan's dictionary, the counterpart of Thurston's theorem \cite{DH93} in geometric group theory is Cannon's Conjecture \cite{Ca94} as mentioned above.

% This conjecture predicts that an infinite, finitely presented Gromov hyperbolic group $G$ whose boundary at infinity $\partial_\infty G$ is a topological $2$-sphere is a Kleinian group. Gromov hyperbolic groups can be considered as metric-topological systems generalizing the conformal systems in the context of geometric group theory, namely, convex-cocompact Kleinian groups. Inspired by Sullivan's dictionary and their interest in Cannon's Conjecture, M.~Bonk and D.~Meyer, along with others, studied a subclass of Thurston maps by imposing some additional condition of expansion. Roughly speaking, we say that a Thurston map is \emph{expanding} if for any two points $x, \, y\in S^2$, their preimages under iterations of the map get closer and closer. 
%
%For each expanding Thurston map, we can equip the $2$-sphere $S^2$ with a natural class of metrics called \emph{visual metrics}. As the name suggests, these metrics are constructed in a similar fashion as the visual metrics on the boundary $\partial_\infty G$ of a Gromov hyperbolic group $G$. As such, these visual metrics on either side of Sullivan's dictionary are corresponding entries in the dictionary. See Subsection~\ref{subsctThurstonMap} for a more detailed discussion on these notions. 

The introduction of visual metrics into complex dynamics via Sullivan's dictionary by M.~Bonk, D.~Meyer, P.~Ha\"issinsky, and K.~M.~Pilgrim has helped to catalyze much recent progress in the area.

Various ergodic properties, including thermodynamic formalism, on which the current paper crucially relies, have been studied by the first-named author in \cite{Li17} (see also \cite{Li15, Li16, Li18}). Generalization of results in \cite{Li17} to the more general branched covering maps studied by P.~Ha\"issinsky, K.~M.~Pilgrim \cite{HP09} has drawn significant interest recently \cite{HRL19, DPTUZ19, LZheH23}. We believe that our ideas introduced in this paper can be used to establish Prime Orbit Theorems in their setting.

M.~Bonk, D.~Meyer \cite{BM10, BM17} and P.~Ha\"issinsky, K.~M.~Pilgrim \cite{HP09} proved that an expanding Thurston map is conjugate to a rational map if and only if the sphere $(S^2,d)$ equipped with a visual metric $d$ is quasisymmetrically equivalent to the Riemann sphere $\widehat\C$ equipped with the chordal metric. The quasisymmetry cannot be promoted to Lipschitz equivalence due to the non-uniform expansion of Thurston maps. There exist expanding Thurston maps not conjugate to rational Thurston maps (e.g.\ ones with periodic critical points). Our theorems below apply to all expanding Thurston maps, which form the most general setting in this series of papers.

\begin{thmx}[Prime Orbit Theorems for expanding Thurston maps]  \label{thmPrimeOrbitTheorem}
Let $f\: S^2 \rightarrow S^2$ be an expanding Thurston map, and $d$ be a visual metric on $S^2$ for $f$. Let $\phi \in \Holder{\alpha}(S^2,d)$ be an eventually positive real-valued H\"{o}lder continuous function with an exponent $\alpha\in (0,1]$. Denote by $s_0$ the unique positive number with topological pressure $P(f,-s_0 \phi) = 0$. Then there exists $N_f\in\N$ depending only on $f$ such that for each $n\in \N$ with $n\geq N_f$, the following statements hold for $F\coloneqq f^n$ and $\Phi\coloneqq \sum_{i=0}^{n-1} \phi \circ f^i$:
\begin{enumerate}
	\smallskip
	\item[(i)] $\pi_{F,\Phi}(T) \sim \frac{n s_0 c \exp(n s_0 c)}{\exp(n s_0 c) - 1} \operatorname{Li}\bigl( e^{s_0 T } \bigr) \sim \frac{n (\deg f)^n }{ ( \deg f)^n -1} \cdot \frac{(\deg f)^{T/c}}{T/c}$ as $T \to + \infty$ and $s_0 = \frac1c \log (\deg f)$ if $\phi$ is cohomologous to a constant $c>0$ in the space $\CCC( S^2 )$ of real-valued continuous functions on $S^2$.
	
	\smallskip
	\item[(ii)] $\pi_{F,\Phi}(T) \sim \operatorname{Li}\bigl( e^{s_0 T} \bigr)$ as $T \to + \infty$ if $\phi$ is not cohomologous to a constant in the space $\CCC( S^2 )$ of real-valued continuous functions on $S^2$.
\end{enumerate}
Here, $\operatorname{Li}(\cdot)$ is the Eulerian logarithmic integral function defined in Theorem~\ref{thmLattesPOT}.
\end{thmx}

Note that $\lim_{y\to+\infty}   \operatorname{Li}(y) / ( y / \log y )  = 1$, thus in statement~(ii) we also get $\pi_{F,\Phi}(T) \sim   e^{s_0 T} \big/ ( s_0 T)$ as $T\to + \infty$. 

The proof of Theorem~\ref{thmPrimeOrbitTheorem} relying on holomorphic extension properties of certain dynamical zeta functions established in Theorem~\ref{thmZetaAnalExt_InvC} follows standard arguments in \cite{PP90} with special adaptations to our setting. See a detailed discussion at the end of Section~{\ref{sctDynOnC_Reduction}.

We remark that our proofs can be modified to derive equidistribution of holonomies similar to the corresponding result in \cite{OW17}, but we choose to omit them in order to emphasize our new ideas and to limit the length of this paper. See \cite{LRL} for a study of holonomies in a parallel setting.

In view of Remark~\ref{rmChordalVisualQSEquiv}, Theorem~\ref{thmPrimeOrbitTheorem_rational} is an immediate consequence of Theorem~\ref{thmPrimeOrbitTheorem}.

\begin{rem}    \label{rmNf}
The integer $N_f$ can be chosen as the minimum of $N(f,\wt{\CC})$ from Lemma~\ref{lmCexistsL} over all Jordan curves $\wt{\CC}$ with $\post f \subseteq \wt{\CC} \subseteq S^2$, in which case $N_f = 1$ if there exists a Jordan curve $\CC\subseteq S^2$ satisfying $f(\CC)\subseteq \CC$, $\post f\subseteq \CC$, and no $1$-tile in $\X^1(f,\CC)$ joins opposite sides of $\CC$ (see Definition~\ref{defJoinOppositeSides}). The same number $N_f$ is used in other results in this paper. We also remark that many properties of expanding Thurston maps $f$ can be established for $f$ after being verified first for $f^n$ for all $n\geq N_f$. However, some of the finer properties established for iterates of $f$ still remain open for the map $f$ itself; see for example, \cite{Me13, Me12}.
\end{rem}

%\begin{rem}
%Combining Theorem~\ref{thmPrimeOrbitTheorem} with the fact that for an expanding Thurston map $f$ the number of periodic points of period $n$, $n\in\N$, is $1+\deg f^n$ counted with a weight given by the local degree $\deg_{f^n}(x)$ (\cite[Theorem~1.1]{Li16}), we get a dichotomy for the asymptotic of $\pi_{F,\Phi}$ depending on whether the potential $\phi$ is cohomologous to a constant in $\CCC(S^2)$ or not. 
%\end{rem}

Note that due to the lack of any algebraic structure of expanding Thurston maps, even the fact that there are only countably many periodic points is not apparent from the definition (see \cite{Li16}). Without any algebraic, differential, or conformal structures, the main tools we rely on are from the interplay between the metric properties of various natural metrics and the combinatorial information on the iterated preimages of certain Jordan curves $\CC$ on $S^2$ (see Subsection~\ref{subsctThurstonMap}).

For many classical smooth dynamical systems with smooth potentials, analogous strong non-integrability conditions are often equivalent to a weaker condition, called \emph{non-local integrability conditions}, introduced in our context in Section~\ref{sctNLI}. In our context, a potential is non-locally integrable if and only if it is not cohomologous to a constant (Theorem~\ref{thmNLI}). However, the existence of critical points in the Julia set, and more seriously of periodic critical points for some expanding Thurston maps, gives rise to obstacles in the proof of this equivalence. For example, an inverse branch (in the universal orbifold covering space) of an expanding Thurston map with a periodic critical point may not have a fixed point. We nevertheless successfully establish such equivalence in Theorem~\ref{thmNLI} by a careful study of the universal orbifold covers, introduced by W.~P.~Thurston \cite{Th80} in the 1970s for the geometry of $3$-manifolds, in our context
in Section~\ref{sctNLI}. The lack of a fixed point for a general inverse branch on the universal orbifold covering space of an expanding Thurston map requires new arguments than that of the universal covering map for hyperbolic rational maps (see \cite{OW17}).

The counting result in Theorem~\ref{thmPrimeOrbitTheorem} follows from some quantitative information on the holomorphic extension of certain dynamical zeta function $\zeta_{F,\,\minus \Phi}$ defined as formal infinite products over periodic orbits. We briefly recall dynamical zeta functions and introduce dynamical Dirichlet series below. See Section~\ref{sctDynZetaFn} for a more detailed discussion.

Let $f\: S^2\rightarrow S^2$ be an expanding Thurston map and $\psi\in\CCC(S^2,\C)$ be a complex-valued continuous function on $S^2$.  Define, for each $n\in\N$ and each $x\in S^2$,
\begin{equation*}
	S_n \psi (x)  \coloneqq \sum_{j=0}^{n-1} \psi(f^j(x)).
\end{equation*}
We denote by the formal infinite product
\begin{equation*}  
\zeta_{f,\,\minus\psi} (s) \coloneqq 
\exp \Biggl( \sum_{n=1}^{+\infty} \frac{1}{n} \sum_{ x = f^n(x) } e^{-s S_n \psi(x)} \Biggr), \qquad s\in\C,
\end{equation*}
the \defn{dynamical zeta function} for the map $f$ and the \emph{potential} $\psi$. We remark that $\zeta_{f,\,\minus\psi}$ is the Ruelle zeta function for the suspension flow over $f$ with roof function $\psi$ if $\psi$ is positive. We define the \emph{dynamical Dirichlet series} associated to $f$ and $\psi$ as the formal infinite product
\begin{equation*}  
\DS_{f,\,\minus\psi,\, \deg_f} (s) \coloneqq \exp \Biggl( \sum_{n=1}^{+\infty} \frac{1}{n} \sum_{x = f^n(x)} e^{-s S_n \psi(x)} \deg_{f^n}(x) \Biggr), \qquad s\in\C.
\end{equation*}
Here $\deg_{f^n}$ is the \emph{local degree} of $f^n$ at $x\in S^2$ (see Definition~\ref{defBranchedCover}).

Note that if $f\: S^2 \rightarrow S^2$ is an expanding Thurston map, then so is $f^n$ for each $n\in\N$ (Remark~\ref{rmExpanding}). 

Recall that a function is holomorphic on a set $A\subseteq \C$ if it is holomorphic on an open set containing $A$.

\begin{thmx}[Analyticity properties of dynamical Dirichlet series and zeta functions for expanding Thurston maps]  \label{thmZetaAnalExt_InvC}
Let $f\: S^2 \rightarrow S^2$ be an expanding Thurston map, and $d$ be a visual metric on $S^2$ for $f$. Fix $\alpha\in(0,1]$. Let $\phi \in \Holder{\alpha}(S^2,d)$ be an eventually positive real-valued H\"{o}lder continuous function. Denote by $s_0$ the unique positive number with topological pressure $P(f,-s_0 \phi) = 0$.  
Then there exists $N_f\in\N$ depending only on $f$ such that for each $n\in \N$ with $n\geq N_f$, the following statement holds for $F\coloneqq f^n$ and $\Phi\coloneqq \sum_{i=0}^{n-1} \phi \circ f^i$:

\smallskip

Both the dynamical zeta function $\zeta_{F,\,\minus \Phi} (s)$ and the dynamical Dirichlet series $\DS_{F,\,\minus \Phi,\,\deg_F} (s)$ converge on $\{s\in\C : \Re(s) > s_0 \}$. If, in addition, $\phi$  is not cohomologous to a constant in $\CCC( S^2 )$, then $\zeta_{F,\,\minus \Phi} (s)$ and $\DS_{F,\,\minus \Phi,\,\deg_F} (s)$ extend to non-vanishing holomorphic functions on $\{s\in\C : \Re(s) \geq s_0\}$ except for the simple pole at $s=s_0$.
\end{thmx}

In order to get information about $\zeta_{F,\,\minus \Phi}$, we need to investigate the zeta function $\zeta_{\sigma_{A_{\ti}},\,\minus \phi\circsmall\pi_{\ti}}$ of a symbolic model of $\sigma_{A_{\ti}} \: \Sigma_{A_{\ti}}^+\rightarrow \Sigma_{A_{\ti}}^+$ of $F$.

\begin{thmx}[Holomorphic extensions of the symbolic zeta functions]  \label{thmZetaAnalExt_SFT}
Let $f\: S^2 \rightarrow S^2$ be an expanding Thurston map with a Jordan curve $\CC\subseteq S^2$ satisfying $f(\CC)\subseteq \CC$, $\post f\subseteq \CC$, and no $1$-tile in $\X^1(f,\CC)$ joins opposite sides of $\CC$. Let $d$ be a visual metric on $S^2$ for $f$. Fix $\alpha\in(0,1]$. Let $\phi \in \Holder{\alpha}(S^2,d)$ be an eventually positive real-valued H\"{o}lder continuous function. Denote by $s_0$ the unique positive number with $P(f,-s_0 \phi) = 0$. Let $\bigl(\Sigma_{A_{\ti}}^+,\sigma_{A_{\ti}}\bigr)$ be the one-sided subshift of finite type associated to $f$ and $\CC$ defined in Proposition~\ref{propTileSFT}, and let $\pi_{\ti}\: \Sigma_{A_{\ti}}^+\rightarrow S^2$ be the factor map as defined in (\ref{eqDefTileSFTFactorMap}).

Then the dynamical zeta function $\zeta_{\sigma_{A_{\ti}},\,\minus \phi\circsmall\pi_{\ti}} (s)$ converges on the open half-plane $\{s\in\C: \Re(s) > s_0 \}$. If, in addition, $\phi$ is not cohomologous to a constant in $\CCC( S^2 )$, then $\zeta_{\sigma_{A_{\ti}},\,\minus \phi\circsmall\pi_{\ti}} (s)$ extends to a non-vanishing holomorphic function on the closed half-plane $\{s\in\C: \Re(s) \geq s_0 \}$ except for the simple pole at $s=s_0$. 
\end{thmx}

The deduction of the information about $\zeta_{F,\,\minus \Phi}$ in Theorem~\ref{thmZetaAnalExt_InvC} from its symbolic version $\zeta_{\sigma_{A_{\ti}},\,\minus \phi\circsmall\pi_{\ti}}$ in Theorem~\ref{thmZetaAnalExt_SFT} is nontrivial in our setting, since the correspondence (i.e., the factor map) from the symbolic coding we have for expanding Thurston maps $F$ to the map $F$ itself is not bijective and the two formal infinite products $\zeta_{F,\,\minus \Phi}$ and $\zeta_{\sigma_{A_{\ti}},\,\minus \phi\circsmall\pi_{\ti}}$ differ in infinitely many terms. In fact, the correspondence is not even finite-to-one, and consequently, A.~Manning's argument used in the literature for symbolic codings that are finite-to-one does not apply here. To overcome this obstacle, we introduce and study in Section~\ref{sctDynOnC} the detailed combinatorial information of several auxiliary dynamical systems and introduce the dynamical Dirichlet series $\DS_{F,\,\minus\psi,\, \deg_F}$ (Definition~\ref{defDynDirichletSeries}) to address their precise relation to the symbolic dynamical system induced by tiles (see Proposition~\ref{propZetaProduct}), reducing the part of Theorem~\ref{thmZetaAnalExt_InvC} on $\DS_{F,\,\minus\psi,\, \deg_F}$ to Theorem~\ref{thmZetaAnalExt_SFT} first. More precisely, such a connection relies on three symbolic dynamical systems on the boundaries of the tiles with a careful study of the combinatorics of tiles (see Subsection~\ref{subsctDynOnC_Combinatorics}) with ideas from \cite{Li16}.

The information about $\zeta_{F,\,\minus \Phi}$ in Theorem~\ref{thmZetaAnalExt_InvC} can then be deduced from that about $\DS_{F,\,\minus\psi,\, \deg_F}$ (in the proof of Theorem~\ref{thmZetaAnalExt_InvC} in Section~\ref{sctDynOnC_Reduction}).   

It seems to be the first instance in the literature where such general dynamical Dirichlet series other than $L$-functions are crucially used.

The proof of Theorem~\ref{thmZetaAnalExt_SFT} relies on a characterization of the condition that a potential is cohomologous to a constant. In this investigation, the local integrability condition is introduced.

\begin{thmx}[Characterization of the local integrability condition]  \label{thmNLI}
Let $f\: S^2\rightarrow S^2$ be an expanding Thurston map and $d$ a visual metric on $S^2$ for $f$. Let $\psi \in \Holder{\alpha}((S^2,d),\C)$ be a complex-valued H\"{o}lder continuous function with an exponent $\alpha\in (0,1]$. Then the following statements are equivalent:

\begin{enumerate}
\smallskip
\item[(i)] The function $\psi$ is locally integrable (in the sense of Definition~\ref{defLI}).

\smallskip
\item[(ii)] There exists $n\in\N$ and a Jordan curve $\CC\subseteq S^2$ with $f^n(\CC)\subseteq \CC$ and $\post f\subseteq \CC$ such that
$
\bigl(S_n^f \psi\bigr)^{f^n,\,\CC}_{\xi,\,\eta}(x,y) = 0
$
for all $\xi=\{ \xi_{\minus i} \}_{i\in\N_0} \in  \Sigma_{f^n,\,\CC}^-$ and $\eta=\{ \eta_{\minus i} \}_{i\in\N_0} \in  \Sigma_{f^n,\,\CC}^-$ with $f^n(\xi_0) = f^n(\eta_0)$, and all $(x,y)\in \bigcup\limits_{\substack{X\in\X^1(f^n,\CC) \\ X\subseteq f^n(\xi_0)}}X \times X$.

\smallskip
\item[(iii)] The function $\psi$ is cohomologous to a constant in $\CCC(S^2,\C)$, i.e.,
$
	\psi= K+ \beta \circ f - \beta
$
for some $K\in\C$ and $\beta \in \CCC(S^2,\C)$.

\smallskip
\item[(iv)] The function $\psi$ is cohomologous to a constant in $\Holder{\alpha}((S^2,d),\C)$, i.e.,
$
	\psi= K+ \tau\circ f -  \tau
$
for some $K\in\C$ and $\tau\in \Holder{\alpha}((S^2,d),\C)$.

\smallskip
\item[(v)] There exists $n\in\N$ and a Jordan curve $\CC \subseteq S^2$ with $f^n(\CC) \subseteq \CC$ and $\post f \subseteq \CC$ such that the following statement holds for $F\coloneqq f^n$, $\Psi\coloneqq S_n^f \psi$, the one-sided subshift of finite type $\bigl( \Sigma_{A_{\ti}}^+, \sigma_{A_{\ti}} \bigr)$ associated to $F$ and $\CC$ defined in Proposition~\ref{propTileSFT}, and the factor map $\pi_\ti \: \Sigma_{A_{\ti}}^+ \rightarrow S^2$ defined in (\ref{eqDefTileSFTFactorMap}):

\smallskip

The function $\Psi \circ \pi_\ti$ is cohomologous to a constant multiple of an integer-valued continuous function in $\CCC\bigl( \Sigma_{A_{\ti}}^+, \C \bigr)$, i.e., 
$
	\Psi \circ \pi_\ti = K M + \varpi \circ \sigma_{A_{\ti}} - \varpi
$
for some $K\in\C$, $M \in \CCC\bigl( \Sigma_{A_{\ti}}^+ , \Z \bigr)$, and $\varpi \in \CCC\bigl( \Sigma_{A_{\ti}}^+ , \C \bigr)$.
\end{enumerate}

\smallskip

If, in addition, $\psi$ is real-valued, then the above statements are equivalent to
\begin{enumerate}
\smallskip
\item[(vi)] The equilibrium state $\mu_\psi$ for $f$ and $\psi$ is equal to the measure of maximal entropy $\mu_0$ of $f$.
\end{enumerate}
\end{thmx}

\subsection{Structure of the paper}      \label{subsctPlan}

We will now give a brief description of the structure of this paper.

After fixing some notation in Section~\ref{sctNotation}, we give a review of basic definitions and results in Section~\ref{sctPreliminaries}. In Section~\ref{sctAssumptions}, we state the assumptions on some of the objects in this paper, which we are going to repeatedly refer to later as \emph{the Assumptions}. 
The goal of Section~\ref{sctDynOnC} is to introduce and study the detailed combinatorial information of several auxiliary dynamical systems (Theorem~\ref{thmNoPeriodPtsIdentity}) and compare their dynamical complexity in terms of topological pressure (Subsection~\ref{subsctDynOnC_TopPressure}).
Section~\ref{sctNLI} is devoted to characterizations of a necessary condition, called \emph{non-local integrability condition}, on the potential $\phi$ for the Prime Orbit Theorems (Theorem~\ref{thmPrimeOrbitTheorem}) using the notion of orbifolds introduced in general by W.~P.~Thurston in 1970s in his study of geometry of $3$-manifolds (see \cite[Chapter~13]{Th80}). The lack of a fixed point for a general inverse branch on the universal orbifold covering space of an expanding Thurston map is the main obstacle addressed here. We provide a proof of Theorem~\ref{thmZetaAnalExt_InvC} in Section~\ref{sctDynOnC_Reduction} to deduce the holomorphic extension of $\DS_{F,\,\minus \Phi,\,\deg_F}$ from that of $\zeta_{\sigma_{A_{\ti}},\,\minus\Phi\circsmall\pi_{\ti}}$, and ultimately to deduce  the holomorphic extension of $\zeta_{ F,\,\minus\Phi}$ from that of $\DS_{F,\,\minus \Phi,\,\deg_F}$.

\subsection*{Acknowledgments} 
The first-named author is grateful to the Institute for Computational and Experimental Research in Mathematics (ICERM) at Brown University for the hospitality during his stay from February to May 2016, where he learned about this area of research while participating in the Semester Program ``Dimension and Dynamics'' as a research postdoctoral fellow. The authors want to express their gratitude to Mark~Pollicott for his introductory discussion on dynamical zeta functions and Prime Orbit Theorems in ICERM and many conversations since then, and to Hee~Oh for her beautiful talks on her work and helpful comments, while the first-named author also wants to thank Dennis~Sullivan for his encouragement to initiating this project, and Jianyu~Chen and  Polina~Vytnova for interesting discussions on related areas of research. Most of this work was done during the first-named author's stay at the Institute for Mathematical Sciences (IMS) at Stony Brook University as a postdoctoral fellow. He wants to thank IMS and his postdoctoral advisor Mikhail~Yu.~Lyubich for the great support and hospitality. The authors also would like to thank Gengrui Zhang for carefully reading the manuscript and pointing out several typos, and the anonymous referees for valuable suggestions to improve the paper.

\section{Notation} \label{sctNotation}
Let $\C$ be the complex plane and $\widehat{\C}$ be the Riemann sphere. For each complex number $z\in \C$, we denote by $\Re(z)$ the real part of $z$, and by $\Im(z)$ the imaginary part of $z$. We denote by $\D$ the open unit disk $\D \coloneqq \{z\in\C : \abs{z}<1 \}$ in the complex plane $\C$. For each $a\in\R$, we denote by $\H_a$ the open (right) half-plane $\H_a\coloneqq \{ z\in\C : \Re(z) > a\}$ in $\C$, and by $\overline{\H}_a$ the closed (right) half-plane $\overline{\H}_a \coloneqq \{ z\in\C : \Re(z) \geq  a\}$. We follow the convention that $\N \coloneqq \{1, \, 2, \, 3, \, \dots\}$, $\N_0 \coloneqq \{0\} \cup \N$, and $\widehat{\N} \coloneqq \N\cup \{+\infty\}$, with the order relations $<$, $\leq$, $>$, $\geq$ defined in the obvious way. For $x\in\R$, we define $\lfloor x\rfloor$ as the greatest integer $\leq x$, and $\lceil x \rceil$ the smallest integer $\geq x$. As usual, the symbol $\log$ denotes the logarithm to the base $e$, and $\log_c$ the logarithm to the base $c$ for $c>0$. The symbol $\I$ stands for the imaginary unit in the complex plane $\C$. The cardinality of a set $A$ is denoted by $\card{A}$. 

Consider real-valued functions $u$, $v$, and $w$ on $(0,+\infty)$. We write 
\begin{equation*}
	u(T) \sim v(T) \text{ as } T\to +\infty \qquad\text{ if } \qquad \lim_{T\to+\infty} \frac{u(T)}{v(T)} = 1,
\end{equation*}
and write 
\begin{equation*}
	u(T) = v(T) + \bigO ( w(T) ) \text{ as } T\to +\infty \qquad \text{ if } \qquad \limsup_{T\to+\infty} \Absbigg{ \frac{u(T) - v(T) }{w(T)} } < +\infty.
\end{equation*}

Let $g\: X\rightarrow Y$ be a map between two sets $X$ and $Y$. We denote the restriction of $g$ to a subset $Z$ of $X$ by $g|_Z$. Consider a map $f\: X\rightarrow X$ on a set $X$. The inverse map of $f$ is denoted by $f^{-1}$. We write $f^n$ for the $n$-th iterate of $f$, and $f^{-n} \coloneqq \left( f^n \right)^{-1}$, for $n\in\N$. We set $f^0 \coloneqq \id_X$, where the identity map $\id_X\: X\rightarrow X$ sends each $x\in X$ to $x$ itself. For each $n\in\N$, we denote by
\begin{equation} \label{eqDefSetPeriodicPts}
P_{n,f} \coloneqq \bigl\{ x\in X : f^n(x)=x, \,  f^k(x)\neq x,  \, k\in\{1, \, 2, \, \dots, \, n-1\} \bigr\}
\end{equation}
the \defn{set of periodic points of $f$ with (precise) periodic $n$}, and by
\begin{equation}  \label{eqDefSetPeriodicOrbits} 
\Orb(n,f) \coloneqq \bigl\{ \bigl\{f^i(x)  : i\in \{0, \, 1, \, \dots, \, n-1\} \bigr\} : x\in P_{n,f}  \bigr\}
\end{equation}
the \defn{set of primitive periodic orbits of $f$ with period $n$}. The set of all primitive periodic orbits of $f$ is denoted by
\begin{equation} \label{eqDefSetAllPeriodicOrbits} 
\Orb(f) \coloneqq \bigcup_{n=1}^{+\infty}  \Orb(n,f).
\end{equation}

Given a complex-valued function $\varphi\: X\rightarrow \C$, we write
\begin{equation}    \label{eqDefSnPt}
S_n \varphi (x)  = S_{n}^f \varphi (x)  \coloneqq \sum_{j=0}^{n-1} \varphi(f^j(x)) 
\end{equation}
for $x\in X$ and $n\in\N_0$. The superscript $f$ is often omitted when the map $f$ is clear from the context. Note that when $n=0$, by definition, we always have $S_0 \varphi = 0$.

Let $(X,d)$ be a metric space. For subsets $A,B\subseteq X$, we set $d(A,B) \coloneqq \inf \{d(x,y): x\in A,\,y\in B\}$, and $d(A,x)\coloneqq d(x,A) \coloneqq d(A,\{x\})$ for $x\in X$. For each subset $Y\subseteq X$, we denote the diameter of $Y$ by $\diam_d(Y) \coloneqq \sup\{d(x,y):x, \, y\in Y\}$, the interior of $Y$ by $\inter Y$, and the characteristic function of $Y$ by $\mathbbm{1}_Y$, which maps each $x\in Y$ to $1\in\R$ and vanishes otherwise. We use the convention that $\mathbbm{1}=\mathbbm{1}_X$ when the space $X$ is clear from the context. For each $r>0$ and each $x\in X$, we denote the open (resp.\ closed) ball of radius $r$ centered at $x$ by $B_d(x, r)$ (resp.\ $\overline{B_d}(x,r)$). 

We set $\CCC(X)$ (resp.\ $\CCC(X,\C)$) to be the space of continuous functions from $X$ to $\R$ (resp.\ $\C$). We adopt the convention that unless specifically referring to $\C$, we only consider real-valued functions. If we do not specify otherwise, we equip $\CCC(X)$ and $\CCC(X,\C)$ with the uniform norm $\Norm{\cdot}_{\CCC^0(X)}$. For a continuous map $g\: X \rightarrow X$, $\MMM(X,g)$ is the set of $g$-invariant Borel probability measures on $X$. %Unless otherwise specified, we equip $\MMM(X)$, $\PPP(X)$, and $\MMM(X,g)$ with the weak$^*$ topology.

The space of real-valued (resp.\ complex-valued) H\"{o}lder continuous functions with an exponent $\alpha\in (0,1]$ on a compact metric space $(X,d)$ is denoted by $\Holder{\alpha}(X,d)$ (resp.\ $\Holder{\alpha}((X,d),\C)$). For each $\psi\in\Holder{\alpha}((X,d),\C)$, we denote
\begin{equation}   \label{eqDef|.|alpha}
\Hseminorm{\alpha}{(X,d)}{\psi} \coloneqq \sup \{ \abs{\psi(x)- \psi(y)} / d(x,y)^\alpha : x, \, y\in X, \,x\neq y \},
\end{equation}
and the standard H\"{o}lder norm of $\psi$ is denoted by
\begin{equation}    \label{eqDefHolderNorm}
\Hnorm{\alpha}{\psi}{(X,d)} \coloneqq  \Hseminorm{\alpha}{(X,d)}{\psi}  + \Norm{\psi}_{\CCC^0(X)}.
\end{equation}

%For a Lipschitz map $g\: (X,d)\rightarrow (X,d)$ on a metric space $(X,d)$, we denote the Lipschitz constant by
%\begin{equation}   \label{eqDefLipConst}
%\LIP_d(g) \coloneqq \sup  \{ d(g(x),g(y)) / d(x,y) : x, \, y\in X \text{ with } x\neq y \}.
%\end{equation}

\section{Preliminaries}  \label{sctPreliminaries}

\subsection{Thermodynamic formalism}  \label{subsctThermodynFormalism}

We first review some basic concepts from dynamical systems. We refer the readers to \cite[Chapter~3]{PU10}, \cite[Chapter~9]{Wal82} or \cite[Chapter~20]{KH95} for more detailed studies of these concepts.

Let $(X,d)$ be a compact metric space and $g\:X\rightarrow X$ a continuous map. For $n\in\N$ and $x, \, y\in X$,
\begin{equation*}
d^n_g(x,y) \coloneqq \operatorname{max}\bigl\{  d \bigl(g^k(x),g^k(y)\bigr)  : k\in\{0, \, 1, \, \dots, \, n-1\} \!\bigr\}
\end{equation*}
defines a new metric on $X$. A set $F\subseteq X$ is \defn{$(n,\epsilon)$-separated}, for some $n\in\N$ and $\epsilon>0$, if for each pair of distinct points $x, \, y\in F$, we have $d^n_g(x,y)\geq \epsilon$. For $\epsilon > 0$ and $n\in\N$, let $F_n(\epsilon)$ be a maximal (in the sense of inclusion) $(n,\epsilon)$-separated set in $X$.

For each real-valued continuous function $\phi \in\CCC(X)$, the following limits exist and are equal, and we denote these limits by $P(g,\phi)$ (see for example, \cite[Theorem~3.3.2]{PU10}):
\begin{equation}  \label{defTopPressure}
P(g,\phi)  \coloneqq    \lim_{\epsilon\to 0} \limsup_{n\to+\infty} \frac{1}{n} \log  \sum_{x\in F_n(\epsilon)} \exp(S_n \phi(x))  
                          =           \lim_{\epsilon\to 0} \liminf_{n\to+\infty} \frac{1}{n} \log  \sum_{x\in F_n(\epsilon)} \exp(S_n \phi(x)), 
\end{equation}
where $S_n \phi (x) = \sum_{j=0}^{n-1} \phi\left(g^j(x)\right)$ is defined in (\ref{eqDefSnPt}). We call $P(g,\phi)$ the \defn{topological pressure} of $g$ with respect to the \emph{potential} $\phi$. The quantity $h_{\operatorname{top}}(g) \coloneqq P(g,0)$ is called the \defn{topological entropy} of $g$. Note that $P(g,\phi)$ is independent of $d$ as long as the topology on $X$ defined by $d$ remains the same (see \cite[Section~3.2]{PU10}).

A \defn{measurable partition} $\xi$ of $X$ is a cover $\xi=\{A_j:j\in J\}$ of $X$ consisting of countably many mutually disjoint Borel sets $A_j$, $j\in J$, where $J$ is a countable index set. 

%A \defn{cover} of $X$ is a collection $\xi=\{A_j : j\in J\}$ of subsets of $X$ with the property that $\bigcup\xi = X$, where $J$ is an index set. The cover $\xi$ is an \defn{open cover} if $A_j$ is an open set for each $j\in J$. The cover $\xi$ is \defn{finite} if the index set $J$ is a finite set.

Let $\xi=\{A_j:j\in J\}$ and $\eta=\{B_k:k\in K\}$ be two covers of $X$, where $J$ and $K$ are the corresponding index sets. We say $\xi$ is a \defn{refinement} of $\eta$ if for each $A_j\in\xi$, there exists $B_k\in\eta$ such that $A_j\subseteq B_k$. The \defn{common refinement} $\xi \vee \eta$ of $\xi$ and $\eta$ defined as
\begin{equation*}
\xi \vee \eta \coloneqq \{A_j\cap B_k : j\in J,\, k\in K\}
\end{equation*}
is also a cover. Note that if $\xi$ and $\eta$ are both open covers (resp., measurable partitions), then $\xi \vee \eta$ is also an open cover (resp., a measurable partition). Define $g^{-1}(\xi) \coloneqq \{g^{-1}(A_j) :j\in J\}$, and denote for $n\in\N$,
\begin{equation*}
\xi^n_g \coloneqq \bigvee_{j=0}^{n-1} g^{-j}(\xi) = \xi\vee g^{-1}(\xi)\vee\cdots\vee g^{-(n-1)}(\xi),
\end{equation*}
and let $\xi^\infty_g$ be the smallest $\sigma$-algebra containing $\bigcup_{n=1}^{+\infty}\xi^n_g$.

%Fix a measurable partition $\xi$ of $X$. For $x\in X$, we denote by $\xi(x)$ the unique element of $\xi$ that contains $x$. Let $\mu\in \MMM(X,g)$ be a $g$-invariant Borel probability measure on $X$. The \defn{information function} $I$ maps a measurable partition $\xi$ of $X$ to a $\mu$-a.e.\ defined real-valued function on $X$ in the following way:
%\begin{equation}   \label{eqDefI}
%I(\xi)(x) \coloneqq -\log \mu(\xi(x)), \qquad \text{for } x\in X.
%\end{equation} 
%
The \defn{entropy} of a measurable partition $\xi$ is
\begin{equation*}
H_{\mu}(\xi) \coloneqq -\sum_{j\in J} \mu(A_j)\log\left(\mu (A_j)\right),
\end{equation*}
where $0\log 0$ is defined to be 0. One can show (see \cite[Chapter 4]{Wal82}) that if $H_{\mu}(\xi)<+\infty$, then the following limit exists:
\begin{equation*}
h_{\mu}(g,\xi) \coloneqq \lim_{n\to+\infty} \frac{1}{n} H_{\mu}(\xi^n_g) \in[0,+\infty).
\end{equation*}

The \defn{measure-theoretic entropy} of $g$ for $\mu$ is given by
\begin{equation}   \label{eqDefMeasThEntropy}
h_{\mu}(g) \coloneqq \sup\{h_{\mu}(g,\xi): \xi   \text{ is a measurable partition of } X  \text{ with } H_{\mu}(\xi)<+\infty\}.   
\end{equation}
For each real-valued continuous function $\phi\in\CCC(X)$, the \defn{measure-theoretic pressure} $P_\mu(g,\phi)$ of $g$ for the measure $\mu$ and the potential $\phi$ is
\begin{equation}  \label{eqDefMeasTheoPressure}
P_\mu(g,\phi) \coloneqq  h_\mu (g) + \int \! \phi \,\mathrm{d}\mu.
\end{equation}

By the Variational Principle (see for example, \cite[Theorem~3.4.1]{PU10}), we have that for each $\phi\in\CCC(X)$,
\begin{equation}  \label{eqVPPressure}
P(g,\phi)=\sup\{P_\mu(g,\phi):\mu\in \MMM(X,g)\}.
\end{equation}
In particular, when $\phi$ is the constant function $0$,
\begin{equation}  \label{eqVPEntropy}
h_{\operatorname{top}}(g)=\sup\{h_{\mu}(g):\mu\in \MMM(X,g)\}.
\end{equation}
A measure $\mu$ that attains the supremum in (\ref{eqVPPressure}) is called an \defn{equilibrium state} for the map $g$ and the potential $\phi$. A measure $\mu$ that attains the supremum in (\ref{eqVPEntropy}) is called a \defn{measure of maximal entropy} of $g$.

Consider a continuous map $g\: X\rightarrow X$ on a compact metric space $(X,d)$ and a real-valued continuous potential $\varphi \in \CCC(X)$. By \cite[Theorem~3.3.2]{PU10}, our definition of the topological pressure $P(g,\varphi)$ in (\ref{defTopPressure}) coincides with the definition presented in \cite[Section~3.2]{PU10}. More precisely, combining (3.2.3), Definition~3.2.3, Lemmas~3.2.1, and~3.2.4 from \cite{PU10}, the topological pressure $P(g,\varphi)$ of $g$ with respect to $\varphi$ is also given by
\begin{equation}   \label{eqEquivDefByCoverForTopPressure}
 P(g,\varphi) 
            =  \lim_{m\to +\infty}  \lim_{n\to +\infty}  \frac{1}{n} \log 
                \inf\Biggl\{ \sum_{V\in\mathcal{V}}  \exp \Bigl( \sup_{x\in V}   S_n\varphi(x)   \Bigr)
                            :  \mathcal{V} \subseteq \bigvee_{i=0}^{n} g^{-i}(\xi_m),\,\bigcup \mathcal{V} = X  \Biggr\}    ,
\end{equation}
where $\{\xi_m\}_{m\in\N_0}$ is an arbitrary sequence of finite open covers of $X$ with 
\begin{equation*}
\lim_{m\to+\infty} \max \{ \diam_d(U) : U\in \xi_m \} = 0.
\end{equation*}

\begin{definition}   \label{defCohomologous}
Let $g\: X\rightarrow X$ be a continuous map on a metric space $(X,d)$. Let $\mathcal{K} \subseteq \CCC(X,\C)$ be a subspace of the space $\CCC(X,\C)$ of complex-valued continuous functions on $X$. Two functions $\phi, \, \psi \in \CCC(X,\C)$ are said to be \defn{cohomologous (in $\mathcal{K}$)} if there exists $u\in\mathcal{K}$ such that $\phi-\psi = u\circ g - u$.
\end{definition}

\subsection{Branched covering maps bewteen surfaces} \label{subsctBranchedCoveringMaps}
This paper is devoted to the discussion of expanding Thurston maps, which are branched covering maps on $S^2$ with certain expansion properties. We will discuss such branched covering maps in detail in Subsection~\ref{subsctThurstonMap}. However, since we are going to use lifting properties of branched covering maps and universal orbifold covers in Section~\ref{sctNLI}, we need to discuss briefly branched covering maps between surfaces in general here. For more detailed discussions on the concepts and results in this subsection, see \cite[Appendix~A.6]{BM17} and references therein. For a study of branched covering maps between more general topological spaces, see P.~Ha\"{\i}ssinsky and K.~M.~Pilgrim \cite{HP09}. 

%This subsection is only used in Section~\ref{sctNLI}. Relevant concepts and results adapted to the special case of branched covering maps on $S^2$ will be reiterated in Subsection~\ref{subsctThurstonMap}. The readers may safely skip this subsection in the first reading.

In this paper, a \defn{surface} is a connected and oriented $2$-dimensional topological manifold.

\begin{definition}[Branched covering maps between surfaces]   \label{defBranchedCover}
Let $X$ and $Y$ be (connected and oriented) surfaces, and $f\: X\rightarrow Y$ be a continuous map. Then $f$ is a \defn{branched covering map} (between $X$ and $Y$) if for each point $q\in Y$ there exists an open set $V \subseteq Y$ with $q\in V$ and there exists a collection $\{U_i\}_{i\in I}$ of open sets $U_i\subseteq X$ for some index set $I\neq \emptyset$ such that the following conditions are satisfied:
\begin{enumerate}
\smallskip
\item[(i)] $f^{-1}(V)$ is a disjoint union $f^{-1}(V) = \bigcup_{i\in I} U_i$,

\smallskip
\item[(ii)] $U_i$ contains precisely one point $p_i \in f^{-1}(q)$ for each $i\in I$, and

\smallskip
\item[(iii)] for each $i\in I$, there exists $d_i \in\N$, and orientation-preserving homeomorphisms $\varphi_i \: U_i\rightarrow\D$ and $\psi_i\: V\rightarrow\D$ with $\varphi_i(p_i) = 0$ and $\psi_i(q) = 0$ such that 
\begin{equation}   \label{eqBranchCoverMapLocalPowerMap}
\bigl( \psi_i \circ f \circ \varphi_i^{-1} \bigr) (z) = z^{d_i}
\end{equation}
for all $z\in\D$.
\end{enumerate}

The positive integer $d_i$ is called the \defn{local degree} of $f$ at $p \coloneqq p_i$, denoted by $\deg_f(p)$.
\end{definition}

Note that in Definition~\ref{defBranchedCover}, we do not require $X$ and $Y$ to be compact. In fact, we will need to use the non-compact case in Section~\ref{sctNLI}. The local degree $\deg_f(p_i)=d_i$ in Definition~\ref{defBranchedCover} is uniquely determined by $p\coloneqq p_i$. If $q'\in V$ is a point close to, but distinct from, $q=f(p)$, then $\deg_f(p)$ is equal to the number of distinct preimages of $q'$ under $f$ close to $p$. In particular, near $p$ (but not at $p$) the map $f$ is $d$-to-$1$, where $d=\deg_f(p)$.

Every branched covering map $f\: X\rightarrow Y$ is surjective, \defn{open} (i.e., images of open sets are open), and \defn{discrete} (i.e., the preimage set $f^{-1}(q)$ of every point $q\in Y$ has no limit points in $X$). Every covering map is also a branched covering map.

A \defn{critical point} of a branched covering map $f \: X\rightarrow Y$ is a point $p\in X$ with $\deg_f(p) \geq 2$. We denote the set of critical points of $f$ by $\crit f$. A \defn{critical value} is a point $q\in Y$ such that $f^{-1}(q)$ contains a critical point of $f$. The set of critical points of $f$ is \defn{discrete} in $X$ (i.e., it has no limit points in $X$), and the set of critical values of $f$ is discrete in $Y$. The map $f$ is an orientation-preserving local homeomorphism near each point $p\in X \setminus \crit f$.

Branched covering maps between surfaces behave well under compositions. We record the facts from Lemmas~A.16 and~A.17 in \cite{BM17} in the following lemma.

\begin{lemma}[Compositions of branched covering maps]   \label{lmBranchedCoverCompositionBM}
Let $X$, $Y$, and $Z$ be (connected and oriented) surfaces, and $f\: X\rightarrow Z$, $g\: Y\rightarrow Z$, and $h\: X\rightarrow Y$ be continuous maps such that $f=g\circ h$.
\begin{enumerate}
\smallskip
\item[(i)] If $g$ and $h$ are branched covering maps, and $Y$ and $Z$ are compact, then $f$ is also a branched covering map, and for each $x\in X$, we have
\begin{equation*}
\deg_f(x) = \deg_g(h(x)) \cdot \deg_h(x).
\end{equation*}

\smallskip
\item[(ii)] If $f$ and $g$ are branched covering maps, then $h$ is a branched covering map. Similarly, if $f$ and $h$ are branched covering maps, then $g$ is a branched covering map.
\end{enumerate}

%If, in addition, $X$, $Y$, and $Z$ are Riemann surfaces and the two branched covering maps in the hypotheses of \textup{(i)} or \textup{(ii)} are holomorphic, then the third map is also holomorphic.
\end{lemma}

Let $\pi\: X \rightarrow Y$ be a branched covering map, $Z$ a topological space, and $f\: Z\rightarrow Y$ be a continuous map. A continuous map $g\: Z\rightarrow X$ is called a \defn{lift} of $f$ (by $\pi$) if $\pi\circ g = f$. 

\begin{lemma}[Lifting paths by branched covering maps]  \label{lmLiftPathBM}
Let $X$ and $Y$ be (connected and oriented) surfaces, $\pi\: X\rightarrow Y$ be a branched covering map, $\gamma \: [0,1]\rightarrow Y$ be a path in $Y$, and $x_0\in \pi^{-1}(\gamma(0))$. Then there exists a path $\lambda\: [0,1]\rightarrow X$ with $\lambda(0)=x_0$ and $\pi\circ \lambda = \gamma$.
\end{lemma}

The above lemma can be found in \cite[Lemma~A.18]{BM17}. Branched covering maps are closely related to covering maps. The following lemma recorded from \cite[Lemma~A.11]{BM17} makes such a connection explicit.

\begin{lemma}   \label{lmBranchCoverToCover}
Let $X$ and $Y$ be (connected and oriented) surfaces, and $f\: X\rightarrow Y$ be a branched covering map. Suppose that $P\subseteq Y$ is a set with $f(\crit f) \subseteq P$ that is discrete in $Y$. Then $f\: X\setminus f^{-1}(P) \rightarrow Y \setminus P$ is a covering map.
\end{lemma}

\begin{comment}

We will use the lifting properties of covering maps in Section~\ref{sctNLI}. The proof and the terminology of the following lemma formulated as Lemma~A.6 in \cite{BM17} can be found in \cite[Section~1.3, Proposition~1.34, and Proposition~1.33]{Ha02} (see also \cite[Section~1.4 and Theorem~4.17]{Fo81}).

\begin{lemma}[Lifting by covering maps]   \label{lmLiftCoveringMap}
Let $X$ and $Y$ be (connected and oriented) surfaces, $\pi\: X\rightarrow Y$ be a covering map, and $Z$ be a path-connected and locally path-connected topological space.
\begin{enumerate}
\smallskip
\item[(i)] Suppose $g_1, g_2\: Z\rightarrow X$ are two continuous maps such that $\pi\circ g_1 = \pi\circ g_2$. If there exists $z_0\in Z$ with $g_1(z_0) = g_2(z_0)$, then $g_1=g_2$.

\smallskip
\item[(ii)] Suppose $Z$ is simply connected, $f\: Z\rightarrow Y$ is a continuous map, and $z_0\in Z$ and $x_0\in X$ are points such that $f(z_0)=\pi(x_0)$. Then there exists a continuous map $g\: Z\rightarrow X$ such that $g(z_0) = x_0$ and $f = \pi \circ g$.
\end{enumerate}
\end{lemma}

\end{comment}

\subsection{Thurston maps} \label{subsctThurstonMap}
In this subsection, we go over some key concepts and results on Thurston maps, and expanding Thurston maps in particular. For a more thorough treatment of the subject, we refer to \cite{BM17}.

Let $S^2$ denote an oriented topological $2$-sphere. A continuous map $f\:S^2\rightarrow S^2$ is called a \defn{branched covering map} on $S^2$ if for each point $x\in S^2$, there exists a positive integer $d\in \N$, open neighborhoods $U$ of $x$ and $V$ of $y=f(x)$, open neighborhoods $U'$ and $V'$ of $0$ in $\widehat{\C}$, and orientation-preserving homeomorphisms $\varphi\:U\rightarrow U'$ and $\eta\:V\rightarrow V'$ such that $\varphi(x)=0$, $\eta(y)=0$, and
\begin{equation*}
(\eta\circ f\circ\varphi^{-1})(z)=z^d
\end{equation*}
for each $z\in U'$. The positive integer $d$ above is called the \defn{local degree} of $f$ at $x$ and is denoted by $\deg_f (x)$. Note that the definition of branched covering maps on $S^2$ mentioned above is compatible with Definition~\ref{defBranchedCover}; see the discussion succeeding Lemma~A.10 in \cite{BM17} for more details.

The \defn{degree} of $f$ is
\begin{equation}   \label{eqDeg=SumLocalDegree}
\deg f=\sum_{x\in f^{-1}(y)} \deg_f (x)
\end{equation}
for $y\in S^2$ and is independent of $y$. If $f\:S^2\rightarrow S^2$ and $g\:S^2\rightarrow S^2$ are two branched covering maps on $S^2$, then so is $f\circ g$, and
\begin{equation} \label{eqLocalDegreeProduct}
 \deg_{f\circsmall g}(x) = \deg_g(x)\deg_f(g(x)), \qquad \text{for each } x\in S^2,
\end{equation}   
and moreover, 
\begin{equation}  \label{eqDegreeProduct}
\deg(f\circ g) =  (\deg f)( \deg g).
\end{equation}

A point $x\in S^2$ is a \defn{critical point} of $f$ if $\deg_f(x) \geq 2$. The set of critical points of $f$ is denoted by $\crit f$. A point $y\in S^2$ is a \defn{postcritical point} of $f$ if $y = f^n(x)$ for some $x\in\crit f$ and $n\in\N$. The set of postcritical points of $f$ is denoted by $\post f$. Note that $\post f=\post f^n$ for all $n\in\N$.

\begin{definition} [Thurston maps] \label{defThurstonMap}
A Thurston map is a branched covering map $f\:S^2\rightarrow S^2$ on $S^2$ with $\deg f\geq 2$ and $\card(\post f)<+\infty$.
\end{definition}

We now recall the notation for cell decompositions of $S^2$ used in \cite{BM17} and \cite{Li17}. A \defn{cell of dimension $n$} in $S^2$, $n \in \{1, \, 2\}$, is a subset $c\subseteq S^2$ that is homeomorphic to the closed unit ball $\overline{\B^n}$ in $\R^n$. We define the \defn{boundary of $c$}, denoted by $\partial c$, to be the set of points corresponding to $\partial\B^n$ under such a homeomorphism between $c$ and $\overline{\B^n}$. The \defn{interior of $c$} is defined to be $\inte (c) = c \setminus \partial c$. For each point $x\in S^2$, the set $\{x\}$ is considered as a \defn{cell of dimension $0$} in $S^2$. For a cell $c$ of dimension $0$, we adopt the convention that $\partial c=\emptyset$ and $\inte (c) =c$. 

We record the following three definitions from \cite{BM17}.

\begin{definition}[Cell decompositions]\label{defcelldecomp}
Let $\DD$ be a collection of cells in $S^2$.  We say that $\DD$ is a \defn{cell decomposition of $S^2$} if the following conditions are satisfied:

\begin{itemize}

\smallskip
\item[(i)]
the union of all cells in $\DD$ is equal to $S^2$,

\smallskip
\item[(ii)] if $c\in \DD$, then $\partial c$ is a union of cells in $\DD$,

\smallskip
\item[(iii)] for $c_1, \, c_2 \in \DD$ with $c_1 \neq c_2$, we have $\inte (c_1) \cap \inte (c_2)= \emptyset$,  

\smallskip
\item[(iv)] every point in $S^2$ has a neighborhood that meets only finitely many cells in $\DD$.

\end{itemize}
\end{definition}

\begin{definition}[Refinements]\label{defrefine}
Let $\DD'$ and $\DD$ be two cell decompositions of $S^2$. We
say that $\DD'$ is a \defn{refinement} of $\DD$ if the following conditions are satisfied:
\begin{itemize}

\smallskip
\item[(i)] every cell $c\in \DD$ is the union of all cells $c'\in \DD'$ with $c'\subseteq c$,

\smallskip
\item[(ii)] for every cell $c'\in \DD'$ there exits a cell $c\in \DD$ with $c'\subseteq c$.

\end{itemize}
\end{definition}

\begin{definition}[Cellular maps and cellular Markov partitions]\label{defcellular}
Let $\DD'$ and $\DD$ be two cell decompositions of  $S^2$. We say that a continuous map $f \: S^2 \rightarrow S^2$ is \defn{cellular} for  $(\DD', \DD)$ if for every cell $c\in \DD'$, the restriction $f|_c$ of $f$ to $c$ is a homeomorphism of $c$ onto a cell in $\DD$. We say that $(\DD',\DD)$ is a \defn{cellular Markov partition} for $f$ if $f$ is cellular for $(\DD',\DD)$ and $\DD'$ is a refinement of $\DD$.
\end{definition}

Let $f\:S^2 \rightarrow S^2$ be a Thurston map, and $\CC\subseteq S^2$ be a Jordan curve containing $\post f$. Then the pair $f$ and $\CC$ induces natural cell decompositions $\DD^n(f,\CC)$ of $S^2$, for $n\in\N_0$, in the following way:

By the Jordan curve theorem, the set $S^2\setminus\CC$ has two connected components. We call the closure of one of them the \defn{white $0$-tile} for $(f,\CC)$, denoted by $X^0_\w$, and the closure of the other the \defn{black $0$-tile} for $(f,\CC)$, denoted by $X^0_\b$. The set of \defn{$0$-tiles} is $\X^0(f,\CC) \coloneqq \bigl\{ X_\b^0, \, X_\w^0 \bigr\}$. The set of \defn{$0$-vertices} is $\V^0(f,\CC) \coloneqq \post f$. We set $\overline\V^0(f,\CC) \coloneqq \bigl\{ \{x\} : x\in \V^0(f,\CC) \bigr\}$. The set of \defn{$0$-edges} $\E^0(f,\CC)$ is the set of the closures of the connected components of $\CC \setminus  \post f$. Then we get a cell decomposition 
\begin{equation*}
\DD^0(f,\CC) \coloneqq \X^0(f,\CC) \cup \E^0(f,\CC) \cup \overline\V^0(f,\CC)
\end{equation*}
of $S^2$ consisting of \emph{cells of level $0$}, or \defn{$0$-cells}.

We can recursively define unique cell decompositions $\DD^n(f,\CC)$, $n\in\N$, consisting of \defn{$n$-cells} such that $f$ is cellular for $\bigl( \DD^{n+1}(f,\CC),\DD^n(f,\CC) \bigr)$. We refer to \cite[Lemma~5.12]{BM17} for more details. We denote by $\X^n(f,\CC)$ the set of $n$-cells of dimension 2, called \defn{$n$-tiles}; by $\E^n(f,\CC)$ the set of $n$-cells of dimension 1, called \defn{$n$-edges}; by $\overline\V^n(f,\CC)$ the set of $n$-cells of dimension 0; and by $\V^n(f,\CC)$ the set $\bigl\{x : \{x\}\in \overline\V^n(f,\CC)\bigr\}$, called the set of \defn{$n$-vertices}. The \defn{$k$-skeleton}, for $k\in\{0, \, 1, \, 2\}$, of $\DD^n(f,\CC)$ is the union of all $n$-cells of dimension $k$ in this cell decomposition. 

We record Proposition~5.16 of \cite{BM17} here in order to summarize properties of the cell decompositions $\DD^n(f,\CC)$ defined above.

\begin{prop}[M.~Bonk \& D.~Meyer \cite{BM17}] \label{propCellDecomp}
Let $k, \, n\in \N_0$, let   $f\: S^2\rightarrow S^2$ be a Thurston map,  $\CC\subseteq S^2$ be a Jordan curve with $\post f \subseteq \CC$, and   $m=\card(\post f)$. 
 
\smallskip
\begin{itemize}

\smallskip
\item[(i)] The map  $f^k$ is cellular for $\bigl( \DD^{n+k}(f,\CC), \DD^n(f,\CC) \bigr)$. In particular, if  $c$ is any $(n+k)$-cell, then $f^k(c)$ is an $n$-cell, and $f^k|_c$ is a homeomorphism of $c$ onto $f^k(c)$.

\smallskip
\item[(ii)]  Let  $c$ be  an $n$-cell.  Then $f^{-k}(c)$ is equal to the union of all 
$(n+k)$-cells $c'$ with $f^k(c')=c$.

\smallskip
\item[(iii)] The $1$-skeleton of $\DD^n(f,\CC)$ is  equal to  $f^{-n}(\CC)$. The $0$-skeleton of $\DD^n(f,\CC)$ is the set $\V^n(f,\CC)=f^{-n}(\post f )$, and we have $\V^n(f,\CC) \subseteq \V^{n+k}(f,\CC)$. 

\smallskip
\item[(iv)] $\card(\X^n(f,\CC))=2(\deg f)^n$,  $\card(\E^n(f,\CC))=m(\deg f)^n$,  and $\card (\V^n(f,\CC)) \leq m (\deg f)^n$.

\smallskip
\item[(v)] The $n$-edges are precisely the closures of the connected components of $f^{-n}(\CC)\setminus f^{-n}(\post f )$. The $n$-tiles are precisely the closures of the connected components of $S^2\setminus f^{-n}(\CC)$.

\smallskip
\item[(vi)] Every $n$-tile  is an $m$-gon, i.e., the number of $n$-edges and the number of $n$-vertices contained in its boundary are equal to $m$.  

\smallskip
\item[(vii)] Let $F\coloneqq f^k$ be an iterate of $f$ with $k \in \N$. Then $\DD^n(F,\CC) = \DD^{nk}(f,\CC)$.
\end{itemize}
\end{prop}

We also note that for each $n$-edge $e\in\E^n(f,\CC)$, $n\in\N_0$, there exist exactly two $n$-tiles $X, \, X'\in\X^n(f,\CC)$ such that $X\cap X' = e$.

For $n\in \N_0$, we define the \defn{set of black $n$-tiles} as
\begin{equation*}
\X_\b^n(f,\CC) \coloneqq \bigl\{X\in\X^n(f,\CC) :  f^n(X)=X_\b^0 \bigr\},
\end{equation*}
and the \defn{set of white $n$-tiles} as
\begin{equation*}
\X_\w^n(f,\CC) \coloneqq \bigl\{X\in\X^n(f,\CC) : f^n(X)=X_\w^0 \bigr\}.
\end{equation*}
It follows immediately from Proposition~\ref{propCellDecomp} that
\begin{equation}   \label{eqCardBlackNTiles}
\card ( \X_\b^n(f,\CC) ) = \card  (\X_\w^n(f,\CC) ) = (\deg f)^n 
\end{equation}
for each $n\in\N_0$.

From now on, if the map $f$ and the Jordan curve $\CC$ are evident from the context, we will sometimes omit $(f,\CC)$ in the notation above.

If we fix the cell decomposition $\DD^n(f,\CC)$, $n\in\N_0$, we can define for each $v\in \V^n$ the \defn{$n$-flower of $v$} as
\begin{equation}   \label{defFlower}
W^n(v) \coloneqq \bigcup  \{\inte (c) : c\in \DD^n,\, v\in c \}.
\end{equation}
Note that flowers are open (in the standard topology on $S^2$). Let $\overline{W}^n(v)$ be the closure of $W^n(v)$. We define the \defn{set of all $n$-flowers} by
\begin{equation}   \label{defSetNFlower}
\W^n \coloneqq \{W^n(v) : v\in\V^n\}.
\end{equation}
\begin{rem}  \label{rmFlower}
For $n\in\N_0$ and $v\in\V^n$, we have 
\begin{equation*}
\overline{W}^n(v)=X_1\cup X_2\cup \cdots \cup X_m,
\end{equation*}
where $m \coloneqq 2\deg_{f^n}(v)$, and $X_1, X_2, \dots X_m$ are all the $n$-tiles that contain $v$ as a vertex (see \cite[Lemma~5.28]{BM17}). Moreover, each flower is mapped under $f$ to another flower in such a way that is similar to the map $z\mapsto z^k$ on the complex plane. More precisely, for $n\in\N_0$ and $v\in \V^{n+1}$, there exist orientation preserving homeomorphisms $\varphi\: W^{n+1}(v) \rightarrow \D$ and $\eta\: W^{n}(f(v)) \rightarrow \D$ such that $\D$ is the unit disk on $\C$, $\varphi(v)=0$, $\eta(f(v))=0$, and 
\begin{equation*}
(\eta\circ f \circ \varphi^{-1}) (z) = z^k
\end{equation*}
for all $z\in \D$, where $k \coloneqq \deg_f(v)$. Let $\overline{W}^{n+1}(v)= X_1\cup X_2\cup \cdots \cup X_m$ and $\overline{W}^n(f(v))= X'_1\cup X'_2\cup \cdots \cup X'_{m'}$, where $X_1, X_2, \dots X_m$ are all the $(n+1)$-tiles that contain $v$ as a vertex, listed counterclockwise, and $X'_1, X'_2, \dots X'_{m'}$ are all the $n$-tiles that contain $f(v)$ as a vertex, listed counterclockwise, and $f(X_1)=X'_1$. Then $m= m'k$, and $f(X_i)=X'_j$ if $i\equiv j \pmod{k}$, where $k=\deg_f(v)$. (See also Case~3 of the proof of Lemma~5.24 in \cite{BM17} for more details.) In particular, the flower $W^n(v)$ is simply connected.
\end{rem}

We denote, for each $x\in S^2$ and $n\in\Z$,
\begin{equation}  \label{defU^n}
U^n(x) \coloneqq \bigcup \{Y^n\in \X^n :    \text{there exists } X^n\in\X^n  
                                        \text{ with } x\in X^n, \, X^n\cap Y^n \neq \emptyset  \}  
\end{equation}
if $n\geq 0$, and set $U^n(x) \coloneqq S^2$ otherwise. %We define the \defn{$n$-partition} $O_n$ of $S^2$ induced by $(f,\CC)$ as
%\begin{equation}   \label{defn-partition}
%O_n \coloneqq \{\inte (X^n):X^n\in\X^n\} \cup \{\inte (e^n) :e^n\in\E^n\} \cup \overline{\V}^n.
%\end{equation}

We can now give a definition of expanding Thurston maps.

\begin{definition} [Expansion] \label{defExpanding}
A Thurston map $f\:S^2\rightarrow S^2$ is called \defn{expanding} if there exists a metric $d$ on $S^2$ that induces the standard topology on $S^2$ and a Jordan curve $\CC\subseteq S^2$ containing $\post f$ such that 
\begin{equation*}
\lim_{n\to+\infty}\max \{\diam_d(X) : X\in \X^n(f,\CC)\}=0.
\end{equation*}
\end{definition}

\begin{rems}  \label{rmExpanding}
It is clear from Proposition~\ref{propCellDecomp}~(vii) and Definition~\ref{defExpanding} that if $f$ is an expanding Thurston map, so is $f^n$ for each $n\in\N$. We observe that being expanding is a topological property of a Thurston map and independent of the choice of the metric $d$ that generates the standard topology on $S^2$. By Lemma~6.2 in \cite{BM17}, it is also independent of the choice of the Jordan curve $\CC$ containing $\post f$. More precisely, if $f$ is an expanding Thurston map, then
\begin{equation*}
\lim_{n\to+\infty}\max \!\bigl\{ \! \diam_{\wt{d}}(X) : X\in \X^n\bigl(f,\wt\CC \hspace{0.5mm}\bigr)\hspace{-0.3mm} \bigr\}\hspace{-0.3mm}=0,
\end{equation*}
for each metric $\wt{d}$ that generates the standard topology on $S^2$ and each Jordan curve $\wt\CC\subseteq S^2$ that contains $\post f$.
\end{rems}

P.~Ha\"{\i}ssinsky and K.~M.~Pilgrim developed a notion of expansion in a more general context for finite branched coverings between topological spaces (see \cite[Sections~2.1 and~2.2]{HP09}). This applies to Thurston maps, and their notion of expansion is equivalent to our notion defined above in the context of Thurston maps (see \cite[Proposition~6.4]{BM17}). Such concepts of expansion are natural analogs, in the contexts of finite branched coverings and Thurston maps, to some of the more classical versions, such as expansive homeomorphisms and forward-expansive continuous maps between compact metric spaces (see for example, \cite[Definition~3.2.11]{KH95}), and distance-expanding maps between compact metric spaces (see for example, \cite[Chapter~4]{PU10}). Our notion of expansion is not equivalent to any such classical notion in the context of Thurston maps. One topological obstruction comes from the presence of critical points for (non-homeomorphic) branched covering maps on $S^2$. In fact, there are subtle connections between our notion of expansion and some classical notions of weak expansion. More precisely, one can prove that an expanding Thurston map is asymptotically $h$-expansive if and only if it has no periodic points. Moreover, such a map is never $h$-expansive. See \cite{Li15} for details.

For an expanding Thurston map $f$, we can fix a particular metric $d$ on $S^2$ called a \emph{visual metric for $f$}. For the existence and properties of such metrics, see \cite[Chapter~8]{BM17}. For a visual metric $d$ for $f$, there exists a unique constant $\Lambda > 1$ called the \emph{expansion factor} of $d$ (see \cite[Chapter~8]{BM17} for more details). One major advantage of a visual metric $d$ is that in $(S^2,d)$, we have good quantitative control over the sizes of the cells in the cell decompositions discussed above. We summarize several results of this type (\cite[Proposition~8.4, Lemmas~8.10,~8.11]{BM17}) in the lemma below.

\begin{lemma}[M.~Bonk \& D.~Meyer \cite{BM17}]   \label{lmCellBoundsBM}
Let $f\:S^2 \rightarrow S^2$ be an expanding Thurston map, and $\CC \subseteq S^2$ be a Jordan curve containing $\post f$. Let $d$ be a visual metric on $S^2$ for $f$ with expansion factor $\Lambda>1$. Then there exist constants $C\geq 1$, $C'\geq 1$, $K\geq 1$, and $n_0\in\N_0$ with the following properties:
\begin{enumerate}
\smallskip
\item[(i)] $d(\sigma,\tau) \geq C^{-1} \Lambda^{-n}$ whenever $\sigma$ and $\tau$ are disjoint $n$-cells for $n\in \N_0$.

\smallskip
\item[(ii)] $C^{-1} \Lambda^{-n} \leq \diam_d(\tau) \leq C\Lambda^{-n}$ for all $n$-edges and all $n$-tiles $\tau$ for $n\in\N_0$.

\smallskip
\item[(iii)] $B_d(x,K^{-1} \Lambda^{-n} ) \subseteq U^n(x) \subseteq B_d(x, K\Lambda^{-n})$ for $x\in S^2$ and $n\in\N_0$.

\smallskip
\item[(iv)] $U^{n+n_0} (x)\subseteq B_d(x,r) \subseteq U^{n-n_0}(x)$ where $n= \lceil -\log r / \log \Lambda \rceil$ for $r>0$ and $x\in S^2$.

\smallskip
\item[(v)] For every $n$-tile $X^n\in\X^n(f,\CC)$, $n\in\N_0$, there exists a point $p\in X^n$ such that $B_d(p,C^{-1}\Lambda^{-n}) \subseteq X^n \subseteq B_d(p,C\Lambda^{-n})$.
\end{enumerate}

Conversely, if $\wt{d}$ is a metric on $S^2$ satisfying conditions \textnormal{(i)} and \textnormal{(ii)} for some constant $C\geq 1$, then $\wt{d}$ is a visual metric with expansion factor $\Lambda>1$.
\end{lemma}

Recall that $U^n(x)$ is defined in (\ref{defU^n}).

In addition, we will need the fact that a visual metric $d$ induces the standard topology on $S^2$ (\cite[Proposition~8.3]{BM17}) and the fact that the metric space $(S^2,d)$ is linearly locally connected (\cite[Proposition~18.5]{BM17}). A metric space $(X,d)$ is \defn{linearly locally connected} if there exists a constant $L\geq 1$ such that the following conditions are satisfied:
\begin{enumerate}
\smallskip

\item  For all $z\in X$, $r > 0$, and $x, \, y\in B_d(z,r)$ with $x\neq y$, there exists a continuum $E\subseteq X$ with $x, \, y\in E$ and $E\subseteq B_d(z,rL)$.

\smallskip

\item For all $z\in X$, $r > 0$, and $x, \, y\in X \setminus B_d(z,r)$ with $x\neq y$, there exists a continuum $E\subseteq X$ with $x, \, y\in E$ and $E\subseteq X \setminus B_d(z,r/L)$.
\end{enumerate}
We call such a constant $L \geq 1$ a \defn{linear local connectivity constant of $d$}. 

In fact, visual metrics serve a crucial role in connecting the dynamical arguments with geometric properties for rational expanding Thurston maps, especially Latt\`{e}s maps.

\begin{definition}  \label{defQuasiSymmetry}
Consider two metric spaces $(X_1,d_1)$ and $(X_2,d_2)$. Let $g\: X_1 \rightarrow X_2$ be a homeomorphism. Then $g$ is a \defn{quasisymmetric homeomorphism} or a \defn{quasisymmetry} if there exists a homeomorphism $\eta \: [0, +\infty) \rightarrow [0, +\infty)$ such that for all $u,\,v,\,w \in X_1$,
\begin{equation*}
\frac{ d_2 ( g(u), g(v) ) }{ d_2 ( g(u), g(w) ) } 
\leq \eta \biggl(  \frac{ d_1 ( u, v ) }{ d_1 ( u, w ) }  \biggr).
\end{equation*}
Moreover, the metric spaces $(X_1,d_1)$ and $(X_2,d_2)$ are \defn{quasisymmetrically equivalent} if there exists a quasisymmetric homeomorphism from $(X_1,d_1)$ to $(X_2,d_2)$.

When $X_1 = X_2 \eqqcolon X$, then we say the metrics $d_1$ and $d_2$ are \defn{quasisymmetrically equivalent} if the identity map from $(X,d_1)$ to $(X,d_2)$ is a quasisymmetric homeomorphism.
\end{definition}

\begin{theorem}[M.~Bonk \& D.~Meyer \cite{BM10, BM17}, P.~Ha\"issinsky \& K.~M.~Pilgrim \cite{HP09}]  \label{thmBM}
An expanding Thurston map is conjugate to a rational map if and only if the sphere $(S^2,d)$ equipped with a visual metric $d$ is quasisymmetrically equivalent to the Riemann sphere $\widehat\C$ equipped with the chordal metric.
\end{theorem}   

See \cite[Theorem~18.1~(ii)]{BM17} for a proof. The chordal metric is recalled below.

\begin{rem}   \label{rmChordalVisualQSEquiv}
If $f\: \widehat\C \rightarrow \widehat\C$ is a rational expanding Thurston map (or equivalently, a postcritically-finite rational map without periodic critical points (see \cite[Proposition~2.3]{BM17})), then each visual metric is quasisymmetrically equivalent to the chordal metric on the Riemann sphere $\widehat\C$ (see \cite[Lemma~18.10]{BM17}). Here the chordal metric $\sigma$ on $\widehat\C$ is given by
$
\sigma(z,w) =\frac{2\abs{z-w}}{\sqrt{1+\abs{z}^2} \sqrt{1+\abs{w}^2}}
$
for $z, \, w\in\C$, and $\sigma(\infty,z)=\sigma(z,\infty)= \frac{2}{\sqrt{1+\abs{z}^2}}$ for $z\in \C$. We also note that quasisymmetric embeddings of bounded connected metric spaces are H\"{o}lder continuous (see \cite[Section~11.1 and Corollary~11.5]{He01}). Accordingly, the class of H\"{o}lder continuous functions on $\widehat\C$ equipped with the chordal metric and that on $S^2=\widehat\C$ equipped with any visual metric for $f$ are the same (up to a change of the H\"{o}lder exponent). We also note that the chordal metric is bi-Lipschitz equivalent to the spherical metric.
\end{rem}

A Jordan curve $\CC\subseteq S^2$ is \defn{$f$-invariant} if $f(\CC)\subseteq \CC$. We are interested in $f$-invariant Jordan curves that contain $\post f$, since for such a Jordan curve $\CC$, we get a cellular Markov partition $(\DD^1(f,\CC),\DD^0(f,\CC))$ for $f$. According to Example~15.11 in \cite{BM17}, such $f$-invariant Jordan curves containing $\post{f}$ need not exist. However, M.~Bonk and D.~Meyer \cite[Theorem~15.1]{BM17} proved that there exists an $f^n$-invariant Jordan curve $\CC$ containing $\post{f}$ for each sufficiently large $n$ depending on $f$. A slightly stronger version of this result was proved in \cite[Lemma~3.11]{Li16}, and we record it below.

\begin{lemma}[M.~Bonk \& D.~Meyer \cite{BM17}, Z.~Li \cite{Li16}]  \label{lmCexistsL}
Let $f\:S^2\rightarrow S^2$ be an expanding Thurston map, and $\wt{\CC}\subseteq S^2$ be a Jordan curve with $\post f\subseteq \wt{\CC}$. Then there exists an integer $N(f,\wt{\CC}) \in \N$ such that for each $n\geq N(f,\wt{\CC})$ there exists an $f^n$-invariant Jordan curve $\CC$ isotopic to $\wt{\CC}$ rel.\ $\post f$ such that no $n$-tile in $\X^n(f,\CC)$ joins opposite sides of $\CC$.
\end{lemma}

The phrase ``joining opposite sides'' has a specific meaning in our context. 

\begin{definition}[Joining opposite sides]  \label{defJoinOppositeSides} 
Consider a Thurston map $f$ with $\card(\post f) \geq 3$ and an $f$-invariant Jordan curve $\CC$ containing $\post f$.  A set $K\subseteq S^2$ \defn{joins opposite sides} of $\CC$ if $K$ meets two disjoint $0$-edges when $\card( \post f)\geq 4$, or $K$ meets  all  three $0$-edges when $\card(\post f)=3$. 
 \end{definition}
 
Note that $\card (\post f) \geq 3$ for each expanding Thurston map $f$ \cite[Lemma~6.1]{BM17}.

We now summarize some basic properties of expanding Thurston maps in the following theorem.

\begin{theorem}[Z.~Li \cite{Li18}, \cite{Li16}]   \label{thmETMBasicProperties}
Let $f\:S^2 \rightarrow S^2$ be an expanding Thurston map, and $d$ be a visual metric on $S^2$ for $f$ with expansion factor $\Lambda>1$. Then the following statements are satisfied:
\begin{enumerate}
\smallskip
\item[(i)] The map $f$ is Lipschitz with respect to $d$.

\smallskip
\item[(ii)] The map $f$ has $1+ \deg f$ fixed points, counted with weight given by the local degree of the map at each fixed point. In particular, $\sum_{x\in P_{1,f^n}} \deg_{f^n}(x) = 1 + \deg f^n$.
\end{enumerate}
\end{theorem}

Theorem~\ref{thmETMBasicProperties}~(i) was shown in \cite[Lemma~3.12]{Li18}. Theorem~\ref{thmETMBasicProperties}~(ii) follows from \cite[Theorem~1.1]{Li16} and Remark~\ref{rmExpanding}.

We record the following lemma from \cite[Lemma~4.1]{Li16}), which gives us almost precise information on the locations of the periodic points of an expanding Thurston map.

\begin{lemma}[Z.~Li \cite{Li16}]  \label{lmAtLeast1}
Let $f$ be an expanding Thurston map with an $f$-invariant Jordan curve  $\CC$ containing $\post f$. Consider the cell decompositions $\DD^n(f,\CC)$. If $X$ is a white $1$-tile contained in the while $0$-tile $X^0_\w$ or a black $1$-tile contained in the black $0$-tile $X^0_\b$, then $X$ contains at least one fixed point of $f$. If $X$ is a white $1$-tile contained in the black $0$-tile $X^0_\b$ or a black $1$-tile contained in the white $0$-tile $X^0_\w$, then $\inte (X)$ contains no fixed points of $f$.
\end{lemma}

\begin{comment}

Recall that cells in the cell decompositions are, by definition, closed sets, and the set of $0$-tiles $\X^0(f,\CC)$ consists of the white $0$-tile $X^0_\w$ and the black $0$-tile $X^0_\b$.

\begin{lemma}[Z.~Li \cite{Li16}]   \label{lmAtMost1}
Let $f$ be an expanding Thurston map with an $f$-invariant Jordan curve  $\CC$ containing $\post f$ such that no $1$-tile in $\X^1(f,\CC)$ joins opposite sides of $\CC$. Then for every $n\in\N$, each $n$-tile $X^n \in\X^n(f,\CC)$ contains at most one fixed point of $f^n$.
\end{lemma}

The following lemma proved in \cite[Lemma~3.13]{Li18} generalizes \cite[Lemma~15.25]{BM17}.

\begin{lemma}[M.~Bonk \& D.~Meyer \cite{BM17}, Z.~Li \cite{Li18}]   \label{lmMetricDistortion}
Let $f\:S^2 \rightarrow S^2$ be an expanding Thurston map, and $\CC \subseteq S^2$ be a Jordan curve that satisfies $\post f \subseteq \CC$ and $f^{n_\CC}(\CC)\subseteq\CC$ for some $n_\CC\in\N$. Let $d$ be a visual metric on $S^2$ for $f$ with expansion factor $\Lambda>1$. Then there exists a constant $C_0 > 1$, depending only on $f$, $d$, $\CC$, and $n_\CC$, with the following property:

If $k, \, n\in\N_0$, $X^{n+k}\in\X^{n+k}(f,\CC)$, and $x, \, y\in X^{n+k}$, then 
\begin{equation}   \label{eqMetricDistortion}
C_0^{-1} d(x,y) \leq \Lambda^{-n}  d(f^n(x),f^n(y)) \leq C_0 d(x,y).
\end{equation}
\end{lemma}

\end{comment}

\subsection{Ergodic theory of expanding Thurston maps}    \label{subsctErgodicTheoryETM}

We summarize the existence, uniqueness, and some basic properties of equilibrium states for expanding Thurston maps in the following theorem.

\begin{theorem}[Z.~Li \cite{Li18}]   \label{thmEquilibriumState}
Let $f\:S^2 \rightarrow S^2$ be an expanding Thurston map and $d$ be a visual metric on $S^2$ for $f$. Let $\phi, \, \gamma\in \Holder{\alpha}(S^2,d)$ be real-valued H\"{o}lder continuous functions with an exponent $\alpha\in(0,1]$. Then the following statements are satisfied:
\begin{enumerate}
\smallskip
\item[(i)] There exists a unique equilibrium state $\mu_\phi$ for the map $f$ and the potential $\phi$.

\smallskip
\item[(ii)] For each $t_0\in\R$, we have
$
\frac{\mathrm{d}}{\mathrm{d}t} P(f,\phi + t\gamma)|_{t=t_0} = \int \!\gamma \,\mathrm{d}\mu_{\phi + t_0 \gamma}.
$

\smallskip
\item[(iii)] If $\CC \subseteq S^2$ is a Jordan curve containing $\post f$ with the property that $f^{n_\CC}(\CC)\subseteq \CC$ for some $n_\CC\in\N$, then
$
\mu_\phi \bigl( \, \bigcup_{i=0}^{+\infty}  f^{-i}(\CC)  \bigr)  = 0.
$
\end{enumerate}
\end{theorem}

Theorem~\ref{thmEquilibriumState}~(i) is part of \cite[Theorem~1.1]{Li18}. Theorem~\ref{thmEquilibriumState}~(ii) follows immediately from \cite[Theorem~6.13]{Li18} and the uniqueness of equilibrium states in Theorem~\ref{thmEquilibriumState}~(i). Theorem~\ref{thmEquilibriumState}~(iii) was established in \cite[Proposition~7.1]{Li18}.

The following distortion lemma serves as the cornerstone in the development of thermodynamic formalism for expanding Thurston maps in \cite{Li18} (see \cite[Lemmas~5.1 and~5.2]{Li18}).

\begin{lemma}[Z.~Li \cite{Li18}]    \label{lmSnPhiBound}
Let $f\:S^2 \rightarrow S^2$ be an expanding Thurston map and $\CC \subseteq S^2$ be a Jordan curve containing $\post f$ with the property that $f^{n_\CC}(\CC)\subseteq \CC$ for some $n_\CC\in\N$. Let $d$ be a visual metric on $S^2$ for $f$ with expansion factor $\Lambda>1$ and a linear local connectivity constant $L\geq 1$. Fix $\phi\in \Holder{\alpha}(S^2,d)$ with $\alpha\in(0,1]$. Then there exist  constants $C_1=C_1(f,\CC,d,\phi,\alpha)$ and $C_2=C_2(f,\CC,d,\phi,\alpha)  \geq 1$ depending only on $f$, $\CC$, $d$, $\phi$, and $\alpha$ such that
\begin{align} 
\Abs{S_n\phi(x)-S_n\phi(y)}  & \leq C_1 d(f^n(x),f^n(y))^\alpha, \label{eqSnPhiBound} \\
\frac{\sum_{z'\in f^{-n}(z)}  \deg_{f^n}(z')  \exp (S_n\phi(z'))}{\sum_{w'\in f^{-n}(w)}  \deg_{f^n}(w')  \exp (S_n\phi(w'))} 
                                                  & \leq \exp\left(4C_1 Ld(z,w)^\alpha\right) \leq C_2,\label{eqSigmaExpSnPhiBound}
\end{align}
for $n, \, m\in\N_0$ with $n\leq m $, $X^m\in\X^m(f,\CC)$, $x, \, y\in X^m$, and $z, \, w\in S^2$. We can choose
\begin{equation}   \label{eqC1C2}
C_1 \coloneqq  \Hseminorm{\alpha}{ (S^2,d)}{\phi} C_0 (1-\Lambda^{-\alpha})^{-1}   \qquad \text{and} \qquad 
C_2 \coloneqq \exp\bigl(4C_1 L \bigl(\diam_d(S^2)\bigr)^\alpha \bigr) 
\end{equation}
where $C_0 > 1$ is the constant depending only on $f$, $\CC$, and $d$ from \cite[Lemma~3.13]{Li18}.
\end{lemma}

The following theorem is part of \cite[Theorem~5.16]{Li18}.

\begin{theorem}[Z.~Li \cite{Li18}]   \label{thmMuExist}
Let $f\:S^2 \rightarrow S^2$ be an expanding Thurston map and $\CC \subseteq S^2$ be a Jordan curve containing $\post f$ with the property that $f^{n_\CC}(\CC)\subseteq \CC$ for some $n_\CC\in\N$. Let $d$ be a visual metric on $S^2$ for $f$ with expansion factor $\Lambda>1$. Let $\phi\in \Holder{\alpha}(S^2,d)$ be a real-valued H\"{o}lder continuous function with an exponent $\alpha\in(0,1]$. Then the sequence 
\begin{equation*}
\biggl\{      \frac{1}{n}\sum_{j=0}^{n-1}    e^{ - j P(f,\phi) }    \sum_{y\in f^{-j}( \cdot )}  \deg_{f^j}(y) e^{ S_j \phi(y) }        \biggr\}_{n\in\N}
\end{equation*}
converges, with respect to the uniform norm, to a function $u_\phi \in \Holder{\alpha}(S^2,d)$, which satisfies
\begin{equation*}
    e^{ -  P(f,\phi) }    \sum_{y\in f^{-1}( x )}  \deg_{f}(y) u_\phi(y)   e^{  \phi(y) }   = u_\phi (x),
\end{equation*}
and $C_2^{-1} \leq u_\phi(x) \leq C_2$ for each $x\in S^2$, 
where $C_2\geq 1$ is the constant from Lemma~\ref{lmSnPhiBound}.
\end{theorem}

\begin{lemma}  \label{lmRR^nConvToTopPressureUniform}
Let $f\:S^2 \rightarrow S^2$ be an expanding Thurston map. Let $d$ be a visual metric on $S^2$ for $f$ with expansion factor $\Lambda>1$. Let $H$ be a bounded subset of $\Holder{\alpha}(S^2,d)$ for some $\alpha \in (0,1]$. Then for each $\phi\in H$ and each $x\in S^2$, we have
\begin{equation}   \label{eqRR^nConvToTopPressure}
P(f,\phi)   = \lim_{n\to +\infty} \frac{1}{n} \log \sum_{y\in f^{-n}(x)} \deg_{f^n}(y) \exp (S_n\phi(y)).
\end{equation} 
Moreover, the convergence in (\ref{eqRR^nConvToTopPressure}) is uniform in $\phi\in H$ and $x\in S^2$.
\end{lemma}

\begin{proof}
Note that $e^{-nP(f,\phi)}  \sum_{y\in f^{-n}(x)} \deg_{f^n}(y) \exp (S_n\phi(y))$ converges to $u_\phi(x)$ uniformly in $\phi\in H$ and $x\in S^2$ as $n\to +\infty$ by \cite[Theorem~6.8 and Corollary~6.10]{Li18}. By Theorem~\ref{thmMuExist}, the function $u_\phi$ is continuous and $C_2^{-1} \leq u_\phi(x)\leq C_2$ for the same constant $C_2>1$ from Lemma~\ref{lmSnPhiBound}. Note that $C_2$ depends only on $f$, $\CC$, $d$, $\phi$, and $\alpha$, but it is bounded on $H$ by (\ref{eqC1C2}). Therefore, the difference between the two sides of (\ref{eqRR^nConvToTopPressure}) converges to $0$ uniformly in $\phi\in H$ and $x\in S^2$ as $n\to+\infty$. 
\end{proof}

Another characterization of the topological pressure in our context analogous to (\ref{eqRR^nConvToTopPressure}), but in terms of periodic points, was obtained in \cite[Proposition~6.8]{Li15}. We record it below.

\begin{prop}[Z.~Li \cite{Li15}]    \label{propTopPressureDefPeriodicPts}
Let $f\:S^2 \rightarrow S^2$ be an expanding Thurston map and $d$ be a visual metric on $S^2$ for $f$ with expansion factor $\Lambda>1$. Let $\phi\in \Holder{\alpha}(S^2,d)$ be a real-valued function. Fix an arbitrary sequence of functions $\{ w_n : S^2\rightarrow \R\}_{n\in\N}$ satisfying $w_n(y) \in [1,\deg_{f^n}(y)]$ for each $n\in\N$ and each $y\in S^2$. Then
\begin{equation}  \label{eqTopPressureDefPrePeriodPts}
P(f,\phi)  = \lim_{n\to +\infty} \frac{1}{n} \log \sum_{y\in P_{1,f^n}} w_n(y) \exp (S_n\phi(y)).
\end{equation}  
\end{prop}

\smallskip

The potentials that satisfy the following property are of particular interest in the considerations of Prime Orbit Theorems and the analytic study of dynamical zeta functions.

\begin{definition}[Eventually positive functions]  \label{defEventuallyPositive}
Let $g\: X\rightarrow X$ be a map on a set $X$, and $\varphi\:X\rightarrow\C$ be a complex-valued function on $X$. Then $\varphi$ is \defn{eventually positive} if there exists $N\in\N$ such that $S_n\varphi(x)>0$ for each $x\in X$ and each $n\in\N$ with $n\geq N$.
\end{definition}

\begin{lemma}  \label{lmSnPhiHolder}
Let $f\: S^2\rightarrow S^2$ be an expanding Thurston map and $d$ be a visual metric on $S^2$ for $f$. If $\psi \in \Holder{\alpha}((S^2,d),\C)$ is a complex-valued H\"{o}lder continuous function with an exponent $\alpha\in(0,1]$, then $S_n\psi$ also satisfies $S_n\psi \in \Holder{\alpha}((S^2,d),\C)$ for each $n\in\N$.
\end{lemma}
\begin{proof}
Since $f$ is Lipschitz with respect to $d$ by Theorem~\ref{thmETMBasicProperties}~(i), so is $f^i$ for each $i\in\N$. Then $\psi\circ f^i \in \Holder{\alpha}((S^2,d),\C)$ for each $i\in\N$. Thus, by (\ref{eqDefSnPt}), $S_n\psi  \in \Holder{\alpha}((S^2,d),\C)$.
\end{proof}

Theorem~\ref{thmEquilibriumState}~(ii) leads to the following corollary that we frequently use, often implicitly, throughout this paper.

\begin{cor}   \label{corS0unique}
Let $f\:S^2 \rightarrow S^2$ be an expanding Thurston map, and $d$ be a visual metric on $S^2$ for $f$. Let $\phi \in \Holder{\alpha}(S^2,d)$ be an eventually positive real-valued H\"{o}lder continuous function with an exponent $\alpha\in(0,1]$. Then the function $t\mapsto P(f,-t\phi)$, $t\in\R$, is strictly decreasing and there exists a unique number $s_0 \in\R$ such that $P(f,-s_0\phi)=0$. Moreover, $s_0>0$.
\end{cor}

\begin{proof}
By Definition~\ref{defEventuallyPositive}, we can choose $m\in\N$ such that $\Phi \coloneqq S_m \phi$ is strictly positive. Denote $A \coloneqq \inf \{ \Phi(x) : x\in S^2 \} > 0$. Then by Theorem~\ref{thmEquilibriumState}~(ii) and the fact that the equilibrium state $\mu_{\minus t\phi}$ for $f$ and $-t\phi$ is an $f$-invariant probability measure (see Theorem~\ref{thmEquilibriumState}~(i) and Subsection~\ref{subsctThermodynFormalism}), we have that for each $t\in\R$,
\begin{equation}   \label{eqPfcorS0unique}
    \frac{\mathrm{d}}{\mathrm{d}t} P(f,- t\phi) 
= - \int \!\phi \,\mathrm{d}\mu_{\minus t\phi} 
= - \frac{1}{m} \int \! S_m \phi \,\mathrm{d}\mu_{\minus t\phi} 
< 0.
\end{equation}   
By Lemma~\ref{lmRR^nConvToTopPressureUniform}, (\ref{eqDeg=SumLocalDegree}), and (\ref{eqLocalDegreeProduct}), for each $t\in\R$ sufficiently large, we have
\begin{align*}
P(f, - t \phi) & =   \lim_{n\to+\infty} \frac{1}{mn} \log \sum_{y\in f^{-mn}(x)} \deg_{f^{mn}}(y)
                                \exp \biggl( -t \sum_{i=0}^{n-1} \bigl(\Phi\circ f^{mi} \bigr) (y) \biggr) \\
               & \leq \lim_{n\to+\infty} \frac{1}{mn} \log  ((\deg f)^{mn} \exp  ( -t n A ) )\\
                &=    \frac{1}{m} \log ( (\deg f )^m \exp ( -t A ) )\\
               &<0.
\end{align*}
Since the topological entropy $h_{\operatorname{top}} (f) = P(f,0) = \log (\deg f)>0$ (see \cite[Corollary~17.2]{BM17}), the corollary follows immediately from (\ref{eqPfcorS0unique}) and the fact that $P(f,\cdot) \: \CCC(S^2)\rightarrow \R$ is continuous (see for example, \cite[Theorem~3.6.1]{PU10}).
\end{proof}

\subsection{Symbolic dynamics for expanding Thurston maps}   \label{subsctSFT}

We give a brief review of the dynamics of one-sided subshifts of finite type in this subsection. We refer the readers to \cite{Ki98} for a beautiful introduction to symbolic dynamics. For a discussion on results on subshifts of finite type in our context, see \cite{PP90, Bal00}.

Let $S$ be a finite nonempty set, and $A \: S\times S \rightarrow \{0, \, 1\}$ be a matrix whose entries are either $0$ or $1$. For $n\in\N_0$, we denote by $A^n$ the usual matrix product of $n$ copies of $A$. We denote the \defn{set of admissible sequences defined by $A$} by
\begin{equation*}
\Sigma_A^+ = \{ \{x_i\}_{i\in\N_0} : x_i \in S, \, A(x_i,x_{i+1})=1  \text{ for each } i\in\N_0\}.
\end{equation*}
Given $\theta\in(0,1)$, we equip the set $\Sigma_A^+$ with a metric $d_\theta$ given by $d_\theta(\{x_i\}_{i\in\N_0},\{y_i\}_{i\in\N_0})=\theta^N$ for $\{x_i\}_{i\in\N_0} \neq \{y_i\}_{i\in\N_0}$, where $N$ is the smallest integer with $x_N \neq y_N$. The topology on the metric space $\left(\Sigma_A^+,d_\theta \right)$ coincides with that induced from the product topology, and is therefore compact.

The \defn{left-shift operator} $\sigma_A \: \Sigma_A^+ \rightarrow \Sigma_A^+$ (defined by $A$) is given by
\begin{equation*}
\sigma_A ( \{x_i\}_{i\in\N_0} ) = \{x_{i+1}\}_{i\in\N_0}  \qquad \text{for } \{x_i\}_{i\in\N_0} \in \Sigma_A^+.
\end{equation*}

The pair $\left(\Sigma_A^+, \sigma_A\right)$ is called the \defn{one-sided subshift of finite type} defined by $A$. The set $S$ is called the \defn{set of states} and the matrix $A\: S\times S \rightarrow \{0, \, 1\}$ is called the \defn{transition matrix}.

We say that a one-sided subshift of finite type $\bigl(\Sigma_A^+, \sigma_A \bigr)$ is \defn{topologically mixing} if there exists $N\in\N$ such that $A^n(x,y)>0$ for each $n\geq N$ and each pair of $x, \, y\in S$.

Let $X$ and $Y$ be topological spaces, and $f\:X\rightarrow X$ and $g\:Y\rightarrow Y$ be continuous maps. We say that the topological dynamical system $(X,f)$ is a \defn{factor} of the topological dynamical system $(Y,g)$ if there is a surjective continuous map $\pi\:Y\rightarrow X$ such that $\pi\circ g=f\circ\pi$. We call the map $\pi\: Y\rightarrow X$ a \emph{factor map}.

\begin{comment}

We get the following commutative diagram:
\begin{equation*}
\xymatrix{  Y  \ar[d]_\pi \ar[r]^{g  }  & Y   \ar[d]^\pi \\ 
            X  \ar[r]_{f}               & X.}
\end{equation*}
It follows immediately that $\pi\circ g^n = f^n \circ \pi$ for each $n\in\N$.

\end{comment}

\smallskip

We will now consider a one-sided subshift of finite type associated with an expanding Thurston map and an invariant Jordan curve on $S^2$ containing $\post f$. The construction of other related symbolic systems will be postponed to Section~\ref{sctDynOnC}. We will need the following lemma in the construction of these symbolic systems. Here, as well as in Section~\ref{sctDynOnC} in a similar fashion, the elements of the symbolic space are in bijective correspondences with descending chains $A^1 \supseteq \cdots \supseteq A^{n-1} \supseteq A^n \supseteq \cdots $, where $A^i$ is an $i$-tile (resp.~$i$-edge). One codes these chains by a sequence $\{B_i\}_{i\in\N_0}$ of $1$-tiles (resp.~$1$-edges) by setting $B_i \coloneqq f^i (A^{i+1})$, $i\in\N_0$.

\begin{lemma}   \label{lmCylinderIsTile}
Let $f\: S^2 \rightarrow S^2$ be an expanding Thurston map with a Jordan curve $\CC\subseteq S^2$ satisfying $f(\CC)\subseteq \CC$ and $\post f\subseteq \CC$. Let $\{X_i\}_{i\in\N_0}$ be a sequence of $1$-tiles in $\X^1(f,\CC)$ satisfying $f(X_i)\supseteq X_{i+1}$ for all $i\in\N_0$. Let $\{e_j\}_{j\in\N_0}$ be a sequence of $1$-edges in $\E^1(f,\CC)$ satisfying $f(e_j)\supseteq e_{j+1}$ for all $j\in\N_0$. Then for each $n\in\N$, we have
\begin{align}  
\bigl( (f|_{X_0})^{-1} \circ (f|_{X_1})^{-1} \circ \cdots \circ (f|_{X_{n-2}})^{-1} \bigr) (X_{n-1}) &=  \bigcap_{i=0}^{n-1}  f^{-i} (X_i) \in \X^n(f,\CC), \label{eqCylinderIsTile}\\
\bigl( (f|_{e_0})^{-1} \circ (f|_{e_1})^{-1} \circ \cdots \circ (f|_{e_{n-2}})^{-1} \bigr) (e_{n-1}) &=\bigcap_{j=0}^{n-1}  f^{-j} (e_j) \in \E^n(f,\CC). \label{eqCylinderIsEdge}
\end{align}
Moreover, both $\bigcap_{i\in\N_0} f^{-i}(X_i)$ and $\bigcap_{j\in\N_0} f^{-j}(e_j)$ are singleton sets.
\end{lemma}

The part of this lemma for tiles can be found in \cite[Lemma~3.24]{LZha23}, and the proof for the part for edges is verbatim the same.

A part of the following proposition is established in \cite[Proposition~3.25]{LZha23}. We need a stronger version.

\begin{prop}   \label{propTileSFT}
Let $f\: S^2 \rightarrow S^2$ be an expanding Thurston map with a Jordan curve $\CC\subseteq S^2$ satisfying $f(\CC)\subseteq \CC$ and $\post f\subseteq \CC$. Let $d$ be a visual metric on $S^2$ for $f$ with expansion factor $\Lambda>1$. Fix $\theta\in(0,1)$. We set $S_{\ti} \coloneqq \X^1(f,\CC)$, and define a transition matrix $A_{\ti}\: S_{\ti}\times S_{\ti} \rightarrow \{0, \, 1\}$ by
\begin{equation*}
A_{\ti}(X,X') = \begin{cases} 1 & \text{if } f(X)\supseteq X', \\ 0  & \text{otherwise}  \end{cases}
\end{equation*}
for $X, \, X'\in \X^1(f,\CC)$. Then $f$ is a factor of the one-sided subshift of finite type $\bigl(\Sigma_{A_{\ti}}^+, \sigma_{A_{\ti}}\bigr)$ defined by the transition matrix $A_{\ti}$ with a surjective and H\"{o}lder continuous factor map $\pi_{\ti}\: \Sigma_{A_{\ti}}^+ \rightarrow S^2$ given by 
\begin{equation}   \label{eqDefTileSFTFactorMap}
\pi_{\ti} \left( \{X_i\}_{i\in\N_0} \right)= x,  \text{ where } \{x\} = \bigcap_{i \in \N_0} f^{-i} (X_i).
\end{equation}
Here $\Sigma_{A_{\ti}}^+$ is equipped with the metric $d_\theta$ defined in Subsection~\ref{subsctSFT}, and $S^2$ is equipped with the visual metric $d$.

Moreover, $\bigl(\Sigma_{A_{\ti}}^+, \sigma_{A_{\ti}} \bigr)$ is topologically mixing and $\pi_{\ti}$ is injective on $\pi_{\ti}^{-1} \bigl( S^2 \setminus \bigcup_{i\in\N_0} f^{-i}(\CC) \bigr)$.
\end{prop}

\begin{rem}
We can show that if $f$ has no periodic critical points, then $\pi$ is uniformly bounded-to-one (i.e., there exists $N\in\N_0$ depending only on $f$ such that $\card \left(\pi_{\ti}^{-1}(x)\right) \leq N$ for each $x\in S^2$); if $f$ has at least one periodic critical point, then $\pi_{\ti}$ is uncountable-to-one on a dense set. We will not use this fact in this paper.
\end{rem}

\begin{proof}
We denote by $\{X_i\}_{i\in\N_0} \in \Sigma_{A_{\ti}}^+$ an arbitrary admissible sequence.

Since $f(X_i)\supseteq X_{i+1}$ for each $i\in\N_0$, by Lemma~\ref{lmCylinderIsTile}, the map $\pi_{\ti}$ is well-defined.

Note that for each $m\in \N_0$ and each $\{X'_i\}_{i\in\N_0} \in \Sigma_{A_{\ti}}^+$ with $X_{m+1} \neq X'_{m+1}$ and $X_j=X'_j$ for each integer $j\in[0,m]$, we have $\{ \pi_{\ti}  (\{X_i\}_{i\in\N_0}  ),  \, \pi_{\ti} ( \{ X'_i \}_{i\in\N_0} ) \} \subseteq \bigcap_{i=0}^{m} f^{-i}(X_i) \in \X^{m+1}$ by Lemma~\ref{lmCylinderIsTile}. Thus, it follows from Lemma~\ref{lmCellBoundsBM}~(ii) that $\pi_{\ti}$ is H\"{o}lder continuous.

To see that $\pi_{\ti}$ is surjective, we observe that for each $x\in S^2$, we can find a sequence $\bigl\{ X^j(x) \bigr\}_{j\in\N}$ of tiles such that $X^j(x) \in \X^j$, $x\in X^j(x)$, and $X^j(x) \supseteq X^{j+1}(x)$ for each $j\in\N$. Then it is clear that $\bigl\{  f^i\bigl( X^{i+1}(x) \bigr) \bigr\}_{i\in\N_0} \in \Sigma_{A_{\ti}}^+$ and $\pi_{\ti} \bigl(  \bigl\{ f^i\bigl( X^{i+1}(x) \bigr) \bigr\}_{i\in\N_0} \bigr) = x$.

To check that $\pi_{\ti} \circ \sigma_{A_{\ti}} = f \circ \pi_{\ti}$, it suffices to observe that
\begin{align*}
            \{ (f\circ \pi_{\ti}) (\{X_i\}_{i\in\N_0}) \}
 & =        f \Bigl(  \bigcap_{j\in\N_0} f^{-j} (X_j) \Bigr) \\
& \subseteq  \bigcap_{j\in\N} f^{-(j-1)} (X_j) \\
 & =         \bigcap_{i\in\N_0} f^{-i} (X_{i+1})  \\
 & =         \{(\pi_{\ti} \circ \sigma_{A_{\ti}} ) (\{X_i\}_{i\in\N_0})  \}.
\end{align*}

To show that $\pi_{\ti}$ is injective on $\pi_{\ti}^{-1}(S^2\setminus \E)$, where we denote $\E \coloneqq \bigcup_{i\in\N_0} f^{-i}(\CC)$, we fix another arbitrary $\{Y_i\}_{i\in\N_0} \in \Sigma_{A_{\ti}}^+$ with $\{X_i\}_{i\in\N_0} \neq \{Y_i\}_{i\in\N_0}$. Suppose that $x = \pi_{\ti} (\{X_i\}_{i\in\N_0}) = \pi_{\ti} (\{Y_i\}_{i\in\N_0}) \notin \E$. Choose $n\in\N_0$ with $X_n\neq Y_n$. Then by Lemma~\ref{lmCylinderIsTile}, $x\in \bigcap_{i=0}^n f^{-i} (X_i) \in \X^{n+1}$ and $x\in \bigcap_{i=0}^n f^{-i} (Y_i) \in \X^{n+1}$. Thus, $f^n(x) \in X_n \cap Y_n \subseteq f^{-1}(\CC)$ by Proposition~\ref{propCellDecomp}~(v). This is a contradiction to the assumption that $x\notin \E$. 

We finally demonstrate that $\bigl( \Sigma_{A_{\ti}}^+, \sigma_{A_{\ti}} \bigr)$ is topologically mixing. It follows from Lemma~\ref{lmCellBoundsBM}~(iii) that there exists a number $M\in\N$ such that for each $m\geq M$, there exist white $m$-tiles $X^m_\w, \, Y^m_\w \in \X^m_\w$ and black $m$-tiles $X^m_\b,  \, Y^m_\b \in \X^m_\b$ satisfying $X^m_\w \cup X^m_\b \subseteq X^0_\b$ and $Y^m_\w \cup Y^m_\b \subseteq X^0_\w$, where $X^0_\b$ and $X^0_\w$ are the black $0$-tile and the white $0$-tile, respectively. Thus, for all $X, \, X'\in\X^1$, and all $n\geq M+1$, we have $f^n(X) = f^{n-1}(f(X)) \supseteq X^0_\b \cup X^0_\w = S^2 \supseteq X'$.
\end{proof}

\section{The Assumptions}      \label{sctAssumptions}
We state below the hypotheses under which we will develop our theory in most parts of this paper. We will repeatedly refer to such assumptions in the later sections. We emphasize again that not all of these assumptions are assumed in all the statements in this paper.

\begin{assumptions}
\quad

\begin{enumerate}

\smallskip

\item $f\:S^2 \rightarrow S^2$ is an expanding Thurston map.

\smallskip

\item $\CC\subseteq S^2$ is a Jordan curve containing $\post f$ with the property that there exists $n_\CC\in\N$ such that $f^{n_\CC} (\CC)\subseteq \CC$ and $f^m(\CC)\nsubseteq \CC$ for each $m\in\{1, \, 2, \, \dots, \, n_\CC-1\}$.

\smallskip

\item $d$ is a visual metric on $S^2$ for $f$ with expansion factor $\Lambda>1$ and a linear local connectivity constant $L\geq 1$.

\smallskip

\item $\alpha\in(0,1]$.

\smallskip

\item $\psi\in \Holder{\alpha}((S^2,d),\C)$ is a complex-valued H\"{o}lder continuous function with an exponent $\alpha$.

\smallskip

\item $\phi\in \Holder{\alpha}(S^2,d)$ is an eventually positive real-valued H\"{o}lder continuous function with an exponent $\alpha$, and $s_0\in\R$ is the unique positive real number satisfying $P(f, -s_0\phi)=0$.

%\smallskip

%\item $\mu_\phi$ is the unique equilibrium state for the  map $f$ and the potential $\phi$.

\end{enumerate}

\end{assumptions}

Note that the uniqueness of $s_0$ in~(6) is guaranteed by Corollary~\ref{corS0unique}. For a pair of $f$ in~(1) and $\phi$ in~(6), we will say that a quantity depends on $f$ and $\phi$ if it depends on $s_0$.

Observe that by Lemma~\ref{lmCexistsL}, for each $f$ in~(1), there exists at least one Jordan curve $\CC$ that satisfies~(2). Since for a fixed $f$, the number $n_\CC$ is uniquely determined by $\CC$ in~(2), in the remaining part of the paper, we will say that a quantity depends on $f$ and $\CC$ even if it also depends on $n_\CC$.

Recall that the expansion factor $\Lambda$ of a visual metric $d$ on $S^2$ for $f$ is uniquely determined by $d$ and $f$. We will say that a quantity depends on $f$ and $d$ if it depends on $\Lambda$.

Note that even though the value of $L$ is not uniquely determined by the metric $d$, in the remainder of this paper, for each visual metric $d$ on $S^2$ for $f$, we will fix a choice of linear local connectivity constant $L$. We will say that a quantity depends on the visual metric $d$ without mentioning the dependence on $L$, even though if we had not fixed a choice of $L$, it would have depended on $L$ as well.

In the discussion below, depending on the conditions we will need, we will sometimes say ``Let $f$, $\CC$, $d$, $\psi$, $\alpha$ satisfy the Assumptions.'', and sometimes say ``Let $f$ and $d$ satisfy the Assumptions.'', etc.

\section{Dynamical zeta functions and Dirichlet series}   \label{sctDynZetaFn}

Let $g\: X\rightarrow X$ be a map on a topological space $X$. Let $\psi \: X \rightarrow \C$ be a complex-valued function on $X$. We write
\begin{equation}  \label{eqDefZn}
Z_{g,\,\minus\psi}^{(n)} (s) \coloneqq  \sum_{x\in P_{1,g^n}} e^{-s S_n \psi(x)}  , \qquad n\in\N \text{ and }  s\in\C.
\end{equation}
Recall that $P_{1,g^n}$ defined in (\ref{eqDefSetPeriodicPts}) is the set of fixed points of $g^n$, and $S_n\psi$ is defined in (\ref{eqDefSnPt}). We denote by the formal infinite product
\begin{equation}  \label{eqDefZetaFn}
\zeta_{g,\,\minus\psi} (s) \coloneqq \exp \Biggl( \sum_{n=1}^{+\infty} \frac{Z_{g,\,\minus\psi}^{(n)} (s)}{n}  \Biggr) 
 = \exp \biggl( \sum_{n=1}^{+\infty} \frac{1}{n} \sum_{x\in P_{1,g^n}} e^{-s S_n \psi(x)} \biggr), \qquad s\in\C,
\end{equation}
the \defn{dynamical zeta function} for the map $g$ and the potential $\psi$. More generally, we can define dynamical Dirichlet series as analogs of Dirichlet series in analytic number theory.

\begin{definition}   \label{defDynDirichletSeries}
Let $g\: X\rightarrow X$ be a map on a topological space $X$. Let $\psi \: X \rightarrow \C$ and $w\: X\rightarrow \C$ be complex-valued functions on $X$. We denote by the formal infinite product
\begin{equation}  \label{eqDefDynDirichletSeries}
\DS_{g,\,\minus\psi,\,w} (s) \coloneqq \exp \biggl( \sum_{n=1}^{+\infty} \frac{1}{n} \sum_{x\in P_{1,g^n}} e^{-s S_n \psi(x)} \prod_{i=0}^{n-1} w \bigl( g^i(x) \bigr) \biggr), \qquad s\in\C,
\end{equation}
the \defn{dynamical Dirichlet series} with coefficient $w$ for the map $g$ and the potential $\psi$.

\end{definition}

\begin{rem}
Dynamical zeta functions are special cases of dynamical Dirichlet series, more precisely, $\zeta_{g,\,\minus\psi}  =  \DS_{g,\,\minus\psi,\,\mathbbm{1}_X}$. The Dynamical Dirichlet series defined above can be considered as analogs of the Dirichlet series equipped with a strongly multiplicative arithmetic function in analytic number theory. We can define a more general dynamical Dirichlet series by replacing $w$ by $w_n$, where $w_n\: X \rightarrow \C$ is a complex-valued function on $X$ for each $n\in\N$. We will not need such generality in this paper.
\end{rem}

\begin{lemma}   \label{lmDynDirichletSeriesConv_general}
Let $g\:X\rightarrow X$ be a map on a topological space $X$. Let $\varphi \: X\rightarrow\R$ and $w\: X\rightarrow \C$ be functions on $X$. Fix $a\in\R$. Suppose that the following conditions are satisfied:
\begin{enumerate}
\smallskip
\item[(i)] $\card  P_{1,g^n}  < +\infty$ for all $n\in\N$.

\smallskip
\item[(ii)] $\limsup_{n\to+\infty} \frac{1}{n} \log  \sum_{x\in P_{1,g^n}} \exp ( - a S_n \varphi(x)) \prod_{i=0}^{n-1} \Absbig{w \bigl( g^i(x) \bigr)} < 0$.
\end{enumerate}
Then the dynamical Dirichlet series $\DS_{g,\,\minus\varphi,\,w}(s)$ as an infinite product converges uniformly and absolutely on $\{ s\in\C : \Re(s) =a \}$, and with $l_\varphi (\tau) \coloneqq \sum_{x\in\tau} \varphi(x)$, we have
\begin{equation}   \label{eqZetaFnOrbForm}
\DS_{g,\,\minus\varphi,\,w} (s) = \prod_{\tau\in \Orb(g)} \biggl( 1- e^{ - s l_\varphi (\tau)} \prod_{x\in\tau} w(x) \biggr)^{-1}.
\end{equation} 

If, in addition, we assume that $\varphi$ is eventually positive, then $\DS_{g,\,\minus\varphi,\,w}(s)$ converges uniformly and absolutely to a non-vanishing continuous function on $\overline{\H}_a = \{s\in\C : \Re(s) \geq a \}$ that is holomorphic on $\H_a = \{s\in\C : \Re(s) > a \}$, and (\ref{eqZetaFnOrbForm}) holds on $\overline{\H}_a$.
\end{lemma}

Recall that $\Orb(g)$ denotes the set of all primitive periodic orbits of $g$ (see (\ref{eqDefSetAllPeriodicOrbits})). We recall that an infinite product of the form $\exp \sum a_i$, $a_i\in\C$, converges uniformly (resp.\ absolutely) if $\sum a_i$ converges uniformly (resp.\ absolutely).

\begin{rem}   \label{rmDynDirichletSeriesZetaFn}
It is often possible to verify condition~(ii) by showing $P(g, - a\varphi) <0$ and $P(g, - a\varphi) \geq \limsup_{n\to+\infty} \frac{1}{n} \log  \sum_{x\in P_{1,g^n}} \exp ( - a S_n \varphi(x)) \prod_{i=0}^{n-1} \Absbig{w \bigl( g^i(x) \bigr)}$ (when the topological pressure $P(g, - a\varphi)$ makes sense). This is how we are going to use Lemma~\ref{lmDynDirichletSeriesConv_general} in this paper. In particular, if $\card X < +\infty$, then it follows immediately from (\ref{defTopPressure}) that
\begin{equation*}
    P(g, - a\varphi) 
= \lim_{n\to+\infty} \frac{1}{n} \log  \sum_{x\in X} \exp ( - a S_n \varphi(x))  
\geq \limsup_{n\to+\infty} \frac{1}{n} \log  \sum_{x\in P_{1,g^n}} \exp ( - a S_n \varphi(x)) .
\end{equation*}
\end{rem}

\begin{proof}[Proof of Lemma~\ref{lmDynDirichletSeriesConv_general}]
Fix $s\in\C$ with $\Re(s)=a$.

By condition~(ii), we can choose constants $N\in\N$ and $\beta\in (0,1)$ such that 
\begin{equation*}
\sum_{x\in P_{1,g^n}} \exp( - a S_n\varphi(x)) \prod_{i=0}^{n-1} \Absbig{w \bigl( g^i(x) \bigr)}  \leq \beta^n
\end{equation*}
for each integer $n\geq N$. Thus,
\begin{align*}
    \sum_{n=N}^{+\infty} \frac{1}{n} \sum_{x\in P_{1,g^n}}  \Absbigg{ e^{ - s S_n\varphi(x)}  \prod_{i=0}^{n-1}w \bigl(g^i(x)\bigr)}
=  \sum_{n=N}^{+\infty} \frac{1}{n} \sum_{x\in P_{1,g^n}}       e^{- a S_n\varphi(x)} 
                                   \prod_{i=0}^{n-1} \Absbig{w \bigl( g^i(x) \bigr)}  
 \leq  \sum_{n=N}^{+\infty}  \beta^n.
\end{align*}
Combining the above inequalities with condition~(i), we can conclude that $\DS_{g,\,\minus\varphi,\,w} (s)$ converges absolutely. Moreover,
\begin{align*}
      \DS_{g,\,\minus\varphi,\,w} (s)
 &=   \exp \biggl( \sum_{n=1}^{+\infty} \frac{1}{n} \sum_{x\in P_{1,g^n}} e^{ - s S_n \varphi(x)} 
                                                           \prod_{i=0}^{n-1} w \bigl( g^i(x) \bigr)   \biggr)  \\ 
 &=   \exp \biggl( \sum_{m=1}^{+\infty} \frac{1}{m} \sum_{x\in P_{m,g}} \sum_{k=1}^{+\infty} \frac{e^{-sk S_m \varphi(x)}}{k} 
                                                           \prod_{i=0}^{km-1} w \bigl( g^i(x) \bigr)  \biggr)  \\
 &=    \exp \biggl( \sum_{\tau\in\Orb(g)} \sum_{k=1}^{+\infty} \frac{e^{ - sk l_\varphi(\tau)} }{k} 
                                                           \prod_{y\in\tau}  w^k  (y) \biggr) \\
 &=   \exp \biggl(-\sum_{\tau\in\Orb(g)} \log \biggl( 1- e^{ - s l_\varphi(\tau)} \prod_{y\in\tau} w(y)  \biggr) \biggr)  \\
 &=    \prod_{\tau\in \Orb(g)} \biggl( 1- e^{ - s l_\varphi (\tau)} \prod_{y\in\tau} w(y) \biggr)^{-1}.
\end{align*}

Now we assume, in addition, that $\varphi$ is eventually positive. Then it is clear from the definition that $S_n\varphi(x) > 0$ for all $n\in\N$ and $x\in P_{1,g^n}$. For each $z \in \overline{\H}_a$ and each $m\in\N$,
\begin{equation*}
  \sum_{n=m}^{+\infty} \frac{1}{n} \sum_{x\in P_{1,g^n}}  \Absbigg{\exp( - z S_n\varphi(x))  \prod_{i=0}^{n-1}w \bigl(g^i(x)\bigr)}  
  \leq \sum_{n=m}^{+\infty} \frac{1}{n} \sum_{x\in P_{1,g^n}}       \exp( - a S_n\varphi(x))
                                                                                         \prod_{i=0}^{n-1} \Absbig{w \bigl( g^i(x) \bigr)}.
\end{equation*}
Hence, $\DS_{g,\,\minus\varphi,\,w}(z)$ converges uniformly and absolutely to a non-vanishing continuous function on $\overline{\H}_a$ that is holomorphic on $\H_a$.

Finally, to verify (\ref{eqZetaFnOrbForm}) for $z\in \overline{\H}_a$ when $\varphi$ is eventually positive, it suffices to apply (\ref{eqZetaFnOrbForm}) to $a\coloneqq \Re(z)$ with the observation that
\begin{align*}
&            \limsup_{n\to+\infty} \frac{1}{n} \log  \sum_{x\in P_{1,g^n}} \exp ( - \Re(z) S_n \varphi(x)) 
                                                                                         \prod_{i=0}^{n-1} \Absbig{w \bigl( g^i(x) \bigr)}  \\
&\qquad \leq \limsup_{n\to+\infty} \frac{1}{n} \log  \sum_{x\in P_{1,g^n}} \exp ( - a      S_n \varphi(x))
                                                                                         \prod_{i=0}^{n-1} \Absbig{w \bigl( g^i(x) \bigr)}
 < 0,
\end{align*}
i.e., condition~(ii) holds with $a = \Re(z)$.
\end{proof}

\smallskip

We now consider the dynamical zeta functions and the Dirichlet series associated to expanding Thurston maps.

\begin{prop}   \label{propZetaFnConv_s0}
Let $f$, $\CC$, $d$, $\alpha$, $\phi$, $s_0$ satisfy the Assumptions in Section~\ref{sctAssumptions}. We assume, in addition, that $f(\CC)\subseteq\CC$. 
Let $\bigl(\Sigma_{A_{\ti}}^+,\sigma_{A_{\ti}}\bigr)$ be the one-sided subshift of finite type associated to $f$ and $\CC$ defined in Proposition~\ref{propTileSFT}, and let $\pi_{\ti}\: \Sigma_{A_{\ti}}^+\rightarrow S^2$ be the factor map defined in (\ref{eqDefTileSFTFactorMap}). Denote by $\deg_f(\cdot)$ the local degree of $f$. Then the following statements are satisfied:
\begin{enumerate}
\smallskip
\item[(i)] $P(\sigma_{A_{\ti}}, \varphi\circ\pi_{\ti}) = P(f,\varphi)$ for each $\varphi \in \Holder{\alpha}(S^2,d)$ . In particular, for an arbitrary number $t\in\R$, we have $P(\sigma_{A_{\ti}}, - t \phi \circ \pi_{\ti}) = 0$ if and only if $t=s_0$.

\smallskip
\item[(ii)] All three infinite products $\zeta_{f,\,\minus \phi}$, $\zeta_{\sigma_{A_{\ti}},\,\minus\phi\circsmall\pi_{\ti}}$, and $\DS_{f,\,\minus\phi,\,\deg_f}$ converge uniformly and absolutely to respective non-vanishing continuous functions on $\overline{\H}_a=\{ s\in\C : \Re(s) \geq a \}$ that are holomorphic on $\H_a=\{ s\in\C : \Re(s) > a \}$, for each $a\in (s_0, +\infty)$.

\smallskip
\item[(iii)] For all $s\in \C$ with $\Re(s)>s_0$, we have
\begin{align}
\zeta_{f,\,\minus\phi} (s)           & = \prod_{\tau\in \Orb(f)} \biggl( 1- \exp\biggl( - s \sum_{y\in\tau} \phi(y) \biggr) \biggr)^{-1},  \label{eqZetaFnOrbitForm_ThurstonMap}\\
\DS_{f,\,\minus\phi,\,\deg_f}(s)  & = \prod_{\tau\in \Orb(f)} \biggl( 1- \exp\biggl( - s \sum_{y\in\tau} \phi(y) \biggr) \prod_{z\in\tau} \deg_f(z) \biggr)^{-1}, \label{eqZetaFnOrbitForm_ThurstonMapDegree}\\
\zeta_{\sigma_{A_{\ti}},\,\minus\phi\circsmall\pi_{\ti}} (s)    & = \prod_{\tau\in \Orb(\sigma_{A_{\ti}}) } \biggl( 1- \exp\biggl( - s \sum_{y\in\tau} \phi\circ\pi_{\ti}(y) \biggr) \biggr)^{-1}  \label{eqZetaFnOrbitForm_Subshift}.
\end{align}
\end{enumerate}
\end{prop}

\begin{proof}
We first claim that for each $\varphi \in \Holder{\alpha}(S^2,d)$, $P(\sigma_{A_{\ti}}, \varphi\circ\pi_{\ti}) = P(f,\varphi)$. Statement~(i) follows from this claim and Corollary~\ref{corS0unique} immediately.

Indeed, by Theorem~\ref{thmEquilibriumState}~(iii), we can choose $x\in S^2 \setminus \bigcup_{i\in\N_0} f^{-i}(\CC)$. By Proposition~\ref{propTileSFT}, the map $\pi_{\ti}$ is H\"{o}lder continuous on $\Sigma_{A_{\ti}}^+$ and injective on $\pi_{\ti}^{-1}(B)$, where $B= \{x\} \cup \bigcup_{i\in\N} f^{-i}(x)$. So we can consider $\pi_{\ti}^{-1}$ as a function from $B$ to $\pi_{\ti}^{-1}(B)$ in the calculation below. By Lemma~\ref{lmRR^nConvToTopPressureUniform}, Proposition~\ref{propSFT}~(iv), and the fact that $\bigl( \Sigma_{A_{\ti}}^+, \sigma_{A_{\ti}} \bigr)$ is topologically mixing (see Proposition~\ref{propTileSFT}),
\begin{align*}
    P(\sigma_{A_{\ti}} , \varphi\circ\pi_{\ti}) 
&=  \lim_{n\to+\infty} \frac{1}{n} \log \sum_{\underline{y}\in \sigma_{A_{\ti}}^{-n}\left(\pi_{\ti}^{-1}(x)\right)} \exp \bigl(S_n^{\sigma_{A_{\ti}}}  (\varphi\circ\pi_{\ti})(\underline{y})\bigr) \\
&=  \lim_{n\to+\infty} \frac{1}{n} \log \sum_{\underline{y}\in  \pi_{\ti}^{-1}( f^{-n} (x))  }      \exp \bigl(S_n^{\sigma_{A_{\ti}}}  (\varphi\circ\pi_{\ti}) (\underline{y})\bigr) \\
&=  \lim_{n\to+\infty} \frac{1}{n} \log \sum_{z\in    f^{-n} (x)   }              \exp \bigl(S_n^f  \varphi           (z) \bigr)  
=   P(f,\varphi).
\end{align*}
The claim is now established.

\smallskip
Next observe that by Corollary~\ref{corS0unique}, for each $a>s_0$,
$
 P(\sigma_{A_{\ti}}, -a\phi\circ \pi_{\ti})  = P(f,-a\phi) < 0.
$

By Theorem~\ref{thmETMBasicProperties}~(ii), Propositions~\ref{propTopPressureDefPeriodicPts}, and~\ref{propSFT}~(i),~(ii), 
we can apply Lemma~\ref{lmDynDirichletSeriesConv_general} and Remark~\ref{rmDynDirichletSeriesZetaFn} to establish statements~(ii) and (iii). 
\end{proof}

\section{Dynamics on the invariant Jordan curve}  \label{sctDynOnC}

One hopes to derive from Theorem~\ref{thmZetaAnalExt_SFT} similar results for the dynamical zeta function for $f$ itself (stated in Theorem~\ref{thmZetaAnalExt_InvC}). However, there is no one-to-one correspondence between the periodic points of $\sigma_{A_{\ti}}$ and those of $f$ through the factor map $\pi_{\ti} \: \Sigma_{A_{\ti}}^+ \rightarrow S^2$. A relation between the two dynamical zeta functions $\zeta_{f,\,\minus\phi}$ and $\zeta_{\sigma_{A_{\ti}},\,\minus\phi\circsmall\pi_{\ti}}$ can nevertheless be established through a careful investigation of the dynamics induced by $f$ on the Jordan curve $\CC$.

\subsection{Constructions}   \label{subsctDynOnC_Construction}

Let $f\: S^2\rightarrow S^2$ be an expanding Thurston map with a Jordan curve $\CC\subseteq S^2$ satisfying $f(\CC)\subseteq \CC$ and $\post f \subseteq \CC$.

Let $\bigl( \Sigma_{A_{\ti}}^+,\sigma_{A_{\ti}} \bigr)$ be the one-sided subshift of finite type associated to $f$ and $\CC$ defined in Proposition~\ref{propTileSFT}, and let $\pi_{\ti}\: \Sigma_{A_{\ti}}^+\rightarrow S^2$ be the factor map defined in (\ref{eqDefTileSFTFactorMap}).

We first construct two more one-sided subshifts of finite type that are related to the dynamics induced by $f$ on $\CC$.

Define the set of states $S_{\e} \coloneqq \bigl\{ e \in \E^1(f,\CC)   : e \subseteq \CC \bigr\}$, and the transition matrix $A_{\e}\: S_{\e} \times S_{\e}\rightarrow \{0, \, 1\}$ by
\begin{equation}    \label{eqDefA|}
A_{\e} \left(  e_1, e_2 \right)  
= \begin{cases} 1 & \text{if } f\left(e_1\right)\supseteq e_2, \\ 0  & \text{otherwise}  \end{cases}
\end{equation}
for $e_1,  \, e_2 \in S_{\e}$.

Define the set of states $S_{\ee} \coloneqq \bigl\{ (e,\c) \in \E^1(f,\CC) \times \{\b, \, \w\} : e \subseteq \CC \bigr\}$. For each $(e,\c)\in S_{\ee}$, we denote by $X^1(e,\c) \in \X^1(f,\CC)$ the unique $1$-tile satisfying 
\begin{equation}   \label{eqDefXec}
e \subseteq X^1(e,\c) \subseteq X^0_\c. 
\end{equation}
The existence and uniqueness of $X^1(f,\c)$ defined by (\ref{eqDefXec}) follows immediately from Proposition~\ref{propCellDecomp}~(iii), (v), and (vi) and the assumptions that $f(\CC)\subseteq \CC$ and $e\subseteq\CC$. We define the transition matrix $A_{\ee}\: S_{\ee} \times S_{\ee}\rightarrow \{0, \, 1\}$ by
\begin{equation}   \label{eqDefA||}
A_{\ee} \left( \left(e_1,\c_1 \right), \left(e_2,\c_2\right)\right)  
= \begin{cases} 1 & \text{if } f\left(e_1\right)\supseteq e_2 \text{ and } f\left(X^1\left(e_1,\c_1\right) \right) \supseteq X^1(e_2, \c_2), \\ 0  & \text{otherwise}  \end{cases}
\end{equation}
for $\left(e_1,\c_1 \right), \left(e_2,\c_2\right) \in S_{\ee}$.

We will consider the one-sided subshift of finite type $\bigl( \Sigma_{A_{\e}}^+, \sigma_{A_{\e}} \bigr)$ defined by the transition matrix $A_{\e}$ and $\bigl( \Sigma_{A_{\ee}}^+, \sigma_{A_{\ee}} \bigr)$ defined by the transition matrix $A_{\ee}$, where  
\begin{align*}
\Sigma_{A_{\e}}^+ & =  \{  \{   e_i  \}_{i\in\N_0}  \:  
             e_i \in S_{\e},\, A_{\e} (  e_i, e_{i+1} ) = 1 \text{ for each } i\in\N_0  \},  \\
\Sigma_{A_{\ee}}^+& =  \{  \{  (e_i,\c_i )  \}_{i\in\N_0}  \:  
            (e_i,\c_i)\in S_{\ee},\, A_{\ee} ( (e_i,\c_i ), (e_{i+1},\c_{i+1})) = 1 \text{ for each } i\in\N_0  \},
\end{align*}
$\sigma_{A_{\e}}$ and $\sigma_{A_{\ee}}$ are the left-shift operators on $\Sigma_{A_{\e}}^+$ and $\Sigma_{A_{\ee}}^+$, respectively (see Subsection~\ref{subsctSFT}).

See Figure~\ref{figDynOnC23} for the sets of states $S_{\ee}$ and $S_{\e}$ associated to an expanding Thurston map $f$ and an invariant Jordan curve $\CC$ whose cell decomposition $\DD^1(f,\CC)$ of $S^2$ is sketched in Figure~\ref{figDynOnC1}. Note that $S_{\ti} = \X^1(f,\CC)$. In this example, $f$ has three postcritical points $A$, $B$, and $C$. Intuitively, each $1$-edge in $S_{\ee}$ or $S_{\e}$ is mapped to the corresponding $0$-edge in the image of the $1$-tile containing it under $f$.

\begin{figure}
    \centering
    \begin{overpic}
    [width=6cm, %grid, 
    tics=20]{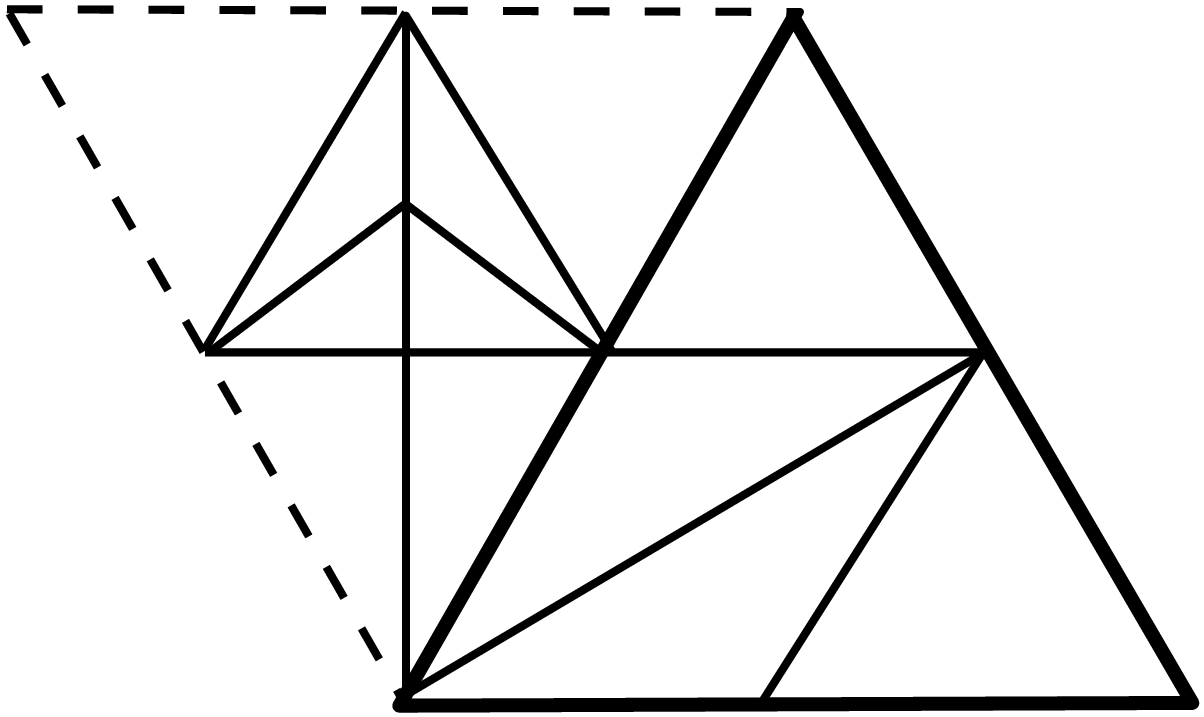}
    \put(109,101){$A$}
    \put(46,-8){$B$}
    \put(171,-8){$C$}
    \put(-9,100){$C$}
    \put(79,-10){$e^1_1$}
    \put(127,-10){$e^1_2$}
    \put(157,29){$e^1_3$}
    \put(130,78){$e^1_4$}
    \put(88,78){$e^1_5$}
    \put(64,36){$e^1_6$}
    \end{overpic}
    \caption{The cell decomposition $\DD^1(f,\CC)$. $S_{\ti} = \X^1(f,\CC)$.}
    \label{figDynOnC1}
%\end{figure}

%\begin{figure}
    \centering
    \begin{overpic}
    [width=11cm, %grid, 
    tics=20]{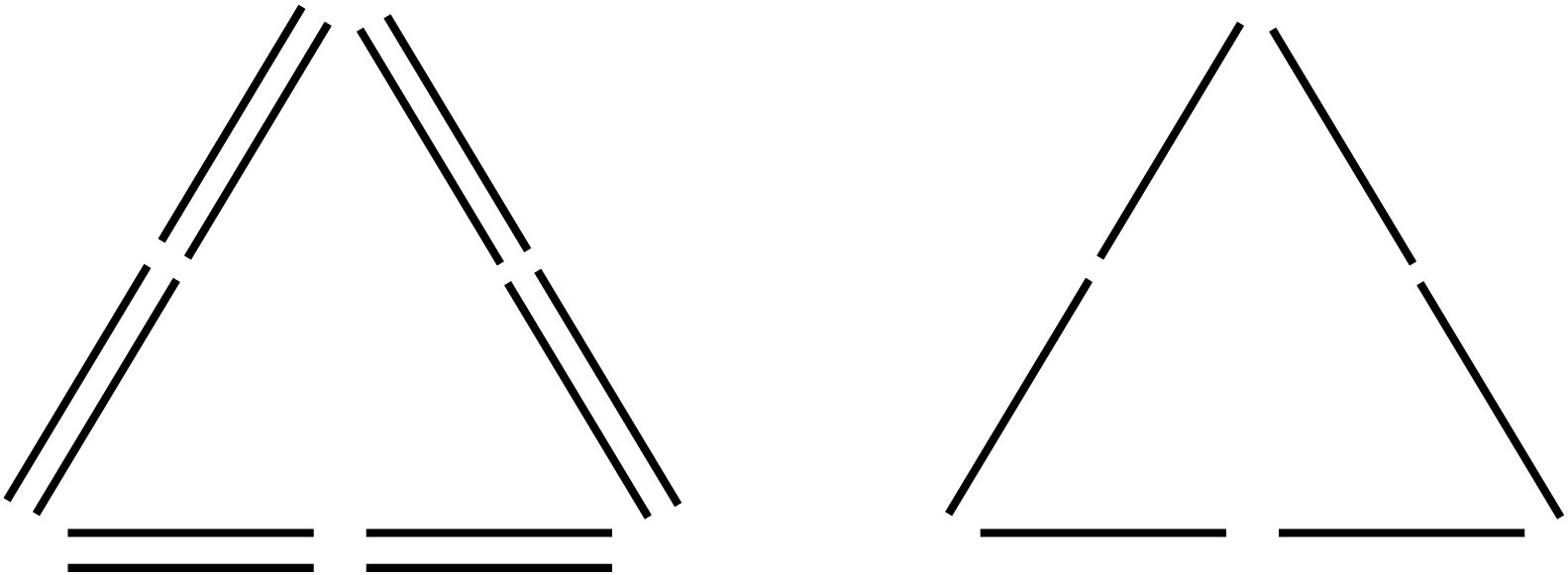}
    \put(24,-11){$(e^1_1,\b)$}
    \put(23, 14){$(e^1_1,\w)$}
    \put(83,-11){$(e^1_2,\b)$}
    \put(82, 14){$(e^1_2,\w)$}
    \put(125,36){$(e^1_3,\b)$}
    \put(77, 36){$(e^1_3,\w)$}
    \put(92,91){$(e^1_4,\b)$}
    \put(15,91){$(e^1_5,\b)$}
    \put(-18,36){$(e^1_6,\b)$}
    
    \put(215,-5){$e^1_1$}
    \put(273,-5){$e^1_2$}
    \put(299, 36){$e^1_3$}
    \put(268,90){$e^1_4$}
    \put(222,90){$e^1_5$}
    \put(190,36){$e^1_6$}
    \end{overpic}
    \caption{Elements in $S_{\ee}$ (left) and elements in $S_{\e}$ (right).}
    \label{figDynOnC23}
%\end{figure}

\end{figure}

\begin{prop}   \label{propSFTs_C}
Let $f$, $\CC$, $d$ satisfy the Assumptions in Section~\ref{sctAssumptions}. We assume, in addition, that $f(\CC)\subseteq\CC$. Let $\bigl( \Sigma_{A_{\ti}}^+,\sigma_{A_{\ti}} \bigr)$ be the one-sided subshift of finite type associated to $f$ and $\CC$ defined in Proposition~\ref{propTileSFT}, and let $\pi_{\ti}\: \Sigma_{A_{\ti}}^+\rightarrow S^2$ be the factor map defined in (\ref{eqDefTileSFTFactorMap}). Fix $\theta \in(0,1)$. Equip the spaces $\Sigma_{A_{\e}}^+$ and $\Sigma_{A_{\ee}}^+$ with the corresponding metrics $d_\theta$ constructed in Subsection~\ref{subsctSFT}. We write $\V(f,\CC) \coloneqq \bigcup_{i\in\N_0} \V^i(f,\CC)$. Then the following statements are satisfied:
\begin{enumerate}
\smallskip
\item[(i)] $\bigl( \Sigma_{A_{\e}}^+, \sigma_{A_{\e}} \bigr)$ is a factor of $\bigl( \Sigma_{A_{\ee}}^+, \sigma_{A_{\ee}} \bigr)$ with a Lipschitz continuous factor map $\pi_{\ee} \:  \Sigma_{A_{\ee}}^+ \rightarrow  \Sigma_{A_{\e}}^+$ defined by
\begin{equation}  \label{eqDefPi||}
 \pi_{\ee} ( \{ (e_i,\c_i) \}_{i\in\N_0} ) = \{ e_i \}_{i\in\N_0}
\end{equation}
for $\{ (e_i,\c_i) \}_{i\in\N_0} \in \Sigma_{A_{\ee}}^+$. Moreover, for each $\{  e_i  \}_{i\in\N_0} \in \Sigma_{A_{\e}}^+$, we have 
\begin{equation*}
\card \left(\pi_{\ee}^{-1} ( \{  e_i  \}_{i\in\N_0} )\right) = 2.
\end{equation*}

\smallskip
\item[(ii)] $( \CC, f|_\CC)$ is a factor of $\bigl( \Sigma_{A_{\e}}^+, \sigma_{A_{\e}} \bigr)$ with a H\"{o}lder continuous factor map $\pi_{\e} \:  \Sigma_{A_{\e}}^+ \rightarrow  \CC$ defined by
\begin{equation}  \label{eqDefPi|}
 \pi_{\e} ( \{  e_i  \}_{i\in\N_0} ) = x, \text{ where } \{x\} = \bigcap_{i \in \N_0} f^{-i}(e_i)
\end{equation}
for $\{  e_i  \}_{i\in\N_0} \in \Sigma_{A_{\e}}^+$. Moreover, for each $x\in\CC$, we have
\begin{equation}  \label{eqPi|Card}
\card \left(\pi_{\e}^{-1} (x)\right) = \begin{cases} 1 & \text{if } x\in \CC \setminus \V(f,\CC), \\ 2  & \text{if } x\in \CC\cap\V(f,\CC). \end{cases}
\end{equation}

\begin{comment}

\smallskip
\item[(iii)] The following identities hold:
\begin{equation*}
  \bigl\{ \tau \in \Orb\left(  \sigma_{A_{\ee}}  \right)   \:  (\pi_{\e}\circ\pi_{\ee}) (\tau) \cap \V(f,\CC) \neq \emptyset \bigr\}
= \bigl\{ \tau \in \Orb\left(  \sigma_{A_{\ee}}  \right)   \:  (\pi_{\e}\circ\pi_{\ee}) (\tau) \subseteq \post f \bigr\},
\end{equation*}
\begin{equation*}
  \bigl\{ \tau \in \Orb\left(  \sigma_{A_{\e}} \right)   \:   \pi_{\e}  (\tau) \cap \V(f,\CC) \neq \emptyset \bigr\}
= \bigl\{ \tau \in \Orb\left(  \sigma_{A_{\e}} \right)   \:   \pi_{\e}  (\tau) \subseteq \post f \bigr\},
\end{equation*}
\begin{align*}
    \{ \tau \in \Orb(f|_\CC)   \: \tau \cap \V(f,\CC) \neq \emptyset \}
= & \{ \tau \in \Orb(f|_\CC)   \: \tau \subseteq \post f \}\\
= & \{ \tau \in \Orb(f     )   \: \tau \subseteq \post f \}
=   \{ \tau \in \Orb(f     )   \: \tau \cap \V(f,\CC) \neq \emptyset \}.
\end{align*}
Moreover, all these sets are finite.
\end{comment}

\end{enumerate}
\end{prop}

Thus, we have the following commutative diagram:
\begin{equation*}
\xymatrix{ \Sigma_{A_{\ee}}^+ \ar[d]_{\sigma_{A_{\ee}}} \ar[r]^{\pi_{\ee}}  & \Sigma_{A_{\e}}^+ \ar[d]^{\sigma_{A_{\e}}} \ar[r]^{\pi_{\e}}  & \CC \ar[d]^{f|_\CC}\\ 
           \Sigma_{A_{\ee}}^+                           \ar[r]_{\pi_{\ee}}  & \Sigma_{A_{\e}}^+                          \ar[r]_{\pi_{\e}}  & \CC.}
\end{equation*}

\begin{proof}
(i) It follows immediately from (\ref{eqDefPi||}), (\ref{eqDefA||}), (\ref{eqDefA|}), and the definitions of $\Sigma_{A_{\ee}}^+$ and $\Sigma_{A_{\e}}^+$ that $\pi_{\ee} \bigl( \Sigma_{A_{\ee}}^+ \bigr) \subseteq \Sigma_{A_{\e}}^+$. By (\ref{eqDefPi||}), it is clear that $\pi_{\ee}$ is Lipschitz continuous.

Next, we show that $\card \bigl(\pi_{\ee}^{-1}  (  \{ e_i \}_{i\in\N_0} ) \bigr) = 2$ for each $\{ e_i \}_{i\in\N_0} \in \Sigma_{A_{\e}}^+$. The fact that $\pi_{\ee}$ is surjective then follows for free.

Fix arbitrary $\c\in\{\b, \, \w\}$ and $\{ e_i \}_{i\in\N_0} \in \Sigma_{A_{\e}}^+$.

We recursively construct $\c_i \in \{\b, \, \w\}$ for each $i\in\N_0$ such that $\c_0 = \c$ and  $\{ (e_i,\c_i) \}_{i\in\N_0} \in \Sigma_{A_{\ee}}^+$, and prove that such a sequence $\{ \c_i \}_{i\in\N_0}$ is unique. Let $\c_0 \coloneqq \c$. Assume that for some $k\in\N_0$, $\c_j$ is determined and is unique for all integer $j\in\N_0$ with $j\leq k$, in the sense that any other choice of $\c_j$ for any $j\in\N_0$ with $j\leq k$ would result in $\{ (e_i,\c_i) \}_{i\in\N_0} \notin \Sigma_{A_{\ee}}^+$ regardless of choices of $\c_j$ for $j>k$. Recall $X^1( e_k, \c_k )$ defined in (\ref{eqDefXec}). Since $f(e_k)\supseteq e_{k+1}$ and $f\bigl(X^1( e_k, \c_k )\bigr)$ is the black $0$-tile $X^0_\b$ or the white $0$-tile $X^0_\w$ by Proposition~\ref{propCellDecomp}~(i), we will have to choose $\c_{k+1} \coloneqq \b$ in the former case and $\c_{k+1} \coloneqq \w$ in the latter case due to (\ref{eqDefA||}). Hence, $\{ (e_i,\c_i) \}_{i\in\N_0} \in \pi_{\ee}^{-1}  (  \{  e_i \}_{i\in\N_0}  )$ is uniquely determined by $\{  e_i \}_{i\in\N_0}$ and $\c\in\{\b, \, \w\}$. This proves $\card \bigl(\pi_{\ee}^{-1}  (  \{ e_i \}_{i\in\N_0} ) \bigr) = 2$ for each $\{ e_i \}_{i\in\N_0} \in \Sigma_{A_{\e}}^+$.

Finally, it follows immediately from (\ref{eqDefPi||}) that $\pi_{\ee} \circ \sigma_{A_{\ee}} = \sigma_{A_{\e}} \circ \pi_{\ee}$.

\smallskip

(ii) Fix an arbitrary $\{e_i\}_{i\in\N_0} \in \Sigma_{A_{\e}}^+$.

Since $f(e_i)\supseteq e_{i+1}$ for each $i\in\N_0$, by Lemma~\ref{lmCylinderIsTile}, the map $\pi_{\e}$ is well-defined.

Note that for each $m\in \N_0$ and each $\{e'_i\}_{i\in\N_0} \in \Sigma_{A_{\e}}^+$ with $e_{m+1} \neq e'_{m+1}$ and $e_j=e'_j$ for each integer $j\in[0,m]$, we have $\{ \pi_{\e}  (\{e_i\}_{i\in\N_0}  ),  \pi_{\e}  ( \{ e'_i \}_{i\in\N_0}  ) \} \subseteq \bigcap_{i=0}^{m} f^{-i}(e_i) \in \E^{m+1}$ by Lemma~\ref{lmCylinderIsTile}. Thus, it follows from Lemma~\ref{lmCellBoundsBM}~(ii) that $\pi_{\e}$ is H\"{o}lder continuous.

To see that $\pi_{\e}$ is surjective, we observe that for each $x\in \CC$, we can find a sequence $\bigl\{ e^j(x) \bigr\}_{j\in\N}$ of edges such that $e^j(x) \in \E^j$, $e^j(x) \subseteq \CC$, $x\in e^j(x)$, and $e^j(x) \supseteq e^{j+1}(x)$ for each $j\in\N$. Then it is clear from Proposition~\ref{propCellDecomp}~(i) that $\bigl\{  f^i \bigl( e^{i+1}(x) \bigr) \bigr\}_{i\in\N_0} \in \Sigma_{A_{\e}}^+$ and $\pi_{\e} \bigl(  \bigl\{ f^i \bigl( e^{i+1}(x) \bigr) \bigr\}_{i\in\N_0} \bigr) = x$.

Next, to check that $\pi_{\e} \circ \sigma_{A_{\e}} = f \circ \pi_{\e}$, it suffices to observe that 
$
            \{ (f\circ \pi_{\e}) (\{e_i\}_{i\in\N_0}) \}
  =         f \bigl(  \bigcap_{j\in\N_0} f^{-j} (e_j) \bigr)
\subseteq        \bigcap_{j\in\N} f^{-(j-1)} (e_j) 
  =          \bigcap_{i\in\N_0} f^{-i} (e_{i+1})  
  =         \{(\pi_{\e} \circ \sigma_{A_{\e}} ) (\{e_i\}_{i\in\N_0})  \}.
$

Finally, we are going to establish (\ref{eqPi|Card}). Fix an arbitrary point $x\in\CC$.

\smallskip

\emph{Case~1.} $x\in \CC\setminus \V(f,\CC)$.

\smallskip
We argue by contradiction and assume that there exist distinct $\{e_i\}_{i\in\N_0},\, \{e'_i\}_{i\in\N_0} \in \pi_{\e}^{-1}(x)$. Choose $m\in\N_0$ to be the smallest non-negative integer with $e_m\neq e'_m$. Then by Lemma~\ref{lmCylinderIsTile}, $x\in \bigcap_{i=0}^m f^{-i} (e_i) \in \E^{m+1}$ and $x\in \bigcap_{i=0}^m f^{-i} (e'_i) \in \E^{m+1}$. Thus, $f^m(x) \in e_m \cap e'_m \subseteq \V^1$ by Proposition~\ref{propCellDecomp}~(v). This is a contradiction to the assumption that $x\in \CC\setminus \V(f,\CC)$. Hence, $\card \bigl( \pi_{\e}^{-1} (x) \bigr) = 1$.

\smallskip

\emph{Case~2.} $x\in \CC\cap \V(f,\CC)$.

\smallskip

Denote $n \coloneqq \min \bigl\{ i\in\N : x\in\V^i \bigr\} \in \N$.

For each $j\in\N$ with $j<n$, we define $e^j_1,e^j_2 \in \E^j$ to be the unique $j$-edge with $x\in e^j_1=e^j_2 \subseteq \CC$. For each $i\in\N$ with $i\geq n$, we choose the unique pair $e^i_i, \, e^i_2 \in \E^i$ of $i$-edges satisfying (1) $e^i_1 \cup e^i_2 \subseteq \CC$, (2) $e^i_1 \cap e^i_2 = \{x\}$, and (3) if $i \geq 2$, then $e^i_1 \subseteq e^{i-1}_1$ and $e^i_2 \subseteq e^{i-1}_2$. Then it is clear from Proposition~\ref{propCellDecomp}~(i) that for each $k\in\{1, \, 2\}$, $\bigl\{ f^i \bigl(e^{i+1}_k\bigr) \bigr\}_{i\in\N_0} \in \Sigma_{A_{\e}}^+$ and $\pi_{\e} \bigl(  \bigl\{ f^i \bigl(e^{i+1}_k\bigr) \bigr\}_{i\in\N_0} \bigr) = x$.

Note that if $n=1$, then $e^1_1 \neq e^1_2$. If $n\geq 2$, then $f^i(x) \notin \V^1$ and thus $f^i(x) \notin \crit f|_\CC$, for each $i\in\{0, \, 1, \, \dots, \, n-2\}$. So $f^{n-1}$ is injective on a neighborhood of $x$ and consequently by Proposition~\ref{propCellDecomp}~(i), $f^{n-1} \bigl(e^n_1\bigr) \neq f^{n-1} \bigl(e^n_2\bigr)$. Hence, $\card \bigl( \pi_{\e}^{-1} (x) \bigr) \geq 2$.

On the other hand, for each $\{ e_i \}_{i\in\N_0} \in \pi_{\e}^{-1} (x)$ and each $j\in\N_0$, (1) if $j<n-1$, then $e_j = f^j \bigl(e^{j+1}_1 \bigr) = f^j \bigl(e^{j+1}_2 \bigr)$ (since $e^{j+1}_1=e^{j+1}_2$) and $f^j(x)\in \inte \bigl( f^j  \bigl(e^{j+1}_1 \bigr) \bigr)$ (by the definition of $n$); (2) if $j=n-1$, then either $e_j= f^j  \bigl(e^{j+1}_1 \bigr)$ or $e_j= f^j  \bigl(e^{j+1}_2 \bigr)$ since $f^{j}(x) \in f^j  \bigl(e^{j+1}_1 \bigr) \cap f^j  \bigl(e^{j+1}_2 \bigr)$; (3) if $j\geq n$, then  $f^j(x) \in \V^0$ and consequently by Proposition~\ref{propCellDecomp}~(i) and (v), there exists exactly one $1$-edge $e_j\in\E^1$ such that $f^j(x) \in e_j$ and $f(e_{j-1}) \supseteq e_j$. Hence, $\card \bigl( \pi_{\e}^{-1} (x) \bigr) \leq 2$. 

Identity (\ref{eqPi|Card}) is now established.
\end{proof}

\subsection{Combinatorics}   \label{subsctDynOnC_Combinatorics}

\begin{lemma}   \label{lmPeriodicPtsLocationCases}
Let $f$ and $\CC$ satisfy the Assumptions in Section~\ref{sctAssumptions}. We assume, in addition, that $f(\CC)\subseteq\CC$. For each $n\in\N$ and each $x\in S^2$ with $f^n(x)=x$, exactly one of the following statements holds:
\begin{enumerate}
\smallskip
\item[(i)] $x\in\inte (X^n)$ for some $n$-tile $X^n\in\X^n(f,\CC)$, where $X^n$ is either a black $n$-tile contained in the black $0$-tile $X^0_\b$ or a white $n$-tile contained in the white $0$-tile $X^0_\w$. Moreover, $x\notin \bigcup_{i\in\N_0} \left( \bigcup \E^i(f,\CC) \right)$.

\smallskip
\item[(ii)] $x\in \inte (e^n)$ for some $n$-edge $e^n\in\E^n(f,\CC)$ satisfying $e^n \subseteq \CC$. Moreover, $x\notin \bigcup_{i\in\N_0}  \V^i(f,\CC)$.

\smallskip
\item[(iii)] $x\in\post f$.
\end{enumerate} 
\end{lemma}

\begin{proof}
Fix $x\in S^2$ and $n\in\N$ with $f^n(x)=x$. It is easy to see that at most one of Cases~(i), (ii), and (iii) holds. By Proposition~\ref{propCellDecomp}~(iii) and (v), it is clear that exactly one of the following cases holds:
\begin{enumerate}
\smallskip
\item[(1)] $x\in\inte (X^n)$ for some $n$-tile $X^n\in\X^n$.

\smallskip
\item[(2)] $x\in \inte (e^n)$ for some $n$-edge $e^n\in\E^n$.

\smallskip
\item[(3)] $x\in \V^n$.
\end{enumerate}

Assume that case~(1) holds. We argue by contradiction and assume that there exists $j\in\N_0$ and $e\in\E^j$ such that $x\in e$. Then for $k \coloneqq \bigl\lceil\frac{j+1}{n}\bigr\rceil \in\N$, $x=f^{kn}(x)\in f^{kn}(e)\subseteq \CC$, contradicting with $x\in\inte(X^n)$. So $x\notin \bigcup_{i\in\N_0} \left( \bigcup \E^i \right)$. By Lemma~\ref{lmAtLeast1}, the rest of statement~(i) holds. Hence, statement~(i) holds in case~(1).

Assume that case~(2) holds. By Proposition~\ref{propCellDecomp}~(i), $x=f^n(x) \in \inte (e^0) \subseteq\CC$ where $e^0=f^n(e^n)\in \E^0$. Since $f(\CC)\subseteq \CC$, $\DD^n$ is a refinement of $\DD^0$ (see Definition~\ref{defrefine}). So we can choose an arbitrary $n$-edge $e^n_*\in\E^n$ contained in $e^0$ with $x\in e^n_*$. Since $x\notin \V^n$, we have $x\in \inte (e^n_*)$. By Definition~\ref{defcelldecomp}, $e^n = e^n_* \subseteq e^0\subseteq \CC$. To verify that $x\notin \bigcup_{i\in\N_0}  \V^i$, we argue by contradiction and assume that there exists $j\in\N_0$ such that $x\in\V^j$. Then for $k \coloneqq \bigl\lceil\frac{j+1}{n}\bigr\rceil \in\N$, $x=f^{kn}(x)\in \V^0$, contradicting with $x\in\inte(e^n)$. Thus, $x\notin \bigcup_{i\in\N_0}  \V^i$. Hence, statement~(ii) holds in case~(2).

Assume that case~(3) holds. By Proposition~\ref{propCellDecomp}~(i), $x = f^n(x) \subseteq \V^0 = \post f$. Hence, statement~(iii) holds in case~(3).
\end{proof}

Let $f$ be an expanding Thurston map with an $f$-invariant Jordan curve $\CC$ containing $\post f$. We orient $\CC$ in such a way that the white $0$-tile lies on the left of $\CC$. Let $p\in \CC$ be a fixed point of $f$. We say that $f|_\CC$ \defn{preserves the orientation at $p$} (resp.\ \defn{reverses the orientation at $p$}) if there exists an open arc $l\subseteq\CC$ with $p\in l$ such that $f$ maps $l$ homeomorphically to $f(l)$ and $f|_\CC$ preserves (resp.\ reverses) the orientation on $l$. Note that it may happen that $f|_\CC$ neither preserves nor reverses the orientation at $p$, because $f|_\CC$ need not be a local homeomorphism near $p$, where it may behave like a ``folding map''.

\begin{theorem}   \label{thmNoPeriodPtsIdentity}
Let $f$ and $\CC$ satisfy the Assumptions in Section~\ref{sctAssumptions}. We assume, in addition, that $f(\CC)\subseteq\CC$. Let $\bigl(\Sigma_{A_{\ti}}^+,\sigma_{A_{\ti}}\bigr)$ be the one-sided subshift of finite type associated to $f$ and $\CC$ defined in Proposition~\ref{propTileSFT}, and let $\pi_{\ti}\: \Sigma_{A_{\ti}}^+\rightarrow S^2$ be the factor map defined in (\ref{eqDefTileSFTFactorMap}). Recall the one-sided subshifts of finite type $\bigl( \Sigma_{A_{\e}}^+, \sigma_{A_{\e}} \bigr)$ and $\bigl( \Sigma_{A_{\ee}}^+, \sigma_{A_{\ee}} \bigr)$ constructed in Subsection~\ref{subsctDynOnC_Construction}, and the factor maps $\pi_{\e}\: \Sigma_{A_{\e}}^+\rightarrow S^2$, $\pi_{\ee}\: \Sigma_{A_{\ee}}^+\rightarrow \Sigma_{A_{\e}}^+$ defined in Proposition~\ref{propSFTs_C}. We denote by $\left(\V^0, f|_{\V^0}\right)$ the dynamical system on $\V^0=\V^0(f,\CC) = \post f$ induced by $f|_{\V^0} \: \V^0\rightarrow \V^0$. For each $y\in S^2$ and each $i\in\N$, we write
\begin{align*}
 M_{\po} (y,i)  &\coloneqq  \card \bigl( P_{1, (f|_{\V^0})^i }       \cap                                  \{y\} \bigr),  \\
 M_{\e}  (y,i)  &\coloneqq  \card \bigl( P_{1, \sigma_{A_{\e}}^i }   \cap \pi_{\e}^{-1}                   (y)\bigr),  \\
 M_{\ee} (y,i)  &\coloneqq  \card \bigl( P_{1, \sigma_{A_{\ee}}^i }  \cap (\pi_{\e} \circ \pi_{\ee})^{-1} (y)\bigr),  \\
 M_{\ti} (y,i)  &\coloneqq  \card \bigl( P_{1, \sigma_{A_{\ti}}^i }  \cap \pi_{\ti}^{-1}                  (y)\bigr). 
\end{align*}
Then for each $n\in\N$ and each $x\in P_{1,f^n}$, we have
\begin{equation}   \label{eqNoPeriodPtsIdentity}
M_{\ti} (x,n)  - M_{\ee} (x,n) + M_{\e} (x,n) + M_{\po} (x,n) = \deg_{f^n} (x).
\end{equation}
\end{theorem}

\begin{proof}
Fix an arbitrary integer $n\in\N$ and an arbitrary fixed point $x\in P_{1,f^n}$ of $f^n$.

We establish (\ref{eqNoPeriodPtsIdentity}) by verifying it in each of the three cases of Lemma~\ref{lmPeriodicPtsLocationCases} depending on the location of $x$.

\smallskip

\emph{Case~(i) of Lemma~\ref{lmPeriodicPtsLocationCases}:} $x\in \inte(X^n)$ for some $n$-tile $X^n \in \X^n$, where $X^n$ is either a black $n$-tile contained in the black $0$-tile $X^0_\b$ or a white $n$-tile contained in the white $0$-tile $X^0_\w$. Moreover, $x\notin \bigcup_{i\in\N_0} \bigl( \bigcup \E^i \bigr) = \bigcup_{i\in\N_0} f^{-i} (\CC)$ (see Proposition~\ref{propCellDecomp}~(iii)). 

Thus, by Proposition~\ref{propTileSFT}, $\card \bigl(\pi_{\ti}^{-1}(x) \bigr) = 1$. For each $i\in\N_0$, we denote by $X^i(x)\in\X^i$ the unique $i$-tile containing $x$. Fix an arbitrary integer $j\in \N_0$. Then $f^j\bigl(X^{j+1}(x)\bigr) \in \X^1$ (see Propostion~\ref{propCellDecomp}~(i)) and $X^{j+1}(x)\subseteq X^j(x)$. Thus, $f \bigl( f^j\bigl(X^{j+1}(x)\bigr) \bigr) \supseteq f^{j+1}\bigl(X^{j+2}(x)\bigr)$. It follows from Lemma~\ref{lmCylinderIsTile} and (\ref{eqDefTileSFTFactorMap}) that $\pi_{\ti}^{-1} (x) = \bigl\{  \bigl\{ f^i\bigl(X^{i+1}(x)\bigr) \bigr\}_{i\in\N_0} \bigr\} \subseteq \Sigma_{A_{\ti}}^+$. Observe that $f^j\bigl(X^{j+1}(x)\bigr)$ is the unique $1$-tile containing $f^j(x)$, and that $f^{j+n}\bigl(X^{j+n+1}(x)\bigr)$ is the unique $1$-tile containing $f^{j+n}(x)$. Since $f^n(x)=x$, we can conclude from Definition~\ref{defcelldecomp} that $f^{j+n}\bigl(X^{j+n+1}(x)\bigr) = f^j\bigl(X^{j+1}(x)\bigr)$. Hence, $\bigl\{ f^i \bigl( X^{i+1}(x) \bigr) \bigr\}_{i\in\N_0} \in P_{1,\sigma_{A_{\ti}}^n}$ and $M_{\ti} = 1$. On the other hand, since $x\notin\CC$, we have $M_{\ee} (x,n) = M_{\e} (x,n) = M_{\po} (x,n) = 0$ by Proposition~\ref{propSFTs_C}. Since $x\in \inte (X^n)$, we have $\deg_{f^n} (x) = 1$. This establishes identity (\ref{eqNoPeriodPtsIdentity}) in Case~(i) of Lemma~\ref{lmPeriodicPtsLocationCases}.

\smallskip

\emph{Case~(ii) of Lemma~\ref{lmPeriodicPtsLocationCases}:} $x \in \inte(e^n)$ for some $n$-edge $e^n\in\E^n$ with $e^n \subseteq \CC$. Moreover, $x\notin \bigcup_{i\in\N_0} \V^i$. So $\deg_{f^n}(x) = 1$ and $M_{\po} (x,n) = 0$.

We will establish (\ref{eqNoPeriodPtsIdentity}) in this case by proving the following two claims.

\smallskip

\emph{Claim~1.} $M_{\e}(x,n) = 1$.

\smallskip

Since $\card \bigl( \pi_{\e}^{-1} (x) \bigr) = 1$ by Proposition~\ref{propSFTs_C}~(ii), it suffices to show that $\sigma_{A_{\e}}^n \bigl( \pi_{\e}^{-1} (x) \bigr) =  \pi_{\e}^{-1} (x)$. For each $y\in \CC \setminus \bigcup_{i\in\N_0} \V^i$ and $i\in\N_0$, we denote by $e^i(y) \in \E^i$ to be the unique $i$-edge containing $y$. Fix an arbitrary integer $j\in \N_0$. Then $f^j\bigl(e^{j+1}(x)\bigr) \in\E^1$ (see Proposition~\ref{propCellDecomp}~(i)) and $e^{j+1}(x) \subseteq e^j(x)$. Thus, $f \bigl( f^j \bigl( e^{j+1}(x) \bigr) \bigr) \supseteq f^{j+1} \bigl( e^{j+2}(x) \bigr)$. It follows from Lemma~\ref{lmCylinderIsTile} and (\ref{eqDefPi|}) that $\pi_{\e}^{-1} (x) = \bigl\{ \bigl\{ f^i \bigl( e^{i+1}(x) \bigr) \bigr\}_{i\in\N_0} \bigr\} \subseteq \Sigma_{A_{\e}}^+$. Observe that $f^j\bigl(e^{j+1}(x)\bigr)$ is the unique $1$-edge containing $f^j(x)$, and that $f^{j+n}\bigl(e^{j+n+1}(x)\bigr)$ is the unique $1$-edge containing $f^{j+n}(x)$. Since $f^n(x)=x$, we can conclude from Definition~\ref{defcelldecomp} that $f^{j+n}\bigl(e^{j+n+1}(x)\bigr) = f^j\bigl(e^{j+1}(x)\bigr)$. Hence, $\bigl\{ f^i \bigl( e^{i+1}(x) \bigr) \bigr\}_{i\in\N_0} \in P_{1,\sigma_{A_{\e}}^n}$ and $M_{\e}(x,n) = 1$, proving Claim~1.

\smallskip

\emph{Claim~2.} $M_{\ee}(x,n) = M_{\ti}(x,n)$.

\smallskip

We prove this claim by constructing a bijection $h\: \pi_{\ti}^{-1} (x) \rightarrow (\pi_{\e}\circ\pi_{\ee})^{-1} (x)$ explicitly and show that $h(\underline{z})\in P_{1,\sigma_{A_{\ee}}^n}$ if and only if $\underline{z}\in P_{1,\sigma_{A_{\ti}}^n}$.

Intuitively, using the point of view of ``descending chains'' as discussed before Lemma~\ref{lmCylinderIsTile}, it suffices to observe that it is effectively equivalent to use a descending chain of ``edges'' in the left figure of Figure~\ref{figDynOnC23} to identify a point in Case~(ii) versus to use a descending chain of tiles attached to the ``edges'' above to identify such a point.

For each $y\in \CC \setminus \bigcup_{i\in\N_0} \V^i$, each $\c\in\{\b, \, \w\}$, and each $i\in\N_0$, we denote by $X^{\c,i}(y) \in \X^i$ the unique $i$-tile satisfying $y\in X^{\c,i}(y)$ and $X^{\c,i}(y) \subseteq X^0_\c$. Here, $X^0_\b$ (resp.\ $X^0_\w$) is the unique black (resp.\ white) $0$-tile. Recall that as defined above, $e^i(x) \in \E^i$ is the unique $i$-edge containing $x$, for $i\in\N_0$. Then for each $\c\in\{\b, \, \w\}$ and each $i\in\N_0$, we have $e^i(x)\subseteq X^{\c,i}(x)$ (see Definition~\ref{defcelldecomp}), $f^i\bigl( X^{\c,i+1}(x) \bigr) \in \X^1$ (see Proposition~\ref{propCellDecomp}~(i)), and $X^{\c,i+1}(x) \subseteq X^{\c,i}(x)$. Thus, 
\begin{equation}   \label{eqPfthmNoPeriodPtsIdentity_X}
f \bigl( f^i \bigl( X^{\c,i+1}(x) \bigr) \bigr) \supseteq f^{i+1} \bigl( X^{\c,i+2}(x) \bigr).
\end{equation}
It follows from Lemma~\ref{lmCylinderIsTile} and (\ref{eqDefTileSFTFactorMap}) that $\bigl\{ f^i \bigl( X^{\c,i+1}(x) \bigr) \bigr\}_{i\in\N_0} \in \Sigma_{A_{\ti}}^+$ and 
\begin{equation*} 
\pi_{\ti} \bigl( \bigl\{ f^i \bigl( X^{\c,i+1}(x) \bigr) \bigr\}_{i\in\N_0} \bigr) = x.
\end{equation*}

Next, we show that $\card \bigl( \pi_{\ti}^{-1}(x) \bigr) = 2$. We argue by contradiction and assume that $\card \bigl( \pi_{\ti}^{-1}(x) \bigr) \geq 3$. We choose $\{X_i\}_{i\in\N_0} \in \pi_{\ti}^{-1}(x)$ different from $\bigl\{ f^i \bigl( X^{\b,i+1}(x) \bigr) \bigr\}_{i\in\N_0}$ and $\bigl\{ f^i \bigl( X^{\w,i+1}(x) \bigr) \bigr\}_{i\in\N_0}$. Since $x\in \CC \setminus \bigcup_{i\in\N_0} \V^i$, for each $j\in\N_0$, there exist exactly two $1$-tiles containing $f^j(x)$, namely, $X^{\b,1}\bigl(f^j(x)\bigr)$ and $X^{\w,1}\bigl(f^j(x)\bigr)$. Since $f^j (x) \in X_j$ for each $j\in\N_0$ (see (\ref{eqDefTileSFTFactorMap})), we get that there exists an integer $k\in\N_0$ and distinct $\c_1, \, \c_2\in\{\b, \, \w\}$ such that $X_k = f^k\bigl( X^{\c_1,k+1}(x) \bigr)$ and $X_{k+1} = f^{k+1}\bigl( X^{\c_2,k+2}(x) \bigr)$. Since $X^{\c_2,k+2}(x) \subseteq X^{\c_2,k+1}(x)$, $X^{\c_2,k+2}(x) \nsubseteq X^{\c_1,k+1}(x)$, and $f^{k+1}$ is injective on $\inte \bigl(  X^{\c_1,k+1}(x) \bigr) \cup \inte \bigl(  X^{\c_2,k+1}(x) \bigr)$, we get 
$
f(X_k) = f^{k+1} \bigl(  X^{\c_1,k+1}(x) \bigr) \nsupseteq f^{k+1} \bigl(  X^{\c_2,k+2}(x) \bigr) = X_{k+1}.
$
This is a contradiction. Hence, $\card \bigl( \pi_{\ti}^{-1}(x) \bigr) = 2$.

We define $h\: \pi_{\ti}^{-1}(x) \rightarrow (\pi_{\e} \circ \pi_{\ee})^{-1}(x)$ by
\begin{equation}   \label{eqPfthmNoPeriodPtsIdentity_Def_h}
   h \bigl( \bigl\{ f^i \bigl( X^{\c,i+1}(x) \bigr) \bigr\}_{i\in\N_0} \bigr) 
= \bigl\{ \bigl( f^i \bigl( e^{i+1} (x) \bigr), \c_i (\c)   \bigr)  \bigr\}_{i\in\N_0},   \qquad \c\in\{\b, \, \w\},
\end{equation}
where $\c_i (\c) \in \{\b, \, \w\}$ is the unique element in $\{\b, \, \w\}$ with the property that
\begin{equation}   \label{eqPfthmNoPeriodPtsIdentity_h2}
f^i \bigl(  X^{\c,i+1} (x) \bigr)  \subseteq X_{\c_i (\c)}^0
\end{equation}
for $i\in\N_0$. 

We first verify that $\bigl\{ \bigl( f^i \bigl( e^{i+1} (x) \bigr), \c_i (\c)   \bigr)  \bigr\}_{i\in\N_0} \in \Sigma_{A_{\ee}}^+$ for each $\c\in\{\b, \, \w\}$. Fix arbitrary $\c\in\{\b, \, \w\}$ and $j\in\N_0$. Since $e^{j+2}(x) \subseteq e^{j+1}(x) \subseteq \CC$, we get $f^j\bigl( e^{j+1}(x) \bigr) \subseteq \CC$ and 
\begin{equation}   \label{eqPfthmNoPeriodPtsIdentity_h1}
f \bigl( f^j \bigl( e^{j+1}(x) \bigr) \bigr) \supseteq f^{j+1} \bigl( e^{j+2}(x) \bigr).
\end{equation}
Recall that $X^1 \bigl(  f^j \bigl( e^{j+1} (x) \bigr), \c_j (\c) \bigr) \in \X^1$ denotes the unique $1$-tile satisfying 
\begin{equation*}
f^j \bigl( e^{j+1} (x) \bigr)    \subseteq     X^1 \bigl(  f^j \bigl( e^{j+1} (x) \bigr), \c_j (\c) \bigr)    \subseteq     X_{\c_j (\c)}^0
\end{equation*}
(see Proposition~\ref{propCellDecomp}~(iii), (v), and (vi) for its existence and uniqueness). Then by (\ref{eqPfthmNoPeriodPtsIdentity_h2}) and the fact that $e^{j+1}(x) \subseteq X^{\c,j+1} (x)$ (see Definition~\ref{defcelldecomp}), we get
\begin{equation}   \label{eqPfthmNoPeriodPtsIdentity_X=X}
   X^1 \bigl(  f^j \bigl( e^{j+1} (x) \bigr), \c_j (\c) \bigr)
=  f^j \bigl( X^{\c,j+1} (x) \bigr).
\end{equation}
Then by (\ref{eqDefA||}), (\ref{eqPfthmNoPeriodPtsIdentity_X=X}), (\ref{eqPfthmNoPeriodPtsIdentity_X}), and (\ref{eqPfthmNoPeriodPtsIdentity_h1}), $\bigl\{ \bigl( f^i \bigl( e^{i+1} (x) \bigr), \c_i (\c)   \bigr)  \bigr\}_{i\in\N_0} \in \Sigma_{A_{\ee}}^+$.

Note that since $X^{\c,1} (x) \subseteq X^0_\c$, we get $\c_0 (\c) = \c$ for $\c\in\{\b, \, \w\}$ from (\ref{eqPfthmNoPeriodPtsIdentity_h2}). Thus,
\begin{equation*}
     \bigl\{ \bigl( f^i \bigl( e^{i+1} (x) \bigr), \c_i (\b)   \bigr)  \bigr\}_{i\in\N_0}
\neq \bigl\{ \bigl( f^i \bigl( e^{i+1} (x) \bigr), \c_i (\w)   \bigr)  \bigr\}_{i\in\N_0},
\end{equation*}
i.e., $h$ is injective. By Proposition~\ref{propSFTs_C}~(i) and (ii), $\card \bigl( (\pi_{\e} \circ \pi_{\ee} )^{-1} ( x ) \bigr) = 2$. Thus, $h$ is a bijection.

It suffices now to show that for each $\underline{z} \in \pi_{\ti}^{-1}(x)$,  $h(\underline{z}) \in P_{1,\sigma_{A_{\ee}}^n }$ if and only if $\underline{z} \in P_{1,\sigma_{A_{\ti}}^n }$. Note that $e^i(x) \subseteq \CC$ for all $i\in\N_0$. Fix arbitrary $\c\in\{\b, \, \w\}$ and $i\in\N$. Note that since $f^i(x) \in f^i \bigl( e^{i+1}(x) \bigr)$, $f^i(x) = f^{i+n}(x) \in f^{i+n} \bigl( e^{i+n+1}(x) \bigr)$, and $f^i(x) \notin \bigcup_{i\in\N_0} \V^i$, we have
\begin{equation}  \label{eqPfthmNoPeriodPtsIdentity_Periodic}
\bigl( f^i \bigl( e^{i+1} (x) \bigr), \c_i (\c) \bigr) = \bigl( f^{i+n} \bigl( e^{i+n+1} (x) \bigr), \c_{i+n} (\c) \bigr)
\end{equation}
if and only if $X^1 \bigl(  f^i \bigl( e^{i+1} (x) \bigr), \c_i (\c) \bigr) = X^1  \bigl( f^{i+n} \bigl( e^{i+n+1} (x) \bigr), \c_{i+n} (\c) \bigr)$. Thus, by (\ref{eqPfthmNoPeriodPtsIdentity_X=X}), we get that (\ref{eqPfthmNoPeriodPtsIdentity_Periodic}) holds if and only if $f^i \bigl( X^{\c,i+1} (x) \bigr) = f^{i+n} \bigl( X^{\c,i+n+1} (x) \bigr)$. Hence, by (\ref{eqPfthmNoPeriodPtsIdentity_Def_h}), $h(\underline{z}) \in P_{1,\sigma_{A_{\ee}}^n }$ if and only if $\underline{z} \in P_{1,\sigma_{A_{\ti}}^n }$, for each $\underline{z} \in \pi_{\ti}^{-1}(x)$.

Claim~2 is now established. Therefore, (\ref{eqNoPeriodPtsIdentity}) holds in Case~(ii) of Lemma~\ref{lmPeriodicPtsLocationCases}.

\smallskip

\emph{Case~(iii) of Lemma~\ref{lmPeriodicPtsLocationCases}:} $x\in\post f$.

We will establish (\ref{eqNoPeriodPtsIdentity}) in this case by verifying it in each of the following subcases.

\begin{figure}
    \centering
    \begin{overpic}
    [width=6cm, %grid, 
    tics=20]{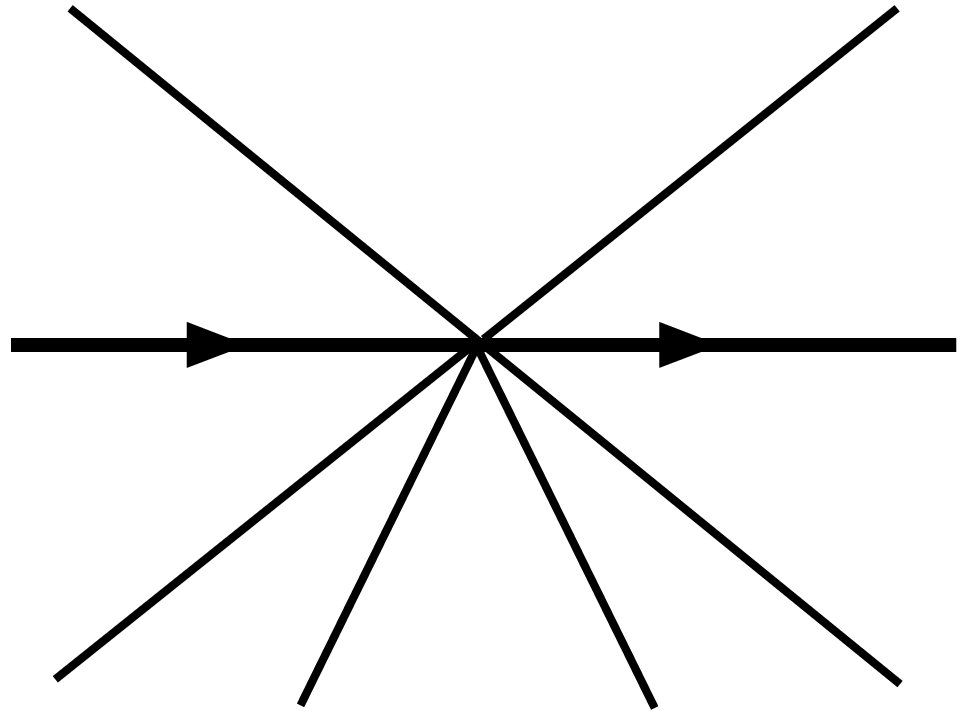}
    \put(75,100){$\b\w$}
    \put(25,80){$\w\w$}
    \put(130,80){$\w\w$}
    \put(25,33){$\b\b$}
    \put(40,20){$\w\b$}
    \put(79,12){$\b\b$}
    \put(112,20){$\w\b$}
    \put(135,33){$\b\b$}
    \put(20,57){$e_1$}
    \put(140,57){$e_2$}
    \put(80,50){$x$}
    \put(173,60){$\CC$}
    \put(163,80){$X^0_\w$}
    \put(163,40){$X^0_\b$}
    \end{overpic}
    \caption{Subcase (2)(a) where $f^n(e_1)\supseteq e_1$ and $f^n(e_2)\supseteq e_2$. $k=2$, $l=3$.}
    \label{figPlota}
%\end{figure}

%\begin{figure}
    \centering
    \begin{overpic}
    [width=6cm, %grid, 
    tics=20]{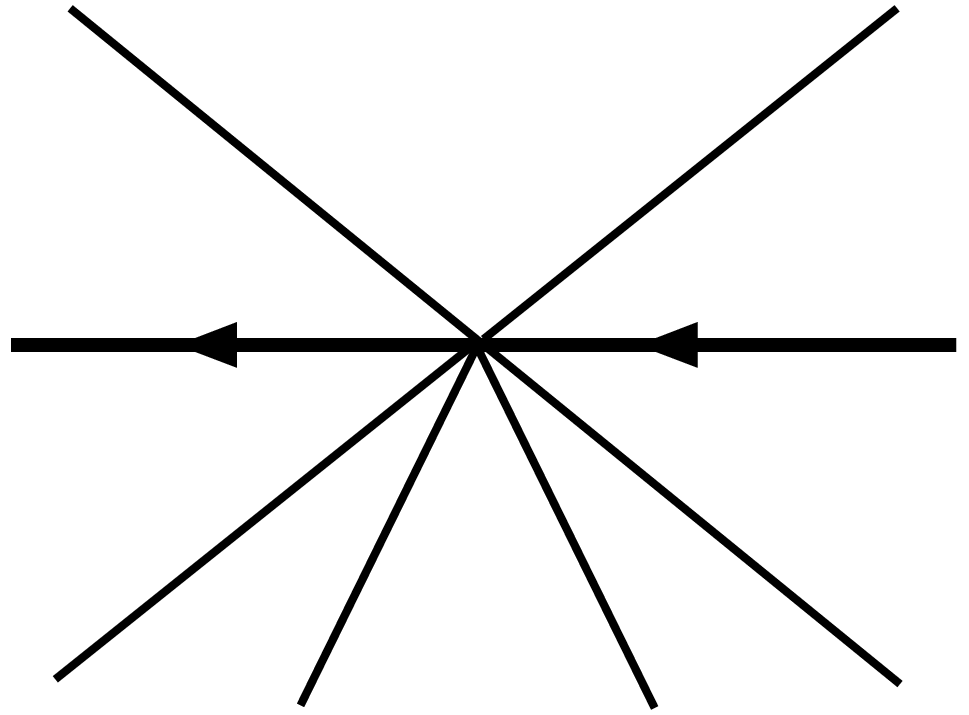}
    \put(75,100){$\w\w$}
    \put(25,80){$\b\w$}
    \put(130,80){$\b\w$}
    \put(25,33){$\w\b$}
    \put(40,20){$\b\b$}
    \put(79,12){$\w\b$}
    \put(112,20){$\b\b$}
    \put(135,33){$\w\b$}
    \put(20,57){$e_1$}
    \put(140,57){$e_2$}
    \put(80,50){$x$}
    \put(173,60){$\CC$}
    \put(163,80){$X^0_\w$}
    \put(163,40){$X^0_\b$}
    \end{overpic}
    \caption{Subcase (2)(b) where $f^n(e_1)\supseteq e_2$ and $f^n(e_2)\supseteq e_1$. $k=2$, $l=3$.}
    \label{figPlotb}
%\end{figure}

%\begin{figure}
    \centering
    \begin{overpic}
    [width=6cm, %grid, 
    tics=20]{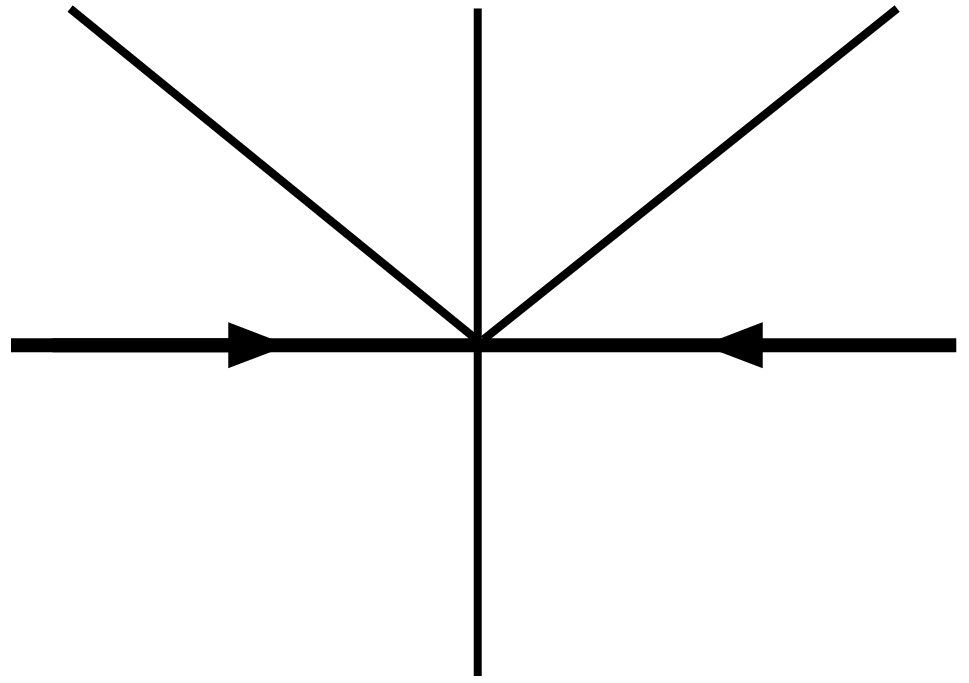}
    \put(55,100){$\b\w$}
    \put(95,100){$\w\w$}
    \put(25,70){$\w\w$}
    \put(130,70){$\b\w$}
    \put(37,22){$\b\b$}
    \put(117,22){$\w\b$}
    \put(25,48){$e_1$}
    \put(140,48){$e_2$}
    \put(75,50){$x$}
    \put(173,57){$\CC$}
    \put(163,77){$X^0_\w$}
    \put(163,37){$X^0_\b$}
    \end{overpic}
    \caption{Subcase (2)(c) where $f^n(e_1)=f^n(e_2)\supseteq e_1$. $k=2$, $l=1$.}
    \label{figPlotc}
%\end{figure}

%\begin{figure}
    \centering
    \begin{overpic}
    [width=6cm, %grid, 
    tics=20]{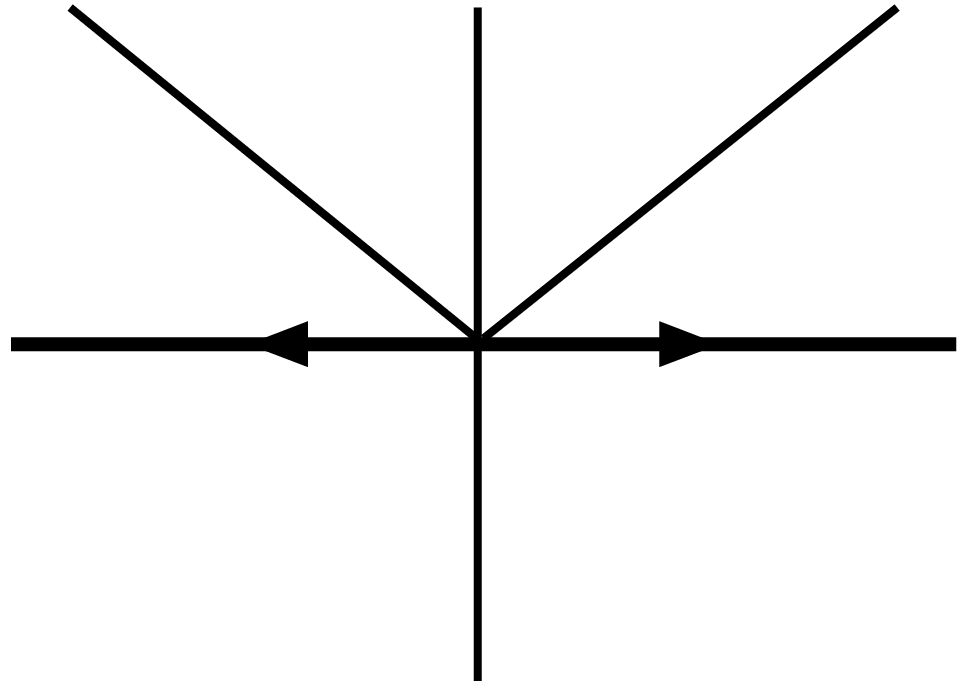}
    \put(55,100){$\w\w$}
    \put(95,100){$\b\w$}
    \put(25,70){$\b\w$}
    \put(130,70){$\w\w$}
    \put(37,22){$\w\b$}
    \put(117,22){$\b\b$}
    \put(25,48){$e_1$}
    \put(140,48){$e_2$}
    \put(75,50){$x$}
    \put(173,57){$\CC$}
    \put(163,77){$X^0_\w$}
    \put(163,37){$X^0_\b$}
    \end{overpic}
    \caption{Subcase (2)(d) where $f^n(e_1)=f^n(e_2)\supseteq e_2$. $k=2$, $l=1$.}
    \label{figPlotd}
\end{figure}

\begin{enumerate}

\smallskip
\item[(1)] If $x \notin \crit f^n$, then $(f^n)|_\CC$ either preserves or reverses the orientation at $x$ and the point $x$ is contained in exactly one white $n$-tile $X^n_\w$ and one black $n$-tile $X^n_\b$.
\begin{enumerate}

\smallskip
\item[(a)] If $(f^n)|_\CC$ preserves the orientation at $x$, then $X^n_\w \subseteq X^0_\w$ and $X^n_\b \subseteq X^0_\b$.

In this subcase, $M_{\ti} = 2$, $M_{\ee} = 4$, $M_{\e} = 2$, $M_{\po} = 1$, and $\deg_{f^n}(x) = 1$.

\smallskip
\item[(b)] If $(f^n)|_\CC$ reverses the orientation at $x$, then $X^n_\w \subseteq X^0_\b$ and $X^n_\b \subseteq X^0_\w$.

In this subcase, $M_{\ti} = 0$, $M_{\ee} = 0$, $M_{\e} = 0$, $M_{\po} = 1$, and $\deg_{f^n}(x) = 1$.
\end{enumerate}

\smallskip
\item[(2)] If $x \in \crit f^n$, then $x=f^n(x)\in\post f$ and so there are two distinct $n$-edges $e_1,e_2\subseteq\CC$ such that $\{x\}=e_1\cap e_2$. We refer to Figures \ref{figPlota} to \ref{figPlotd}.
\begin{enumerate}

\smallskip
\item[(a)] If $e_1\subseteq f^n(e_1)$ and $e_2\subseteq f^n(e_2)$, then $x$ is contained in exactly $k$ white and $k-1$ black $n$-tiles that are contained in the white $0$-tile, as well as in exactly $l-1$ white and $l$ black $n$-tiles that are contained in the black $0$-tile, for some $k, \, l\in\N$ with $k+l-1=\deg_{f^n}(x)$. Note that in this case $(f^n)|_\CC$ preserves the orientation at $x$.

In this subcase, $M_{\ti} = k+l$, $M_{\ee} = 4$, $M_{\e} = 2$, $M_{\po} = 1$, and $\deg_{f^n}(x) = k+l-1$.

\smallskip
\item[(b)] If $e_2\subseteq f^n(e_1)$ and $e_1\subseteq f^n(e_2)$, then $x$ is contained in exactly $k-1$ white and $k$ black $n$-tiles that are contained in the white $0$-tile, as well as in exactly $l$ white and $l-1$ black $n$-tiles that are contained in the black $0$-tile, for some $k, \, l\in\N$ with $k+l-1=\deg_{f^n}(x)$. Note that in this case $(f^n)|_\CC$ reverses the orientation at $x$.

In this subcase, $M_{\ti} = k+l-2$, $M_{\ee} = 0$, $M_{\e} = 0$, $M_{\po} = 1$, and $\deg_{f^n}(x) = k+l-1$.

\smallskip
\item[(c)] If $e_1\subseteq f^n(e_1)= f^n(e_2)$, then $x$ is contained in exactly $k$ white and $k$ black $n$-tiles that are contained in the white $0$-tile, as well as in exactly $l$ white and $l$ black $n$-tiles that are contained in the black $0$-tile, for some $k, \, l\in\N$ with $k+l=\deg_{f^n}(x)$. Note that in this case, $(f^n)|_\CC$ neither preserves nor reverses the orientation at $x$.

In this subcase, $M_{\ti} = k+l$, $M_{\ee} = 2$, $M_{\e} = 1$, $M_{\po} = 1$, and $\deg_{f^n}(x) = k+l$.

\smallskip
\item[(d)] If $e_2\subseteq f^n(e_1)= f^n(e_2)$, then $x$ is contained in exactly $k$ white and $k$ black $n$-tiles that are contained in the white $0$-tile, as well as in exactly $l$ white and $l$ black $n$-tiles that are contained in the black $0$-tile, for some $k, \, l\in\N$ with $k+l=\deg_{f^n}(x)$. Note that in this case, $(f^n)|_\CC$ neither preserves nor reverses the orientation at $x$.

In this subcase, $M_{\ti} = k+l$, $M_{\ee} = 2$, $M_{\e} = 1$, $M_{\po} = 1$, and $\deg_{f^n}(x) = k+l$.
\end{enumerate}
\end{enumerate}

This finishes the verification of (\ref{eqNoPeriodPtsIdentity}) in Case~(iii) of Lemma~\ref{lmPeriodicPtsLocationCases}.

\smallskip 

The proof of the theorem is now complete.
\end{proof}

Since all periodic points of $\bigl( \Sigma_{A_{\e}}^+, \sigma_{A_{\e}} \bigr)$ and $\bigl( \Sigma_{A_{\ee}}^+, \sigma_{A_{\ee}} \bigr)$ are mapped to periodic points of $f$ by the corresponding factor maps, we can write the dynamical Dirichlet series $\DS_{f,\,\minus\phi,\,\deg_f} (s)$ formally as a combination of products and quotient of the dynamical zeta functions for $\bigl( \Sigma_{A_{\ti}}^+, \sigma_{A_{\ti}} \bigr)$, $\bigl( \Sigma_{A_{\ee}}^+, \sigma_{A_{\ee}} \bigr)$, $\bigl( \Sigma_{A_{\e}}^+, \sigma_{A_{\e}} \bigr)$, and $\left(\V^0, f|_{\V^0}\right)$. In order to deduce Theorem~\ref{thmZetaAnalExt_InvC} from Theorem~\ref{thmZetaAnalExt_SFT}, we will need to verify that the zeta functions for the last three systems converge on an open half-plane on $\C$ containing $\{ s\in\C: \Re(s) \geq s_0 \}$.

\subsection{Calculation of topological pressure}  \label{subsctDynOnC_TopPressure}

Let $f\:S^2\rightarrow S^2$ be an expanding Thurston map with a Jordan curve $\CC\subseteq S^2$ satisfying $f(\CC)\subseteq \CC$ and $\post f\subseteq \CC$. We define for $m\in \N_0$ and $p \in\CC \cap \V^m(f,\CC)$,
\begin{equation}   \label{eqDefEdgePair}
\ae^m(p) \coloneqq \inte (e_1) \cup \{p\} \cup \inte (e_2) \quad \qquad  \text{and} \quad  \qquad
\overline\ae^m(p) \coloneqq e_1  \cup e_2,
\end{equation} 
where $e_1, \, e_2\in\E^m(f,\CC)$ are the unique pair of $m$-edges with $e_1\cup e_2 \subseteq \CC$ and $e_1\cap e_2 = \{p\}$. We denote for $m\in \N_0$, $n\in\N$, $q\in\CC$, and $q_j\in\CC \cap \V^m(f,\CC)$ for $j\in\{1, \, 2, \, \dots, \, n\}$,
\begin{align}  \label{eqDefEm}
 E_m(q_n,\,q_{n-1},\,\dots,\,q_1;\,q)  
 & = \bigl\{ x\in (f|_\CC)^{-n} (q)  :  (f|_\CC)^i(x) \in \overline\ae^m(q_{n-i}), \, i\in\{0, \, 1, \, \dots, \, n-1\}  \bigr\} \notag\\
 & = (f|_\CC)^{-n} (q)  \cap \biggl( \bigcap_{i=0}^{n-1}  (f|_\CC)^{-i} (\overline\ae^m (q_{n-i}) ) \biggr)   
   \subseteq \CC \cap \V^{m+n}(f,\CC) .
\end{align}

Intuitively, the set $E_m(q_n,\,q_{n-1},\,\dots,\,q_1;\,q)$ captures the preimages of $q$ under $(f|_\CC)^{n}$ that ``travel'' along the prescribed ``itinerary'' $q_n,\,q_{n-1},\,\dots,\,q_1$ along forward iterations of $f|_\CC$.

\begin{lemma}   \label{lmEmInductive}
Let $f$ and $\CC$ satisfy the Assumptions in Section~\ref{sctAssumptions}. We assume, in addition, that $f(\CC)\subseteq \CC$. Then
\begin{equation*}
\bigcup_{x\in E_m(p_n,\,p_{n-1},\,\dots,\,p_1;\,p_0)} E_m(p_{n+1};\,x) =  E_m(p_{n+1},\,p_n,\,\dots,\,p_1;\,p_0).
\end{equation*}
for $m\in\N_0$, $n\in\N$, and $p_i \in\CC\cap \V^m(f,\CC)$ for $i\in\{1, \, 2, \, \dots, \, n+1\}$. Here $E_m$ is defined in (\ref{eqDefEm}).
\end{lemma}

\begin{proof}
By (\ref{eqDefEm}), we get
\begin{align*}
&            \bigcup_{x\in E_m(p_n,\,p_{n-1},\,\dots,\,p_1;\,p_0)} E_m(p_{n+1};\,x)     \\
&\qquad =    \biggl\{  y\in (f|_\CC)^{-1} (x) :  y\in \overline{\ae}^m(p_{n+1}), \, 
                           x\in (f|_\CC)^{-n}(p_0) \cap \biggl(  \bigcap_{i=0}^{n-1} (f|_\CC)^{-i} (\overline{\ae}^m(p_{n-i}) ) \biggr)\biggr\}\\
&\qquad =    \biggl\{  y\in (f|_\CC)^{-n-1} (p_0) :  y\in \overline{\ae}^m(p_{n+1}), \, 
                           f(y) \in  \bigcap_{i=0}^{n-1} (f|_\CC)^{-i} (\overline{\ae}^m(p_{n-i}) ) \biggr\}\\
&\qquad =    E_m(p_{n+1},\,p_n,\,\dots,\,p_1;\,p_0).                       
\end{align*}
The lemma is now established.
\end{proof}

\begin{lemma}   \label{lmEdgeInPair}
Let $f$ and $\CC$ satisfy the Assumptions in Section~\ref{sctAssumptions}. Fix $m, \, n\in\N_0$ with $m\leq n$. If $f(\CC)\subseteq\CC$, then the following statements hold:
\begin{enumerate}
\smallskip
\item[(i)] For each $n$-edge $e^n\in\E^n(f,\CC)$ and each $m$-edge $e^m\in\E^m(f,\CC)$, if $e^m \cap \inte(e^n) \neq \emptyset$, then $e^n\subseteq e^m$.

\smallskip
\item[(ii)] For each $n$-vertex $v\in\CC\cap \V^n(f,\CC)$ and each $m$-vertex $w\in\CC\cap\V^m(f,\CC)$ on the curve $\CC$, if $v\notin \overline{\ae}^m(w)$, then $\overline{\ae}^m(w) \cap \ae^n(v) = \emptyset$.

\smallskip
\item[(iii)] Assume that $m\geq 1$ and no $1$-tile in $\X^1(f,\CC)$ joins opposite sides of $\CC$. For each pair $v_0, \, v_1 \in \CC\cap\V^m(f,\CC)$ of $m$-vertices on $\CC$, we denote by $e^1_i, \, e^2_i \in \E^m(f,\CC)$ the unique pair of $m$-edges with $e^1_i \cup e^2_i = \overline{\ae}^m(v_i)$ and $e^1_i \cap e^2_i = \{v_i\}$, for each $i\in\{0, \, 1\}$. Then $f$ is injective on $e^j_1$ for each $j\in\{1, \, 2\}$, and exactly one of the following cases is satisfied:

\begin{enumerate}
\smallskip
\item[(1)] $f( \ae^m(v_1) ) \cap \ae^m (v_0) = \emptyset$. In this case, 
$\card\{ x\in \overline{\ae}^m(v_1) : f(x) \in \overline{\ae}^m(v_0)\} \leq 2$.

\smallskip
\item[(2)] There exist $j,\, k\in\{1, \, 2\}$ such that 
\begin{itemize}
\smallskip
\item $f\bigl(e^j_1 \bigr) \supseteq e^k_0$, 

\smallskip
\item $f\bigl(e^j_1\bigr) \cap e^{k'}_0 \setminus \{v_0\} = \emptyset$ for $k'\in \{1, \, 2\}\setminus \{k\}$, 

\smallskip
\item $f\bigl(e^{j'}_1\bigr) \cap \overline{\ae}^m(v_0) = \emptyset$ for $j'\in \{1, \, 2\}\setminus\{j\}$, and

\smallskip
\item $v_1\notin\crit f|_\CC$.
\end{itemize}

\smallskip
\item[(3)] There exists $j\in\{1, \, 2\}$ such that
\begin{itemize}
\smallskip
\item $f\bigl(e^j_1 \bigr) \supseteq \overline{\ae}^m(v_0) = e^1_0 \cup e^2_0$, 

\smallskip
\item $f\bigl(e^{j'}_1 \setminus \{v_1\} \bigr) \cap \overline{\ae}^m(v_0) = \emptyset$ for $j'\in \{1, \, 2\}\setminus \{j\}$,

\smallskip
\item $v_1\notin\crit f|_\CC$, and $f(v_1)\neq v_0$.
\end{itemize}

\smallskip
\item[(4)] There exists $k\in\{1, \, 2\}$ such that
\begin{itemize}
\smallskip
\item $f\bigl(e^1_1 \bigr) \supseteq e^k_0$, $f\bigl(e^2_1 \bigr) \supseteq e^{k'}_0$, 

\smallskip
\item $f\bigl(e^1_1 \setminus \{v_1\} \bigr) \cap e^{k'}_0 = \emptyset$, $f\bigl(e^2_1 \setminus \{v_1\} \bigr) \cap e^k_0 = \emptyset$,

\smallskip
\item $v_1\notin\crit f|_\CC$, and $f(v_1)=v_0$,
\end{itemize}
\smallskip
where $k'\in \{1, \, 2\} \setminus \{k\}$.

\smallskip
\item[(5)] There exists $k\in\{1, \, 2\}$ such that 
\begin{itemize}
\smallskip
\item $f\bigl(e^1_1 \bigr) \cap f\bigl(e^2_1 \bigr) \supseteq e^k_0$, 

\smallskip
\item $f\bigl(e^1_1\bigr) \cap e^{k'}_0  \setminus \{v_0\} = \emptyset$ for $k'\in \{1, \, 2\}\setminus \{k\}$,

\smallskip
\item $f(e^1_1)=f(e^2_1)$, and $v_1\in\crit f|_\CC$.
\end{itemize}

\smallskip
\item[(6)] For each $j\in\{1, \, 2\}$, $f\bigl(e^j_1\bigr) \supseteq \overline{\ae}^m(v_0)$. In this case, we have $f(e^1_1)=f(e^2_1)$ and $v_1\in\crit f|_\CC$.
\end{enumerate}
 
\end{enumerate}
\end{lemma}

We say that a point $x\in\CC$ is a \defn{critical point of $f|_\CC$}, denoted by $x\in\crit f|_\CC$, if there is no open neighborhood $U\subseteq \CC$ of $x$ on which $f$ is injective. Clearly, $\crit f|_\CC \subseteq \crit f$. Our proof below relies crucially on the fact that $\CC$ is a Jordan curve.

To help understand the proof, the reader wants to keep in mind the geometric picture that adjacent $m$-edges are either mapped onto distinct adjacent $(m-1)$-edges or ``folded together'' to an $(m-1)$-edge. On the other hand, the role of the condition that no $1$-tile joins opposite sides is essentially to guarantee that three edges among which at least one is not a $0$-edge cannot form the entire Jordan curve $\CC$ as a union.

\begin{proof}
(i) Fix arbitrary $e^n\in\E^n$ and $e^m\in\E^m$. Since $f(\CC)\subseteq \CC$, $e^m = \bigcup \{e\in\E^n : e\subseteq e^m\}$. If $e^m \cap \inte(e^n) \neq \emptyset$, then there exists $e\in\E^n$ with $e\subseteq e^m$ and $\inte(e) \cap \inte(e^n) \neq \emptyset$. Then $e=e^n$ (see Definition~\ref{defcelldecomp}). Hence, $e^n\subseteq e^m$.

\smallskip

(ii) Fix $v\in\CC\cap\V^n$ and $w\in\CC\cap\V^m$ with $v\notin \overline{\ae}^m(w)$. Suppose $\overline{\ae}^m(w) \cap \ae^n(v) \neq \emptyset$. Then there exist $e^n\in\E^n$ and $e^m\in\E^m$ such that $e^n\subseteq \overline{\ae}^n(v)$, $e^m\subseteq \overline{\ae}^m(w)$, and $\inte(e^n)\cap e^m \neq\emptyset$. Then by Lemma~\ref{lmEdgeInPair}~(i), $e^n\subseteq e^m$. Hence, $v\in e^n \subseteq \overline{\ae}^m(w)$, a contradiction.

\smallskip

(iii) Fix arbitrary $v_0, \, v_1\in\CC\cap\V^m$. Recall $m\geq 1$.

By Proposition~\ref{propCellDecomp}~(i), the map $f$ is injective on $e^j_1$ for each $j\in\{1, \, 2\}$.

Recall that $\card (\post f) \geq 3$ (see \cite[Lemma~6.1]{BM17}).

We denote $I \coloneqq \bigl\{ (j,k) \in \{1, \, 2\}\times\{1, \, 2\} : f\bigl( e^j_1 \bigr) \supseteq e^k_0 \bigr\}$. Then $\card I \in \{0, \, 1, \, 2, \, 3, \, 4\}$. 

We will establish Lemma~\ref{lmEdgeInPair}~(iii) by proving the following statements:

\begin{enumerate}
\smallskip
\item[(a)] $\card I \neq 3$.

\smallskip
\item[(b)] Case~(1), (2), or (6) holds if and only if $\card I = 0$, $1$, or $4$, respectively.

\smallskip
\item[(c)] Case~(3) holds if and only if $\card I = 2$, $v_1 \notin \crit f|_\CC$, and $f(v_1) \neq v_0$.

\smallskip
\item[(d)] Case~(4) holds if and only if $\card I = 2$, $v_1 \notin \crit f|_\CC$, and $f(v_1) = v_0$.

\smallskip
\item[(e)] Case~(5) holds if and only if $\card I = 2$, $v_1 \in \crit f|_\CC$.
\end{enumerate}

\smallskip

(a) Suppose $\card I = 3$. Without loss of generality, we assume that $f\bigl( e^1_1 \bigr) \supseteq e^1_0 \cup e^2_0$, $f\bigl( e^2_1 \bigr) \supseteq e^1_0$, and $f\bigl( e^2_1 \bigr) \nsupseteq e^2_0$. Since $f\bigl( e^1_1 \bigr),  \, f\bigl( e^2_1 \bigr) \in \E^{m-1}$ (see Proposition~\ref{propCellDecomp}~(i)), $f\bigl( e^1_1 \bigr) \cap f\bigl( e^2_1 \bigr) \supseteq e^1_0$, we get $f\bigl( e^1_1 \bigr) \cap \inte \bigl( f\bigl( e^2_1 \bigr) \bigr) \neq \emptyset$, thus by Lemma~\ref{lmEdgeInPair}~(i), $f\bigl( e^1_1 \bigr) = f\bigl( e^2_1 \bigr)$. Hence, $\card I = 4$. This is a contradiction.

\smallskip

(b) If Case~(1) holds, then clearly $\card I = 0$. Conversely, we assume that $\card I = 0$. Fix arbitrary $j, \, k\in\{1, \, 2\}$. Since $f\bigl( e^j_1 \bigr) \in \E^{m-1}$ and $f$ is injective on $e^j_1$ (see Proposition~\ref{propCellDecomp}~(i)), by Lemma~\ref{lmEdgeInPair}~(i), $f\bigl( \inte \bigl(  e^j_1 \bigr) \bigr) \cap   e^k_0  =  f\bigl( e^j_1 \bigr) \cap \inte \bigl(  e^k_0 \bigr) = \emptyset$. Observe that it follows from $\card I = 0$ and Lemma~\ref{lmEdgeInPair}~(i) that $f(v_1)\neq v_0$. Thus, $f( \ae^m(v_1) ) \cap \ae^m (v_0) = \emptyset$. In order to show $\card\{ x\in \overline{\ae}^m(v_1) : f(x) \in \overline{\ae}^m(v_0)\} \leq 2$, it suffices to prove $\card \bigl\{ x\in e^j_1 \setminus \inte\bigl(e^j_1\bigr) : f(x)\in \overline{\ae}^m(v_0) \bigr\} \leq 1$. Suppose not, then since $f\bigl( \inte \bigl(  e^j_1 \bigr) \bigr) \cap   \overline{\ae}^m(v_0) = \emptyset$, $f$ is injective on $e^j_1$, and $\CC$ is a Jordan curve, we get $f\bigl(  e^j_1 \bigr) \cup \overline{\ae}^m(v_0) = \CC$. This contradicts the fact that $\card (\post f) \geq 3$ and the condition that no $1$-tile in $\X^1$ joins opposite sides of $\CC$. Therefore, Case~(1) holds.

\smallskip

If Case~(2) holds, then clearly $\card I = 1$. Conversely, we assume that $\card I = 1$. Without loss of generality, we assume that $f\bigl( e^1_1 \bigr) \supseteq e^1_0$. We observe that $v_1\notin \crit f|_\CC$. For otherwise, $f\bigl( e^1_1 \bigr) = f\bigl( e^2_1 \bigr) \in\E^{m-1}$ (see Proposition~\ref{propCellDecomp}~(i)), thus $\card I \neq 1$, which is a contradiction. 

To show $f\bigl( e^1_1 \bigr) \cap e^2_0 \setminus \{v_0\} = \emptyset$, we argue by contradiction and assume that $f\bigl( e^1_1 \bigr) \cap e^2_0 \setminus \{v_0\} \neq \emptyset$. Since $e^2_0 \nsubseteq f\bigl( e^1_1 \bigr) \in \E^{m-1}$ (see Proposition~\ref{propCellDecomp}~(i)), by Lemma~\ref{lmEdgeInPair}~(i), $f\bigl( e^1_1 \bigr) \cap \inte \bigl(  e^2_0 \bigr) = \emptyset$. Note that $v_0\in e_0^1 \subseteq f \bigl( e_1^1 \bigr)$. Since $f\bigl( e^1_1 \bigr)$ is connected and $\CC$ is a Jordan curve, we get $f\bigl( e^1_1 \bigr) \cup e^2_0 = \CC$. This contradicts the fact that $\card (\post f) \geq 3$.

Next, we verify that $f(v_1) \notin e^1_0$ as follows. We argue by contradiction and assume that $f(v_1)\in e^1_0$. Since $f(v_1) \in \V^{m-1}$ (see Proposition~\ref{propCellDecomp}~(i)), we get $f(v_1) \in e^1_0 \setminus \inte\bigl(e^1_0\bigr)$. Since $\card I = 1$, it is clear that $f(v_1) \neq v_0$. Thus, $f(v_1) \in e^1_0 \setminus \bigl( \inte\bigl(e^1_0 \bigr) \cup \{v_0\} \bigr)$. Since $v_1\notin \crit f|_\CC$ and no $1$-tile in $\X^1$ joins opposite sides of $\CC$, we get from Proposition~\ref{propCellDecomp}~(i) that either $f\bigl(e^1_1 \bigr) \supseteq e^1_0 \cup e^2_0$ or $f\bigl(e^2_1 \bigr) \supseteq e^1_0 \cup e^2_0$. This contradicts the assumption that $\card I = 1$. Hence, $f(v_1) \notin e^1_0$.

Finally we show that $f\bigl(e^2_1\bigr) \cap \overline{\ae}^m(v_0) = \emptyset$. To see this, we argue by contradiction and assume that $f\bigl(e^2_1\bigr) \cap \overline{\ae}^m(v_0) \neq \emptyset$. Since $\CC$ is a Jordan curve, $v_1\notin \crit f|_\CC$, $f(v_1) \notin e^1_0$, $f\bigl( e^1_1 \bigr) \supseteq e^1_0$, and $f \bigl( e_1^1 \bigr) \nsupseteq e_0^2$, we get that $f\bigl( e^1_1 \bigr) \cup f\bigl( e^2_1 \bigr) \cup e^2_0 = \CC$. This contradicts the fact that $\card (\post f) \geq 3$ and the condition that no $1$-tile in $\X^1$ joins opposite sides of $\CC$.

Therefore, Case~(2) holds.

\smallskip

If Case~(6) holds, then clearly $\card I = 4$. Conversely, we assume that $\card = 4$. Then $f\bigl( e^j_1\bigr) \supseteq e^1_0 \cup e^2_0 = \overline{\ae}^m(v_0)$ for each $j\in\{1, \, 2\}$. We show that $v_1\in\crit f|_\CC$ as follows. We argue by contradiction and assume that $v_1\notin\crit f|_\CC$. Then Since $\CC$ is a Jordan curve and $f\bigl( e^1_1\bigr) \cap f\bigl( e^2_1\bigr) \supseteq e^1_0$, we get that $f\bigl( e^1_1\bigr) \cup f\bigl( e^2_1\bigr) = \CC$. This contradicts the fact that $\card (\post f) \geq 3$. Hence, $v_1\in\crit f|_\CC$, and consequently $f\bigl( e^1_1\bigr) = f\bigl( e^2_1\bigr) \in \E^{m-1}$ by Proposition~\ref{propCellDecomp}~(i). Therefore, Case~(6) holds.

\smallskip

(c) If Case~(3) holds, then clearly $\card I = 2$, $v_1 \notin \crit f|_\CC$, and $f(v_1) \neq v_0$. Conversely, we assume that $\card I = 2$, $v_1 \notin \crit f|_\CC$, and $f(v_1) \neq v_0$. Suppose that $f \bigl( e_1^1 \bigr) \supseteq e_0^k$ and $f \bigl( e_1^2 \bigr) \supseteq e_0^{k'}$ for some $k, \, k'\in\{1, \, 2\}$ with $k\neq k'$, then since $v_1 \notin  \crit f|_\CC$ and $f(v_1) \neq v_0$, we get from Proposition~\ref{propCellDecomp}~(i) that $f \bigl( e_1^1 \bigr) \cup f \bigl( e_1^2 \bigr) = \CC$. This contradicts the fact that $\card (\post f) \geq 3$. Thus, without loss of generality, we can assume that $f\bigl(e^1_1\bigr) \supseteq e^1_0 \cup e^2_0$. In order to show that Case~(3) holds, it suffices now to verify that $f\bigl( e^2_1 \setminus \{v_1\} \bigr) \cap \overline{\ae}^m(v_0) = \emptyset$. We argue by contradiction and assume that $f\bigl( e^2_1 \setminus \{v_1\} \bigr) \cap \overline{\ae}^m(v_0) \neq \emptyset$. Then $f\bigl( e^2_1 \setminus \{v_1\} \bigr) \cap f\bigl( e^1_1  \bigr) \neq \emptyset$. Since $\CC$ is a Jordan curve and $v_1\notin\crit f|_\CC$, we get from Proposition~\ref{propCellDecomp}~(i) that $f\bigl( e^2_1   \bigr) \cup f\bigl( e^1_1  \bigr) = \CC$. This contradicts the fact that $\card (\post f) \geq 3$. Therefore, Case~(3) holds.

\smallskip

(d) If Case~(4) holds, then clearly $\card I = 2$, $v_1 \notin \crit f|_\CC$, and $f(v_1) = v_0$. Conversely, we assume that $\card I = 2$, $v_1 \notin \crit f|_\CC$, and $f(v_1) = v_0$. Without loss of generality, we assume that $f\bigl(e^1_1\bigr) \supseteq e^1_0$ and $f\bigl(e^2_1\bigr) \supseteq e^2_0$. In order to show that Case~(4) holds, by symmetry, it suffices to show that $f\bigl( e^1_1 \setminus \{v_1\} \bigr) \cap e^2_0 = \emptyset$. We argue by contradiction and assume that $f\bigl( e^1_1 \setminus \{v_1\} \bigr) \cap e^2_0 \neq \emptyset$. Then $f\bigl( e^1_1 \setminus \{v_1\} \bigr) \cap f\bigl( e^2_1  \bigr) \neq \emptyset$. Since $\CC$ is a Jordan curve and $v_1\notin\crit f|_\CC$, we get from Proposition~\ref{propCellDecomp}~(i) that $f\bigl( e^2_1   \bigr) \cup f\bigl( e^1_1  \bigr) = \CC$. This contradicts the fact that $\card (\post f) \geq 3$. Therefore, Case~(4) holds.

\smallskip

(e) If Case~(5) holds, then clearly $\card I = 2$ and $v_1 \in \crit f|_\CC$. Conversely, we assume that $\card I = 2$ and $v_1 \in \crit f|_\CC$. Since $v_1 \in \crit f|_\CC$, $f\bigl(e^1_1\bigr) = f\bigl(e^2_1\bigr) \in \E^{m-1}$ by Proposition~\ref{propCellDecomp}~(i). Without loss of generality, we assume that $f\bigl(e^1_1\bigr) \cap f\bigl(e^2_1\bigr) \supseteq e^1_0$. In order to show that Case~(5) holds, it suffices now to show that $f\bigl( e^1_1  \bigr) \cap e^2_0 \setminus \{v_0\}= \emptyset$. We argue by contradiction and assume that $f\bigl( e^1_1 \bigr) \cap e^2_0 \setminus \{v_0\} \neq \emptyset$.  Since $e^2_0 \nsubseteq f\bigl( e^1_1 \bigr) \in \E^{m-1}$ (see Proposition~\ref{propCellDecomp}~(i)), by Lemma~\ref{lmEdgeInPair}~(i), $f\bigl( e^1_1 \bigr) \cap \inte \bigl(  e^2_0 \bigr) = \emptyset$. Thus, $e^2_0\setminus \inte\bigl(e^2_0\bigr) \subseteq f\bigl( e^1_1 \bigr)$ as we already know $v_0\in e^1_0 \subseteq f\bigl( e^1_1 \bigr)$. Since $f\bigl( e^1_1 \bigr)$ is connected and $\CC$ is a Jordan curve, we get $f\bigl( e^1_1 \bigr) \cup e^2_0 = \CC$. This contradicts the fact that $\card (\post f) \geq 3$. Therefore, Case~(5) holds.
\end{proof}

\begin{lemma}   \label{lmCoverByEdgePair}
Let $f$ and $\CC$ satisfy the Assumptions in Section~\ref{sctAssumptions}. We assume, in addition, that $f(\CC)\subseteq \CC$. Then
\begin{equation*}
\bigcap_{i=0}^n (f|_\CC)^{-i} (\ae^m(p_{n-i}))  \subseteq \bigcup_{x\in E_m(p_n,\,p_{n-1},\,\dots,\,p_1;\,p_0)}  \ae^{m+n} (x),
\end{equation*}
for all $m\in\N_0$, $n\in\N$, and $p_i \in \CC \cap \V^m(f,\CC)$ for each $i\in \{0, \, 1, \, \dots, \, n\}$. Here $\ae^m$ is defined in (\ref{eqDefEdgePair}) and $E_m$ in (\ref{eqDefEm}).
\end{lemma}

\begin{proof}
We fix $m\in\N_0$ and an arbitrary sequence $\{p_i\}_{i\in\N_0}$ in $\CC\cap\V^m$. We prove the lemma by induction on $n\in\N$.

For $n=1$, we get
\begin{align*}
               \ae^m(p_1) \cap (f|_\CC)^{-1} (\ae^m(p_0)) 
&\subseteq      \bigcup \bigl\{ \ae^{m+1}(x) : x\in (f|_\CC)^{-1} (p_0),\, x\in \overline{\ae}^m(p_1)  \bigr\} 
   =           \bigcup_{x\in E_m(p_1;\,p_0)} \ae^{m+1}(x)
\end{align*}
by (\ref{eqDefEm}), Proposition~\ref{propCellDecomp}~(ii), and the fact that $\ae^{m+1}(x) \cap \ae^m(p_1) = \emptyset$ if both $x\in\CC \cap \V^{m+1}$ and $x\notin \overline{\ae}^m(p_1)$ are satisfied (see Lemma~\ref{lmEdgeInPair}~(ii)).

We now assume that the lemma holds for $n=l$ for some integer $l\in\N$. Then by the induction hypothesis, we have
\begin{align*}
              \bigcap_{i=0}^{l+1}  (f|_\CC)^{-i} (\ae^m(p_{l+1-i}))  
& =          \ae^m(p_{l+1}) \cap (f|_\CC)^{-1} \biggl( \bigcap_{i=1}^{l+1}  (f|_\CC)^{-(i-1)} (\ae^m(p_{l+1-i})) \biggr) \\
& \subseteq   \ae^m(p_{l+1}) \cap (f|_\CC)^{-1} \biggl( \bigcup_{x\in E_m(p_l,\,p_{l-1},\,\dots,\,p_1;\,p_0)}  \ae^{m+l}(x) \biggr) \\
&  =          \bigcup_{x\in E_m(p_l,\,p_{l-1},\,\dots,\,p_1;\,p_0)}  \bigl(  \ae^m(p_{l+1}) \cap (f|_\CC)^{-1} \bigl(\ae^{m+l}(x)\bigr)\bigr) \\
& \subseteq   \bigcup_{x\in E_m(p_l,\,p_{l-1},\,\dots,\,p_1;\,p_0)}  
                           \Bigl(  \bigcup \bigl\{ \ae^{m+l+1}(y) : y\in(f|_\CC)^{-1} (x),\, y\in \overline{\ae}^m(p_{l+1}) \bigr\} \Bigr)  \\
&  =          \bigcup_{x\in E_m(p_l,\,p_{l-1},\,\dots,\,p_1;\,p_0)}  \bigcup_{y\in E_m(p_{l+1};\,x)} \ae^{m+l+1}(y),
\end{align*}
where the last two lines are due to (\ref{eqDefEm}), Proposition~\ref{propCellDecomp}~(ii), and the fact that $\ae^{m+l+1}(y) \cap \ae^m(p_{l+1}) = \emptyset$ if both $y\in\CC \cap \V^{m+l+1}$ and $y\notin \overline{\ae}^m(p_{l+1})$ are satisfied (see Lemma~\ref{lmEdgeInPair}~(ii)).

By Lemma~\ref{lmEmInductive}, we get
$
\bigcap_{i=0}^{l+1} (f|_\CC)^{-i} (\ae^m(p_{l+1-i}))  \subseteq \bigcup_{x\in E_m(p_{l+1},\,p_l,\,\dots,\,p_1;\,p_0)}  \ae^{m+l+1} (x).
$

The induction step is now complete.
\end{proof}

\begin{prop}    \label{propEmBound}
Let $f$ and $\CC$ satisfy the Assumptions in Section~\ref{sctAssumptions}. We assume, in addition, that $f(\CC)\subseteq \CC$ and no $1$-tile in $\X^1(f,\CC)$ joins opposite sides of $\CC$. Then
\begin{equation}   \label{eqEmBound}
\card (  E_m(p_n,\,p_{n-1},\,\dots,\,p_1;\,p_0)  )  \leq m 2^{\frac{n}{m}}
\end{equation}
for all $m,  \, n\in\N$ with $m\geq 14$, and $p_i \in \CC \cap \V^m(f,\CC)$ for each $i\in \{0, \, 1, \, \dots, \, n\}$. Here $E_m$ is defined in (\ref{eqDefEm}).
\end{prop}

 In order to obtain the upper bound in (\ref{eqEmBound}), we use in the proof below careful inductive arguments to discuss how the left-hand side of (\ref{eqEmBound}) grows as $n$ increases in a case-by-case fashion according to the combinatorial information obtained in Lemma~\ref{lmEdgeInPair}.
 
\begin{proof}
We fix $m\in\N$ with $m\geq 14$, and fix an arbitrary sequence $\{p_i\}_{i\in\N_0}$ of $m$-vertices in $\CC\cap\V^m$.

For each $n\in\N$, we write $E_{m,n} \coloneqq E_m (p_n,\,p_{n-1},\,\dots,\,p_1;\,p_0)$. Note that for each $n\in\N$, by (\ref{eqDefEm}),
\begin{equation}  \label{eqPfpropEmBound_EmnInAEpn}
E_{m,n} = E_m (p_n,\,p_{n-1},\,\dots,\,p_1;\,p_0) \subseteq \overline{\ae}^m(p_n),
\end{equation}
where $\overline{\ae}^m$ is defined in (\ref{eqDefEdgePair}). We denote by $e_{n,1},  \, e_{n,2}\in\E^m$ the unique pair of $m$-edges with $e_{n,1} \cup e_{n,2} = \overline{\ae}^m(p_n)$ and $e_{n,1} \cap e_{n,2} = \{p_n\}$. For each $i\in\{1, \, 2\}$, we define
\begin{equation} \label{eqPfpropEmBound_Lni}
L_{n,i} \coloneqq \begin{cases} \card (E_{m,n} \cap e_{n,i}) & \text{if } E_{m,n} \cap e_{n,i} \setminus \{p_n\} \neq \emptyset, \\ 0  & \text{otherwise}.  \end{cases}
\end{equation}

\smallskip

We first observe that by (\ref{eqDefEm}), (\ref{eqPfpropEmBound_EmnInAEpn}), Lemma~\ref{lmEmInductive}, and the fact that $f$ is injective on each $m$-edge (see Proposition~\ref{propCellDecomp}~(i)), we get that
\begin{equation}   \label{eqPfpropEmBound_EmDouble}
\card E_{m,1} \leq 2  \qquad \text{ and } \qquad  \card E_{m,n+1} \leq 2 \card E_{m,n},
\end{equation}
for each $n\in\N$.

Next, we need to establish two claims.

\smallskip

\emph{Claim~1.} For each $n\in\N$, if $L_{n,1}L_{n,2}\neq 0$ then $L_{n,1}=L_{n,2}$.

\smallskip

We will establish Claim~1 by induction on $n\in\N$.

For $n=1$, we apply Lemma~\ref{lmEdgeInPair}~(iii) with $v_0=p_0$ and $v_1=p_1$. By (\ref{eqDefEm}), it is easy to verify that in Cases~(1) through (4) discussed in Lemma~\ref{lmEdgeInPair}~(iii), we have $L_{1,1}L_{1,2} = 0$, and in Cases~(5) and (6), we have $L_{1,1}=L_{1,2}$.

We now assume that Claim~1 holds for $n=l$ for some integer $l\in\N$. We apply Lemma~\ref{lmEdgeInPair}~(iii) with $v_0=p_l$ and $v_1=p_{l+1}$. Then by (\ref{eqPfpropEmBound_EmnInAEpn}) and Lemma~\ref{lmEmInductive} with $n=l$, it is easy to verify that in Case~(1) discussed in Lemma~\ref{lmEdgeInPair}~(iii), we have either $L_{l+1,1}L_{l+1,2} = 0$ or $L_{l+1,1}=L_{l+1,2}=1$; in Cases~(2) and (3), we have $L_{l+1,1}L_{l+1,2} = 0$; in Case~(4), we have $L_{l+1,1}=L_{l,k}$ and $L_{l+1,2}=L_{l,k'}$, where $k, \, k'\in\{1, \, 2\}$ satisfy $f(e_{l+1,1})\supseteq e_{l,k}$, $f(e_{l+1,2})\supseteq e_{l,k'}$, and $k\neq k'$; and in Cases~(5) and (6), we have $L_{l+1,1}=L_{l+1,2}$.

The induction step is now complete. Claim~1 follows.

\smallskip

\emph{Claim~2.} For each $n\in\N$ with $\card E_{m,n} \geq 4$, the following statements hold:
\begin{enumerate}
\smallskip
\item[(i)] If $\card E_{m,n+1} < \card E_{m,n}$, then 
$
\card E_{m,n+1} \leq \bigl\lceil \frac{1}{2} \card E_{m,n} \bigr\rceil.
$

\smallskip
\item[(ii)] If $\card E_{m,n+1} = \card E_{m,n}$ and $p_{n+1} \in \crit f|_\CC$, then
$
\card (e_{n+1,1} \cap E_{m,n+1}) = \card (e_{n+1,2} \cap E_{m,n+1}).
$

\smallskip
\item[(iii)] If $\card E_{m,n+1} > \card E_{m,n}$, then
\begin{enumerate}
\smallskip
\item[(a)] $\card (e_{n+1,1} \cap E_{m,n+1}) = \card (e_{n+1,2} \cap E_{m,n+1})$,

\smallskip
\item[(b)] $p_{n+1} \in \crit f|_\CC$, and

\smallskip
\item[(c)] $E_{m,n} \subseteq e^{m-1} \in \E^{m-1}$, where $e^{m-1} \coloneqq f(e_{n+1,1}) = f(e_{n+1,2})$.
\end{enumerate}
\end{enumerate}

\smallskip
To prove Claim~2, we first note that by (\ref{eqPfpropEmBound_EmDouble}), $\card E_{m,1} \leq 2$, so it suffices to consider $n\geq 2$. We fix an integer $n\geq 2$ with $\card E_{m,n} \geq 4$. If such $n$ does not exist, then Claim~2 holds trivially.

We will verify statements~(i) through (iii) according to the cases discussed in Lemma~\ref{lmEdgeInPair}~(iii) with $v_0 = p_n$, $v_1 = p_{n+1}$, $e^i_0=e_{n,i}$, and $e^i_1=e_{n+1,i}$ for each $i\in\{1, \, 2\}$.

\smallskip

\emph{Case~(1).} It is easy to see that $\card E_{m,n+1} \leq 2 \leq \bigl\lceil \frac{1}{2} \card E_{m,n} \bigr\rceil$.

\smallskip

\emph{Case~(2).} We have $p_{n+1}\notin \crit f|_\CC$. Without loss of generality, we assume that $f(e_{n+1,1}) \supseteq e_{n,1}$, $f(e_{n+1,1}) \cap e_{n,2} \setminus \{p_n\} = \emptyset$, and $f(e_{n+1,2}) \cap \overline{\ae}^m(p_n) = \emptyset$. Since $f$ is injective on each $m$-edge (see Proposition~\ref{propCellDecomp}~(i)), $E_{m,n} \subseteq \overline{\ae}^m(p_{n})$ (see (\ref{eqDefEm})), and either $L_{n,1}  L_{n,2} = 0$ or $L_{n,1} = L_{n,2}$ by Claim~1, it is easy to verify from Lemma~\ref{lmEmInductive} that either $\card E_{m,n+1} = \card E_{m,n}$ or $\card E_{m,n+1} \leq \bigl\lceil \frac{1}{2} \card E_{m,n} \bigr\rceil$.

\smallskip

\emph{Case~(3).} We have $p_{n+1}\notin \crit f|_\CC$. Without loss of generality, we assume that $f(e_{n+1,1}) \supseteq \overline{\ae}^m (p_n)$ and $f(e_{n+1,2} \setminus \{ p_{n+1} \} ) \cap \overline{\ae}^m(p_n) = \emptyset$. Since $f$ is injective on each $m$-edge (see Proposition~\ref{propCellDecomp}~(i)), $E_{m,n} \subseteq \overline{\ae}^m(p_{n})$ (see (\ref{eqDefEm})), and either $L_{n,1}  L_{n,2} = 0$ or $L_{n,1} = L_{n,2}$ by Claim~1, it is easy to verify from Lemma~\ref{lmEmInductive} that $\card E_{m,n+1} = \card E_{m,n}$.

\smallskip

\emph{Case~(4).} We have $p_{n+1} \notin \crit f|_\CC$. By Proposition~\ref{propCellDecomp}~(i), $f$ maps $\overline{\ae}^m(p_{n+1})$ bijectively onto $f( \overline{\ae}^m(p_{n+1}) )$. Since $f( \overline{\ae}^m(p_{n+1}) ) \supseteq \overline{\ae}^m(p_{n})$ and $E_{m,n} \subseteq \overline{\ae}^m(p_{n})$ (see (\ref{eqDefEm})), we get $\card E_{m,n+1} = \card E_{m,n}$ by Lemma~\ref{lmEmInductive}.

\smallskip

\emph{Case~(5).} We have $p_{n+1}\in \crit f|_\CC  \subseteq \crit f$. Without loss of generality, we assume that $f(e_{n+1,j}) \supseteq e_{n,1}$ and $f(e_{n+1,j}) \cap e_{n,2} \setminus \{p_n\} = \emptyset$ for each $j\in\{1, \, 2\}$. Since $E_{m,n} \subseteq \overline{\ae}^m(p_{n})$ (see (\ref{eqDefEm})) and either $L_{n,1}  L_{n,2} = 0$ or $L_{n,1} = L_{n,2}$ by Claim~1, it is easy to verify from Lemma~\ref{lmEmInductive} that either $\card E_{m,n+1} \leq 2 \leq \bigl\lceil \frac{1}{2} \card E_{m,n} \bigr\rceil$ or $\card E_{m,n+1} \geq \card E_{m,n}$. Note that in either case, we have $\card (e_{n+1,1} \cap E_{m,n+1}) = \card (e_{n+1,2} \cap E_{m,n+1})$. Moreover, if $\card E_{m,n+1} > \card E_{m,n}$, then it follows that $L_{n,2} = 0$, and thus $E_{m,n} \subseteq e_{n,1} \subseteq f(e_{n+1,1}) = f(e_{n+1,2}) \in \E^{m-1}$ (see Proposition~\ref{propCellDecomp}~(i)). 

\smallskip

\emph{Case~(6).} We have $p_{n+1}\in \crit f|_\CC  \subseteq \crit f$. Since $f$ is injective on each $m$-edge (see Proposition~\ref{propCellDecomp}~(i)) and $E_{m,n} \subseteq \overline{\ae}^m(p_{n})$ (see (\ref{eqDefEm})), it is easy to verify that $\card E_{m,n+1} > \card E_{m,n}$ and $E_{m,n} \subseteq \overline{\ae}^m(p_n)  \subseteq f(e_{n+1,1}) = f(e_{n+1,2}) \in \E^{m-1}$ (see Proposition~\ref{propCellDecomp}~(i)).

\smallskip

Claim~2 now follows.

\smallskip

Finally, we will establish (\ref{eqEmBound}) by induction on $n\in\N$. Recall that we assume $m\geq 14$.

For $n=1$, by (\ref{eqPfpropEmBound_EmDouble}), $\card E_{m,1} \leq 2 < m 2^{\frac{n}{m}}$. 

We now assume that (\ref{eqEmBound}) holds for all $n\leq l$ for some integer $l\in\N$. If $\card E_{m,l} < 8$, then (\ref{eqEmBound}) holds for $n=l+1$ by (\ref{eqPfpropEmBound_EmDouble}) and the assumption that $m\geq 14$. So we can assume that $\card E_{m,l} \geq 8$. Moreover, if $\card E_{m,l+1} \leq \card E_{m,l}$, then (\ref{eqEmBound}) holds for $n=l+1$ trivially from the induction hypothesis. Thus, we can also assume that $\card E_{m,l+1} > \card E_{m,l}$.

Since $\card E_{m,1} \leq 2$ (see (\ref{eqPfpropEmBound_EmDouble})) and $\card E_{m,l} \geq 8$, we can define a number
\begin{equation}  \label{eqPfpropEmBound_Defk}
k \coloneqq \max \{ i\in\N : i<l, \, \card E_{m,i} \neq \card E_{m,l} \} \leq l-1.
\end{equation}
Note that $l\geq 3$ and $\card E_{m,k} \geq \frac{1}{2} \card E_{m,k+1} = \frac{1}{2} \card E_{m,l} \geq 4$ by (\ref{eqPfpropEmBound_EmDouble}).

We will establish (\ref{eqEmBound}) for $n=l+1$ by considering the following two cases:

\smallskip

\emph{Case I.} $\card E_{m,k} > \card E_{m,l}=\card E_{m,k+1}$. Then by (\ref{eqPfpropEmBound_EmDouble}) and Claim~2(i), we have
\begin{align*}
 \card E_{m,l+1} & \leq 2 \card E_{m,l}
                    =   2 \card E_{m,k+1}
                   \leq 2 \Bigl\lceil \frac{1}{2} \card E_{m,k} \Bigr\rceil \\
                 & \leq 1 + \card E_{m,k} 
                   \leq 1 + m 2^{\frac{k}{m}}
                   \leq m 2^{\frac{k+2}{m}}
                   \leq m 2^{\frac{l+1}{m}},
\end{align*}
where the second-to-last inequality follows from the fact that the function $h(x) \coloneqq x\bigl(2^{\frac{2}{x}} - 1 \bigr)$, $x>1$, satisfies $\lim_{x\to+\infty} h(x) = \log 4 >1$, $\lim_{x\to+\infty} \frac{\mathrm{d}}{\mathrm{d} x} h(x) = 0$, and $\frac{\mathrm{d}^2}{\mathrm{d} x^2} h(x) > 0$ for $x>1$.

\smallskip

\emph{Case II.} $\card E_{m,k} < \card E_{m,l}=\card E_{m,k+1}$. Then by Claim~2(iii), we have $p_{k+1} \in \crit f|_\CC$. Put
\begin{equation}  \label{eqPfpropEmBound_Defk'}
k' \coloneqq \max \{ i\in \N : i\leq l, \, p_i\in\crit f|_\CC \}  \in [k+1, l].
\end{equation}
Note that by (\ref{eqPfpropEmBound_Defk}), (\ref{eqPfpropEmBound_Defk'}), and (\ref{eqPfpropEmBound_EmDouble}), we get $\card E_{m,k'-1} \geq \frac{1}{2} \card E_{m,l} \geq 4$. By Claim~2(ii) and (iii), regardless of whether $k'=k+1$ or not, we have
\begin{equation}  \label{eqPfpropEmBound_k'TwoSides}
\card (e_{k',1} \cap E_{m,k'} ) = \card (e_{k',2} \cap E_{m,k'} ) \geq 2.
\end{equation}
By Claim~2(iii), we have $p_{l+1} \in \crit f|_\CC \subseteq \V^1$, $E_{m,l} \subseteq e^{m-1} \in \E^{m-1}$ where $e^{m-1} \coloneqq f(e_{l+1,1}) = f(e_{l+1,2})$. Note that $f(p_{l+1}) \in \V^0 \cap e^{m-1}$. We now show that $k'< l$. We argue by contradiction and assume that $k'\geq l$. By (\ref{eqPfpropEmBound_Defk'}), $k'=l$. Then $p_l = p_{k'} \in \crit f|_\CC \subseteq \V^1$ and $p_l = p_{k'} \in \inte \bigl( e^{m-1} \bigr)$ (see (\ref{eqPfpropEmBound_k'TwoSides})). This contradicts the fact that no $(m-1)$-edge can contain a $1$-vertex in its interior. Thus, $k'<l$.

We now show that
\begin{equation}  \label{eqPfpropEmBound_jump}
l-k'\geq m-1.
\end{equation}
Fix an arbitrary integer $i\in[k'+1,l]$. Since $\card E_{m,i} = \card E_{m,i-1} \in [8,+\infty)$ (see (\ref{eqPfpropEmBound_Defk'}) and (\ref{eqPfpropEmBound_Defk})), $f(E_{m,i}) \subseteq E_{m,i-1}$ (see (\ref{eqDefEm})), and $f^{l-k'}$ is injective on $e^{m-1} \supseteq E_{m,l}$ (see Proposition~\ref{propCellDecomp}~(i)), we get $f(E_{m,i}) = E_{m,i-1}$. By Proposition~\ref{propCellDecomp}~(i), $f^{l-k'}\bigl( e^{m-1} \bigr) \in \E^{m-1-l+k'}$. Since $f^{l-k'}\bigl( e^{m-1} \bigr) \supseteq f^{l-k'}(E_{m,l}) = E_{m,k'}$ and $f^{l-k'}$ is injective and continuous on $e^{m-1}$ (see Proposition~\ref{propCellDecomp}~(i)), it follows from (\ref{eqPfpropEmBound_k'TwoSides}) that $p_{k'} \in \inte \bigl( f^{l-k'}\bigl( e^{m-1} \bigr) \bigr)$. Since $p_{k'} \in \crit f|_\CC \subseteq \V^1$, we get $m-1-l+k' \leq 0$, proving (\ref{eqPfpropEmBound_jump}).

Hence, by (\ref{eqPfpropEmBound_EmDouble}), (\ref{eqPfpropEmBound_Defk}), (\ref{eqPfpropEmBound_Defk'}), (\ref{eqPfpropEmBound_jump}), and the induction hypothesis, we get 
\begin{equation*}
\card E_{m,l+1} \leq 2 \card E_{m,l} = 2 \card E_{m,k+1} \leq 2 m 2^{\frac{k+1}{m}}  \leq 2 m 2^{\frac{k'}{m}} \leq m 2^{\frac{l+1}{m}}.
\end{equation*}
The induction step is now complete, establishing Proposition~\ref{propEmBound}.
\end{proof}

\begin{theorem}  \label{thmPressureOnC}
Let $f$, $\CC$, $d$, $\alpha$ satisfy the Assumptions in Section~\ref{sctAssumptions}. We assume, in addition, that $f(\CC)\subseteq \CC$. Let $\varphi \in \Holder{\alpha}(S^2,d)$ be a real-valued H\"{o}lder continuous function with an exponent $\alpha$. Recall the one-sided subshifts of finite type $\bigl( \Sigma_{A_{\e}}^+, \sigma_{A_{\e}} \bigr)$ and $\bigl( \Sigma_{A_{\ee}}^+, \sigma_{A_{\ee}} \bigr)$ constructed in Subsection~\ref{subsctDynOnC_Construction}. We denote by $\left(\V^0, f|_{\V^0}\right)$ the dynamical system on $\V^0=\V^0(f,\CC) = \post f$ induced by $f|_{\V^0} \: \V^0\rightarrow \V^0$. Then the following relations between the topological pressure of these systems hold:
\begin{enumerate}
	\smallskip
	\item[(i)] $P(f,\varphi) > P  (   f|_{\V^0} ,       \varphi|_{\V^0}                         )$,
	\smallskip
	\item[(ii)] $P(f,\varphi) > P  (   f|_\CC,           \varphi|_\CC                            )     
	=   P  (   \sigma_{A_{\e}},  \varphi \circ \pi_{\e}                  )
	=   P  (   \sigma_{A_{\ee}}, \varphi \circ \pi_{\e} \circ \pi_{\ee}  )$.
\end{enumerate}
\end{theorem}

\begin{proof}
The identity $P  (   \sigma_{A_{\ee}}, \varphi \circ \pi_{\e} \circ \pi_{\ee}  )  =  P  (   \sigma_{A_{\e}},  \varphi \circ \pi_{\e}  )$ follows directly from Lemma~\ref{lmUnifBddToOneFactorPressure} and H\"{o}lder continuity of $\pi_{\e}$ and $\pi_{\ee}$ (see Proposition~\ref{propSFTs_C}).

The strict inequalities $P (f|_\CC, \varphi|_\CC) < P(f,\varphi)$ and $P ( f|_{\V^0}), \varphi|_{\V^0} ) < P(f,\varphi)$ follow from the uniqueness of the equilibrium state $\mu_\phi$ for the map $f$ and the potential $\varphi$ (see Theorem~\ref{thmEquilibriumState}~(i)), the fact that $\mu_\phi (\CC) = 0$ (see Theorem~\ref{thmEquilibriumState}~(iii)), the Variational Principle (\ref{eqVPPressure}), and the observation that any equilibrium state for the map $f|_\CC$ (resp.\ $f|_{\V^0}$) and the potential $\varphi|_\CC$ (resp.\ $\varphi|_{\V^0}$) is also invariant under $f$.

We observe that since $(\CC,f|_\CC)$ is a factor of $\bigl(\Sigma_{A_{\e}}^+, \sigma_{A_{\e}}\bigr)$ with a factor map $\pi_{\e} \: \Sigma_{A_{\e}}^+ \rightarrow \CC$ (Proposition~\ref{propSFTs_C}~(ii)), it follows from \cite[Lemma~3.2.8]{PU10} that $P  ( \sigma_{A_{\e}},  \varphi \circ \pi_{\e} ) \geq P  (  f|_\CC,  \varphi|_\CC )$. It remains to show $P  ( \sigma_{A_{\e}},  \varphi \circ \pi_{\e} ) \leq P  (  f|_\CC,  \varphi|_\CC )$. 

By Lemma~\ref{lmCellBoundsBM}~(ii) and Proposition~\ref{propCellDecomp}~(vii), no $1$-tile in $\X^1(f^n,\CC)$ joins opposite sides of $\CC$ for all sufficiently large $n\in\N$. Note that for all $m\in\N$, $P\bigl( f^m|_\CC, S^f_m \varphi|_\CC \bigr) = m P ( f|_\CC,  \varphi|_\CC  )$ and $P  \bigl( \sigma_{A_{\e}}^m,  S^{\sigma_{A_{\e}}}_m ( \varphi \circ \pi_{\e}) \bigr)  = mP  ( \sigma_{A_{\e}},  \varphi \circ \pi_{\e} )$ (see for example, \cite[Theorem~9.8]{Wal82}). It is clear that, without loss of generality, we can assume that no $1$-tile in $\X^1(f,\CC)$ joins opposite sides of $\CC$.

We define a sequence of finite open covers $\{\eta_i\}_{i\in\N_0}$ of $\CC$ by 
\begin{equation*}
\eta_i \coloneqq \bigl\{  \ae^i(v)  : v\in \CC \cap \V^i  \bigr\}
\end{equation*}
for $i\in\N_0$. We note that since we are considering the metric space $(\CC,d)$, $\ae^i(v)$ is indeed an open set for each $i\in\N_0$ and each $v \in \CC \cap \V^i$. By Lemma~\ref{lmCellBoundsBM}~(ii), 
\begin{equation*}
\lim_{i\to+\infty} \max \{ \diam_d(V) : V\in\eta_i \} = 0.
\end{equation*}

Fix arbitrary integers $l, \, m, \, n\in\N$ with $l\geq m\geq 14$. Choose $U\in \bigvee_{i=0}^n (f|_\CC)^{-i} (\eta_m)$ arbitrarily, say
\begin{equation*}
U = \bigcap_{i=0}^n (f|_\CC)^{-i} (\ae^m(p_{n-i})),
\end{equation*}
where $p_0, \, p_1, \, \dots, \, p_n \in \CC\cap\V^m$. By Lemma~\ref{lmCoverByEdgePair},
\begin{equation*}
U \subseteq \bigcup_{x\in E_m(p_n,\,p_{n-1},\,\dots,\,p_1;\,p_0)}  \ae^{m+n} (x),
\end{equation*}
where $E_m$ is defined in (\ref{eqDefEm}). It follows immediately from (\ref{eqDeg=SumLocalDegree}) and Proposition~\ref{propCellDecomp}~(i) and (v) that
\begin{equation*}
\card \bigl\{ e^{l+n} \in \E^{l+n} : e^{l+n} \subseteq e \bigr\}  \leq (\deg f)^{l-m} \card(\post f)
\end{equation*}
for each $(m+n)$-edge $e\in\E^{m+n}$. Thus, we can construct a collection $\mathcal{E}^{l+n}(U) \subseteq \E^{l+n}$ of $(l+n)$-edges such that $U\subseteq \bigcup \mathcal{E}^{l+n}(U)$ by setting
\begin{equation*}
\mathcal{E}^{l+n}(U)  \coloneqq \bigcup_{x\in E_m(p_n,\,p_{n-1},\,\dots,\,p_1;\,p_0)}   \bigl\{ e^{l+n} \in \E^{l+n} 
                                                                                  : e^{l+n} \subseteq \overline{\ae}^{m+n}(x)  \bigr\}
\end{equation*}
Then 
\begin{align}   \label{eqPfthmPressureOnC_EdgeNoBound}
       \card \bigl(  \mathcal{E}^{l+n}(U) \bigr) 
& \leq 2 (\deg f)^{l-m} \card(\post f) \card (E_m(p_n,\,p_{n-1},\,\dots,\,p_1;\,p_0))  \notag \\
& \leq 2 m 2^{\frac{n}{m}} (\deg f)^{l-m} \card(\post f),
\end{align}
where the last inequality follows from Proposition~\ref{propEmBound}.

Hence, by (\ref{eqEquivDefByCoverForTopPressure}), Lemma~\ref{lmSnPhiBound}, and (\ref{eqPfthmPressureOnC_EdgeNoBound}), we get
\begin{align*}
            P(f|_\CC, \varphi|_\CC)  
&    = \lim_{m\to+\infty}   \lim_{l\to+\infty}   \lim_{n\to+\infty}  \frac{1}{n} \log
                 \inf_\xi \biggl\{ \sum_{U\in\xi} \exp  \bigl( \sup  \bigl\{S_n^f\varphi(x) : x\in U \bigr\}  \bigr)    \biggr\}     \\
&\geq  \liminf_{m\to+\infty}   \liminf_{l\to+\infty}   \liminf_{n\to+\infty}  \frac{1}{n} \log
                 \inf_\xi \biggl\{ \sum_{U\in\xi} \sum_{e\in\mathcal{E}^{l+n}(U)}
                   \frac{  \exp  ( \sup \{S_n^f\varphi(x) : x\in e\cap U \}  )}
                        {  \card  (  \mathcal{E}^{l+n}(U)   ) }  \biggr\}  \\        
&\geq  \liminf_{m\to+\infty}   \liminf_{l\to+\infty}   \liminf_{n\to+\infty}  \frac{1}{n} \log
                 \inf_\xi \biggl\{ \sum_{U\in\xi} \sum_{e\in\mathcal{E}^{l+n}(U)}
                   \frac{  \exp  ( \sup    \{S_n^f\varphi(x) : x\in e \}  )}
                        { 2 m 2^{\frac{n}{m}} (\deg f)^{l-m}   D } \biggr\}     \\ 
&\geq  \liminf_{m\to+\infty}   \liminf_{l\to+\infty}   \liminf_{n\to+\infty}  \frac{1}{n} \log
                   \sum\limits_{\substack{e\in\E^{l+n} \\ e\subseteq \CC  }}
                   \frac{  \exp  ( \sup    \{S_n^f\varphi(x) : x\in e \}  )}
                        { 2 m 2^{\frac{n}{m}} (\deg f)^{l-m}   D }     \\  
&  =   \liminf_{m\to+\infty}   \liminf_{l\to+\infty}   \liminf_{n\to+\infty}  \frac{1}{n} 
              \biggl( \log   \sum\limits_{\substack{e\in\E^{l+n} \\ e\subseteq \CC  }}
                              \exp  ( \sup    \{S_n^f\varphi(x) : x\in e \}  ) \\
 &\qquad\qquad\qquad\qquad\qquad\qquad\qquad     - \log \bigl( 2 m 2^{\frac{n}{m}} (\deg f)^{l-m}   D \bigr)   \biggr)\\
&  =   \liminf_{l\to+\infty}   \liminf_{n\to+\infty}  \frac{1}{n} \log
                \sum\limits_{\substack{e\in\E^{l+n} \\ e\subseteq \CC  }}
                              \exp  \bigl( \sup \bigl\{ S_n^f\varphi(x) : x\in e  \bigr\}  \bigr)  ,                   
\end{align*}
where the constant $D>1$ is defined to be 
\begin{equation*}
	D \coloneqq \card(\post f)  e^{C_1 (\diam_d(S^2) )^\alpha}
\end{equation*}
with $C_1 = C_1(f,\CC,d,\varphi,\alpha)>0$ depending only on $f$, $\CC$, $d$, $\varphi$, and $\alpha$ defined in (\ref{eqC1C2}), and the infima are taken over all finite open subcovers $\xi$ of $\bigvee_{i=0}^n (f|_\CC)^{-i}(\eta_m)$, i.e., 
\begin{equation*}
	\xi \in \biggl\{ \varsigma: \varsigma \subseteq \bigvee_{i=0}^n (f|_\CC)^{-i}(\eta_m), \, \bigcup \varsigma = \CC \biggr\}.
\end{equation*}
The second inequality follows from Lemma~\ref{lmSnPhiBound}, and the last inequality follows from the fact that 
\begin{equation*}
\CC\supseteq \bigcup \Bigl( \bigcup_{U\in\xi}  \mathcal{E}^{l+n}(U) \Bigr) =  \bigcup_{U\in\xi} \Bigl( \bigcup  \mathcal{E}^{l+n}(U)\Bigr) \supseteq  \bigcup \xi = \CC
\end{equation*}
and thus $\bigcup_{U\in\xi}  \mathcal{E}^{l+n}(U) = \bigl\{ e\in\E^{l+n} : e\subseteq \CC \}$.

Finally, we will show that 
\begin{equation*}
 P  ( \sigma_{A_{\e}},  \varphi \circ \pi_{\e} ) 
 =  \lim_{l\to+\infty}   \lim_{n\to+\infty}  \frac{1}{n} \log
                \sum\limits_{\substack{e\in\E^{l+n} \\ e\subseteq \CC  }}
                              \exp  \bigl( \sup \bigl\{ S_n^f\varphi(x) : x\in e \bigr\}  \bigr).
\end{equation*}
We denote by $C_n(e_0,e_1,\dots,e_n)$ the $n$-cylinder set 
\begin{equation*}
C_n(e_0,e_1,\dots,e_n) \coloneqq \bigl\{ \{e'_i\}_{i\in\N_0} \in \Sigma_{A_{\e}}^+ : e'_i=e_i \text{ for all } i\in\N_0 \text{ with } i\leq n  \bigr\}
\end{equation*}  
in $\Sigma_{A_{\e}}^+$ containing $\{e_i\}_{i\in\N_0}$, for each $n\in \N_0$ and each $\{e_i\}_{i\in\N_0} \in \Sigma_{A_{\e}}^+$. For each $n\in\N_0$, we denote by $\mathfrak{C}_n$ the set all $n$-cylinder sets in $\Sigma_{A_{\e}}^+$, i.e., 
\begin{equation*}
	\mathfrak{C}_n = \bigl\{ C_n(e_0,e_1,\dots,e_n) : \{e_i\}_{i\in\N_0} \in \Sigma_{A_{\e}}^+ \bigr\}.
\end{equation*}
Then it is easy to verify that for all $n, \, l\in\N_0$, $\mathfrak{C}_n$ is a finite open cover of $\Sigma_{A_{\e}}^+$, and $\bigvee_{i=0}^n \sigma_{A_{\e}}^{-i} (\mathfrak{C}_l)  = \mathfrak{C}_{l+n}$. Hence, by (\ref{eqEquivDefByCoverForTopPressure}), Lemma~\ref{lmCylinderIsTile}, and Proposition~\ref{propSFTs_C}~(i), we have
\begin{align*}
&           P  ( \sigma_{A_{\e}},  \varphi \circ \pi_{\e} )  \\
&\qquad  =  \lim_{l\to +\infty}  \lim_{n\to +\infty}  \frac{1}{n} \log 
               \inf\biggl\{ \sum_{V\in\mathcal{V}}  \exp 
                      \biggl( \sup_{\underline{z}\in V} S_n^{\sigma_{A_{\e}}} (\varphi \circ \pi_{\e})(\underline{z})   \biggr)
                             :  \mathcal{V} \subseteq \bigvee_{i=0}^{n} \sigma_{A_{\e}}^{-i}(\mathfrak{C}_l),\,
                                                        \bigcup \mathcal{V} = \Sigma_{A_{\e}}^+ \biggr\}  \\
&\qquad  =  \lim_{l\to +\infty}  \lim_{n\to +\infty}  \frac{1}{n} \log 
              \sum_{V\in\mathfrak{C}_{l+n}}  \exp 
                      \Bigl( \sup  \Bigl\{ S_n^{\sigma_{A_{\e}}} (\varphi \circ \pi_{\e})(\underline{z}) : \underline{z}\in V \Bigr\}   \Bigr)  \\
&\qquad  =  \lim_{l\to+\infty}   \lim_{n\to+\infty}  \frac{1}{n} \log
                \sum\limits_{\substack{e\in\E^{l+n+1} \\ e\subseteq \CC  }}
                              \exp  \bigl( \sup \bigl\{ S_n^f \varphi(x) : x\in e \bigr\}  \bigr) .                               
\end{align*}
Hence, $P(f|_\CC, \varphi|_\CC) \geq P  ( \sigma_{A_{\e}},  \varphi \circ \pi_{\e} )$. The proof is, therefore, complete.
\end{proof}

\section{Orbifolds and non-local integrability}  \label{sctNLI}
This section is devoted to characterizations of a necessary condition, called \emph{non-local integrability condition}, on the potential $\phi\: S^2 \rightarrow \R$ for the Prime Orbit Theorems for expanding Thurston maps. The characterizations are summarized in Theorem~\ref{thmNLI}. In particular, a real-valued H\"{o}lder continuous $\phi$ on $S^2$ is non-locally integrable if and only if $\phi$ is cohomologous to a constant in the set of real-valued continuous functions on $S^2$. As we will eventually show, such a condition is actually equivalent in our context to the Prime Orbit Theorem (see Theorem~\ref{thmPrimeOrbitTheorem}). In our proof of Theorem~\ref{thmNLI}, we use the notion of orbifolds introduced in general by W.~P.~Thurston in the 1970s in his study of the geometry of $3$-manifolds (see \cite[Chapter~13]{Th80}).

\subsection{Non-local integrability}   \label{subsctNLI_Def}

Let $f\: S^2\rightarrow S^2$ be an expanding Thurston map, $d$ be a visual metric on $S^2$ for $f$, and $\CC\subseteq S^2$ be a Jordan curve satisfying $f(\CC)\subseteq \CC$ and $\post f\subseteq \CC$. Recall the one-sided subshift of finite type $\bigl(\Sigma_{A_{\ti}}^+,\sigma_{A_{\ti}} \bigr)$ associated to $f$ and $\CC$ defined in Proposition~\ref{propTileSFT}. In this section, we write $\Sigma_{f,\, \CC}^+ \coloneqq \Sigma_{A_{\ti}}^+$ and $\sigma \coloneqq \sigma_{A_{\ti}}$, i.e.,
\begin{equation}    \label{eqDefSigma+}
\Sigma_{f,\,\CC}^+ = \bigl\{\{X_i\}_{i\in\N_0} : X_i\in \X^1(f,\CC) \text{ and } f(X_i) \supseteq X_{i+1},\text{ for } i\in\N_0 \bigr\},
\end{equation}
and $\sigma$ is the left-shift operator defined by $\sigma(\{X_i\}_{i\in\N_0} ) = \{X_{i+1}\}_{i\in\N_0}$ for $\{X_i\}_{i\in\N_0} \in \Sigma_{f,\,\CC}^+$.

Similarly, we define
\begin{equation}   \label{eqDefSigma-}
\Sigma_{f,\,\CC}^- \coloneqq \bigl\{\{X_{\minus i}\}_{i\in\N_0} : X_{\minus i}\in \X^1(f,\CC) \text{ and } f \bigl(X_{\minus (i+1)}\bigr) \supseteq X_{\minus i},\text{ for } i\in\N_0 \bigr\}.
\end{equation}
For each $X\in\X^1(f,\CC)$, since $f$ is injective on $X$ (see Proposition~\ref{propCellDecomp}~(i)), we denote the inverse branch of $f$ restricted on $X$ by $f_X^{-1}\: f(X) \rightarrow X$, i.e., $f_X^{-1} \coloneqq (f|_X)^{-1}$.

Let $\psi\in \Holder{\alpha}((S^2,d),\C)$ be a complex-valued H\"{o}lder continuous function with an exponent $\alpha\in (0,1]$. For each $\xi=\{ \xi_{\minus i} \}_{i\in\N_0} \in  \Sigma_{f,\,\CC}^-$, we define the function
\begin{equation}   \label{eqDelta}
\Delta^{f,\,\CC}_{\psi,\,\xi} (x,y) \coloneqq \sum_{i=0}^{+\infty} \bigl( \bigl(\psi \circ f^{-1}_{\xi_{\minus i}} \circ \cdots \circ f^{-1}_{\xi_{0}}\bigr) (x) -  \bigl( \psi \circ f^{-1}_{\xi_{\minus i}} \circ \cdots \circ f^{-1}_{\xi_{0}}\bigr) (y)   \bigr)
\end{equation}
for $(x,y)\in \bigcup\limits_{\substack{X\in\X^1(f,\CC) \\ X\subseteq f(\xi_0)}}X \times X$.  

We will see in the following lemma that the series in (\ref{eqDelta}) converges.

\begin{lemma}  \label{lmDeltaHolder}
Let $f$, $\CC$, $d$, $\psi$, $\alpha$ satisfy the Assumptions in Section~\ref{sctAssumptions}. We assume, in addition, that $f(\CC)\subseteq \CC$. Let $\xi=\{ \xi_{\minus i} \}_{i\in\N_0} \in  \Sigma_{f,\,\CC}^-$. Then for each $X\in\X^1(f,\CC)$ with $X\subseteq f(\xi_0)$, we get that $\Delta^{f,\,\CC}_{\psi,\,\xi} (x,y)$ as a series defined in (\ref{eqDelta}) converges absolutely and uniformly in $x, \, y\in X$, and moreover, for each triple of $x, \, y,z\in X$, the identity 
\begin{equation}   \label{eqDeltaDifference}
\Delta^{f,\,\CC}_{\psi,\,\xi} (x,y)   = \Delta^{f,\,\CC}_{\psi,\,\xi} (z,y) - \Delta^{f,\,\CC}_{\psi,\,\xi} (z,x)
\end{equation}
holds with
$
\Absbig{ \Delta^{f,\,\CC}_{\psi,\,\xi} (x,y) }   \leq C_1 d(x,y)^\alpha,
$
where $C_1=C_1(f,\CC,d,\psi,\alpha)$ is the constant depending on $f$, $\CC$, $d$, $\psi$, and $\alpha$ from Lemma~\ref{lmSnPhiBound}.
\end{lemma}

\begin{proof}
We fix $X\in\X^1(f,\CC)$ with $X\subseteq f(\xi_0)$. By Proposition~\ref{propCellDecomp}~(i) and Lemma~\ref{lmCellBoundsBM}~(ii), for each $i\in\N_0$,
\begin{align} \label{eqPflmDeltaHolder}
&         \Absbig{    \bigl( \psi \circ f^{-1}_{\xi_{\minus i}} \circ \cdots \circ f^{-1}_{\xi_{0}}\bigr) (x) 
                   -  \bigl( \psi \circ f^{-1}_{\xi_{\minus i}} \circ \cdots \circ f^{-1}_{\xi_{0}}\bigr) (y)   } \notag \\
&\qquad   \leq \Hseminorm{\alpha}{ (S^2,d)}{\psi} \diam_d\bigl( \bigl(  f^{-1}_{\xi_{\minus i}} \circ \cdots \circ f^{-1}_{\xi_{0}} \bigr) (X) \bigr)^{\alpha} 
    \leq \Hseminorm{\alpha}{ (S^2,d)}{\psi}  C^\alpha \Lambda^{-i\alpha-\alpha}  
\end{align}
for $x, \, y\in X$, where $C\geq 1$ is the constant from Lemma~\ref{lmCellBoundsBM}. Thus, the series on the right-hand side of (\ref{eqDelta}) converges absolutely. Hence, by (\ref{eqDelta}),  $\Delta^{f,\,\CC}_{\psi,\,\xi} (x,y) =  \Delta^{f,\,\CC}_{\psi,\,\xi} (z,y) - \Delta^{f,\,\CC}_{\psi,\,\xi} (z,x)$. Moreover, by Proposition~\ref{propCellDecomp}~(i), Lemma~\ref{lmSnPhiBound}, and (\ref{eqPflmDeltaHolder}), for each pair of $x, \, y\in X$, and each $j\in\N$,
\begin{align*}
&                \AbsBig{ \Delta^{f,\,\CC}_{\psi,\,\xi} (x,y)}   
          =      \Absbigg{\sum_{i=0}^{+\infty} \bigl(   \bigl( \psi \circ f^{-1}_{\xi_{\minus i}} \circ \cdots \circ f^{-1}_{\xi_{0}}\bigr) (x) 
                                                                         -  \bigl( \psi \circ f^{-1}_{\xi_{\minus i}} \circ \cdots \circ f^{-1}_{\xi_{0}}\bigr) (y) \bigr) }\\
&\qquad \leq  \Absbigg{\sum_{i=0}^{j-1}     \bigl(   \bigl( \psi \circ f^{-1}_{\xi_{\minus i}} \circ \cdots \circ f^{-1}_{\xi_{0}}\bigr) (x) 
                                                                          -  \bigl( \psi \circ f^{-1}_{\xi_{\minus i}} \circ \cdots \circ f^{-1}_{\xi_{0}}\bigr) (y) \bigr) }
          + \sum_{i=j}^{+\infty}  \frac{  \Hseminorm{\alpha}{ (S^2,d)}{\psi}  C^\alpha}{ \Lambda^{i\alpha+\alpha}} \\
&\qquad   \leq C_1 d(x,y)^\alpha + \Hseminorm{\alpha}{(S^2,d)}{\psi} C^\alpha (1-\Lambda^\alpha)^{-1}  \Lambda^{-j\alpha-\alpha}.
\end{align*}
We complete our proof by taking $j$ to infinity.
\end{proof}

\begin{definition}[Temporal distance]  \label{defTemporalDist}
Let $f$, $\CC$, $d$, $\psi$, $\alpha$ satisfy the Assumptions in Section~\ref{sctAssumptions}. We assume, in addition, that $f(\CC)\subseteq \CC$. For $\xi=\{ \xi_{\minus i} \}_{i\in\N_0} \in  \Sigma_{f,\,\CC}^-$ and $\eta=\{ \eta_{\minus i} \}_{i\in\N_0} \in  \Sigma_{f,\,\CC}^-$ with $f(\xi_0) = f(\eta_0)$, we define the \defn{temporal distance} $\psi^{f,\,\CC}_{\xi,\,\eta}$ as 
\begin{equation*}
\psi^{f,\,\CC}_{\xi,\,\eta}(x,y) \coloneqq \Delta^{f,\,\CC}_{\psi,\,\xi} (x,y) - \Delta^{f,\,\CC}_{\psi,\,\eta} (x,y)
\end{equation*}
for each
$
(x,y)\in \bigcup\limits_{\substack{X\in\X^1(f,\CC) \\ X\subseteq f(\xi_0)}}X \times X.
$
\end{definition}

Recall that $f^n$ is an expanding Thurston map with $\post f^n = \post f$ for each expanding Thurston map $f\: S^2\rightarrow S^2$ and each $n\in\N$ (see Remark~\ref{rmExpanding}).

\begin{definition}[Local integrability]   \label{defLI}
Let $f\: S^2\rightarrow S^2$ be an expanding Thurston map and $d$ a visual metric on $S^2$ for $f$. A complex-valued H\"{o}lder continuous function $\psi \in \Holder{\alpha}((S^2,d),\C)$ is \defn{locally integrable} (with respect to $f$ and $d$) if for each natural number $n\in\N$, and each Jordan curve $\CC\subseteq S^2$ satisfying $f^n(\CC)\subseteq \CC$ and $\post f \subseteq \CC$, we have 
\begin{equation*}
\bigl(S_n^f \psi\bigr)^{f^n,\,\CC}_{\xi,\,\eta}(x,y)=0
\end{equation*}
for all $\xi=\{ \xi_{\minus i} \}_{i\in\N_0} \in  \Sigma_{f^n,\,\CC}^-$ and $\eta=\{ \eta_{\minus i} \}_{i\in\N_0} \in  \Sigma_{f^n,\,\CC}^-$ satisfying $f^n(\xi_0) = f^n(\eta_0)$, and all
$
(x,y)\in \bigcup\limits_{\substack{X\in\X^1(f^n,\CC) \\ X\subseteq f^n(\xi_0)}}X \times X.
$

The function $\psi$ is \defn{non-locally integrable} if it is not locally integrable.
\end{definition}

\subsection{Orbifolds and universal orbifold covers}   \label{subsctNLI_Orbifold}

In order to establish Theorem~\ref{thmNLI}, we need to consider orbifolds associated to Thurston maps. An orbifold is a space that is locally represented as a quotient of a model space by a group action (see \cite[Chapter~13]{Th80}). For the purpose of this work, we restrict ourselves to orbifolds on $S^2$. In this context, only cyclic groups can occur, so a simpler definition will be used. We follow closely the setup from \cite{BM17}.

An \defn{orbifold} is a pair $\mathcal{O} = (S, \alpha)$, where $S$ is a surface and $\alpha\: S \rightarrow \widehat{\N} = \N \cup \{+\infty\}$ is a map such that the set of points $p\in S$ with $\alpha(p) \neq 1$ is a discrete set in $S$, i.e., it has no limit points in $S$. We call such a function $\alpha$ a \defn{ramification function} on $S$. The set $\supp(\alpha) \coloneqq \{p\in S : \alpha(p)\geq 2 \}$ is the \defn{support} of $\alpha$. We will only consider orbifolds with $S=S^2$, an oriented $2$-sphere, in this paper.

The \defn{Euler characteristic} of an orbifold $\mathcal{O}= (S^2,\alpha)$ is defined as
\begin{equation*}
\chi(\mathcal{O}) \coloneqq 2 - \sum_{x\in S^2} \biggl(1- \frac{1}{\alpha(x)} \biggr),
\end{equation*}
where we use the convention $\frac{1}{+\infty} = 0$, and note that the terms in the summation are nonzero on a finite set of points. The orbifold $\mathcal{O}$ is \defn{parabolic} if $\chi(\mathcal{O}) = 0$ and \defn{hyperbolic} if $\chi(\mathcal{O}) < 0$.

Every Thurston map $f$ has an associated orbifold $\mathcal{O}_f = (S^2,\alpha_f)$, which plays an essential role in this section.

\begin{definition}  \label{defRamificationFn}
Let $f\: S^2\rightarrow S^2$ be a Thurston map. The \defn{ramification function} of $f$ is the map $\alpha_f\: S^2\rightarrow\widehat{\N}$ defined as
\begin{equation} \label{eqDefRamificationFn}
\alpha_f(x) \coloneqq \lcm \bigl\{\deg_{f^n}(y) : y\in S^2,  \, n\in \N, \text{ and } f^n(y)=x \bigr\}
\end{equation}
for $x\in S^2$.
\end{definition}
Here $\widehat{\N} = \N\cup \{+\infty\}$ with the order relations $<$, $\leq$, $>$, $\geq$ extended in the obvious way, and $\lcm$ denotes the least common multiple on $\widehat{\N}$ defined by $\lcm(A)=+\infty$ if $A\subseteq \widehat{\N}$ is not a bounded set of natural numbers, and otherwise $\lcm(A)$ is calculated in the usually way. Note that different Thurston maps can share the same ramification function. In particular, we have the following fact from \cite[Proposition~2.16]{BM17}.

\begin{prop}   \label{propSameRamificationFn}
Let $f\: S^2\rightarrow S^2$ be a Thurston map. Then $\alpha_f = \alpha_{f^n}$ for each $n\in \N$.
\end{prop}

\begin{definition}[Orbifolds associated to Thurston maps]  \label{defOrbifoldThurstonMaps}
Let $f\: S^2\rightarrow S^2$ be a Thurston map. The \defn{orbifold associated to $f$} is a pair $\mathcal{O}_f \coloneqq (S^2,\alpha_f)$, where $S^2$ is an oriented $2$-sphere and $\alpha_f\: S^2\rightarrow \widehat{\N}$ is the ramification function of $f$.
\end{definition}

Orbifolds associated to Thurston maps are either parabolic or hyperbolic (see \cite[Proposition~2.12]{BM17}).

\smallskip

For an orbifold $\mathcal{O} = (S^2,\alpha)$, we set
\begin{equation}  \label{eqDefS0}
S_0^2 \coloneqq S^2 \setminus \bigl\{x\in S^2 : \alpha(x) = +\infty \bigr\}.
\end{equation}

We record the following facts from \cite{BM17}, whose proofs can be found in \cite{BM17} and references therein (see Theorem~A.26 and Corollary~A.29 in \cite{BM17}).

\begin{theorem}  \label{thmUniOrbCoverBM}
Let $\mathcal{O}=(S^2,\alpha)$ be an orbifold that is parabolic or hyperbolic. Then the following statements are satisfied:
\begin{enumerate}
\smallskip
\item[(i)] There exists a simply connected surface $\XX$ and a branched covering map $\Theta\: \XX\rightarrow S_0^2$ such that $ \deg_\Theta(x) = \alpha(\Theta(x)) $ for each $x\in \XX$.

\smallskip
\item[(ii)] The branched covering map $\Theta$ in $\operatorname{(i)}$ is unique. More precisely, if $\wt \XX$ is a simply connected surface and $\wt\Theta \: \wt \XX \rightarrow S_0^2$ satisfies $\deg_{\wt\Theta}(y) = \alpha\bigl(\wt\Theta(y)\bigr)$ for each $y\in \wt \XX$, then for all points $x_0\in \XX$ and $\wt{x}_0 \in \wt \XX$ with $\Theta(x_0)=\wt\Theta(\wt{x}_0)$ there exists orientation-preserving homeomorphism $A\: \XX\rightarrow \wt \XX$ with $A(x_0)= \wt{x}_0$ and $\Theta=\wt\Theta\circ A$. Moreover, if $\alpha(\Theta(x_0)) = 1$, then $A$ is unique.
\end{enumerate}
\end{theorem}

\begin{definition}[Universal orbifold covering maps]  \label{defUniOrbCover}
Let $\mathcal{O}=(S^2,\alpha)$ be an orbifold that is parabolic or hyperbolic. The map $\Theta\: \XX \rightarrow S_0^2$ from Theorem~\ref{thmUniOrbCoverBM} is called the \defn{universal orbifold covering map} of $\mathcal{O}$.
\end{definition}

We now discuss the deck transformations of the universal orbifold covering map. 

\begin{definition}[Deck transformations]  \label{defDeckTransf}
Let $\mathcal{O}=(S^2,\alpha)$ be an orbifold that is parabolic or hyperbolic, and $\Theta\: \XX \rightarrow S_0^2$ be the universal orbifold covering map of $\mathcal{O}$. A homeomorphism $\wt\sigma\: \XX\rightarrow\XX$ is called a \defn{deck transformation} of $\Theta$ if $\Theta\circ \wt\sigma = \Theta$. The group of deck transformations with composition as the group operation, denoted by $\pi_1(\mathcal{O})$, is called the \defn{fundamental group} of the orbifold $\mathcal{O}$. 
\end{definition}

Note that deck transformations are orientation-preserving. We record the following proposition from \cite[Proposition~A.31]{BM17}. 

\begin{prop}  \label{propDeckTransf}
Let $\mathcal{O}=(S^2,\alpha)$ be an orbifold that is parabolic or hyperbolic, and $\Theta\: \XX \rightarrow S_0^2$ be the universal orbifold covering map of $\mathcal{O}$. Fix $u, \, v\in \XX$. Then $\Theta(u)=\Theta(v)$ if and only if there exists a deck transformation $\wt\sigma\in \pi_1(\mathcal{O})$ with $v=\wt\sigma(u)$.
\end{prop}

\smallskip

We now focus on the orbifold $\mathcal{O}_f = (S^2,\alpha_f)$ associated to a Thurston map $f\: S^2 \rightarrow S^2$.

Orbifolds enable us to lift branches of the inverse map $f^{-1}$ by the universal orbifold covering map. See Lemma~A.32 in \cite{BM17} for a proof of the following Lemma.

\begin{lemma} \label{lmLiftInverse}
Let $f\: S^2\rightarrow S^2$ be a Thurston map, $\mathcal{O}_f=(S^2,\alpha_f)$ be the orbifold associated to $f$, and $\Theta\: \XX \rightarrow S_0^2$ be the universal orbifold covering map of $\mathcal{O}_f$. Fix $u_0, \, v_0\in \XX$ with $(f\circ \Theta)(v_0) = \Theta(u_0)$. Then there exists a branched covering map $\wt{g} \: \XX\rightarrow\XX$ with $\wt{g}(u_0)=v_0$ and $f\circ \Theta \circ \wt{g} = \Theta$. If $u_0\notin \crit \Theta$, then the map $\wt{g}$ is unique.
\end{lemma}

\begin{definition}  \label{defInverseBranch}
Let $f\: S^2\rightarrow S^2$ be a Thurston map, $\mathcal{O}_f=(S^2,\alpha_f)$ be the orbifold associated to $f$, and $\Theta\: \XX \rightarrow S_0^2$ be the universal orbifold covering map of $\mathcal{O}_f$.  A branched covering map $\wt{g}\:\XX\rightarrow\XX$ is called an \defn{inverse branch of $f$ on $\XX$} if $f\circ \Theta\circ\wt{g}=\Theta$. We denote the set of inverse branches of $f$ on $\XX$ by $\Inv(f)$.
\end{definition}

By the definition of branched covering maps, $\wt{g} \: \XX\rightarrow \XX$ is surjective for each $\wt{g}\in \Inv(f)$.

\begin{lemma}   \label{lmHomotopyCurveOnX}
Let $f$ and $\CC$ satisfy the Assumptions in Section~\ref{sctAssumptions}. Let $\mathcal{O}_f=(S^2,\alpha_f)$ be the orbifold associated to $f$, and $\Theta\: \XX \rightarrow S_0^2$ be the universal orbifold covering map of $\mathcal{O}_f$. Then there exists $N\in\N$ such that for each $n\in\N$ with $n\geq N$ and each continuous path $\wt\gamma\: [0,1]\rightarrow \XX\setminus \Theta^{-1}(\post f)$, there exists a continuous path $\gamma\: [0,1]\rightarrow \XX\setminus \Theta^{-1}(\post f)$ with the following properties:
\begin{enumerate}
\smallskip
\item[(i)] $\gamma$ is homotopic to $\wt\gamma$ relative to $\{0, \, 1\}$ in $\XX\setminus \Theta^{-1}(\post f)$.

\smallskip
\item[(ii)] There exists a number $k\in\N$, a strictly increasing sequence of numbers $0\eqqcolon a_0<a_1<\cdots<a_{k-1}<a_k\coloneqq 1$, and a sequence $\{X^n_i\}_{i\in \{1, \, 2, \, \dots, \, k\}}$ of $n$-tiles in $\X^n(f,\CC)$ such that for each $i\in\{1, \, 2, \, \dots, \, k\}$, we have $(\Theta \circ \gamma)((a_{i-1},a_i)) \subseteq \inte(X^n_i)$.
\end{enumerate} 
\end{lemma}

Let $Z$ and $X$ be two topological spaces and $Y\subseteq Z$ be a subset of $Z$. A continuous function $f\: Z\rightarrow X$ is \defn{homotopic to} a continuous function $g\: Z\rightarrow X$ \defn{relative to} $Y$ (in $X$) if there exists a continuous function $H\: Z\times[0,1] \rightarrow X$ such that for each $z\in Z$, each $y\in Y$, and each $t\in[0,1]$, $H(z,0)= f(z)$, $H(z,1)=g(z)$, and $H(y,t)=f(y)=g(y)$.

\begin{remark}
We can choose $N$ to be the smallest number satisfying that no $n$-tile joins opposite sides of $\CC$ for all $n\geq N$.
\end{remark}

\begin{proof}
Since $\post f$ is a finite set, by Lemma~\ref{lmCellBoundsBM}~(ii), we can choose $N\in\N$ large enough such that for each $n\in\N$ with $n\geq N$ and each $n$-tile $X^n\in\X^n$ we have
\begin{equation}   \label{eqPflmHomotopyCurveOnX_AtMost1}
\card(X^n\cap \post f) \leq 1.
\end{equation}

Fix $n\geq N$ and a continuous path $\wt\gamma\:[0,1]\rightarrow \XX\setminus \Theta^{-1}(\post f)$.

We first claim that for each $x\in[0,1]$, there exists an $n$-vertex $v^n_x\in \V^n\setminus \post f$ and an open interval $I_x\subseteq \R$ such that $x\in I_x$ and $(\Theta\circ\wt\gamma)(I_x) \subseteq W^n(v^n_x) \subseteq S^2_0$.

We establish the claim by explicit construction in the following three cases:
\begin{enumerate}
\smallskip
\item[(1)] Assume $(\Theta\circ\wt\gamma)(x) \in \V^n$. Then we let $v^n_x\coloneqq (\Theta\circ\wt\gamma)(x)$. Since $(\Theta\circ\wt\gamma)(x)$ is contained in the open set $W^n(v^n_x)$, we can choose an open interval $I_x\subseteq \R$ containing $x$ with $(\Theta\circ\wt\gamma)(I_x) \subseteq W^n(v^n_x) \subseteq S^2_0$.

\smallskip
\item[(2)] Assume $(\Theta\circ\wt\gamma)(x) \in \inte(e^n)$ for $e^n\in\E^n$. Since $\card(e^n\cap \post f) \leq 1$ by (\ref{eqPflmHomotopyCurveOnX_AtMost1}), we can choose $v^n_x \in e^n \cap \V^n \setminus \post f \subseteq S^2_0$. Then $(\Theta\circ\wt\gamma)(x) \in \inte(e^n) \subseteq W^n(v^n_x) \subseteq S^2_0$. Thus, we can choose an open interval $I_x\subseteq \R$ containing $x$ with $(\Theta\circ\wt\gamma)(I_x) \subseteq W^n(v^n_x) \subseteq S^2_0$.

\smallskip
\item[(3)] Assume $(\Theta\circ\wt\gamma)(x) \in \inte(X^n)$ for $X^n\in\X^n$. By (\ref{eqPflmHomotopyCurveOnX_AtMost1}), we can choose $v^n_x \in X^n \cap \V^n \setminus \post f \subseteq S^2_0$. Then $(\Theta\circ\wt\gamma)(x) \subseteq \inte(X^n) \subseteq W^n(v^n_x) \subseteq S^2_0$. Thus, we can choose an open interval $I_x\subseteq \R$ containing $x$ with $(\Theta\circ\wt\gamma)(I_x) \subseteq W^n(v^n_x) \subseteq S^2_0$. 
\end{enumerate}

The claim is now established.

\smallskip

Since $[0,1]$ is compact, we can choose finitely many numbers $0\eqqcolon x_0 < x_1 < \cdots < x_{m'-1} < x_{m'} \coloneqq 1$ for some $m'\in\N$ such that $\bigcup_{i=1}^m I_{x_i} \supseteq [0,1]$. Then it is clear that we can choose $m\leq m'$ and $0\eqqcolon b_0 < b_1 < \cdots < b_{m-1} < b_m \coloneqq 1$ such that for each $i\in\{1, \, 2, \, \dots, \, m\}$, $[b_{i-1},b_i]\subseteq I_{x_{j(i)}}$ for some $j(i) \in \{1, \, 2, \, \dots, \,  m'\}$.

Fix an arbitrary $i\in\{1, \, 2, \, \dots, \, m\}$. From the discussion above, we have
\begin{equation*}   \label{eqPflmHomotopyCurveOnX_InFlower}
(\Theta\circ\wt\gamma) ([b_{i-1},b_i]) \subseteq (\Theta\circ\wt\gamma) \bigl( I_{x_{j(i)}} \bigr)  \subseteq W^n\bigl( v^n_{x_{j(i)}} \bigr)  \subseteq S^2_0.
\end{equation*}

It follows from Remark~\ref{rmFlower} that we can choose a continuous path $\gamma_i\: [b_{i-1},b_i]\rightarrow W^n\bigl(v^n_{x_{j(i)}}\bigr)$ such that $\gamma_i$ is injective, $\gamma_i(b_{i-1}) = (\Theta \circ \wt\gamma) (b_{i-1})$, $\gamma_i(b_i) = (\Theta \circ \wt\gamma) (b_i)$, and that for each $n$-tile $X^n\in\X^n$ with $X^n\subseteq \overline{W}^n\bigl(v^n_{x_{j(i)}}\bigr)$, $\gamma_i^{-1}( \inte(X^n))$ is connected and $\card \bigl( \gamma_i^{-1}( \partial X^n ) \bigr) \leq 2$. See Figure~\ref{figHomotopy}. Since $W^n\bigl(v^n_{x_{j(i)}}\bigr)$ is simply connected (see Remark~\ref{rmFlower}), $(\Theta\circ\wt\gamma)|_{[b_{i-1},b_i]}$ is homotopic to $\gamma_i$ relative to $\{b_{i-1}, \, b_i\}$ in $W^n\bigl(v^n_{x_{j(i)}}\bigr)$. It follows from Definition~\ref{defRamificationFn}, Definition~\ref{defUniOrbCover}, Lemma~\ref{lmBranchCoverToCover}, and the lifting property of covering maps (see \cite[Lemma~A.6]{BM17} or \cite[Section~1.3, Propositions~1.33, and~1.34]{Ha02}) that there exists a unique continuous path $\wt\gamma_i\: [b_{i-1},b_i] \rightarrow \XX\setminus\Theta^{-1}(\post f)$ such that $\Theta\circ \wt\gamma_i = \gamma_i$ and $\wt\gamma_i$ is homotopic to $\wt\gamma|_{[b_{i-1},b_i]}$ relative to $\{b_{i-1}, \, b_i\}$ in $\XX\setminus\Theta^{-1}(\post f)$.

\begin{figure}
    \centering
    \begin{overpic}
    [width=6cm, %grid, 
    tics=20]{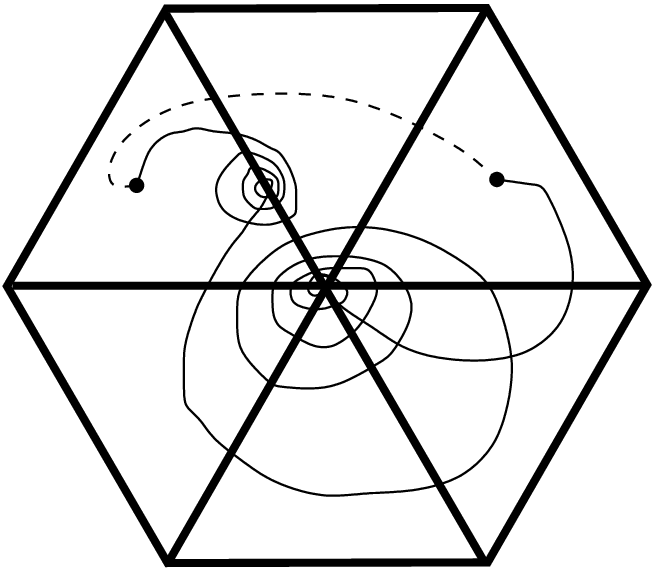}
    \put(80,133){$\gamma_i$}
    \put(55,9){$\Theta \circ \wt\gamma|_{[b_{i-1},b_i]}$}
    \end{overpic}
    \caption{Homotopic curves in $W^n\bigl(v^n_{x_{j(i)}}\bigr)$.}
    \label{figHomotopy}
\end{figure}

We define $\gamma\:[0,1]\rightarrow \XX\setminus\Theta^{-1}(\post f)$ by setting $\gamma|_{[b_{i-1},b_i]} = \wt\gamma_i$. Then it is clear that $\gamma$ is continuous and homotopic to $\wt\gamma$ relative to $\{0, \, 1\}$ in $\XX\setminus\Theta^{-1}(\post f)$. It also follows immediately from our construction that Property~(ii) is satisfied.
\end{proof}

\begin{cor} \label{corGoodPath}
Let $f$ and $\CC$ satisfy the Assumptions in Section~\ref{sctAssumptions}. Let $\mathcal{O}_f=(S^2,\alpha_f)$ be the orbifold associated to $f$, and $\Theta\: \XX \rightarrow S_0^2$ be the universal orbifold covering map of $\mathcal{O}_f$. For each pair of points $x, \, y\in\XX$, there exists a continuous path $\wt\gamma\:[0,1] \rightarrow \XX$, numbers $k, \, n\in\N$, a strictly increasing sequence of numbers $0\eqqcolon a_0<a_1<\cdots<a_{k-1}<a_k\coloneqq 1$, and a sequence $\{X^n_i\}_{i\in \{1, \, 2, \, \dots, \, k\}}$ of $n$-tiles in $\X^n(f,\CC)$ such that $\wt\gamma(0)=x$, $\wt\gamma(1)=y$, and
\begin{equation}   \label{eqGoodPathInTileInte}
(\Theta \circ \wt\gamma)((a_{i-1},a_i)) \subseteq \inte(X^n_i)
\end{equation}
for each $i\in\{1, \, 2, \, \dots, \, k\}$. 

Moreover, if $\{ \wt{g}_j \}_{j\in\N}$ is a sequence in $\Inv(f)$ of inverse branches of $f$ on $\XX$, (i.e., $f\circ \Theta \circ \wt{g}_j = \Theta$ for each $j\in\N$) then for each $m\in\N$, there exists a sequence  $\{X^{n+m}_i\}_{i\in \{1, \, 2, \, \dots, \, k\}}$ of $(n+m)$-tiles in $\X^{n+m} (f,\CC)$ such that
\begin{equation}    \label{eqGoodPathInTileInteSmaller}
(\Theta \circ \wt{g}_m \circ \cdots \circ \wt{g}_1 \circ \wt\gamma )((a_{i-1},a_i)) \subseteq \inte ( X^{n+m}_i )
\end{equation}
for each $i\in\{1, \, 2, \, \dots, \, k\}$. If $d$ is a visual metric on $S^2$ for $f$ with expansion factor $\Lambda >1$, then
\begin{equation}   \label{eqGoodPathLimitLength0}
 \diam_d  (  ( \Theta  \circ \wt{g}_m \circ \cdots \circ \wt{g}_1 \circ \wt\gamma  ) ([0,1])  )  \leq kC \Lambda^{-(n+m)}
\end{equation}
for $m\in\N$, where $C\geq 1$ is the constant from Lemma~\ref{lmCellBoundsBM} depending only on $f$, $\CC$, and $d$.
\end{cor}

\begin{proof}
Let $\CC \subseteq S^2$ be a Jordan curve on $S^2$ with $\post f \subseteq \CC$. Fix an arbitrary number $n\geq N$, where $N\in\N$ is the constant depending only on $f$ and $\CC$ from Lemma~\ref{lmHomotopyCurveOnX}.

Choose $n$-tiles $X^n_1, \, X^n_{\minus 1} \in \X^n(f,\CC)$ with $\Theta(x)\in X^n_1$ and $\Theta(y)\in X^n_{\minus 1}$. Since $n$-tiles are cells of dimension $2$ as discussed in Subsection~\ref{subsctThurstonMap}, we can choose continuous paths $\gamma_x \: \bigl[0,\frac14\bigr] \rightarrow S^2_0$ and $\gamma_y \: \bigl[\frac34, 1\bigr] \rightarrow S^2_0$ with $\gamma_x(0) = \Theta(x)$, $\gamma_y(1) = \Theta(y)$, $\gamma_x\bigl( \bigl(0, \frac14 \bigr] \bigr) \subseteq \inte (X^n_1)$, and $\gamma_y\bigl( \bigl(\frac34, 1 \bigr] \bigr) \subseteq \inte (X^n_{\minus 1})$. Since $\Theta$ is a branched covering map (see Theorem~\ref{thmUniOrbCoverBM}), by Lemma~\ref{lmLiftPathBM} we can lift $\gamma_x$ (resp.\ $\gamma_y$) to $\wt\gamma_x \: \bigl[ 0, \frac14 \bigr] \rightarrow \XX$ (resp.\ $\wt\gamma_y \: \bigl[ \frac34,  1 \bigr] \rightarrow \XX$) such that $\wt\gamma_x(0) = x$ and $\Theta \circ \wt\gamma_x = \gamma_x$ (resp.\ $\wt\gamma_y(1) = y$ and $\Theta \circ \wt\gamma_y = \gamma_y$).

Since $u\coloneqq \wt\gamma_x \bigl( \frac14 \bigr) \in \Theta^{-1} (\inte(X^n_1))$ and $v\coloneqq \wt\gamma_y \bigl( \frac34 \bigr) \in \Theta^{-1} (\inte(X^n_{\minus 1}))$, we have $\{u, \, v\} \subseteq \XX \setminus \Theta^{-1}(\post f)$. Since $\post f$ is a finite set and $\Theta$ is a branched covering map (and is, consequently, discrete), we can choose a continuous path $\widehat\gamma\: \bigl[\frac14, \frac34\bigr] \rightarrow \XX \setminus \Theta^{-1}(\post f)$ with $\widehat\gamma \bigl(\frac14\bigr) = u$ and $\widehat\gamma \bigl(\frac34\bigr) = v$. By Lemma~\ref{lmHomotopyCurveOnX}, there exists a number $k\in\N$, a continuous path $\gamma \: \bigl[ \frac14, \frac34 \bigr] \rightarrow \XX \setminus \Theta^{-1}(\post f)$, a sequence of numbers $\frac14 \eqqcolon a_1 < a_2 < \cdots < a_{k-2} < a_{k-1} \coloneqq \frac34$, and a sequence $\{X^n_i\}_{i\in\{2, \, 3, \, \dots,  \, k-1\}}$ of $n$-tiles in $\X^n(f,\CC)$ such that $\gamma \bigl( \frac14 \bigr) = u$, $\gamma \bigl( \frac34 \bigr) = v$, and $(\Theta \circ \gamma) ( (a_{i-1},a_i) ) \subseteq \inte(X^n_i)$ for each $i\in \{2, \, 3, \, \dots, \,  k-1\}$.

We define a continuous path $\wt\gamma \: [0,1] \rightarrow \XX$ by
\begin{equation*}
\wt\gamma(t) \coloneqq \begin{cases} 
\wt\gamma_x(t) & \text{if } t\in \bigl[ 0, \frac14 \bigr), \\ 
   \gamma (t)  & \text{if } t\in \bigl[ \frac14, \frac34 \bigr], \\
\wt\gamma_y(t) & \text{if } t\in \bigl( \frac34, 1 \bigr].   \end{cases}
\end{equation*}
Let $X^n_k \coloneqq X^n_{\minus 1}$, $a_0 \coloneqq 0$, and $a_k\coloneqq 1$. By our construction, we have $\wt\gamma(0) = x$, $\wt\gamma(1) = y$, and $(\Theta \circ \wt\gamma) ( (a_{i-1},a_i) ) \subseteq \inte(X^n_i)$ for each $i\in \{1, \, 2, \, \dots,  \, k\}$, establishing (\ref{eqGoodPathInTileInte}).

\smallskip

Fix a  sequence $\{ \wt{g}_j \}_{j\in\N}$ of inverse branches of $f$ on $\XX$ in $\Inv(f)$. Fix arbitrary integers $m\in\N$ and  $i\in \{1, \, 2, \, \dots,  \, k\}$. Denote $I_i \coloneqq (a_{i-1}, a_i)$.

By (\ref{eqGoodPathInTileInte}), each connected component of $f^{-m}((\Theta \circ \wt\gamma)(I_i))$ is contained in some connected component of $f^{-m}(\inte(X^n_i))$. Since both $( \Theta \circ \wt{g}_m \circ \cdots \circ \wt{g}_1 \circ \wt \gamma ) (I_i)$ and $f^m ( ( \Theta \circ \wt{g}_m \circ \cdots \circ \wt{g}_1\circ \wt\gamma ) (I_i) ) = (\Theta \circ \wt\gamma)(I_i)$ are connected, by Proposition~\ref{propCellDecomp}~(i), (ii), and (v), there exists an $(n+m)$-tile $X^{n+m}_i \in \X^{n+m}(f,\CC)$ such that $( \Theta \circ \wt{g}_m \circ \cdots \circ \wt{g}_1 \circ \wt\gamma ) (I_i) \subseteq \inte ( X^{n+m}_i )$. Since $m\in\N$ and  $i\in \{1, \, 2, \, \dots, \,  k\}$ are arbitrary, (\ref{eqGoodPathInTileInteSmaller}) is established. Finally, it is clear that (\ref{eqGoodPathLimitLength0}) follows immediately from (\ref{eqGoodPathInTileInteSmaller}) and Lemma~\ref{lmCellBoundsBM}~(ii).
\end{proof}

If we assume that $f$ is expanding, then roughly speaking, each inverse branch on the universal orbifold cover has a unique attracting fixed point (possibly at infinity). The precise statement is formulated in the following proposition.

\begin{prop}  \label{propInvBranchFixPt}
Let $f$, $d$, $\Lambda$ satisfy the Assumptions in Section~\ref{sctAssumptions}. Let $\mathcal{O}_f=(S^2,\alpha_f)$ be the orbifold associated to $f$, and $\Theta\: \XX \rightarrow S_0^2$ be the universal orbifold covering map of $\mathcal{O}_f$. Fix a branched covering map $\wt{g}\: \XX\rightarrow \XX$ satisfying $f\circ \Theta \circ \wt{g} = \Theta$. Then the map $\wt{g}$ has at most one fixed point. Moreover, 
\begin{enumerate}
\smallskip
\item[(i)] if $w\in \XX$ is a fixed point of $\wt{g}$, then $\lim_{i\to +\infty} \wt{g}^i(u) = w$ for all $u\in\XX$;

\smallskip
\item[(ii)] if $\wt{g}$ has no fixed point in $\XX$, then $f$ has a fixed critical point $z\in S^2$ such that $\lim_{i\to +\infty} \Theta\bigl(\wt{g}^i(u)\bigr) = z$ for all $u\in \XX$.
\end{enumerate}
\end{prop}

\begin{proof}
Fix an arbitrary Jordan curve $\CC\subseteq S^2$ on $S^2$ with $\post f\subseteq\CC$.

We observe that it follows immediately from statement~(i) that $\wt{g}$ has at most one fixed point.

\smallskip

(i) We assume that $w\in \XX$ is a fixed point of $\wt{g}$. We argue by contradiction and assume that $\wt{g}^i(u)$ does not converge to $w$  as $i\to+\infty$ for some $u\in\XX$. By Corollary~\ref{corGoodPath} (with $\wt{g}_j\coloneqq \wt{g}$ for each $j\in\N$), we choose a continuous path $\wt\gamma \: [0,1]\rightarrow \XX$ with $\wt\gamma(0)=u$, $\wt\gamma(1)=w$, and
\begin{equation}   \label{eqPfpropInvBranchFixPt_CurveDiamLimit}
\lim_{i\to+\infty} \diam_d \bigl( \bigl(\Theta\circ\wt{g}^i\circ\wt\gamma \bigr) ([0,1]) \bigr) = 0.
\end{equation}
Denote $q\coloneqq  \Theta(w)$. Since $\Theta$ is a branched covering map (see Theorem~\ref{thmUniOrbCoverBM}), we can choose open sets $V\subseteq S^2_0$, $U_i\subseteq\XX$, and homeomorphisms $\varphi_i\: U_i \rightarrow \D$, $\psi_i\: V\rightarrow \D$, for $i\in I$, as in Definition~\ref{defBranchedCover} (with $X\coloneqq \XX$, $Y \coloneqq S^2_0$, and $f \coloneqq \Theta$). We choose $i_0 \in I$ such that $w \in U_{i_0}$. Then by our assumption there exists $r \in (0,1)$ and a strictly increasing sequence $\{k_j\}_{j\in\N}$ of positive integers such that 
$
\wt{g}^{k_j} (u) \notin \varphi_{i_0}^{-1} ( \{z\in\C : \abs{z} < r \} )
$
for each $j\in\N$.

For each $j\in\N$, since $\bigl( \wt{g}^{k_j} \circ \wt\gamma \bigr) ([0,1])$ is a path on $\XX$ connecting $\wt{g}^{k_j}(u)$ and $\wt{g}^{k_j}(w)=w$, we have
$
\bigl( \wt{g}^{k_j} \circ \wt\gamma \bigr) ([0,1])    \cap \varphi_{i_0}^{-1} ( \{z\in\C : \abs{z} = r \} )  \neq \emptyset.
$
Combining the above with (\ref{eqBranchCoverMapLocalPowerMap}) in Definition~\ref{defBranchedCover}, we get
\begin{equation*}
\diam_{\rho}   \bigl(  \bigl( \psi_{i_0} \circ \Theta \circ \wt{g}^{k_j} \circ \wt\gamma \bigr) ([0,1]) \bigr) 
\geq  \rho \bigl(0, \bigl( \psi_{i_0} \circ \Theta \circ \varphi_{i_0}^{-1} \bigr) ( \{z\in\C : \abs{z} = r \} ) \bigr) =  r^{d_{i_0}} >0
\end{equation*}
for $j\in\N$, where $d_{i_0} \coloneqq \deg_\Theta (w)$ as in Definition~\ref{defBranchedCover} and $\rho$ is the Euclidean metric on $\C$. This immediately leads to a contradiction with the fact that $\wt\gamma$ satisfies (\ref{eqPfpropInvBranchFixPt_CurveDiamLimit}), proving statement~(i).

\smallskip
(ii) We assume that $\wt{g}$ has no fixed point in $\XX$. Fix an arbitrary point $v\in\XX$. Let $x \coloneqq u$ and $y \coloneqq \wt{g}(u)$.

By Corollary~\ref{corGoodPath} (with $\wt{g}_j \coloneqq \wt{g}$ for each $j\in\N$), there exists a continuous path $\wt\gamma\:[0,1] \rightarrow \XX$, numbers $k, \, n\in\N$, a strictly increasing sequence of numbers $0\eqqcolon a_0<a_1<\cdots<a_{k-1}<a_k\coloneqq 1$, and for each $m\in\N_0$ there exists a sequence $\{X^{n+m}_i\}_{i\in \{0,1,\dots,k\}}$ of $(n+m)$-tiles in $\X^{(n+m)}$ such that $\wt\gamma(0)=x$, $\wt\gamma(1)=y$, and
\begin{equation}   \label{eqPfpropInvBranchFixPt_InTileInte}
(\Theta \circ \wt{g}^m \circ \wt\gamma)((a_{i-1},a_i)) \subseteq \inte(X^{n+m}_i)
\end{equation}
for each $i\in\{1, \, 2, \, \dots, \, k\}$ and each $m\in\N_0$.  Moreover, for each $m\in\N_0$,
\begin{equation}   \label{eqPfpropInvBranchFixPt_GoodPathLimitLength0}
 \diam_d  ( ( \Theta  \circ \wt{g}^m \circ \wt\gamma  ) ([0,1]) ) \leq kC \Lambda^{-(n+m)}.
\end{equation}
where $C\geq 1$ is the constant from Lemma~\ref{lmCellBoundsBM} depending only on $f$, $\CC$, and $d$.

Since $\wt\gamma(0)=x$ and $\wt\gamma(1)=\wt{g}(x)$, by (\ref{eqPfpropInvBranchFixPt_GoodPathLimitLength0}), for each $m\in\N$ we have  
\begin{equation}  \label{eqPfpropInvBranchFixPt_gmxIterateDistance}
 d  \bigl( \Theta  ( \wt{g}^m(x)  ), \Theta \bigl( \wt{g}^{m+1}(x) \bigr) \bigr)  \leq kC \Lambda^{-(n+m)}.
\end{equation}
Since $S^2$ is compact, we get $\lim_{m\to+\infty}  \Theta  ( \wt{g}^m(x)  ) = z$ for some $z\in S^2$. Since $\bigl( f \circ \Theta \circ \wt{g}^{m+1} \bigr) = \Theta \circ \wt{g}^m$ for each $m\in\N$, we have $f(z) = z$. To see that $z$ is independent of $x$, we choose arbitrary points $x'\in\XX$ and $z'\in S^2$ with $\lim_{m\to+\infty} \Theta  ( \wt{g}^m(x')) = z'$. Then by the same argument as above, we get $f(z')=z'$. Applying Corollary~\ref{corGoodPath} (with $y\coloneqq x'$), we get $z=z'$.

It suffices to show $z\in \crit f$ now. We observe that it follows from (\ref{eqDefS0}), (\ref{eqDefRamificationFn}), and $f(z)=z$ that it suffices to prove $z\notin S^2_0$. We argue by contradiction and assume that $z\in S^2_0$. Since $\Theta$ is a branched covering map (see Theorem~\ref{thmUniOrbCoverBM}), we can choose open sets $V\subseteq S^2_0$ and $U_i\subseteq \XX$, $i\in I$, as in Definition~\ref{defBranchedCover} (with $X \coloneqq \XX$, $Y \coloneqq S^2_0$, $q \coloneqq z$, and $f \coloneqq \Theta$). By Lemma~\ref{lmCellBoundsBM}~(ii) and the fact that flowers are open sets (see Remark~\ref{rmFlower}), it is clear that there exist numbers $l,\, L\in\N$, an $l$-vertex $v^l\in \V^l$, and an $L$-vertex $v^L\in\V^L$ such that $l<L$ and
\begin{equation} \label{eqPfpropInvBranchFixPt_zInVInV}
z \in W^L \bigl( v^L \bigr) \subseteq \overline{W}^L \bigl( v^L \bigr) \subseteq W^l \bigl( v^l \bigr) \subseteq V.
\end{equation}  
By (\ref{eqPfpropInvBranchFixPt_gmxIterateDistance}) there exists $N\in\N$ large enough so that for each $m\in\N$ with $m\geq N$, we have $\Theta  ( \wt{g}^m(x)  ) \subseteq W^L \bigl( v^L \bigr) \subseteq V$. Since $\wt{g}$ has no fixed points in $\XX$, $\wt{g}^m(x)$ does not converge to any point in $\Theta^{-1}(z)$ as $m\to+\infty$, for otherwise, suppose $\lim_{m\to+\infty} \wt{g}^m(x) \eqqcolon p \in \XX$, then $\wt{g}(p)=p$, a contradiction. Hence, there exists a strictly increasing sequence $\{m_j\}_{j\in\N}$ of positive integers such that $\wt{g}^{m_j}(x)$ and $\wt{g}^{m_j + 1}(x)$ are contained in different connected components of $\Theta^{-1}(V)$. Since $\wt{g}^{m_j} ( \wt\gamma(0) ) = \wt{g}^{m_j} (x)$, $\wt{g}^{m_j} ( \wt\gamma(1) ) = \wt{g}^{m_j + 1} (x)$, and the set $\wt{g}^{m_j} (\wt\gamma ( [0,1] ))$ is connected, we get from (\ref{eqPfpropInvBranchFixPt_zInVInV}) that
\begin{equation*}
 ( \Theta \circ \wt{g}^{m_j} \circ \wt\gamma  ) ( [0,1] )  \cap \partial W^L \bigl( v^L \bigr)     \neq  \emptyset \neq ( \Theta \circ \wt{g}^{m_j} \circ \wt\gamma  ) ( [0,1] )  \cap \partial W^l \bigl( v^l \bigr)    .
\end{equation*}
This contradicts with (\ref{eqPfpropInvBranchFixPt_GoodPathLimitLength0}). Therefore, $z\notin S^2_0$ and $z\in \crit f$.
\end{proof}

\subsection{Proof of the characterization Theorem~\ref{thmNLI}}     \label{subsctNLI_Proof}

We first lift the local integrability condition by the universal orbifold covering map.

\begin{lemma}   \label{lmLiftLI}
Let $f$,  $\CC$, $d$, $\psi$ satisfy the Assumptions in Section~\ref{sctAssumptions}. We assume, in addition, that $f(\CC)\subseteq \CC$. Let $\mathcal{O}_f=(S^2, \alpha_f)$ be the orbifold associated to $f$, and $\Theta\: \XX\rightarrow S^2_0$ the universal orbifold covering map of $\mathcal{O}_f$. Assume that $\psi^{f,\,\CC}_{\xi,\,\eta}(x,y) = 0$ for all $\xi \coloneqq \{ \xi_{\minus i} \}_{i\in\N_0} \in  \Sigma_{f,\,\CC}^-$ and $\eta \coloneqq \{ \eta_{\minus i} \}_{i\in\N_0} \in  \Sigma_{f,\,\CC}^-$ with $f(\xi_0) = f(\eta_0)$, and all $(x,y)\in \bigcup\limits_{\substack{X\in\X^1(f,\CC) \\ X\subseteq f(\xi_0)}}X \times X$. Then for each pair of sequences $\{ \wt{g}_i \}_{i\in\N}$ and $\{ \wt{h}_i \}_{i\in\N}$ of inverse branches of $f$ on $\XX$, (i.e., $f\circ\Theta\circ\wt{g}_i = \Theta$ and $f\circ\Theta\circ\wt{h}_i = \Theta$ for $i\in\N$,) we have
\begin{align}  \label{eqLiftLI}
	\begin{aligned}
		&  \sum_{i=1}^{+\infty} \bigl(  \bigl(   \wt{\psi} \circ \wt{g}_i \circ\cdots\circ \wt{g}_1 \bigr)(u)  
		- \bigl(   \wt{\psi} \circ \wt{g}_i \circ\cdots\circ \wt{g}_1 \bigr)(v) \bigr)  \\
		&\qquad =  \sum_{i=1}^{+\infty}\bigl(  \bigl(   \wt{\psi} \circ \wt{h}_i \circ\cdots\circ \wt{h}_1 \bigr)(u) 
		- \bigl(   \wt{\psi} \circ \wt{h}_i \circ\cdots\circ \wt{h}_1 \bigr)(v) \bigr)   
	\end{aligned}
\end{align}
for $u, \, v\in\XX$, where $\wt{\psi} \coloneqq \psi\circ \Theta$.
\end{lemma}

\begin{proof}
Fix sequences $\{\wt{g}_i\}_{i\in\N}$ and $\{\wt{h}_i\}_{i\in\N}$ of inverse branches of $f$ on $\XX$. Fix arbitrary points $u, \, v\in\XX$.

By Corollary~\ref{corGoodPath}, there exists a continuous path $\wt\gamma \: [0,1] \rightarrow \XX$, integers $k, \, n\in\N$, a strictly increasing sequence of numbers $0 \eqqcolon a_0 < a_1 < \cdots < a_{k-1} < a_k \coloneqq 1$, and a sequence $\{X^n_i\}_{i\in\{1, \, 2, \, \dots,  \, k\}}$ of $n$-tiles in $\X^n(f,\CC)$ such that $\wt\gamma(0) = u$, $\wt\gamma(1) = v$, and 
\begin{equation} \label{eqPflmLiftLI_InTileInte}
( \Theta \circ \wt\gamma ) ( (a_{i-1},a_i) ) \subseteq \inte (X^n_i).
\end{equation}
Moreover, for each $m\in\N$, there exist two sequences  $\{X^{n+m}_i\}_{i\in \{1, \, 2, \, \dots, \, k\}}$ and $\{Y^{n+m}_i\}_{i\in \{1, \, 2, \, \dots, \, k\}}$ of $(n+m)$-tiles in $\X^{n+m} (f,\CC)$ such that
\begin{align}   
( \Theta \circ \wt{g}_m \circ \cdots \circ \wt{g}_1 \circ \wt\gamma ) ( (a_{i-1}, a_i) )  & \subseteq \inte ( X^{n+m}_i )  \quad \text{ and} \label{eqPflmLiftLI_InTileInteSmallerX} \\
\bigl( \Theta \circ \wt{h}_m \circ \cdots \circ \wt{h}_1 \circ \wt\gamma \bigr) ( (a_{i-1}, a_i) )  & \subseteq \inte ( Y^{n+m}_i )  \label{eqPflmLiftLI_InTileInteSmallerY}
\end{align}
for each $i\in\{1, \, 2, \, \dots, \, k\}$.

We denote $u_0 \coloneqq \wt{\gamma} (a_0) =  u$, $u_i \coloneqq \wt\gamma (a_i)$, and $I_i \coloneqq (a_{i-1},a_i)$ for $i\in\{1,  \, 2,  \, \dots,  \, k\}$.

Observe that it suffices to show that (\ref{eqLiftLI}) holds with $u$ and $v$ replaced by $u_{i-1}$ and $u_i$, respectively, for each $i\in\{1, \, 2, \, \dots, \, k\}$.

Fix an arbitrary integer $i\in\{1, \, 2, \, \dots, \, k\}$.

For each $j\in\N_0$,  we denote by $\xi_{\minus j}$ the unique $1$-tile in $\X^1$ containing $X^{n+j+1}_i$, and denote by $\eta_{\minus j}$ the unique $1$-tile in $\X^1$ containing $Y^{n+j+1}_i$.

We will show that $\xi \coloneqq \{\xi_{\minus j}\}_{j\in\N_0}$ and $\eta \coloneqq \{ \eta_{\minus j} \}_{j\in\N_0}$ satisfy the following properties:
\begin{enumerate}
\smallskip
\item[(1)]  $\xi, \, \eta \in  \Sigma_{f,\,\CC}^-$.

\smallskip
\item[(2)] $f(\xi_0) = f(\eta_0) \eqqcolon X^0 \supseteq X^n_i \supseteq (\Theta \circ \wt\gamma) (I_i)$.

\smallskip
\item[(3)] For each $j\in\N_0$,
\begin{align*}
	\bigl( \Theta \circ \wt{g}_{j+1} \circ \cdots \circ \wt{g}_1 \circ \wt\gamma \bigr) (I_i)
             &\subseteq \bigl( f_{\xi_{\minus j}}^{-1} \circ  \cdots \circ f_{\xi_0}^{-1} \bigr)  (X^0) \qquad\text{ and} \\
    \bigl( \Theta \circ \wt{h}_{j+1} \circ \cdots \circ \wt{h}_1 \circ \wt\gamma \bigr) (I_i)
             &\subseteq \bigl( f_{\eta_{\minus j}}^{-1} \circ  \cdots \circ f_{\eta_0}^{-1} \bigr)  (X^0).
\end{align*}
\end{enumerate}

\smallskip

(1) Fix an arbitrary integer $m\in\N_0$. We note that by (\ref{eqPflmLiftLI_InTileInteSmallerX}),
\begin{align}   \label{eqPflmLiftLI_Prop1}
&        f \bigl(\xi_{\minus(m+1)} \bigr) \cap \inte(\xi_{\minus m}) \notag      \\
&\qquad \supseteq f \bigl( X^{n+m+2}_i \bigr) \cap \inte \bigl( X^{n+m+1}_i \bigr)   \notag\\
&\qquad \supseteq ( f \circ \Theta \circ \wt{g}_{m+2} \circ \cdots \circ \wt{g}_1 \circ \wt\gamma ) (I_i) 
               \cap ( \Theta \circ \wt{g}_{m+1}           \circ \cdots \circ \wt{g}_1 \circ \wt\gamma ) (I_i)  \\
&\qquad    =              ( \Theta \circ \wt{g}_{m+1}            \circ \cdots \circ \wt{g}_1 \circ \wt\gamma ) (I_i) \notag       \\
 &\qquad      \neq \emptyset. \notag                    
\end{align}
Since $f \bigl( \xi_{\minus(m+1)} \bigr) \in \X^0$ (see Proposition~\ref{propCellDecomp}~(i)), we get from (\ref{eqPflmLiftLI_Prop1}) that  $f \bigl( \xi_{\minus(m+1)} \bigr) \supseteq \xi_{\minus m}$. Since $m\in\N_0$ is arbitrary, we get $\xi \in \Sigma_{f,\,\CC}^-$. Similarly, we have $\eta \in \Sigma_{f,\,\CC}^-$.

\smallskip

(2) We note that by (\ref{eqPflmLiftLI_InTileInteSmallerX}), (\ref{eqPflmLiftLI_InTileInteSmallerY}), and (\ref{eqPflmLiftLI_InTileInte}),
\begin{align}   \label{eqPflmLiftLI_Prop2}
&f ( \inte (\xi_0) ) \cap f (\inte (\eta_0) )   \cap  \inte (X^n_i)  \notag\\
&\qquad \supseteq ( f \circ \Theta \circ \wt{g}_1 \circ \wt\gamma ) (I_i)  \cap \bigl( f \circ \Theta \circ \wt{h}_1 \circ \wt\gamma \bigr) (I_i)   \cap  \inte (X^n_i)  
    =     (  \Theta \circ \wt\gamma ) (I_i)  
  \neq      \emptyset.  
\end{align}
It follows from (\ref{eqPflmLiftLI_Prop2}) and Proposition~\ref{propCellDecomp}~(i) that $f(\xi_0) = f(\eta_0) \eqqcolon X^0 \supseteq  X^n_i \supseteq (\Theta \circ \wt\gamma) (I_i)$. This verifies Property~(2).

\smallskip

(3) We will establish the first relation in Property~(3) since the proof of the second one is the same. We note that by (\ref{eqPflmLiftLI_InTileInteSmallerX}), it suffices to show that
\begin{equation}   \label{eqPflmLiftLI_Prop3Equiv}
X^{n+j+1}_i \subseteq  \bigl( f_{\xi_{\minus j}}^{-1} \circ  \cdots \circ f_{\xi_0}^{-1} \bigr)  (X^0)
\end{equation}
for each $j\in\N_0$.

We prove (\ref{eqPflmLiftLI_Prop3Equiv}) by induction on $j\in\N_0$.

For $j=0$, we have $f_{\xi_0}^{-1} (X^0) = \xi_0 \supseteq X^{n+1}_i$ by Property~(2) and our construction above.

Assume that (\ref{eqPflmLiftLI_Prop3Equiv}) holds for some $j\in\N_0$. Then by the induction hypothesis, (\ref{eqPflmLiftLI_InTileInteSmallerX}), and the fact that $f$ is injective on $\xi_{\minus (j+1)} \supseteq X^{n+j+2}_i$ (see Proposition~\ref{propCellDecomp}~(i)), we get
\begin{align*}
&                  \bigl( f_{\xi_{\minus (j+1)}}^{-1} \circ  \cdots \circ f_{\xi_0}^{-1} \bigr)  (X^0)  \cap   \inte \bigl( X^{n+j+2}_i \bigr)  \\
&\qquad \supseteq  f_{\xi_{\minus (j+1)}}^{-1} \bigl( X^{n+j+1}_i \bigr)   \cap   \inte \bigl( X^{n+j+2}_i \bigr)  \\
&\qquad  =       f_{\xi_{\minus (j+1)}}^{-1} \bigl( X^{n+j+1}_i    \cap  f \bigl(  \inte \bigl( X^{n+j+2}_i \bigr) \bigr) \bigr) \\
&\qquad \supseteq  f_{\xi_{\minus (j+1)}}^{-1} \bigl( \bigl( \Theta \circ \wt{g}_{j+1} \circ \cdots \circ \wt{g}_1 \circ \wt\gamma \bigr) (I_i)
                                         \cap \bigl( f \circ \Theta \circ \wt{g}_{j+2} \circ \cdots \circ \wt{g}_1 \circ \wt\gamma \bigr) (I_i)  \bigr) \\
&\qquad    =       f_{\xi_{\minus (j+1)}}^{-1} \bigl( \bigl( \Theta \circ \wt{g}_{j+1} \circ \cdots \circ \wt{g}_1 \circ \wt\gamma \bigr) (I_i)     \bigr). 
\end{align*}
The set on the right-hand side of the last line above is nonempty, since by (\ref{eqPflmLiftLI_InTileInteSmallerX}) and our construction,
\begin{equation*}
        \bigl( \Theta \circ \wt{g}_{j+1} \circ \cdots \circ \wt{g}_1 \circ \wt\gamma \bigr) (I_i) 
   =   \bigl( f \circ \Theta \circ \wt{g}_{j+2} \circ \cdots \circ \wt{g}_1 \circ \wt\gamma \bigr) (I_i) 
\subseteq   f \bigl( X^{n+j+2}_i \bigr)
\subseteq    f \bigl( \xi_{\minus (j+1)} \bigr).
\end{equation*}
On the other hand, since $\xi \in \Sigma_{f,\,\CC}^-$, it follows immediately from Lemma~\ref{lmCylinderIsTile} that
\begin{equation}  \label{eqPflmLiftLI_ImageIsTile}
\bigl( f_{\xi_{\minus (j+1)}}^{-1} \circ  \cdots \circ f_{\xi_0}^{-1} \bigr)  (X^0)  \in \X^{j+2}.
\end{equation}
Hence, $X^{n+j+2}_i \subseteq \bigl( f_{\xi_{\minus (j+1)}}^{-1} \circ  \cdots \circ f_{\xi_0}^{-1} \bigr)  (X^0)$.

The induction step is now complete, establishing Property~(3).

\smallskip

Finally, by Property~(3), for each $m\in\N_0$ and each $w\in\{ u_{i-1},  \, u_i \}$ we have
\begin{equation*}
 ( \Theta \circ \wt{g}_{m+1} \circ \cdots \circ \wt{g}_1 ) (w)
             \subseteq \bigl( f_{\xi_{\minus m}}^{-1} \circ  \cdots \circ f_{\xi_0}^{-1} \bigr)  (X^0).
\end{equation*}
Since $f^{m+1} ( ( \Theta \circ \wt{g}_{m+1} \circ \cdots \circ \wt{g}_1 ) (w) ) = \Theta (w)$ and $f^{m+1}$ is injective on $\bigl( f_{\xi_{\minus m}}^{-1} \circ  \cdots \circ f_{\xi_0}^{-1} \bigr)  (X^0)$ (by (\ref{eqPflmLiftLI_ImageIsTile}) and Proposition~\ref{propCellDecomp}~(i)) with $\bigl( f^{m+1} \circ f_{\xi_{\minus m}}^{-1} \circ  \cdots \circ f_{\xi_0}^{-1} \bigr)  (x) = x$ for each $x\in X^0$, we get
$
 ( \Theta \circ \wt{g}_{m+1} \circ \cdots \circ \wt{g}_1 ) (w)   =  \bigl( f_{\xi_{\minus m}}^{-1} \circ  \cdots \circ f_{\xi_0}^{-1} \bigr)  ( \Theta(w) )$. Hence,
\begin{align*}
&                       \sum_{j=0}^{+\infty}    \Bigl(  \Bigl(   \wt{\psi} \circ \wt{g}_{j+1} \circ\cdots\circ \wt{g}_1 \Bigr)(u_{i-1})  
                                                                                      - \Bigl(   \wt{\psi} \circ \wt{g}_{j+1} \circ\cdots\circ \wt{g}_1 \Bigr)(u_i) \Bigr) \\ 
&\qquad  =      \sum_{j=0}^{+\infty}    \Bigl(  \Bigl(   \psi \circ f_{\xi_{\minus j}}^{-1} \circ \cdots \circ f_{\xi_{0}}^{-1} \Bigr)  (   \Theta (u_{i-1})  )  
                                                                                     - \Bigl(   \psi \circ f_{\xi_{\minus j}}^{-1} \circ \cdots \circ f_{\xi_{0}}^{-1} \Bigr)  (   \Theta (u_i)  )   \Bigr) \\
&\qquad   =     \Delta^{f,\,\CC}_{\psi,\,\xi} (  \Theta(u_{i-1}),   \Theta(u_i)  )   ,                                                                                  
\end{align*}
where $\Delta^{f,\,\CC}_{\psi,\,\xi} $ is defined in (\ref{eqDelta}). Similarly, the right-hand side of (\ref{eqLiftLI}) with $u$ and $v$ replaced by $u_{i-1}$ and $u_i$, respectively, is equal to $\Delta^{f,\,\CC}_{\psi,\,\eta} (  \Theta(u_{i-1}),   \Theta(u_i)  )$.

Since $\{ \Theta(u_{i-1}),  \, \Theta(u_i) \} \subseteq X^n_i  \subseteq f ( \xi_0 )  = f ( \eta_0 )$ by Property~(2), we get from our assumption on $\psi^{f,\,\CC}_{\xi,\,\eta}$ and Definition~\ref{defTemporalDist} that
\begin{equation*}
0 = \psi^{f,\,\CC}_{\xi,\,\eta}( \Theta(u_{i-1}),   \Theta(u_i) ) = \Delta^{f,\,\CC}_{\psi,\,\xi} ( \Theta(u_{i-1}),   \Theta(u_i) ) - \Delta^{f,\,\CC}_{\psi,\,\eta} ( \Theta(u_{i-1}),   \Theta(u_i) ).
\end{equation*}
Therefore,  (\ref{eqLiftLI}) holds with $u$ and $v$ replaced by $u_{i-1}$ and $u_i$, respectively. This establishes the lemma.
\end{proof}

\begin{proof}[Proof of Theorem~\ref{thmNLI}]
In the case where $\psi\in\Holder{\alpha}(S^2,d)$ is real-valued, the implication (vi)$\implies$(iii) follows immediately from Theorem~5.45 in \cite{Li17}. Conversely, statement~(iii) implies that $S_n\psi(x)=nK=S_n(K\mathbbm{1}_{S^2})(x)$ for all $x\in S^2$ and $n\in\N$ satisfying $f^n(x)=x$. So $K\in\R$. By Proposition~5.52 in \cite{Li17}, the function $\beta$ in statement~(iii) can be assumed to be real-valued. Then (vi) follows from Theorem~5.45 in \cite{Li17}.

\smallskip 
We now focus on the general case where $\psi\in\Holder{\alpha}((S^2,d),\C)$ is complex-valued. The implication (i)$\implies$(ii) is trivial.

\smallskip

(iii)$\implies$(iv): Note that statement~(iii) implies that $S_n\psi(x)=nK=S_n(K\mathbbm{1}_{S^2})(x)$ for all $x\in S^2$ and $n\in\N$ satisfying $f^n(x)=x$. Now statement~(iv) follows from Proposition~5.52 in \cite{Li17}.

\smallskip

(iv)$\implies$(i): Fix a Jordan curve $\CC\subseteq S^2$ satisfying $\post f\subseteq\CC$ and $f^n(\CC)\subseteq \CC$ for some $n\in\N$ (see Lemma~\ref{lmCexistsL}). We assume that statement~(iv) holds. Denote $F \coloneqq f^n$ and $\Psi \coloneqq S_n^f\psi=nK+\tau\circ F - \tau \in \Holder{\alpha} ((S^2,d),\C)$ (by Lemma~\ref{lmSnPhiHolder}). Fix any $\xi=\{ \xi_{\minus i} \}_{i\in\N_0} \in  \Sigma_{F,\,\CC}^-$ and $\eta=\{ \eta_{\minus i} \}_{i\in\N_0} \in  \Sigma_{F,\,\CC}^-$ with $F(\xi_0) = F(\eta_0)$. By (\ref{eqDelta}), we get that for all $(x,y)\in \bigcup\limits_{\substack{X\in\X^1(F,\CC) \\ X\subseteq F(\xi_0)}}X \times X$,
\begin{align*}
    \Delta^{F,\,\CC}_{\Psi,\,\xi}(x,y) 
&=  \lim_{i\to+\infty}  \Bigl(  \tau(x)  -   \tau\bigl(F^{-1}_{\xi_{\minus i}} \circ\cdots\circ F^{-1}_{\xi_0}(x)\bigr)  
                                                - \tau(y) +   \tau\bigl(F^{-1}_{\xi_{\minus i}} \circ\cdots\circ F^{-1}_{\xi_0}(y)\bigr)  \Bigr)  \\
&=  \tau(x)-\tau(y).
\end{align*}
The second equality here follows from the H\"{o}lder continuity of $\tau$ and Lemma~\ref{lmCellBoundsBM}~(ii). Similarly, we have $\Delta^{F,\,\CC}_{\Psi,\,\eta}(x,y)= \tau(x)-\tau(y)$. Therefore, by Definition~\ref{defTemporalDist}, $\Psi_{\xi,\,\eta}^{F,\,\CC}(x,y)=0$ for all $(x,y)\in \bigcup\limits_{\substack{X\in\X^1(F,\CC) \\ X\subseteq F(\xi_0)}}X \times X$, establishing (i).

\smallskip

(iv)$\implies$(v): Let $M\coloneqq n \cdot \mathbbm{1}_{\Sigma_{A_{\ti}}^+}$ and $\varpi \coloneqq \tau \circ \pi_{\ti}$. Then (v) follows immediately from Proposition~\ref{propTileSFT}.

\smallskip

(v)$\implies$(ii): We argue by contradiction and assume that (ii) does not hold but (v) holds. Fix $n\in\N$, $\CC\subseteq S^2$, $K\in\C$, $M\in\CCC \bigl( \Sigma_{A_{\ti}}^+ , \Z \bigr)$, and $\varpi\in\CCC \bigl( \Sigma_{A_{\ti}}^+ , \C \bigr)$ as in (v). Recall $F\coloneqq f^n$, $\Psi \coloneqq S_n^f \psi$, and
\begin{equation}   \label{eqPfthmNLI_DefPhiPi_PsiCohomology}
\Psi \circ \pi_{\ti} = KM + \varpi \circ \sigma_{A_{\ti}} - \varpi.
\end{equation}
Since $ \Sigma_{A_{\ti}}^+$ is compact, we know that $\card \bigl( M \bigl(  \Sigma_{A_{\ti}}^+ \bigr) \bigr)$ is finite. Thus, considering that the topology on $\Sigma_{A_{\ti}}^+$ is induced from the product topology, we can choose $m\in\N$ such that $M( \underline{u} ) = M ( \underline{v} )$ for all $\underline{u} = \{u_i\}_{i\in\N_0}\in  \Sigma_{A_{\ti}}^+$ and $\underline{v} = \{v_i\}_{i\in\N_0} \in  \Sigma_{A_{\ti}}^+$ with $u_i=v_i$ for each $i\in\{0, \, 1, \, \dots,  \, m\}$. Fix $\xi=\{ \xi_{\minus i} \}_{i\in\N_0} \in  \Sigma_{F,\,\CC}^-$, $\eta=\{ \eta_{\minus i} \}_{i\in\N_0} \in  \Sigma_{F,\,\CC}^-$, $X \in \X^1(F,\CC)$, and $x, \, y\in X$ with $X \subseteq F(\xi_0) = F(\eta_0)$ and
$
\Psi_{\xi,\,\eta}^{F,\,\CC}(x,y) \neq 0
$
as in (ii). Since $D\coloneqq S^2 \setminus \bigcup_{i\in\N_0} F^{-i} (\CC)$ is dense in $S^2$, by (\ref{eqDeltaDifference}) in Lemma~\ref{lmDeltaHolder} and Definition~\ref{defTemporalDist}, we can assume without loss of generality that there exists $X^m \in \X^m(F,\CC)$ with $x, \, y\in X^m \setminus D \subseteq X$. Thus, by Proposition~\ref{propTileSFT}, $\pi_{\ti}$ is injective on $\pi_{\ti}^{-1} (P)$ where $P\coloneqq \bigcup_{i\in\N_0} F^{-i} (\{x, \, y\})$. With abuse of notation, we denote by $\pi_{\ti}^{-1} \: P \rightarrow \pi_{\ti}^{-1}(P)$ the inverse of $\pi_{\ti}$ on $P$. Then by (\ref{eqDelta}), Proposition~\ref{propTileSFT}, and (\ref{eqPfthmNLI_DefPhiPi_PsiCohomology}),
\begin{align*}
        \Delta^{F,\,\CC}_{\Psi,\,\xi}(x,y) 
&=   \sum_{i=0}^{+\infty} \Bigl( 
                     \bigl( ( \Psi \circ \pi_{\ti} )  \circ \bigl( \pi_{\ti}^{-1} \circ F^{-1}_{\xi_{\minus i}} \circ \pi_{\ti} \bigr) \circ \cdots \circ \bigl( \pi_{\ti}^{-1} \circ F^{-1}_{\xi_0}  \circ \pi_{\ti} \bigr) \bigr) \bigl( \pi_{\ti}^{-1} (x) \bigr) \\
   & \qquad\quad - \bigl( ( \Psi \circ \pi_{\ti} )  \circ \bigl( \pi_{\ti}^{-1} \circ F^{-1}_{\xi_{\minus i}} \circ \pi_{\ti} \bigr) \circ \cdots \circ \bigl( \pi_{\ti}^{-1} \circ F^{-1}_{\xi_0}  \circ \pi_{\ti} \bigr) \bigr) \bigl( \pi_{\ti}^{-1} (y) \bigr)
                                      \Bigr) \\
&=  \lim_{i\to+\infty}  \Bigl( \varpi \bigl( \pi_{\ti}^{-1}(x) \bigr)  -\varpi \bigl( \pi_{\ti}^{-1}(y) \bigr) \\
    &\qquad\qquad\quad                -\bigl( \varpi \circ \bigl( \pi_{\ti}^{-1} \circ F^{-1}_{\xi_{\minus i}} \circ \pi_{\ti} \bigr) \circ \cdots \circ \bigl( \pi_{\ti}^{-1} \circ F^{-1}_{\xi_0}  \circ \pi_{\ti} \bigr) \bigr) \bigl( \pi_{\ti}^{-1} (x) \bigr) \\
    &\qquad\qquad\quad              +\bigl( \varpi \circ \bigl( \pi_{\ti}^{-1} \circ F^{-1}_{\xi_{\minus i}} \circ \pi_{\ti} \bigr) \circ \cdots \circ \bigl( \pi_{\ti}^{-1} \circ F^{-1}_{\xi_0}  \circ \pi_{\ti} \bigr) \bigr) \bigl( \pi_{\ti}^{-1} (y) \bigr)  \Bigr) \\
&=  \varpi \bigl( \pi_{\ti}^{-1}(x) \bigr)  -\varpi \bigl( \pi_{\ti}^{-1}(y) \bigr).
\end{align*}
The last identity follows from the uniform continuity of $\varpi$. Similarly, $ \Delta^{F,\,\CC}_{\Psi,\,\eta}(x,y) = \varpi \bigl( \pi_{\ti}^{-1}(x) \bigr)  -\varpi \bigl( \pi_{\ti}^{-1}(y) \bigr) $. Thus, by Definition~\ref{defTemporalDist}, $\Psi_{\xi,\,\eta}^{F,\,\CC}(x,y)=0$, a contradiction. The implication (v)$\implies$(ii) is now established.

\smallskip

It remains to show the implication (ii)$\implies$(iii).

Assume that statement~(ii) holds. Fix $n\in\N$ and a Jordan curve $\CC\subseteq S^2$ with $\post f \subseteq \CC$ and $f^n(\CC)\subseteq \CC$. We denote $F \coloneqq f^n$, and $\Psi \coloneqq S_n^f \psi\in\Holder{\alpha}((S^2,d),\C)$ (see Lemma~\ref{lmSnPhiHolder}).

Let $\mathcal{O}_F=(S^2,\alpha_F)$ be the orbifold associated to $F$. By Proposition~\ref{propSameRamificationFn}, we have $\mathcal{O}_F=\mathcal{O}_f$. Let $\Theta\:\XX\rightarrow S^2_0$ be the universal orbifold covering map of $\mathcal{O}_F=\mathcal{O}_f$, which depends only on $f$, and in particular, is independent of $n$.

For each branched covering map $\wt{h}\in \Inv(F)$, i.e., $\wt{h}\: \XX\rightarrow\XX$ satisfying $F\circ \Theta \circ \wt{h} = \Theta$, we define a function $\wt{\beta}_{\wt{h}} \: \XX\rightarrow\C$ by
\begin{equation}  \label{eqDef_b_h}
\wt{\beta}_{\wt{h}} (u)   \coloneqq   \sum_{i=1}^{+\infty}  \bigl( (\Psi\circ\Theta)\bigl(\wt{h}^i (u)\bigr)  -  \Psi_{\wt{h}}   \bigr)
\end{equation}
for $u\in\XX$, where
\begin{equation}  \label{eqPfthmNLI_DefPhi_wt_h}
\Psi_{\wt{h}} \coloneqq \lim_{j\to+\infty} (\Psi\circ\Theta)\bigl(\wt{h}^j (u)\bigr)
\end{equation}
converges and the limit in (\ref{eqPfthmNLI_DefPhi_wt_h}) is independent of $u\in\XX$ by Proposition~\ref{propInvBranchFixPt}. 

Fix an arbitrary point $u\in\XX$. We will show that the series in (\ref{eqDef_b_h}) converges absolutely. Note that it follows immediately from (\ref{eqGoodPathLimitLength0}) in Corollary~\ref{corGoodPath} (applied to $x\coloneqq u$ and $y\coloneqq \wt{h}(u)$) that there exists an integer $k\in\N$ such that for each $i\in\N$,
\begin{align*}
&                        \AbsBig{  (\Psi\circ\Theta)\bigl(\wt{h}^i (u)\bigr) - \lim_{j\to+\infty} (\Psi\circ\Theta)\bigl(\wt{h}^{i+j} (u)\bigr)  }  \\
&\qquad \leq  \sum_{j=1}^{+\infty}  \AbsBig{  (\Psi\circ\Theta) \bigl(\wt{h}^{i+j-1} (u)\bigr) - (\Psi\circ\Theta) \bigl(\wt{h}^{i+j} (u)\bigr)  }  \\
&\qquad \leq \sum_{j=1}^{+\infty} k^\alpha C^\alpha \Lambda^{-(i+j-1)\alpha} \Hnorm{\alpha}{\Psi}{(S^2,d)}  \\
 &\qquad \leq \frac{ k C}{ 1-\Lambda^{-\alpha}}      \Lambda^{-i\alpha}  \Hnorm{\alpha}{\Psi}{(S^2,d)},
\end{align*} 
where $C\geq 1$ is the constant from Lemma~\ref{lmCellBoundsBM} depending only on $f$, $\CC$, and $d$. Thus, the series in (\ref{eqDef_b_h}) converges uniformly and absolutely, and $\wt{\beta}_{\wt{h}}$ is well-defined and continuous on $\XX$.

Hence, for arbitrary $l\in\N_0$ and $u\in\XX$, it follows from (\ref{eqGoodPathLimitLength0}) in Corollary~\ref{corGoodPath} (applied to $x\coloneqq u$ and $y\coloneqq \wt{h}(u)$) that
\begin{align*}
&                      \Absbigg{ \wt{\beta}_{\wt{h}} (u) - \sum_{i=1}^{+\infty}  \bigl( (\Psi\circ\Theta)\bigl(\wt{h}^i (u)\bigr)- (\Psi\circ\Theta)\bigl(\wt{h}^{i} \bigl(\wt{h}^{l}(u)\bigr)\bigr)  \bigr) } \\
&\qquad  \leq \sum_{i=1}^{+\infty} \AbsBig{ (\Psi\circ\Theta)\bigl(\wt{h}^{i+l}  (u) \bigr)   -   \lim_{j\to+\infty} (\Psi\circ\Theta)\bigl(\wt{h}^{i+j} (u)\bigr)  \bigr) } \\
&\qquad  \leq \sum_{i=1}^{+\infty}   \sum_{j=l+1}^{+\infty}    \AbsBig{  (\Psi\circ\Theta) \bigl(\wt{h}^{i+j-1} (u)\bigr) - (\Psi\circ\Theta) \bigl(\wt{h}^{i+j} (u)\bigr)  }  \\
&\qquad  \leq \sum_{i=1}^{+\infty}   \sum_{j=l+1}^{+\infty}    k^\alpha C^\alpha \Lambda^{-(i+j-1)\alpha} \Hnorm{\alpha}{\Psi}{(S^2,d)}   \\
&\qquad   \leq k C(1-\Lambda^{-\alpha})^{-2 }      \Lambda^{-l\alpha}  \Hnorm{\alpha}{\Psi}{(S^2,d)}.
\end{align*}
Hence, for each $u\in\XX$,
\begin{equation}  \label{eqSwitchLimitSum}
\wt{\beta}_{\wt{h}} (u)  = \lim_{j\to+\infty} \sum_{i=1}^{+\infty}  \bigl( (\Psi\circ\Theta)\bigl(\wt{h}^i (u)\bigr)- (\Psi\circ\Theta)\bigl(\wt{h}^{i} \bigl(\wt{h}^{j}(u)\bigr)\bigr)  \bigr).
\end{equation}

We now fix an inverse branch $\wt{g}\in\Inv(F)$ of $F$ on $\XX$ and consider the map $\wt{\beta}_{\wt{g}}$. Note that by the absolute convergence of the series in (\ref{eqDef_b_h}), for each $u\in \XX$,
\begin{equation}  \label{eqPreCohomology}
\wt{\beta}_{\wt{g}}(u) - \wt{\beta}_{\wt{g}}  ( \wt{g}(u)  ) = \Psi (\Theta  (\wt{g}(u)  )  ) - \Psi_{\wt{g}} .
\end{equation}

We claim that $\wt{\beta}_{\wt{g}}(u) = \wt{\beta}_{\wt{g}}(\wt{\sigma}(u))$ for each $u\in\XX$ and each deck transformation $\wt{\sigma}\in \pi_1(\mathcal{O}_f)$.

By Proposition~\ref{propDeckTransf} and the fact that $\mathcal{O}_f$ is either parabolic or hyperbolic (see \cite[Proposition~2.12]{BM17}), the claim is equivalent to 
\begin{equation}  \label{eqBeta_hInvFiberTheta}
\wt{\beta}_{\wt{g}}(u)= \wt{\beta}_{\wt{g}}(v), \qquad\qquad \text{for all }y\in S^2,  \, u, \, v\in\Theta^{-1}(y).
\end{equation}

We assume that the claim holds for now and postpone its proof to the end of this discussion. Then by (\ref{eqBeta_hInvFiberTheta}), the function $\beta \: S^2_0 \rightarrow \C$ defined by assigning, for each $y\in S^2_0$,
\begin{equation}  \label{eqDefBeta}
\beta(y) \coloneqq \wt{\beta}_{\wt{g}} (v)
\end{equation}
with $v\in \Theta^{-1}(y)$ is well-defined and independent of the choice of $v$.

For each $x\in S^2_0$, by the surjectivity of $\wt{g}$, we can choose $u_0\in\XX$ such that $\Theta (\wt{g}(u_0) ) = x$. Note that $\Theta(u_0)=( F\circ\Theta\circ\wt{g} ) (u_0) = F(x)$. So $\wt{g}(u_0) \in \Theta^{-1}(x)$ and $u_0\in\Theta^{-1}(F(x))$. Then by (\ref{eqPreCohomology}), 
\begin{equation*}
\beta(F(x)) - \beta(x)  
= \wt{\beta}_{\wt{g}} (u_0) - \wt{\beta}_{\wt{g}}  (\wt{g}(u_0) ) 
= \Psi (\Theta  (\wt{g}(u_0)  )  ) - \Psi_{\wt{g}}
= \Psi(x) -  \Psi_{\wt{g}}.
\end{equation*}

We will show that $\beta \in \Holder{\alpha} \bigl( \bigl( X^0_\c \setminus \post F, d \bigr), \C \bigr)$ for each $\c\in\{\b, \, \w\}$. Here $X^0_\b$ (resp.\ $X^0_\w$) is the black (resp.\ white) $0$-tile in $\X^0(F,\CC)$.

Fix arbitrary $\c\in\{\b, \, \w\}$ and $x_0, \, y_0 \in X^0_\c \setminus \post F$. Let $\gamma \: [0,1] \rightarrow X^0_\c \setminus \post F$ be an arbitrary continuous path with $\gamma(0) = x_0$, $\gamma(1) = y_0$, and $\gamma((0,1)) \subseteq \inte \bigl(X^0_\c \bigr)$.

By Lemma~\ref{lmLiftPathBM}, we can lift $\gamma$ to $\wt\gamma \: [0,1] \rightarrow \XX$ such that $\Theta \circ \wt\gamma = \gamma$. Denote $u \coloneqq \wt\gamma(0)$, $v \coloneqq \wt\gamma(1)$, and $I \coloneqq (0,1)$. Thus,
\begin{equation}   \label{eqPfthmNLI_InTileInte}
(\Theta \circ \wt\gamma) (I) \subseteq \inte \bigl( X^0_\c \bigr).
\end{equation}

Fix an arbitrary integer $m\in\N$.

By (\ref{eqPfthmNLI_InTileInte}), each connected component of $F^{-m}((\Theta \circ \wt\gamma)(I))$ is contained in some connected component of $F^{-m}\bigl(\inte \bigl(X^0_\c\bigr) \bigr)$. Since both $\bigl( \Theta \circ \wt{g}^m \circ \wt \gamma \bigr) (I)$ and $F^m \bigl( \bigl( \Theta \circ \wt{g}^m \circ \wt\gamma \bigr) (I) \bigr) = (\Theta \circ \wt\gamma)(I)$ are connected, by Proposition~\ref{propCellDecomp}~(v), there exists an $m$-tile $X^m \in \X^m(F,\CC)$ such that 
\begin{equation*}  
\bigl( \Theta \circ \wt{g}^m \circ \wt\gamma \bigr) (I)   \subseteq   \inte ( X^m ). 
\end{equation*}
We denote $x_m \coloneqq (\Theta \circ \wt{g}^m ) (u)$ and $y_m \coloneqq (\Theta \circ \wt{g}^m ) (v)$. Then $F^m(x_m) =  ( F^m \circ \Theta \circ \wt{g}^m ) (u) = x_0$, $F^m(y_m) =  ( F^m \circ \Theta \circ \wt{g}^m ) (v) = y_0$, and $x_m,  \, y_m \in X^m$. Hence, by (\ref{eqDefBeta}), the absolute convergence of the series defining $\wt{\beta}_{\wt{g}}$ in (\ref{eqDef_b_h}), and Lemma~\ref{lmSnPhiBound},
\begin{align*}
            \abs{ \beta(x_0) - \beta(y_0)} 
&=       \AbsBig{ \wt\beta_{\wt{g}}(u) -\wt\beta_{\wt{g}}(v) }  \\
&=       \lim_{m\to+\infty}  \Absbigg{  \sum_{i=1}^m \bigl( \bigl( \Psi \circ \Theta \circ \wt{g}^i \bigr) (u) - \bigl( \Psi \circ \Theta \circ \wt{g}^i \bigr) (v)  \bigr) }  \\
&=       \lim_{m\to+\infty}  \Absbigg{  \sum_{j=0}^{m-1} \bigl( \bigl( \Psi \circ F^j \circ \Theta \circ \wt{g}^m \bigr) (u) - \bigl( \Psi \circ F^j \circ \Theta \circ \wt{g}^m \bigr) (v)  \bigr) }  \\
&=       \lim_{m\to+\infty} \Absbig{ S^F_m \Psi (x_m) - S^F_m \Psi (y_m) }   \\
&\leq  \limsup_{m\to+\infty} C_1 d ( F^m (x_m), F^m (y_m) )^\alpha  \\
&= C_1 d(x_0, y_0)^\alpha,
\end{align*}
where $C_1 = C_1(F,\CC,d,\Psi,\alpha)>0$ is the constant depending only on $F$, $\CC$, $d$, $\Psi$, and $\alpha$ from Lemma~\ref{lmSnPhiBound}.

Hence, $\beta \in \Holder{\alpha} \bigl( \bigl( X^0_\c \setminus \post F, d \bigr), \C \bigr)$.

We can now extend $\beta$ continuously to $X^0_\c$, denoted by $\beta_\c$. Since $\c\in\{\b, \, \w\}$ is arbitrary, $\bigl( X^0_\b \setminus \post F \bigr) \cap \bigl( X^0_\b \setminus \post F \bigr) = \CC \setminus \post F$, and $\post F$ is a finite set, we get $\beta_\b|_\CC = \beta_\w|_\CC$. Thus, $\beta$ can be extended to $S^2$, and this shows that $\Psi=S_n^f\psi$ is cohomologous to a constant $K_1$ in $\CCC(S^2,\C)$ (see Definition~\ref{defCohomologous}). Therefore, by Lemma~5.53 in \cite{Li17}, $\psi = K - \beta + \beta\circ f$ for some constant $K\in\C$, establishing statement~(iii).

Thus, it suffices to prove the claim. The verification of the claim occupies the remaining part of this proof.

We denote $\wt\Psi \coloneqq \Psi \circ \Theta$.

Let $\wt\sigma \in \pi_1(\mathcal{O}_f)$ be a deck transformation on $X$. Then by Definition~\ref{defInverseBranch} and Definition~\ref{defDeckTransf} it is clear that $\wt\sigma \circ \wt{g} \in \Inv(F)$. Thus, by the absolute convergence of the series defining $\wt{\beta}_{\wt{g}}$ in (\ref{eqDef_b_h}), for each $u\in\XX$,
\begin{align}    \label{eqPfthmNLI_PfClaim1}
  \wt{\beta}_{\wt{g}} ( (\wt\sigma \circ \wt{g})(u) )  - \wt{\beta}_{\wt{g}} (  \wt{g}(u) ) 
&  =  \bigl(\wt{\beta}_{\wt{g}} ( (\wt\sigma \circ \wt{g})(u) ) - \wt{\beta}_{\wt{g}}  (u) \bigr)
    -\bigl(\wt{\beta}_{\wt{g}} ( \wt{g}(u) ) - \wt{\beta}_{\wt{g}}  (u) \bigr) \notag \\
&  =  \sum_{i=1}^{+\infty} \bigl( \wt\Psi \bigl(\wt{g}^i ((\wt\sigma \circ \wt{g}) (u) ) \bigr)   -  \wt\Psi \bigl(\wt{g}^i (u) \bigr) \bigr)
                                                                 +   \wt\Psi(\wt{g}(u)) - \Psi_{\wt{g}}  .  
\end{align}
Fix arbitrary $i_0, \, j_0\in\N$. It follows immediately from (\ref{eqGoodPathLimitLength0}) in Corollary~\ref{corGoodPath} that all three series
\begin{align*}
&   \sum_{i=1}^{+\infty} \Absbig{ \wt\Psi \bigl(\wt{g}^i ((\wt\sigma \circ \wt{g}) (u) ) \bigr) 
                                           -\wt\Psi \bigl( \wt{g}^{i} \bigl( (\wt\sigma \circ \wt{g})^{j_0} (u) \bigr) \bigr)   },   \\
&  \sum_{i=1}^{+\infty}  \Absbig{ \wt\Psi \bigl(\wt{g}^i (u) \bigr) 
                                          -\wt\Psi \bigl( \wt{g}^{i} \bigl( (\wt\sigma \circ \wt{g})^{j_0} (u) \bigr) \bigr)   } ,   \text{ and}\\                                       
&   \sum_{j=1}^{+\infty} \Absbig{ \wt\Psi \bigl(\wt{g}^{i_0} ((\wt\sigma \circ \wt{g})^j (u) ) \bigr) 
                                           -\wt\Psi \bigl( \wt{g}^{i_0} \bigl( (\wt\sigma \circ \wt{g})^{j+1} (u) \bigr) \bigr)   }
\end{align*}
are majorized by convergent geometric series. So the right-hand side of (\ref{eqPfthmNLI_PfClaim1}) is equal to
\begin{align*}
&    \sum_{i=1}^{+\infty} \bigl(\wt\Psi \bigl(\wt{g}^i ((\wt\sigma \circ \wt{g}) (u) ) \bigr) 
                                           -\wt\Psi \bigl( \wt{g}^{i} \bigl( (\wt\sigma \circ \wt{g})^{j_0} (u) \bigr) \bigr)   \bigr)   \\
&\quad  - \sum_{i=1}^{+\infty} \bigl(\wt\Psi \bigl(\wt{g}^i (u) \bigr) 
                                          -\wt\Psi \bigl( \wt{g}^{i} \bigl( (\wt\sigma \circ \wt{g})^{j_0} (u) \bigr) \bigr)   \bigr)
   +  \wt\Psi(\wt{g}(u)) -  \Psi_{\wt{g}}  \\
&\qquad =   \lim_{j\to+\infty} \sum_{i=1}^{+\infty} \bigl(\wt\Psi \bigl(\wt{g}^i ((\wt\sigma \circ \wt{g}) (u) ) \bigr) 
                                           -\wt\Psi \bigl( \wt{g}^{i} \bigl( (\wt\sigma \circ \wt{g})^j (u) \bigr) \bigr)   \bigr)   \\
&\qquad\quad      - \lim_{j\to+\infty} \sum_{i=1}^{+\infty} \bigl(\wt\Psi \bigl(\wt{g}^i (u) \bigr) 
                                          -\wt\Psi \bigl( \wt{g}^{i} \bigl( (\wt\sigma \circ \wt{g})^j (u) \bigr) \bigr)   \bigr)
   +  \wt\Psi(\wt{g}(u)) -  \Psi_{\wt{g}}  .
\end{align*}
Then by Lemma~\ref{lmLiftLI}, (\ref{eqSwitchLimitSum}), and (\ref{eqPreCohomology}), the right-hand side of the above equation is equal to
\begin{align*}
 &     \lim_{j\to+\infty} \sum_{i=1}^{+\infty} \bigl(\wt\Psi \bigl((\wt\sigma\circ\wt{g})^i ((\wt\sigma \circ \wt{g}) (u) ) \bigr) 
                                           -\wt\Psi \bigl( (\wt\sigma\circ\wt{g})^{i} \bigl( (\wt\sigma \circ \wt{g})^j (u) \bigr) \bigr)   \bigr)   \\
&\quad     - \lim_{j\to+\infty} \sum_{i=1}^{+\infty} \bigl(\wt\Psi \bigl((\wt\sigma\circ\wt{g})^i (u) \bigr)  
                                          -\wt\Psi \bigl( (\wt\sigma\circ\wt{g})^{i} \bigl( (\wt\sigma \circ \wt{g})^j (u) \bigr) \bigr)   \bigr) 
                           +   \wt\Psi(\wt{g}(u)) - \Psi_{\wt{g}} \\
&\qquad=    \wt{\beta}_{ \wt{\sigma} \circsmall \wt{g} }  ( ( \wt{\sigma} \circ \wt{g}  ) (u) )  -   \wt{\beta}_{ \wt{\sigma} \circsmall \wt{g} }  ( u ) +  \wt\Psi(\wt{g}(u)) - \Psi_{\wt{g}}  \\
&\qquad=   - \wt\Psi((\wt\sigma \circ \wt{g})(u)) + \Psi_{ \wt{\sigma} \circsmall \wt{g} }    +  \wt\Psi(\wt{g}(u)) -  \Psi_{\wt{g}}  \\
&\qquad=    \Psi_{ \wt{\sigma} \circsmall \wt{g} }     -  \Psi_{\wt{g}} .
\end{align*}
The last equality follows from $\wt{\Psi}\circ\wt{\sigma}= \Psi\circ(\Theta\circ\wt{\sigma}) = \Psi\circ\Theta = \wt{\Psi}$. Since $\wt{g} \: \XX \rightarrow \XX$ is surjective, we can conclude that for each $v\in\XX$ and each $\wt{\sigma} \in \pi_1( \mathcal{O}_f )$,
\begin{equation}  \label{eqPfthmNLI_PfClaim2}
\wt{\beta}_{\wt{g}} ( \wt{\sigma} (v) ) - \wt{\beta}_{\wt{g}} ( v ) =  \Psi_{ \wt{\sigma} \circsmall \wt{g} }     -  \Psi_{\wt{g}} .
\end{equation}

The claim follows once we show that $ \Psi_{ \wt{\sigma} \circsmall \wt{g} }    =  \Psi_{\wt{g}}$ for each $\wt{\sigma} \in \pi_1( \mathcal{O}_f )$. We argue by contradiction and assume that $\Psi_{ \wt{\sigma} \circsmall \wt{g} }    -  \Psi_{\wt{g}} \neq 0$ for some $\wt{\sigma} \in \pi_1( \mathcal{O}_f )$. Then by (\ref{eqPfthmNLI_PfClaim2}), for each $k\in \N$,
\begin{equation*}
        \Psi_{ \wt{\sigma}^k \circsmall \wt{g} }    -  \Psi_{\wt{g}}  
    = \wt{\beta}_{\wt{g}} \bigl( \wt{\sigma}^k (v) \bigr) - \wt{\beta}_{\wt{g}} ( v ) 
    = \sum_{i=0}^{k-1}  \Bigl(   \wt{\beta}_{\wt{g}} \bigl( \wt{\sigma} \bigl( \bigl( \wt{\sigma}^i (v) \bigr) \bigr) \bigr) -  \wt{\beta}_{\wt{g}} \bigl( \wt{\sigma}^i (v) \bigr)   \Bigr)
    = k \bigl( \Psi_{ \wt{\sigma} \circsmall \wt{g} }    -  \Psi_{\wt{g}}   \bigr).
\end{equation*}
However, by (\ref{eqPfthmNLI_DefPhi_wt_h}), Proposition~\ref{propInvBranchFixPt}, and Theorem~\ref{thmETMBasicProperties}~(ii), 
\begin{equation*}
\card \bigl\{  \Psi_{ \wt{\sigma}^k \circsmall \wt{g} }  : k\in\N \bigr\} 
\leq  \card \bigl\{  \Psi_{ \wt{h} }  : \wt{h}\in \Inv(F) \bigr\}  
\leq  \card \{ \Psi(x) : x\in S^2, \, F(x)=x \}
< +\infty. 
\end{equation*}
This is a contradiction.

The claim is now proved, establishing the implication  (ii)$\implies$(iii).
\end{proof}

\section{Proofs of Theorems~\ref{thmZetaAnalExt_SFT},~\ref{thmZetaAnalExt_InvC}, and~\ref{thmPrimeOrbitTheorem}}  \label{sctDynOnC_Reduction}

Recall that $s_0$ is defined in the Assumptions in Section~\ref{sctAssumptions}.

\begin{proof}[Proof of Theorem~\ref{thmZetaAnalExt_SFT}]
By Proposition~\ref{propZetaFnConv_s0}, we can assume that $\phi$ is not cohomologous to a constant in $\CCC( S^2 )$.
	
In this proof, for $s\in\C$ and $r\in\R$, $B(s,r)$ denotes the open Euclidean ball in $\C$ centered at $s$ with radius $r>0$. For an arbitrary number $t\in\R$, by Proposition~\ref{propZetaFnConv_s0}~(i), we have 
\begin{equation}  \label{eqPfthmZetaAnalExt_SFT_Unique0}
P (\sigma_{A_{\ti}}, - t \phi \circ \pi_{\ti} ) = 0 \quad \text{ if and only if } \quad   t = s_0.
\end{equation}

Fix an arbitrary number $\theta \in (0,1)$. By \cite[Theorem~4.5, Propositions~4.6, 4.7, and 4.8]{PP90} and the discussion preceding them in \cite{PP90}, the exponential of the topological pressure $\exp  (P( \sigma_{A_{\ti}}, \cdot ))$ as a function on $\Holder{1} \bigl( \Sigma_{A_{\ti}}^+, d_\theta \bigr)$ can be extended to a new function (still denoted by $\exp( P( \sigma_{A_{\ti}}, \cdot))$) with the following properties:
\begin{enumerate}
\smallskip
\item[(1)] The domain $\operatorname{dom} (\exp ( P ( \sigma_{A_{\ti}}, \cdot )))$ of $\exp (P ( \sigma_{A_{\ti}}, \cdot ))$ is an nonempty open subset of $\Holder{1} \bigl( \bigl( \Sigma_{A_{\ti}}^+, d_\theta \bigr), \C\bigr)$.

\smallskip
\item[(2)] The function $s \mapsto \exp( P  ( \sigma_{A_{\ti}}, - s \phi \circ \pi_{\ti}  ))$ is a holomorphic map in an open neighborhood $U \subseteq \C$ of $s$ if  $- s \phi \circ \pi_{\ti}  \in \operatorname{dom} (\exp (P ( \sigma_{A_{\ti}}, \cdot )))$ .

\smallskip
\item[(3)] If $\psi \in \operatorname{dom} ( \exp (P ( \sigma_{A_{\ti}}, \cdot )))$ and $\eta \coloneqq \psi + c + 2\pi\I M + u - u\circ \sigma_{A_{\ti}}$ for some $c\in\C$, $M \in \CCC \bigl( \Sigma_{A_{\ti}}^+, \Z \bigr)$, and $u\in \Holder{1} \bigl( \bigl( \Sigma_{A_{\ti}}^+, d_\theta \bigr), \C\bigr)$, then $\eta \in \operatorname{dom} (\exp ( P ( \sigma_{A_{\ti}}, \cdot )))$ and $\exp (P(  \sigma_{A_{\ti}}, \eta)) = e^c \exp (P(  \sigma_{A_{\ti}}, \psi) )$.
\end{enumerate}

\smallskip

We first show that $s_0$ is not an accumulation point of zeros of the function $s\mapsto 1-\exp ( P ( \sigma_{A_{\ti}}, -s \phi \circ \pi_{\ti} ))$. We argue by contradiction and assume otherwise. Then by Property~(2) above, $\exp ( P ( \sigma_{A_{\ti}}, -s \phi \circ \pi_{\ti} )) = 1$ for all $s$ in a neighborhood of $s_0$. This contradicts with (\ref{eqPfthmZetaAnalExt_SFT_Unique0}).

Thus, by \cite[Theorems~5.5~(ii) and~5.6~(b), (c)]{PP90}, we can choose $\vartheta_0 >0$ small enough such that  $\zeta_{\sigma_{A_{\ti}}, \, \minus \phi \circsmall \pi_{\ti} } (s)$ has a non-vanishing holomorphic extension 
\begin{equation}  \label{eqPfthmZetaAnalExt_SFT_Extension}
\zeta_{\sigma_{A_{\ti}}, \, \minus \phi \circsmall \pi_{\ti} } (s) 
=   \frac{ \exp \Bigl( \sum_{n=1}^{+\infty}  \frac1n   \sum_{ \underline{x}  \in P_{1,\sigma_{A_{\ti}}^n}} E_{s,n,\underline{x}}   \Bigr)  }        
             { 1 - \exp ( P( \sigma_{A_{\ti}}, - s \phi \circ \pi_{\ti} ) )}  ,
\end{equation}
where $E_{s,n,\underline{x}}$ is given by
\begin{equation*} 
E_{s,n,\underline{x}} \coloneqq \exp \bigl( -s S_n^{\sigma_{A_{\ti}}} (\phi \circ \pi_{\ti} )( \underline{x} )\bigr)
- \exp ( n P( \sigma_{A_{\ti}}, - s \phi \circ \pi_{\ti} ) ) ,
\end{equation*}
to $B(s_0, \vartheta_0 ) \setminus \{s_0\}$, and $\zeta_{\sigma_{A_{\ti}}, \, \minus \phi \circsmall \pi_{\ti} } (s)$ has a pole at $s=s_0$. Moreover, the numerator on the right-hand side of (\ref{eqPfthmZetaAnalExt_SFT_Extension}) is a non-vanishing holomorphic function on $B(s_0, \vartheta_0 )$.

Next, we show that $\zeta_{\sigma_{A_{\ti}}, \, \minus \phi \circsmall \pi_{\ti} } (s)$ has a simple pole at $s=s_0$. It suffices to show that $1 - \exp ( P( \sigma_{A_{\ti}}, -s \phi \circ \pi_{\ti} ) )$ has a simple zero at $s=s_0$. Indeed, since $\phi$ is eventually positive, we fix $m\in\N$ such that $S_m^f \phi $ is strictly positive on $S^2$ (see Definition~\ref{defEventuallyPositive}). By Proposition~\ref{propZetaFnConv_s0}~(i), Theorem~\ref{thmEquilibriumState}~(ii), and the fact that the equilibrium state $\mu_{\minus t\phi}$ for $f$ and $-t\phi$ is an $f$-invariant probability measure (see Theorem~\ref{thmEquilibriumState}~(i) and Subsection~\ref{subsctThermodynFormalism}), we have for $t\in \R$,
\begin{align}   \label{eqPfthmZetaAnalExt_SFT_Simple0}
        \frac{\mathrm{d}}{\mathrm{d}t} ( 1 - \exp ( P( \sigma_{A_{\ti}}, -t \phi \circ \pi_{\ti} ) )   ) 
&=   \frac{\mathrm{d}}{\mathrm{d}t} \bigl( 1 - e^{ P(f, -t \phi  )  } \bigr)   \notag  \\
&=     - e^{ P(f, -t \phi  )  }   \frac{\mathrm{d}}{\mathrm{d}t} P(f,- t\phi)  \notag \\
&= e^{ P(f, -t \phi  )  } \int \!\phi \,\mathrm{d}\mu_{\minus t\phi} \\
&=  \frac{e^{ P(f, -t \phi  )  }}{m} \int \! S_m^f \phi \,\mathrm{d}\mu_{\minus t\phi} \notag \\
&> 0.\notag 
\end{align}  
Hence, by (\ref{eqPfthmZetaAnalExt_SFT_Simple0}) and Property~(2) above, we get that $\zeta_{\sigma_{A_{\ti}}, \, \minus \phi \circsmall \pi_{\ti} } (s)$ has a simple pole at $s=s_0$.

We now show that for each $b\in \R \setminus \{ 0 \}$, there exists $\vartheta_b>0$ such that  $\zeta_{\sigma_{A_{\ti}}, \, \minus \phi \circsmall \pi_{\ti} } (s)$ has a non-vanishing holomorphic extension to $B (s_0+ \I b, \vartheta_b)$.

By \cite[Theorems~5.5~(ii) and~5.6]{PP90}, and the fact that $\operatorname{dom} ( \exp ( P ( \sigma_{A_{\ti}}, \cdot )))$ is open and $\exp ( P( \sigma_{A_{\ti}}, \cdot ))$ is continuous on $\operatorname{dom} (\exp (P ( \sigma_{A_{\ti}}, \cdot )))$ (see Properties~(2) and (3) above), we get that for each $b\in \R \setminus \{ 0 \}$, we can always choose  $\vartheta_b>0$ such that $\zeta_{\sigma_{A_{\ti}}, \, \minus \phi \circsmall \pi_{\ti} } (s)$ has a non-vanishing holomorphic extension to $B (s_0+ \I b, \vartheta_b)$ unless the following two conditions are both satisfied:
\begin{enumerate}
\smallskip
\item[(i)] $- (s_0 + \I b) \phi \circ \pi_{\ti} = - s_0 \phi \circ \pi_{\ti} + \I c + 2\pi\I M + u - u\circ \sigma_{A_{\ti}} \in \operatorname{dom} ( \exp ( P ( \sigma_{A_{\ti}}, \cdot )))$ for some $c\in \C$, $M \in \CCC \bigl(\Sigma_{A_{\ti}}^+, \Z \bigr)$ and  $u\in \Holder{1} \bigl( \bigl( \Sigma_{A_{\ti}}^+, d_\theta \bigr), \C \bigr)$.

\smallskip
\item[(ii)] $1- \exp ( P( \sigma_{A_{\ti}}, - (s_0 + \I b) \phi \circ \pi_{\ti} ) ) = 0$.
\end{enumerate}

We will show that conditions~(i) and (ii) cannot be both satisfied. We argue by contradiction and assume that conditions~(i) and (ii) are both satisfied. Then by Property~(3) above, $c \equiv 0 \pmod{2\pi}$. Thus, by taking the imaginary part of both sides of the identity in condition~(i), we get that $\phi\circ \pi_{\ti} = K M + \tau - \tau \circ \sigma_{A_{\ti}}$ for some $K\in\R$, $M \in \CCC \bigl( \Sigma_{A_{\ti}}^+ , \Z \bigr)$, and $\tau \in \Holder{1} \bigl( \bigl( \Sigma_{A_{\ti}}^+, d_\theta \bigr), \C \bigr)$. Then by Theorem~\ref{thmNLI}, $\phi$ is cohomologous to a constant in $\CCC( S^2)$, a contradiction, establishing Theorem~\ref{thmZetaAnalExt_SFT}.
\end{proof}

\begin{prop}  \label{propZetaProduct}
	Let $f$, $\CC$, $d$, $\phi$, $s_0$ satisfy the Assumptions in Section~\ref{sctAssumptions}. Assume $f(\CC) \subseteq \CC$.	
	Then on $\H_{s_0} =\{s\in\C : \Re(s) > s_0 \}$, $\DS_{F,\,\minus \Phi,\,\deg_F} (s)$ converges and satisfies
	\begin{equation}    \label{eqDirichletSeriesIsCombZetaFns}
		\DS_{f,\,\minus\phi,\,\deg_f} (s) 
		=  \zeta_{ \sigma_{A_{\ti}}, \, \minus \phi \circsmall \pi_{\ti} } (s) 
		\zeta_{ \sigma_{A_{\e}}, \, \minus \phi \circsmall \pi_{\e} } (s)   \zeta_{ f|_{\V^0}, \, \minus \phi|_{\V^0} } (s)   
		/   \zeta_{ \sigma_{A_{\ee}}, \, \minus \phi \circsmall \pi_{\e} \circsmall \pi_{\ee} } (s) .
	\end{equation}
	If, in addition, $\phi$ is not cohomologous to a constant in $\CCC( S^2 )$ and no $1$-tile in $\X^1(f,\CC)$ joins opposite sides of $\CC$, then $\DS_{f,\,\minus\phi,\,\deg_f} (s)$ extends to non-vanishing holomorphic function on $\overline{\H}_{s_0} =\{s\in\C : \Re(s) \geq s_0\}$ except for the simple pole at $s=s_0$.
\end{prop}

\begin{proof}
	%We first observe that by the continuity of the topological pressure (see for example, \cite[Theorem~3.6.1]{PU10}) and Theorem~\ref{thmPressureOnC}, there exists a real number $\epsilon'_0 \in (0, \min \{\wt{\epsilon}_0, \, s_0 \} )$ such that $P (f|_{\V^0}, -(s_0 - \epsilon'_0) \phi|_{\V^0} ) < 0 $ and
	
	We first observe that by the continuity of the topological pressure (see for example, \cite[Theorem~3.6.1]{PU10}) and Theorem~\ref{thmPressureOnC}, there exists a real number $\epsilon'_0 \in (0,  s_0   )$ such that $P (f|_{\V^0}, -(s_0 - \epsilon'_0) \phi|_{\V^0} ) < 0 $ and
	\begin{equation*}
		P  (   \sigma_{A_{\ee}}, -(s_0 - \epsilon'_0) \varphi \circ \pi_{\e} \circ \pi_{\ee}  )
		=   P  (   \sigma_{A_{\e}},  -(s_0 - \epsilon'_0) \varphi \circ \pi_{\e}                  )
		<   0.
	\end{equation*}
	%Here $\wt{\epsilon}_0 > 0$ is the constant from Theorem~\ref{thmZetaAnalExt_SFT} depending only on $f$, $\CC$, $d$, and $\phi$.
	
	By Lemma~\ref{lmDynDirichletSeriesConv_general}, Remark~\ref{rmDynDirichletSeriesZetaFn}, Proposition~\ref{propSFT}~(ii), and the fact that $\phi$ is eventually positive, each of the zeta functions $\zeta_{ f|_{\V^0}, \, \minus \phi|_{\V^0} }$, $\zeta_{ \sigma_{A_{\e}}, \, \minus \phi \circsmall \pi_{\e} }$, and $\zeta_{ \sigma_{A_{\ee}}, \, \minus \phi \circsmall \pi_{\e} \circsmall \pi_{\ee} }$ converges uniformly and absolutely to a non-vanishing bounded holomorphic function on the closed half-plane $\overline{\H}_{s_0 - \epsilon'_0} = \{ s\in\C: \Re(s) \geq s_0 - \epsilon'_0 \}$.
	
	On the other hand, for each $n\in\N$, we have $P_{1, (f|_{\V^0})^n} \subseteq P_{1, f^n}$, and by Proposition~\ref{propSFTs_C},
	\begin{equation*}
		(\pi_{\e} \circ \pi_{\ee} )  \bigl(   P_{1, \sigma_{A_{\ee}}^n } \bigr)
		\subseteq     \pi_{\e}                    \bigl(  P_{1, \sigma_{A_{\e }}^n } \bigr)
		\subseteq                                         P_{1, (f|_\CC)^n }
		\subseteq                                         P_{1, f^n }.
	\end{equation*} 
	Thus, by (\ref{eqDefDynDirichletSeries}), (\ref{eqDefZetaFn}), and Theorem~\ref{thmNoPeriodPtsIdentity}, we get that for each $s\in \H_{s_0}$, (\ref{eqDirichletSeriesIsCombZetaFns}) holds.	
	
	The proposition now follows from Theorem~\ref{thmZetaAnalExt_SFT}.	
\end{proof}

\begin{proof}[Proof of Theorem~\ref{thmZetaAnalExt_InvC}]
We choose $N_f\in\N$ as in Remark~\ref{rmNf}. Note that $P \bigl( f^i, - s_0 S_i^f \phi \bigr) = i P(f, - s_0 \phi) = 0$ for each $i\in\N$ (see for example, \cite[Theorem~9.8]{Wal82}). We observe that by Lemma~\ref{lmCexistsL}, it suffices to prove the case $n=N_f = 1$. In this case, $F=f$, $\Phi=\phi$, and there exists a Jordan curve $\CC\subseteq S^2$ satisfying $f(\CC)\subseteq \CC$, $\post f\subseteq \CC$, and no $1$-tile in $\X^1(f,\CC)$ joins opposite sides of $\CC$.

In this proof, we write $l_\phi(\tau) \coloneqq \sum_{y\in\tau} \phi(y)$ and $\deg_f(\tau) \coloneqq \prod_{y\in\tau} \deg_f(y)$ for each primitive periodic orbit $\tau\in\Orb(f)$ and each $y\in\tau$. 

We first note that the conclusion on $\DS_{f,\,\minus\phi,\,\deg_f}$ follows from Proposition~\ref{propZetaProduct}.

Next, we observe that by (\ref{eqZetaFnOrbitForm_ThurstonMap}) and  (\ref{eqZetaFnOrbitForm_ThurstonMapDegree}) in Proposition~\ref{propZetaFnConv_s0},
\begin{equation}   \label{eqPfthmZetaAnalExt_InvC_zeta}
   \zeta_{f,\,\minus\phi} (s) \prod_{\tau\in\Orb^>(f|_{\V^0})}  \Bigl(  1- e^{ - s l_\phi(\tau)  } \Bigr)   
=   \DS_{f,\,\minus\phi,\,\deg_f} (s)  \prod_{\tau\in\Orb^>(f|_{\V^0})}  \Bigl(  1- \deg_f(\tau) e^{ - s l_\phi(\tau)  } \Bigr)                  
\end{equation}
for all $s\in\C$ with $\Re(s)> s_0$, where 
\begin{equation}  \label{eqPfthmZetaAnalExt_InvC_Orb>}
\Orb^>(f|_{\V^0}) \coloneqq \bigl\{\tau\in \Orb(f|_{\V^0}) :  \deg_f(\tau) >1 \bigr\}
\end{equation}
is a finite set since $\V^0=\post f$ is a finite set.

We denote, for each $\tau\in\Orb^>(f|_{\V^0})$,
\begin{equation*}
\beta_\tau \coloneqq \deg_f(\tau) e^{-s_0 l_\phi(\tau)}.
\end{equation*}

Fix an arbitrary $\tau \in \Orb^>(f|_{\V^0})$. We show now that $1-\beta_\tau \geq 0$. We argue by contradiction and assume that $\beta_\tau >1$. Let $k \coloneqq \card \tau$, and fix an arbitrary $y\in\tau$. Then $y\in P_{1,f^{km}}$ for each $m\in\N$. Thus, by Proposition~\ref{propTopPressureDefPeriodicPts},
\begin{align*}
0 =     P(f,-s_0\phi)
 &\geq  \lim_{m\to+\infty} \frac{1}{km} \log \bigl(  \deg_{f^{km}} (y)  \exp \bigl(-s_0 S^f_{km} \phi (y) \bigr) \bigr)   \\
 &=     \lim_{m\to+\infty} \frac{1}{km} \log \bigl( \beta_\tau^m \bigr)
  =     \frac{\log \beta_\tau }{k}
  >     0.
\end{align*}
This is a contradiction, proving $1-\beta_\tau \geq 0$ for each $\tau\in\Orb^>(f|_{\V^0})$.

By Proposition~\ref{propZetaFnConv_s0}, we can now assume that $\phi$ is not cohomologous to a constant in $\CCC( S^2 )$.

\smallskip

\emph{Claim.} We have $1-\beta_\tau > 0$ for each $\tau\in\Orb^>(f|_{\V^0})$.

\smallskip

We argue by contradiction and assume that there exists $\eta \in \Orb^>(f|_{\V^0})$ with $1-\beta_\eta = 0$. We define a function $w\: S^2\rightarrow \C$ by
\begin{equation*}
w(x) \coloneqq \begin{cases} \deg_f(x) & \text{if } x\in S^2 \setminus \eta, \\ 0  & \text{otherwise}. \end{cases}
\end{equation*}

Fix an arbitrary real number $a>s_0$. By (\ref{eqLocalDegreeProduct}), Proposition~\ref{propTopPressureDefPeriodicPts}, and Corollary~\ref{corS0unique}, for each $n\in\N$,
\begin{align*}
&             \limsup_{n\to+\infty} \frac{1}{n} \log \sum_{y\in P_{1,f^n}} \exp(-a S_n\phi(y)) \prod_{i=0}^{n-1} w\bigl(f^i(y)\bigr) \\
&\qquad  \leq \limsup_{n\to+\infty} \frac{1}{n} \log \sum_{y\in P_{1,f^n}} \deg_{f^n}(y) \exp(-a S_n\phi(y))    
  = P(f, -a\phi) <0.
\end{align*}
Hence, by Lemma~\ref{lmDynDirichletSeriesConv_general} and Theorem~\ref{thmETMBasicProperties}~(ii), $\DS_{f,\,\minus\phi,\,w}(s)$ converges uniformly and absolutely on the closed half-plane $\overline{\H}_a$, and 
\begin{equation}  \label{eqPfthmZetaAnalExt_InvC_DynDSeriesIdenw}
	\begin{aligned}
		\DS_{f,\,\minus\phi,\,w}(s)  
		&= \prod_{\tau\in\Orb(f)\setminus\{\eta\}}  \Bigl(1 - \deg_f(\tau) e^{-s l_\phi(\tau)   } \Bigr)^{-1}  \\
		&= \DS_{f,\,\minus\phi,\,\deg_f}(s) \Bigl(1 - \deg_f(\eta) e^{-s l_\phi(\eta)  } \Bigr)  
	\end{aligned}
\end{equation}
for $s\in \overline{\H}_a$. Note that by our assumption that $1-\beta_\tau = 0$, we know that $1 - \deg_f(\eta) e^{-s l_\phi(\eta)  }$ is an entire function with simple zeros at $s=s_0 +\I j h_0 $, $j\in\Z$, where $h_0 \coloneqq \frac{2\pi}{l_\phi(\eta)}$. Note that $l_\phi(\eta) > 0$ since $\phi$ is eventually positive (see Definition~\ref{defEventuallyPositive}). Since $\DS_{f,\,\minus\phi,\,\deg_f}$ has a non-vanishing holomorphic extension to $\overline{\H}_{s_0}$ except a simple pole at $s=s_0$, we get from (\ref{eqPfthmZetaAnalExt_InvC_DynDSeriesIdenw}) that $\DS_{f,\,\minus\phi,\,w}$ has a holomorphic extension to $\overline{\H}_{s_0}$ with $\DS_{f,\,\minus\phi,\,w}(s_0) \neq 0$ and $\DS_{f,\,\minus\phi,\,w}(s_0 + \I jh_0) = 0$ for each $j\in\Z \setminus \{0\}$.

On the other hand, for each $s\in \overline{\H}_a$,
\begin{equation} \label{eqPfthmZetaAnalExt_InvC_DynDSeriesIdenwBound}
      \AbsBigg{ \sum_{n=1}^{+\infty}  \frac{1}{n} \sum_{x\in P_{1,f^n}} e^{-s S_n\phi(x)} \prod_{i=0}^{n-1} w\bigl(f^i(x)\bigr) }
\leq  \sum_{n=1}^{+\infty}  \frac{1}{n} \sum_{x\in P_{1,f^n}} e^{-\Re(s) S_n\phi(x)} \prod_{i=0}^{n-1} w\bigl(f^i(x)\bigr).
\end{equation}
Since $a>s_0$ is arbitrary, it follows from (\ref{eqDefDynDirichletSeries}), (\ref{eqPfthmZetaAnalExt_InvC_DynDSeriesIdenwBound}), and $\DS_{f,\,\minus\phi,\,w}(s_0) \neq 0$ that
\begin{equation*}
\limsup_{a\to s_0^+}   \AbsBigg{ \sum_{n=1}^{+\infty}  \frac{1}{n} \sum_{x\in P_{1,f^n}} e^{- (a + \I b) S_n\phi(x)} \prod_{i=0}^{n-1} w\bigl(f^i(x)\bigr) }  < +\infty
\end{equation*}
for each $b\in \R$. By (\ref{eqDefDynDirichletSeries}), this is a contradiction to the fact that $\DS_{f,\,\minus\phi,\,w}$ has a holomorphic extension to $\overline{\H}_{s_0}$ with $\DS_{f,\,\minus\phi,\,w}(s_0 + \I jh_0) = 0$ for each $j\in\Z \setminus \{0\}$. Claim is now established.

\smallskip

Hence, $\prod_{\tau \in \Orb^>(f|_{\V^0})}  \frac{  1 - \deg_f(\tau) \exp(-s l_\phi(\tau))  }{  1 -  \exp(-s l_\phi(\tau))   }$ is uniformly bounded away from $0$ and $+\infty$ on the closed half-plane $\overline{\H}_{s_0-\epsilon_0}$ for some $\epsilon_0 \in ( 0, \epsilon'_0 )$.

Therefore, Theorem~\ref{thmZetaAnalExt_InvC} follows from Proposition~\ref{propZetaProduct} and (\ref{eqPfthmZetaAnalExt_InvC_zeta}).
\end{proof}

\begin{prop}  \label{propOrbitLength}
	Let $f$, $\CC$, $d$, $\phi$, $s_0$ satisfy the Assumptions in Section~\ref{sctAssumptions}. Assume $f(\CC) \subseteq \CC$ and $\phi \equiv 1$. Then $s_0 = \log (\deg f)$ and
	\begin{equation}   \label{eqOrbitLengthAsymp}
		\pi_{f,\phi}(T) \sim \frac{s_0 e^{s_0}}{e^{s_0} - 1}\operatorname{Li}\bigl( e^{s_0 T} \bigr)  \sim \frac{(\deg f)^{T+1}}{ (-1 + \deg f) T}\text{ as } T \to + \infty.
	\end{equation}
\end{prop}

\begin{proof}
	It follows immediately from $h_{\operatorname{top}} (f) = \log (\deg f)$ that $s_0 = \log (\deg f)$.
		
	Similar to the claim in the proof of Theorem~\ref{thmZetaAnalExt_InvC} above, we claim that for each $\tau \in \Orb(f)$, we have $\deg (\tau) \coloneqq \prod_{x \in \tau} \deg_f(x) < (\deg f)^{\card \tau}$. We argue by contradiction and assume that there exists $\tau \in \Orb(f)$ such that $\deg (\tau) \geq (\deg f)^{\card \tau}$. Write $k\coloneqq \card \tau$. So by (\ref{eqLocalDegreeProduct}), for each $m\in\N$, 
	\begin{equation*}
		\frac{1}{(\deg f)^{m k}} \sum_{x \in \tau} \deg_{f^{m k}} (x) \geq k.
	\end{equation*}
	But by \cite[Lemma~5.11]{Li16} (with $M \coloneqq \tau$),
	\begin{equation*}
	 \lim_{n \to +\infty} \frac{1}{(\deg f)^n} \sum_{x \in \tau} \deg_{f^n} (x)  = 0.
	\end{equation*}
	This is a contradiction, establishing the claim.
	
	We note that in the notation of page~101 in \cite{PP90}, writing $z\coloneqq e^s$, we have
	\begin{equation*}
		\zeta_{\sigma_A,\,1} (s)   
		    = \prod_{\tau\in \Orb(\sigma_A)} \bigl( 1- e^{s \card \tau} \bigr)^{-1}
		    = \zeta_A (z)   
	\end{equation*}
	for each transition matrix $A \in\{ A_{\ti}, \, A_{\e}, \, A_{\ee}, \, A_{\po}\}$ (see Proposition~\ref{propTileSFT}, (\ref{eqDefA|}), and (\ref{eqDefA||})). Here $A_{\po} \: \V^0 \times \V^0 \rightarrow \{0,\,1\}$ is the transition matrix induced by $f|_{\V^0} \: \V^0 \rightarrow \V^0$.
	
	Write $\beta \coloneqq \deg f$.
	
	It follows immediately from the claim above that for each $\tau \in \Orb(f)$ with $\deg(\tau) > 1$, both $\bigl(  1- \deg_f(\tau) z^{ \card\tau  } \bigr)^{-1}$ and $\bigl(  1-  z^{ \card \tau  } \bigr)^{-1}$ are holomorphic in $\{z \in \C : \abs{z} < e^{\epsilon_0} / \beta\}$ for some $\epsilon_0 >0$.
	
	Hence, by (\ref{eqDirichletSeriesIsCombZetaFns}) in Proposition~\ref{propZetaProduct}, Theorem~\ref{thmPressureOnC} (with $\varphi \equiv 0$), the fact that $h_{\operatorname{top}} (\sigma_{A_{\ti}}) = \log(\deg f)$ (by (\ref{eqCardBlackNTiles}), Propositions~\ref{propCellDecomp},~\ref{propSFT}~(iv), and~\ref{propTileSFT}), and the arguments on pages~100--101 in \cite{PP90}, we get that for each $s\in\C$ with $\Re(s) < - s_0$,
	\begin{equation*}
				\DS_{f,\,1,\,\deg_f} (s) 
		=  \zeta_{ A_{\ti} } (z) \zeta_{ A_{\e} } (z) \zeta_{ A_{\ee} } (z) /\zeta_{ A_{\po} } (z) ,
	\end{equation*}
	where $z\coloneqq e^s$, and additionally by (\ref{eqPfthmZetaAnalExt_InvC_zeta}) in the proof of Theorem~\ref{thmZetaAnalExt_InvC} above and the claim, we also get that
	\begin{equation*}
		\frac{ \zeta'_f (z)}{\zeta_f(z)} = \frac{\beta}{1-\beta z} + \alpha (z), 
	\end{equation*}
	where $\zeta_f(z) \coloneqq \prod_{\tau \in \Orb(f)} \bigl( 1- z^{\card \tau}\bigr)^{-1}$ and $\alpha(z)$ is holomorphic in $\{z \in \C : \abs{z} < e^{\epsilon} / \beta\}$ for some $\epsilon >0$.
	
	Following verbatim the same arguments as that of \cite[Theorem~6.5]{PP90} on pages~101--105 in \cite{PP90} (with $\zeta_A$ replaced by $\zeta_f$), we get (\ref{eqOrbitLengthAsymp}).	
\end{proof}

Theorem~\ref{thmPrimeOrbitTheorem} now follows from standard number-theoretic arguments. More precisely, a proof of statement~(ii) in Theorem~\ref{thmPrimeOrbitTheorem}, relying on Theorem~\ref{thmZetaAnalExt_InvC} and the Ikehara--Wiener Tauberian Theorem (see \cite[Appendix~I]{PP90}), is verbatim the same as that of \cite[Theorem~6.9]{PP90} on pages 106--109 of \cite{PP90} (after defining $h\coloneqq s_0$, $\lambda(\tau) \coloneqq l_{F,\Phi}(\tau)$, $\pi\coloneqq \pi_{F,\Phi}$, and $\zeta \coloneqq \zeta_{F, \, \minus s_0 \Phi}$ in the notation of \cite{PP90}) with an additional observation that $\lim_{y\to+\infty} \frac{ \operatorname{Li}(y) }{  \frac{y}{ \log y } }  = 1$. We omit this proof here and direct the interested readers to the references cited above. On the other hand, statement~(i) in Theorem~\ref{thmPrimeOrbitTheorem} follows immediately from Proposition~\ref{propOrbitLength}, Remark~\ref{rmNf}, and Theorem~\ref{thmNLI}.

\appendix

\section{Basic facts about subshifts of finite type}    \label{apxSFT}

We collect some basic facts about subshifts of finite type.

\begin{prop}  \label{propSFT}
Consider a finite set of states $S$ and a transition matrix $A \: S\times S \rightarrow \{0, \, 1\}$. Let $\left( \Sigma_A^+ , \sigma_A \right)$ be the one-sided subshift of finite type defined by $A$, and $\phi\in\Holder{1}\bigl(\Sigma_A^+, d_\theta\bigr)$ be a real-valued Lipschitz continuous function with $\theta \in (0,1)$. Then the following statements are satisfied:
\begin{enumerate}
\smallskip
\item[(i)] $\card P_{1,\sigma_A^n} \leq (\card S)^n$ for all $n\in\N$.

\smallskip
\item[(ii)] $P(\sigma_A,\phi) \geq \limsup_{n\to +\infty} \frac{1}{n} \log  \sum_{\underline{x}\in P_{1,\sigma_A^n}}  \exp (S_n \phi(\underline{x}))$. 

\smallskip
\item[(iii)] $P(\sigma_A,\phi) = \lim_{n\to +\infty} \sup_{\underline{x}\in\Sigma_A^+} 
                 \frac{1}{n} \log \sum_{\underline{y}\in \sigma_A^{-n}(\underline{x})}  \exp (S_n \phi(\underline{y}))$.
\end{enumerate}

If, in addition, $(\Sigma_A^+,\sigma_A)$ is topologically mixing, then
\begin{enumerate}
\smallskip
\item[(iv)] $P(\sigma_A,\phi) = \lim_{n\to +\infty} \frac{1}{n} \log \sum_{\underline{y}\in \sigma_A^{-n}(\underline{x})}  \exp (S_n \phi(\underline{y}))$ for each $\underline{x}\in \Sigma_A^+$.
\end{enumerate}
\end{prop}

\begin{proof}
(i) Fix $n\in\N$. The inequality follows trivially from the observation that each $\{x_i\}_{i\in\N_0} \in P_{1,\sigma_A^n}$ is uniquely determined by the first $n$ entries in the sequence $\{x_i\}_{i\in\N_0}$.

\smallskip

(ii) Fix $n\in\N$. Since each $\{x_i\}_{i\in\N_0} \in P_{1,\sigma_A^n}$ is uniquely determined by the first $n$ entries in the sequence $\{x_i\}_{i\in\N_0}$,
it is clear that each pair of distinct $\{x_i\}_{i\in\N_0}, \,  \{x'_i\}_{i\in\N_0}\in P_{1,\sigma_A^n}$ are $(n,1)$-separated (see Subsection~\ref{subsctThermodynFormalism}). The inequality now follows from (\ref{defTopPressure}) and the observation that an $(n,1)$-separated set is also $(n,\epsilon)$-separated for all $\epsilon\in(0,1)$.

\smallskip

(iii) As remarked in \cite[Remark~1.3]{Bal00}, the proof of statement~(iii) follows from \cite[Lemma~4.5]{Rue89} and the beginning of the proof of Theorem~3.1 in \cite{Rue89}.

\smallskip

(iv) One can find a proof of this well-known fact in \cite[Proposition~4.4.3]{PU10} (see the first page of Chapter~4 in \cite{PU10} for relevant definitions).
\end{proof}

\begin{lemma}   \label{lmUnifBddToOneFactorPressure}
For each $i\in\{1, \, 2\}$, given a finite set of states $S_i$ and a transition matrix $A_i \: S_i \times S_i \rightarrow \{0, \, 1\}$, we denote by $\bigl( \Sigma_{A_i}^+ , \sigma_{A_i} \bigr)$ the one-sided subshift of finite type defined by $A_i$. Let $\phi\in\Holder{1}\bigl(\Sigma_{A_2}^+, d_\theta\bigr)$ be a real-valued Lipschitz continuous function on $\Sigma_{A_2}^+$ with $\theta \in (0,1)$. Suppose that there exists a uniformly bounded-to-one H\"{o}lder continuous factor map $\pi\: \Sigma_{A_1}^+ \rightarrow \Sigma_{A_2}^+$, i.e., $\pi$ is a H\"{o}lder continuous surjective map with $\sigma_{A_2} \circ \pi = \pi \circ \sigma_{A_1}$ and $\sup \bigl\{ \card \bigl(\pi^{-1}(\underline{x}) \bigr) : \underline{x}\in \Sigma_{A_2}^+ \bigr\} < +\infty$. Then
$
P(\sigma_{A_1}, \phi\circ \pi ) = P(\sigma_{A_2}, \phi).
$
\end{lemma}

\begin{proof}
We observe that since $\bigl(\Sigma_{A_2}^+, \sigma_{A_2}\bigr)$ is a factor of $\bigl(\Sigma_{A_1}^+, \sigma_{A_1}\bigr)$ with the factor map $\pi$, it follows from \cite[Lemma~3.2.8]{PU10} that $P  ( \sigma_{A_1},  \phi \circ \pi  ) \geq P  ( \sigma_{A_2},  \phi  )$. It remains to show $P  ( \sigma_{A_1},  \phi \circ \pi  ) \leq P  ( \sigma_{A_2},  \phi  )$.

Denote $M \coloneqq \sup \bigl\{ \card \bigl(\pi^{-1}(\underline{x}) \bigr) : \underline{x}\in \Sigma_{A_2}^+ \bigr\}$.

Note that $\phi \circ \pi \in\Holder{1}\bigl(\Sigma_{A_1}^+, d_{\theta'}\bigr)$ for some $\theta' \in (0,1)$. By Proposition~\ref{propSFT}~(iii), for each $\epsilon>0$, we can choose a sequence $\{\underline{x}^n\}_{n\in\N_0}$ in $\Sigma_{A_1}^+$ such that
\begin{equation}   \label{eqPflmUnifBddToOneFactorPressure}
     \liminf_{n\to +\infty}  \frac{1}{n} \log  \sum_{\underline{y}\in \sigma_{A_1}^{-n}(\underline{x}^n)}
          \exp \biggl( \sum_{i=0}^{n-1} \bigl( \phi \circ \pi \circ \sigma_{A_1}^i \bigr)(\underline{y}) \biggr)
\geq P(\sigma_{A_1}, \phi \circ \pi) - \epsilon.
\end{equation}
Observe that for all $\underline{x}, \underline{y} \in \Sigma_{A_1}^+$, and $n\in \N_0$, if $\sigma_{A_1}^n(\underline{y})=\underline{x}$, then 
$\pi(\underline{x}) = \bigl( \pi\circ \sigma_{A_1}^n \bigr) (\underline{y}) =  \bigl( \sigma_{A_2}^n \circ \pi \bigr) (\underline{y})$. Thus, by (\ref{eqPflmUnifBddToOneFactorPressure}) and Proposition~\ref{propSFT}~(iii),
\begin{align*}
        P(\sigma_{A_1}, \phi \circ \pi) - \epsilon
&\leq   \liminf_{n\to +\infty}  \frac{1}{n} \log  \sum_{\underline{y}\in \sigma_{A_1}^{-n}(\underline{x}^n)}
          \exp \biggl( \sum_{i=0}^{n-1} \bigl( \phi \circ \sigma_{A_2}^i \bigr)( \pi(\underline{y} ) ) \biggr) \\
&\leq   \liminf_{n\to +\infty}  \frac{1}{n} \log  \sum_{\underline{y}\in \pi^{-1} \bigl( \sigma_{A_2}^{-n}(\pi(\underline{x}^n)) \bigr)}
          \exp \biggl( \sum_{i=0}^{n-1} \bigl( \phi \circ \sigma_{A_2}^i \bigr)( \pi(\underline{y} ) ) \biggr) \\
&\leq  \limsup_{n\to +\infty}  \frac{1}{n} \log \Biggl( M \sum_{\underline{z}\in  \sigma_{A_2}^{-n}(\pi(\underline{x}^n))  }
          \exp \biggl( \sum_{i=0}^{n-1} \bigl( \phi \circ \sigma_{A_2}^i \bigr)(  \underline{z}  ) \biggr)  \Biggr) \\
&\leq    P(\sigma_{A_2}, \phi ) .
\end{align*}
Since $\epsilon>0$ is arbitrary, we get $P(\sigma_{A_1}, \phi \circ \pi) \leq P(\sigma_{A_2}, \phi )$. The proof is complete.
\end{proof}


\begin{thebibliography}{99}

\bibitem[An00a]{An00a}
\textsc{Anantharaman,~N.},
G\'eod\'esiques ferm\'ees d'une surface sous contraintes homologiques. Th\`{e}se de doctorat, Universit\'e Paris 6, 2000.

\bibitem[An00b]{An00b}
\textsc{Anantharaman,~N.},
Precise counting results for closed orbits of Anosov flows.
\textit{Ann.\ Sci.\ \'Ec.\ Norm.\ Sup\'er.\ (4)} 33 (2000), 33--56.


\bibitem[AM65]{AM65}
\textsc{Artin,~M.} and \textsc{Mazur,~B.},
On periodic points.
\textit{Ann.\ of Math.\ (2)} 81 (1965), 82--99.



\bibitem[AGY06]{AGY06}
\textsc{Avila,~A.}, \textsc{Gou\"ezel,~S.}, and \textsc{Yoccoz,~J.C.},
Exponential mixing for the Teichm\"uller flow.
\textit{Publ.\ Math.\ Inst.\ Hautes \'Etudes Sci.} 104 (2006), 143--211.





\bibitem[BabLe98]{BabLe98}
\textsc{Babillot,~M.} and \textsc{Ledrappier,~F.},
Lalley's theorem on periodic orbits of hyperbolic flows.
\textit{Ergodic Theory Dynam.\ Systems} 18 (1998), 17--39.


%\bibitem[Bak85]{Bak85}
%\textsc{Baker,~A.},
%\textit{A concise introduction to the theory of numbers},
%Cambridge Univ.\ Press, Cambridge, 1985.

\bibitem[Bai04]{Bai04}
\textsc{Baillif,~M.},
Kneading operators, sharp determinants and weighted Lefschetz zeta functions in higher dimension.
\textit{Duke Math.\ J.} 124 (2004), 145--175.


\bibitem[BB05]{BB05}
\textsc{Baillif,~M.} and \textsc{Baladi,~V.},
Kneading determinants and spectra of transfer operators in higher dimensions: the isotropic case.
\textit{Ergodic Theory Dynam.\ Systems} 25 (2005), 1437--1470.



\bibitem[Bal00]{Bal00}
\textsc{Baladi,~V.},
\textit{Positive transfer operators and decay of correlations}, volume~16 of \textit{Adv.\ Ser.\ Nonlinear Dynam.}
World Sci.\ Publ., Singapore, 2000.

\bibitem[Bal18]{Bal18}
\textsc{Baladi,~V.},
\textit{Dynamical zeta functions and dynamical determinants for hyperbolic maps: a functional approach}, volume~68 of \textit{Ergeb.\ Math.\ Grenzgeb.\ (3)}.
Springer, Berlin, 2018.






\bibitem[BDL18]{BDL18}
\textsc{Baladi,~V.}, \textsc{Demers,~M.}, and \textsc{Liverani,~C.},
Exponential decay of correlations for finite horizon Sinai billiard flows.
\textit{Invent.\ Math.} 211 (2018), 39--177.


\bibitem[BJR02]{BJR02}
\textsc{Baladi,~V.}, \textsc{Jiang,~Y.}, and \textsc{Rugh,~H.H.},
Dynamical determinants via dynamical conjugacies for postcritically finite polynomials.
\textit{J.\ Stat.\ Phys.} 108 (2002), 973--993.

\bibitem[BKRS97]{BKRS97}
\textsc{Baladi,~V.}, \textsc{Kitaev,~A.}, \textsc{Ruelle,~D.}, and \textsc{Semmes,~S.}
Sharp determinants and kneading operators for holomorphic maps. 
\textit{Tr.\ Mat.\ Inst.\ Steklova} 216, Din.\ Sist.\ i Smezhnye Vopr., (1997) 193--235;
translation in \textit{Proc.\ Steklov Inst.\ Math.} 216 (1997), 186--228.


\bibitem[BalLiv12]{BalLiv12}
\textsc{Baladi,~V.} and \textsc{Liverani,~C.},
Exponential decay of correlations for piecewise cone hyperbolic contact flows.
\textit{Comm.\ Math.\ Phys.} 314 (2012), 689--773.

\bibitem[BR96]{BR96}
\textsc{Baladi,~V.} and \textsc{Ruelle,~D.},
Sharp determinants. 
\textit{Invent.\ Math.} 123 (1996), 553--574.


%\bibitem[Ba88]{Ba88}
%\textsc{Barnsley,~M.},
%\textit{Fractals everywhere},
%Academic Press Professional, San Diego, 1988.

%\bibitem[BD11]{BD11}
%\textsc{Baker,~M.} and \textsc{DeMarco,~L.},
%Preperiodic points and unlikely intersections.
%\textit{ Duke Math.\ J.} 159 (2011), 1--29.


%\bibitem[Bi95]{Bi95}
%\textsc{Billingsley,~P.},
%\textit{Probability and measure},
%John Wiley \& Sons, New York, 1995.


%\bibitem[Bi99]{Bi99}
%\textsc{Billingsley,~P.},
%\textit{Convergence of probability measures}, 3rd ed.,
%John Wiley \& Sons, New York, 1999.

%\bibitem[BK98]{BK98}
%\textsc{Bruin,~H.} and \textsc{Keller,~G.},
%Equilibrium states for $S$-unimodal maps.
%\textit{Ergodic Theory Dynam.\ Systems} 18 (1998), 765--789.

\bibitem[Bon06]{Bon06}
\textsc{Bonk,~M.},
Quasiconformal geometry of fractals. In \textit{Proc.\ Internat.\ Congr.\ Math.\ (Madrid 2006)}, Volume~II,
Eur.\ Math.\ Soc.\ Z\"{u}rich, 2006, pp.\ 1349--1373.


\bibitem[BM10]{BM10} 
\textsc{Bonk,~M.} and \textsc{Meyer,~D.},
Expanding Thurston maps. Preprint, (arXiv:1009.3647v1), 2010.

%\bibitem[BM16]{BM16} 
%\textsc{Bonk,~M.} and \textsc{Meyer,~D.},
%Expanding Thurston maps. Preprint, (arXiv:1009.3647v2), 2016.



\bibitem[BM17]{BM17} 
\textsc{Bonk,~M.} and \textsc{Meyer,~D.},
\textit{Expanding Thurston maps}, volume 225 of \textit{Math.\ Surveys Monogr.}, Amer.\ Math.\ Soc., Providence, RI, 2017.


\bibitem[BCRW08]{BCRW08} 
\textsc{Borwein,~P.}, \textsc{Choi,~St.}, \textsc{Rooney,~B.}, and \textsc{Weirathmueller,~A.},
\textit{The Riemann Hypothesis: a resource for the afficionado and virtuoso alike}, 
Springer, New York, 2008.


\bibitem[BD17]{BD17} 
\textsc{Bourgain,~J.} and \textsc{Dyatlov,~S.},
Fourier dimension and spectral gaps for hyperbolic surfaces.
\textit{Geom.\ Funct.\ Anal.} 27 (2017), 744--771.


\bibitem[BGS11]{BGS11} 
\textsc{Bourgain,~J.}, \textsc{Gamburd,~A.}, and \textsc{Sarnak,~P.},
Generalization of Selberg’s $\frac{3}{16}$ theorem and affine sieve.
\textit{Acta Math.} 207 (2011), 255--290.



%\bibitem[Bo71]{Bo71}
%\textsc{Bowen,~R.},
%Periodic points and measures for Axiom A diffeomorphisms. 
%\textit{Trans.\ Amer.\ Math.\ Soc.} 154 (1971), 377--397.


%\bibitem[Bow72]{Bow72}
%\textsc{Bowen,~R.},
%Periodic orbits for hyperbolic flows.
%\textit{Amer.\ J.\ Math.} 94 (1972), 1--30.


\bibitem[Bow72]{Bow72}
\textsc{Bowen,~R.},
Entropy-expansive maps. 
\textit{Trans.\ Amer.\ Math.\ Soc.} 164 (1972), 323--333.





%\bibitem[Bo75]{Bo75}
%\textsc{Bowen,~R.},
%\textit{Equilibrium states and the ergodic theory of Anosov diffeomorphisms}, volume~470 of \textit{Lecture Notes in Math.}, Springer, Berlin, 1975.         
         
   
%\bibitem[BR75]{BR75}
%\textsc{Bowen,~R.} and \textsc{Ruelle,~D.},
%The ergodic theory of Axiom A flows.
%\textit{Invent.\ Math.} 29 (1975), 181--202.

   
   
   
%\bibitem[Br10]{Br10}
%\textsc{Bracci,~F.},
%Local holomorphic dynamics of diffeomorphisms in dimension one.
%In \textsc{Contreras,~M.D.} and \textsc{D\'{i}az-Madrigal,~S.} (Eds.), \textit{Five lectures in complex analysis} (pp.\ 1--42). 
%Amer.\ Math.\ Soc., Providence, RI, 2010.



%\bibitem[BMD02]{BMD02}
%\textsc{Buzzi,~J.} and \textsc{Maume-Deschamps,~V.},
%Decay of correlations for piecewise invertible maps in higher dimensions.
%\textit{Israel J.\ Math.} 131 (2002), 203--220.


%\bibitem[BS03]{BS03}
%\textsc{Buzzi,~J.} and \textsc{Sarig,~O.},
%Uniqueness of equilibrium measures for countable Markov shifts and multidimensional piecewise expanding maps.
%\textit{Ergodic Theory Dynam.\ Systems} 23 (2003), 1383--1400.


%\bibitem[Bu99]{Bu99}
%\textsc{Buzzi,~J.},
%Markov extensions for multi-dimensional dynamical systems.
%\textit{Israel J.\ Math. } 112 (1999), 357--380.



\bibitem[Ca94]{Ca94}
\textsc{Cannon,~J.W.},
The combinatorial Riemann mapping theorem.
\textit{Acta Math.} 173 (1994), 155--234.



%\bibitem[CFP07]{CFP07}
%\textsc{Cannon,~J.W.}, \textsc{Floyd,~W.J.}, and \textsc{Parry,~W.R.},
%Constructing subdivision rules from rational maps.
%\textit{Conform.\ Geom.\ Dyn.} 11 (2007), 128--136.


%\bibitem[CFPP09]{CFPP09}
%\textsc{Cannon,~J.}, \textsc{Floyd,~W.}, \textsc{Parry,~R.}, and \textsc{Pilgrim,~K.M.},
%Subdivision rules and virtual endomorphisms.
%\textit{Geom.\ Dedicata} 141 (2009), 181--195.





%\bibitem[CG93]{CG93}
%\textsc{Carleson,~L.} and \textsc{Gamelin,~T.H.},
%\textit{Complex dynamics},
%Springer, New York, 1993.



%\bibitem[Ch98]{Ch98}
%\textsc{Chernov,~N.},
%Markov approximations and decay of correlations for Anosov flows.
%\textit{Ann.\ of Math.} 147 (1998), 269--324.



%\bibitem[CRL11]{CRL11}
%\textsc{Comman,~H.} and \textsc{Rivera-Letelier,~J.},
%Large deviation principles for non-uniformly hyperbolic rational maps.  
%\textit{Ergodic Theory Dynam.\ Systems} 31 (2011), 321--349.



\bibitem[DPTUZ19]{DPTUZ19}
\textsc{Das,~T}, \textsc{Przytycki,~F.}, \textsc{Tiozzo,~G.}, \textsc{Urba\'{n}ski,~M.}, and \textsc{Zdunik,~A.},
Thermodynamic formalism for coarse expanding dynamical systems, (arXiv:1908.08270), 2019.




\bibitem[Dol98]{Dol98}
\textsc{Dolgopyat,~D.},
On decay of correlations of Anosov flows.  
\textit{Ann.\ of Math.\ (2)} 147 (1998), 357--390. 


\bibitem[DH93]{DH93}
\textsc{Douady,~A.} and \textsc{Hubbard,~J.H.},
A proof of Thurston's topological characterization of rational functions.  
\textit{Acta Math.} 171 (1993), 263--297. 


%\bibitem[DPU96]{DPU96}
%\textsc{Denker,~M.}, \textsc{Przytycki,~F.}, and \textsc{Urba\'nski,~M.},
%On the transfer operator for rational functions on the Riemann sphere.  
%\textit{Ergodic Theory Dynam.\ Systems} 16 (1996), 255--266.




%\bibitem[DU91]{DU91}
%\textsc{Denker,~M.} and \textsc{Urba\'nski,~M.},
%Ergodic theory of equilibrium states for rational maps.
%\textit{Nonlinearity} 4 (1991), 103--134. 



%\bibitem[DU92]{DU92}
%\textsc{Denker,~M.} and \textsc{Urba\'nski,~M.},
%The dichotomy of Hausdorff measures and equilibrium states for parabolic rational maps.
%In \textit{Ergodic theory and related topics III}, volume~1514 of \textit{Lecture Notes in Math.}, pp.\ 90--113, Springer, Berlin, 1992. 


%\bibitem[Do68]{Do68}
%\textsc{Dobru\v{s}in,~R.L.},
%Description of a random field by means of conditional probabilities and conditions for its regularity.
%\textit{Theor.\ Verojatnost.\ i Primenen} 13 (1968), 201--229. 



\bibitem[DZ16]{DZ16} 
\textsc{Dyatlov,~S.} and \textsc{Zahl,~J.},
Spectral gaps, additive energy, and a fractal uncertainty principle.
\textit{Geom.\ Funct.\ Anal.} 26 (2016), 1011--1094.




%\bibitem[EE85]{EE85}
%\textsc{Ellison,~W.J.} and \textsc{Ellison,~F.},
%\textit{Prime numbers},
%Hermann, Paris, 1985.

%\bibitem[Fi11]{Fi11} 
%\textsc{Field,~M.J.},
%Exponential mixing for smooth hyperbolic suspension flows.
%\textit{Regul.\ Chaotic Dyn.} 16 (2011), 90--103.


%\bibitem[Fo81]{Fo81}
%\textsc{Forster,~O.},
%\textit{Lectures on Riemann surfaces},
%Springer, New York, 1981.


%\bibitem[Fo99]{Fo99}
%\textsc{Folland,~G.B.},
%\textit{Real analysis: modern techniques and their applications}, 2nd ed.,
%Wiley, New York, 1999.


%\bibitem[FK83]{FK83}
%\textsc{Furstenberg,~H.} and \textsc{Kifer,~Y.},
%Random matrix products and measures on projective spaces.
%\textit{Israel J.\ Math.} 46 (1983), 12--32.

               
            
%\bibitem[FLM83]{FLM83}
%\textsc{Freire,~A.}, \textsc{Lopes,~A.}, and \textsc{Ma\~{n}\'{e},~R.},
%An invariant measure for rational maps.
%\textit{Bol.\ Soc.\ Brasil.\ Mat.} 14 (1983), 45--62.
             


%\bibitem[GW10]{GW10}
%\textsc{Gelfert,~K.} and \textsc{Wolf,~C.},
%On the distribution of periodic orbits.
%\textit{Discrete Contin.\ Dyn.\ Syst.} 26 (2010), 949--966.


%\bibitem[GP10]{GP10}
%\textsc{Guillemin,~V.} and \textsc{Pollack,~A.},
%\textit{Differential topology},
%Amer.\ Math.\ Soc., Providence, RI, 2010.



%\bibitem[HT03]{HT03}
%\textsc{Hawkins,~J.} and \textsc{Taylor,~M.},
%Maximal entropy measure for rational maps and a random iteration algorithm for Julia sets.
%\textit{Intl.\ J.\ of Bifurcation and Chaos} 13 (6) (2003), 1442--1447.


\bibitem[GLP13]{GLP13}
\textsc{Giulietti,~P.}, \textsc{Liverani,~C.}, and \textsc{Pollicott,~M.},
Anosov flows and dynamical zeta functions.
\textit{Ann.\ of Math.\ (2)} 178 (2013), 687--773.


\bibitem[Gu86]{Gu86}
\textsc{Guillop\'e,~L.},
Sur la distribution des longueurs des g\'eode\'esiques ferm\'ees d'une surface compacte \`a bord totalement g\'eode\'esique.
\textit{Duke Math.\ J.} 53 (1986), 827--848.


\bibitem[HP09]{HP09}
\textsc{Ha\"{\i}ssinsky,~P.} and \textsc{Pilgrim,~K.M.},
Coarse expanding conformal dynamics.
\textit{Ast\'{e}risque} 325 (2009).


%\bibitem[HP11]{HP11}
%\textsc{Ha\"{\i}ssinsky,~P.} and \textsc{Pilgrim,~K.M.},
%Finite type coarse expanding conformal dynamics.
%\textit{Groups Geom.\ Dyn.} 5 (2011), 603--661.

%\bibitem[HP12]{HP12}
%\textsc{Ha\"{\i}ssinsky,~P.} and \textsc{Pilgrim,~K.M.},
%An algebraic characterization of expanding Thurston maps.
%\textit{J.\ of Mod.\ Dyn.} 6 (2012), 451--476.

\bibitem[HRL19]{HRL19}
\textsc{Ha\"{\i}ssinsky,~P.} and \textsc{Rivera-Letelier,~J.},
Private communication.


\bibitem[Ha02]{Ha02}
\textsc{Hatcher,~A.},
\textit{Algebraic topology},
Cambridge Univ.\ Press, Cambridge, 2002.



%\bibitem[HT03]{HT03}
%\textsc{Hawkins,~J.} and \textsc{Taylor,~M.},
%Maximal entropy measure for rational maps and a random iteration algorithm for Julia sets.
%\textit{Intl.\ J.\ of Bifurcation and Chaos} 13 (6) (2003), 1442--1447.



\bibitem[He01]{He01}
\textsc{Heinonen,~J.}
\textit{Lectures on analysis on metric spaces},
Springer, New York, 2001.



%\bibitem[Hu06]{Hu06}
%\textsc{Hubbard,~J.H.},
%\textit{Teichm\"uller theory and applications to geometry, topology, and dynamics}, Vol.\ 1,
%Matrix Editions, Ithaca, 2006.

\bibitem[Hu61]{Hu61}
\textsc{Huber,~H.},
Zur analytischen Theorie hyperbolischer Raumformen und Bewegungsgruppen. II.
\textit{Math.\ Ann.} 142 (1961), 385--398.



%\bibitem[IRRL12]{IRRL12}
%\textsc{Inoquio-Renteria,~I.} and \textsc{ Rivera-Letelier,~J.},
%A characterization of hyperbolic potentials of rational maps.
%\textit{Bull.\ Braz.\ Math.\ Soc.\ (N.S.)} 43 (2012), 99--127.




\bibitem[KH95]{KH95}
\textsc{Katok,~A.} and \textsc{Hasselblatt,~B.},
\textit{Introduction to the modern theory of dynamical systems},
Cambridge Univ.\ Press, Cambridge, 1995.



\bibitem[KS90]{KS90}
\textsc{Katsuda,~A.} and \textsc{ Sunada,~T.},
Closed orbits in homology classes.
\textit{Publ.\ Math.\ Inst.\ Hautes \'Etudes Sci.} 71 (1990), 5--32.



%\bibitem[KPS09]{KPS09}
%\textsc{Koch,~S.}, \textsc{Pilgrim,~K.M.}, and \textsc{Selinger,~N.},
%Pullback invariants of Thurston maps.
%In \textit{Complex Dynamics, Families and Friends}, 561--583, AK Peters, Wellesley, 2009. 



\bibitem[Ki98]{Ki98}
\textsc{Kitchens,~B.P.},
\textit{Symbolic dynamics: one-sided, two-sided, and countable state Markov shifts},
Springer, Berlin, 1998.


\bibitem[La89]{La89} 
\textsc{Lalley,~S.P.},
Renewal theorems in symbolic dynamics, with applications to geodesic flows, noneuclidean tessellations and their fractal limits.
\textit{Acta Math.} 163 (1989), 1--55.


\bibitem[Li15]{Li15} 
\textsc{Li,~Z.},
Weak expansion properties and large deviation principles for expanding Thurston maps.
\textit{Adv.\ Math.} 285 (2015), 515--567.


\bibitem[Li16]{Li16} 
\textsc{Li,~Z.},
Periodic points and the measure of maximal entropy of an expanding Thurston map.
\textit{Trans.\ Amer.\ Math.\ Soc.} 368 (2016), 8955--8999.

\bibitem[Li17]{Li17} 
\textsc{Li,~Z.},
\textit{Ergodic theory of expanding Thurston maps,} volume~4 of \textit{Atlantis Stud.\ Dyn.\ Syst.}, Atlantis Press, 2017.



\bibitem[Li18]{Li18} 
\textsc{Li,~Z.},
Equilibrium states for expanding Thurston maps. 
\textit{Comm.\ Math.\ Phys.} 357 (2018), 811--872.


\bibitem[LRL]{LRL} 
\textsc{Li,~Z.}  and \textsc{Rivera-Letelier,~J.},
Prime orbit theorems for topological Collet--Eckmann maps.
Work in progress.



\bibitem[LZha23]{LZha23} 
\textsc{Li,~Z.}  and \textsc{Zhang,~Y.},
Ground states and periodic orbits for expanding Thurston maps.
Preprint, (arXiv:2303.00514), 2023.

\bibitem[LZheH23]{LZheH23} 
\textsc{Li,~Z.}  and \textsc{Zheng,~H.},
Weak expansion properties and a large deviation principle for coarse expanding conformal systems.
Preprint, (arXiv:2311.07305), 2023.


\bibitem[LZhe23b]{LZhe23b} 
\textsc{Li,~Z.}  and \textsc{Zheng,~T.},
Prime orbit theorems for expanding Thurston maps: Latt\`{e}s maps and split Ruelle operators.
Preprint, (arXiv:2312.06688), 2023.


\bibitem[LZhe23c]{LZhe23c} 
\textsc{Li,~Z.}  and \textsc{Zheng,~T.},
Prime orbit theorems for expanding Thurston maps: Genericity of strong non-integrability condition.
Preprint, (arXiv:2312.06687), 2023.


\bibitem[Liv04]{Liv04} 
\textsc{Liverani,~C.},
On contact Anosov flows.
\textit{Ann.\ of Math.\ (2)} 159 (2004), 1275--1312.


%\bibitem[Ly83]{Ly83}
%\textsc{Lyubich,~M.Yu.},
%Entropy properties of rational endomorphisms of the Riemann sphere.
%\textit{Ergodic Theory Dynam.\ Systems} 3 (1983), 351--385.

\bibitem[LM97]{LM97}
\textsc{Lyubich,~M.Yu.} and \textsc{Minsky,~Y.},
Laminations in holomorphic dynamics.
\textit{J.\ Differential Geom.} 47 (1997), 17--94.

%\bibitem[Ma83]{Ma83}
%\textsc{Ma\~n\'e,~R.},
%On the uniqueness of the maximising measure for rational maps.
%\textit{Bol.\ Soc.\ Brasil.\ Mat.} 14 (1983), 27--43.


\bibitem[Mar69]{Mar69}
\textsc{Margulis,~G.A.},
On some applications of ergodic theory to the study of manifolds on negative curvature.
\textit{Fun.\ Anal.\ Appl.} 3 (1969), 89--90.



\bibitem[Mar04]{Mar04}
\textsc{Margulis,~G.A.},
\textit{On some aspects of theory of Anosov systems},
Springer, Berlin, 2004.

\bibitem[MMO14]{MMO14}
\textsc{Margulis,~G.A.}, \textsc{Mohammadi,~A.}, and \textsc{Oh,~H.}, 
Closed geodesics and holonomies for Kleinian manifolds.
\textit{Geom.\ Funct.\ Anal.} 24 (2014), 1608--1636.

%\bibitem[MauU03]{MauU03}
%\textsc{Mauldin,~D.} and \textsc{Urba\'{n}ski,~M.},
%\textit{Graph directed Markov systems: geometry and dynamics of limit sets}, 
%Cambridge Univ.\ Press, Cambridge, 2003.


%\bibitem[MayU10]{MayU10}
%\textsc{Mayer,~V.} and \textsc{Urba\'{n}ski,~M.}, 
%Thermodynamical formalism and multifractal analysis for meromorphic functions of finite order.
%\textit{Mem.\ Amer.\ Math.\ Soc.} 203 (2010), no.~954, vi+107 pp.
 
\bibitem[Me12]{Me12}
\textsc{Meyer,~D.},
Expanding Thurston maps as quotients.
Preprint, (arXiv:0910.2003v1), 2012.
 

\bibitem[Me13]{Me13}
\textsc{Meyer,~D.},
Invariant Peano curves of expanding Thurston maps.
\textit{Acta Math.} 210 (2013), 95--171.







%\bibitem[Mil06]{Mil06}
%\textsc{Milnor,~J.},
%On Latt\`{e}s maps. 
%In \textit{Dynamics on the Riemann sphere},
%Eur.\ Math.\ Soc., Z\"urich, 2006, pp.9--43.

%\bibitem[Mil06]{Mil06}
%\textsc{Milnor,~J.},
%\textit{Dynamics in one complex variable}, 3rd ed., 
%Princeton Univ.\  Press, Princeton, 2006.

\bibitem[MT88]{MT88}
\textsc{Milnor,~J.} and \textsc{Thurston,~W.P.},
Iterated maps of the interval.
In \textsc{Alexander,~J.C.} (Eds.), \textit{Dynamical systems (Maryland 1986--87)}, 
volume~1342 of \textit{Lecture Notes in Math.}, pp.\ 465--563, Springer, Berlin, 1988. 


\bibitem[Mi73]{Mi73}
\textsc{Misiurewicz,~M.},
Diffeomorphisms without any measure with maximal entropy.
\textit{Bull.\ Acad.\ Pol.\ Sci.} 21 (1973), 903--910.



\bibitem[Mi76]{Mi76}
\textsc{Misiurewicz,~M.},
Topological conditional entropy.
\textit{Studia Math.} 55 (1976), 175--200.



%\bibitem[Mu00]{Mu00}
%\textsc{Munkres,~J.R.},
%\textit{Topology}, 2nd ed., 
%Prentice Hall, Upper Saddle River, NJ, 2000.


\bibitem[Na05]{Na05}
\textsc{Naud,~F.},
Expanding maps on Cantor sets and analytic continuation of zeta functions.
\textit{Ann.\ Sci.\ \'Ec.\ Norm.\ Sup\'er.\ (4)} 38 (2005), 116--153.

\bibitem[Na14]{Na14}
\textsc{Naud,~F.},
Density and location of resonances for convex co-compact hyperbolic surfaces.
\textit{Invent.\ Math.} 195 (2014), 723--750.


\bibitem[OP18]{OP18}
\textsc{Oh,~H.} and \textsc{Pan,~W.},
Local mixing and invariant measures for horospherical subgroups on abelian covers.
\textit{Int.\ Math.\ Res.\ Not.\ IMRN}  (2018).

\bibitem[OW16]{OW16}
\textsc{Oh,~H.} and \textsc{Winter,~D.},
Uniform exponential mixing and resonance free regions for convex cocompact congruence subgroups of $\mathrm{SL}_2(\Z)$.
\textit{J.\ Amer.\ Math.\ Soc.} 29 (2016), 1069--1115.



\bibitem[OW17]{OW17}
\textsc{Oh,~H.} and \textsc{Winter,~D.},
Prime number theorems and holonomies for hyperbolic rational maps.
\textit{Invent.\ Math.} 208 (2017), 401--440.


%\bibitem[Ol03]{Ol03}
%\textsc{Oliveira,~K.},
%Equilibrium states for non-uniformly expanding maps.
%\textit{Ergod.\ Th.\ Dynam.\ Sys.} 23 (2003), 1891--1905.


%\bibitem[Ol12]{Ol12}
%\textsc{Oliveira,~K.},
%Every expanding measure has the nonuniform specification property.
%\textit{Proc.\ Amer.\ Math.\ Soc.} 140 (2012), 1309--1320.




%\bibitem[OV06]{OV06}
%\textsc{Oliveira,~K.} and \textsc{Viana,~M.}, 
%Existence and uniqueness of maximizing measures for robust classes of local diffeomorphisms.
%\textit{Discrete Contin.\ Dyn.\ Syst.} 15 (2006), 225--236.




%\bibitem[Pa64]{Pa64}
%\textsc{Parry,~W.},
%On intrinsic Markov chains.
%\textit{Trans.\ Amer.\ Math.\ Soc.} 112 (1964), 55--66.


\bibitem[PP90]{PP90}
\textsc{Parry,~W.} and \textsc{Pollicott,~M.},
Zeta functions and the periodic orbit structure of hyperbolic dynamics.
\textit{Ast\'{e}risque} 187--188 (1990), 1--268.


\bibitem[PhSa87]{PhSa87}
\textsc{Phillips,~R.} and \textsc{Sarnak,~P.},
Geodesics in homology classes.
\textit{Duke Math. J.} 55 (1987), 287--297.


\bibitem[Po91]{Po91}
 \textsc{Pollicott,~M.},
Homology and closed geodesics in a compact negatively curved surface.
\textit{Amer.\ J.\ Math.} 113 (1991), 379--385.

\bibitem[PoSh98]{PoSh98}
 \textsc{Pollicott,~M.} and \textsc{Sharp,~R.},
Exponential error terms for growth functions on negatively curved surfaces.
\textit{Amer.\ J.\ Math.} 120 (1998), 1019--1042.



%\bibitem[Po90]{Po90}
%\textsc{Poincar\'e,~H.},
%Sur le probl\'eme des trois corps et les \'equations de la dynamique.
%\textit{Acta Math.} 13 (1890), 1--270.


%\bibitem[Pr90]{Pr90}
%\textsc{Przytycki,~F.},
%On the Perron--Frobenius--Ruelle operator for rational maps on the Riemann sphere and for H\"{o}lder continuous functions.
%\textit{Bol.\ Soc.\ Brasil.\ Mat.} 20(2) (1990), 95--125.


%\bibitem[PoU17]{PoU17}
%\textsc{Pollicott,~M.} and \textsc{Urba\'{n}ski,~M.},
%Asymptotic counting in conformal dynamical systems.
%Preprint, (arXiv:1704.06896v2), 2017.



\bibitem[PRL07]{PRL07}
\textsc{Przytycki,~F.} and \textsc{Rivera-Letelier,~J.},
Statistical properties of topological Collet--Eckmann maps.
\textit{Ann.\ Sci.\ \'Ec.\ Norm.\ Sup\'er.\ (4)} 40 (2007), 135--178.

\bibitem[PRL11]{PRL11}
\textsc{Przytycki,~F.} and \textsc{Rivera-Letelier,~J.},
Nice inducing schemes and the thermodynamics of rational maps.
\textit{Comm.\ Math.\ Phys.} 301 (2011), 661--707.


\bibitem[PRLS03]{PRLS03}
\textsc{Przytycki,~F.}, \textsc{Rivera-Letelier,~J.}, and \textsc{Smirnov,~S.}
Equivalence and topological invariance of conditions for non-uniform hyperbolicity in the iteration of rational maps.
\textit{Invent.\ Math.} 151 (2003), 29--63.




\bibitem[PU10]{PU10}
\textsc{Przytycki,~F.} and \textsc{Urba\'{n}ski,~M.},
\textit{Conformal fractals: ergodic theory methods},
Cambridge Univ.\ Press, Cambridge, 2010.


\bibitem[RLS14]{RLS14}
\textsc{Rivera-Letelier,~J.} and \textsc{Shen,~W.},
Statistical properties of one-dimensional maps under weak hyperbolicity assumptions.
\textit{Ann.\ Sci.\ \'Ec.\ Norm.\ Sup\'er.\ (4)} 47 (2014), 1027--1083.




\bibitem[Ro03]{Ro03}
\textsc{Roblin,~T.},
Ergodicit\'{e} et \'{e}quidistribution en courbure n\'{e}gative.
\textit{M\'{e}m.\ Soc.\ Math.\ Fr.\ (N.S.)} 95 (2003), vi+96.



%\bibitem[Ro49]{Ro49}
%\textsc{Rokhlin,~V.},
%On the fundamental ideas of measure theory.
%\textit{Mat.\ Sb.\ (N.S.),} 25(67) (1949), 107--150; English transl., \textit{Amer.\ Math.\ Soc.\ Transl.} (1) 10 (1962), 1--54.


%\bibitem[Ro61]{Ro61}
%\textsc{Rokhlin,~V.},
%Exact endomorphisms of a Lebesgue space.
%\textit{Izv.\ Akad.\ Nauk SSSR Ser.\ Mat.} 25 (1961), 499--530; English transl., \textit{Amer.\ Math.\ Soc.\ Transl.} (2) 39 (1964), 1--36.

\bibitem[Rue76a]{Rue76a}
\textsc{Ruelle,~D.},
Zeta functions and statistical mechanics.
\textit{Ast\'{e}risque} 40 (1976), 167--176.

\bibitem[Rue76b]{Rue76b}
\textsc{Ruelle,~D.},
Generalized zeta-functions for axiom A basic sets.
\textit{Bull.\ Amer.\ Math.\ Soc.} 80 (1976), 153--156.

\bibitem[Rue76c]{Rue76c}
\textsc{Ruelle,~D.},
Zeta-functions for expanding maps and Anosov flows.
\textit{Invent.\ Math.} 34 (1976), 231--242.


\bibitem[Rue89]{Rue89}
\textsc{Ruelle,~D.},
The thermodynamical formalism for expanding maps.
\textit{Comm.\ Math.\ Phys.} 125 (1989), 239--262.


\bibitem[Rue90]{Rue90}
\textsc{Ruelle,~D.},
An extension of the theory of Fredholm determinants.
\textit{Publ.\ Math.\ Inst.\ Hautes \'Etudes Sci.} 72 (1990), 175--193.


\bibitem[Rug16]{Rug16}
\textsc{Rugh,~H.H.},
The Milnor--Thurston determinant and the Ruelle transfer operator.
\textit{Comm.\ Math.\ Phys.} 342 (2016), 603--614.


%\bibitem[Sa03]{Sa03}
%\textsc{Sarig,~O.},
%Existence of Gibbs measures for countable Markov shifts.
%\textit{Proc.\ Amer.\ Math.\ Soc.} 131 (2003), 1751--1758.



\bibitem[Sa80]{Sa80}
\textsc{Sarnak,~P.},
Prime geodesic theorems. PhD thesis, Stanford University, 1980.



\bibitem[Se56]{Se56}
\textsc{Selberg,~A.},
Harmonic analysis and discontinuous groups in weakly symmetric Riemannian spaces with applications to Dirichlet series.
\textit{J.\ Indian Math.\ Soc.} 20 (1956), 47--87.


\bibitem[Sh93]{Sh93}
\textsc{Sharp,~R.},
Closed orbits in homology classes for Anosov flows.
\textit{Ergodic Theory Dynam.\ Systems} 13 (1993), 387--408.



%\bibitem[Si72]{Si72}
%\textsc{Sinai,~Ya.},
%Gibbs measures in ergodic theory.
%\textit{Russian Math.\ Surveys} 27 (1972), 21--69.


\bibitem[Sm67]{Sm67}
\textsc{Smale,~S.},
Differentiable dynamical systems.
\textit{Bull.\ Amer.\ Math.\ Soc.} 73 (1967), 747--817.


\bibitem[St01]{St01}
\textsc{Stoyanov,~L.N.},
Spectrum of the Ruelle operator and exponential decay of correlations for open billiard flows.
\textit{Amer.\ J.\ Math.} 123 (2001), 715--759.

\bibitem[St11]{St11}
\textsc{Stoyanov,~L.N.},
Spectra of Ruelle transfer operators for Axiom A flows.
\textit{Nonlinearity} 24 (2011), 1089--1120.

\bibitem[Su83]{Su83}
\textsc{Sullivan,~D.P.},
Conformal dynamical systems. In \textit{Geometric dynamics}, volume~1007 of \textit{Lecture Notes in Math.}, pp.\ 725--752.
Springer, Berlin, 1983.

\bibitem[Su85]{Su85}
\textsc{Sullivan,~D.P.},
Quasiconformal homeomorphisms and dynamics I. Solution of the Fatou--Julia problem on wandering domains.
\textit{Ann.\ of Math.\ (2)} 122 (1985), 401--418.


%\bibitem[Ti39]{Ti39}
%\textsc{Titchmarsh,~E.C.},
%\textit{The theory of functions},
%Oxford Univ.\ Press, Oxford, 1939.



\bibitem[Th80]{Th80}
\textsc{Thurston,~W.P.},
\textit{The geometry and topology of three-manifolds},
lecture notes from 1980, available at \url{http://library.msri.org/books/gt3m/}.




%\bibitem[Ur98]{Ur98}
%\textsc{Urba\'nski,~M.},
%Hausdorff measures versus equilibrium states of conformal infinite iterated function systems.
%\textit{Period.\ Math.\ Hungar.} 37 (1998), 153--205. 


\bibitem[vK01]{vK01}
\textsc{von~Koch,~H.},
Sur la distribution des nombres premiers.
\textit{Acta Math.} 24 (1901), 159--182.


\bibitem[Wad97]{Wad97}
\textsc{Waddington,~S.},
Zeta functions and asymptotic formulae for preperiodic orbits of hyperbolic rational maps.
\textit{Math.\ Nachr.} 186 (1997), 259--284.


%\bibitem[Wa76]{Wa76}
%\textsc{Walters,~P.},
%A variational principle for the pressure of continuous transformations.
%\textit{Amer.\ J.\ Math.} 17 (1976), 937--971. 

\bibitem[Wal82]{Wal82}
\textsc{Walters,~P.},
\textit{An introduction to ergodic theory},
Springer, New York, 1982.


%\bibitem[Wi07]{Wi07}
%\textsc{Wildrick,~K.M.},
%Quasisymmetric parameterizations of two-dimensional metric spaces. PhD thesis, University of Michigan, 2007.


\bibitem[Wi16]{Wi16} 
\textsc{Winter,~D.},
Exponential mixing of frame flows for convex cocompact hyperbolic manifolds. 
Preprint, (arXiv:1612.00909v1), 2016.

\bibitem[Yi15]{Yi15}
\textsc{Yin,~Q.},
Thurston maps and asymptotic upper curvature.
\textit{Geom.\ Dedicata} 176 (2015), 271--293.


%\bibitem[Yu99]{Yu99}
%\textsc{Yuri,~M.},
%Thermodynamic formalism for certain nonhyperbolic maps.
%\textit{Ergodic Theory Dynam.\ Systems} 19 (1999), 1365--1378.


%\bibitem[Yu00]{Yu00}
%\textsc{Yuri,~M.},
%Weak Gibbs measures for certain non-hyperbolic systems.
%\textit{Ergodic Theory Dynam.\ Systems} 20 (2000), 1495--1518.


%\bibitem[Yu03]{Yu03}
%\textsc{Yuri,~M.},
%Thermodynamic formalism for countable to one Markov systems.
%\textit{Trans.\ Amer.\ Math.\ Soc.} 335 (2003), 2949--2971.




%\bibitem[Zi96]{Zi96}
%\textsc{Zinsmeister,~M.},
%\textit{Formalisme thermodynamique et syst\`{e}mes dynamiques holomorphes}, Panoramas et Synth\`{e}ses 4,
%Soci\'et\'e Math\'ematique de France, Paris, 1996.





\end{thebibliography}
\end{document}